\documentclass[11pt]{article}
\usepackage[toc,page]{appendix}

\usepackage{amsmath,amsthm,amssymb,amsfonts,mathtools,mathrsfs,enumerate}
\usepackage[usenames,dvipsnames]{color}
\usepackage{graphicx,psfrag,epsfig,ifpdf}
\usepackage{mathabx}
\usepackage{pstricks}
\usepackage{pst-bezier}
\usepackage[all]{xy}
\usepackage{pst-all}

\newcounter{parag}[subsection]

\newcounter{paraga}[subsection]
\renewcommand{\theparaga}{{\bf\arabic{paraga}.}}
\newcommand{\paraga}{\medskip \addtocounter{paraga}{1} 
\noindent{\theparaga\ } }

\newcounter{pparag}

\newtheorem{thm}{Theorem}[section]

\newtheorem{lemma}{Lemma}[section]
\newtheorem{prop}{Proposition}[section]

\newtheorem{cor}{Corollary}[section]
\newtheorem{Def}{Definition}

\newtheorem{notation}{Notation}

\def\text#1{\,\hbox{#1}\;}
\def\dst{\displaystyle}

\def\al{\alpha}
\def\be{\beta}
\def\ga{\gamma}
\def\Ga{{\Gamma}}
\def\de{\delta}
\def\De{\Delta}
\def\eps{{\varepsilon}}
\def\ka{\kappa}
\def\la{\lambda}
\def\La{\Lambda}
\def\om{\omega}
\def\Om{\Omega}
\def\vpi{\varpi}

\def\sig{{\sigma}}
\def\Sig{{\Sigma}}

\def\th{{\theta}}
\def\Th{\Theta}

\def\ph{\varphi}
\def\ze{{\zeta}}

\def\vpi{{\varpi}}

\def\bGa{{\boldsymbol \Gamma}}

\def\beps{{\boldsymbol \varepsilon}}

\def\bla{{\boldsymbol \lambda}}

\def\bom{{\boldsymbol \omega}}

\def\bth{{\boldsymbol \theta}}

\def\jA{{\mathscr A}}
\def\jB{{\mathscr B}}
\def\jC{{\mathscr C}}
\def\jD{{\mathscr D}}
\def\jE{{\mathscr E}}
\def\jF{{\mathscr F}}

\def\jH{{\mathscr H}}
\def\jI{{\mathscr I}}
\def\jJ{{\mathscr J}}

\def\jL{{\mathscr L}}
\def\jM{{\mathscr M}}
\def\jN{{\mathscr N}}
\def\jO{{\mathscr O}}
\def\jP{{\mathscr P}}

\def\jR{{\mathscr R}}
\def\jS{{\mathscr S}}
\def\jT{{\mathscr T}}
\def\jU{{\mathscr U}}
\def\jV{{\mathscr V}}
\def\jW{{\mathscr W}}
\def\jX{{\mathscr X}}

\def\A{{\mathbb A}}

\def\C{{\mathbb C}}

\def\N{{\mathbb N}}

\def\R{{\mathbb R}}
\def\T{{\mathbb T}}
\def\Z{{\mathbb Z}}

\def\cY{{\mathcal Y}}

\def\rk{{\rm rank\,}}
\def\sgn{{\rm sgn\,}}

\def\Im{{\rm Im\,}}

\def\Sup{\mathop{\rm Sup\,}\limits}
\def\Inf{\mathop{\rm Inf\,}\limits}
\def\Max{\mathop{\rm Max\,}\limits}
\def\Min{\mathop{\rm Min\,}\limits}

\def\max{{\rm Max\,}}

\def\min{{\rm Min\,}}
\def\dist{{\rm dist\,}}
\def\Ln{{\rm Ln\,}}
\def\Log{{\rm Log\,}}

\def\sh{{\rm sh\,}}

\def\Lip{{\rm Lip\,}}

\def\setm{\setminus}

\def\ov{\overline}
\def\til{\widetilde}
\def\ha{\widehat}
\def\d{\partial}
\def\inv{^{-1}}
\def\demi{{\frac{1}{2}}}
\def\pdemi{{\tfrac{1}{2}}}
\def\ppdemi{{\tfrac{1}{2}}}
\def\abs#1{\left\vert#1\right\vert}
\def\norm#1{\Vert#1\Vert}
\def\setm{\setminus}
\def\gabs#1{\left\vert#1\right\vert}

\def\cT{{\mathcal T}}
\def\cM{{\mathcal M}}
\def\bS{{\mathcal S}}
\def\bD{{\mathcal D}}
\def\cO{{\mathcal O}}
\def\cR{{\mathcal R}}

\def\cR{{\mathcal R}}
\def\cT{{\mathcal T}}
\def\cM{{\mathcal M}}
\def\bS{{\mathcal S}}
\def\bD{{\mathcal D}}

\def\sA{{\mathsf A}}
\def\sD{{\mathsf D}}
\def\sN{{\mathsf N}}

\def\bt{{\mathsf t}}
\def\br{{\mathsf r}} 

\def\bN{{\mathsf N}}
\def\bC{{\mathsf C}}

\def\beq{\begin{equation}}
\def\eeq{\end{equation}}
\def\spa{\phantom{\frac{\int}{\int}}}

\def\bu{{\bullet}}
\def\e{{\bf e}}

\def\sA{{\mathsf A}}

\def\sD{{\mathsf D}}
\def\sN{{\mathsf N}}

\def\bez{B^*_\Z}
\def\sH{{\mathsf H}}
\def\sh{{\mathsf h}}
\def\sf{{\mathsf f}}
\def\sr{{\mathsf r}} 

\def\sN{{\mathsf N}}
\def\sR{{\mathsf R}}

\def\Frac{\frac}

\def\Id{{\rm Id}}

\def\beq{\begin{equation}}
\def\eeq{\end{equation}}
\def\spa{\phantom{\frac{\int}{\int}}}

\def\bu{{\bullet}}

\def\jC{{\mathscr C}} 

\def\H{{\bf H}} 
\def\bA{{\bf A}}

\def\trans{\pitchfork}
\def\sM{{\mathsf M}} 

\def\codim{{\rm codim\,}}
\def\ev{{\bf ev\,}}

\def\bQ{{\bf Q}}
\def\bR{{\bf R}}
\def\bL{{\bf L}}
\def\bM{{\bf M}}
\def\taum{{\tau_m}}

\def\H{{\bf H}} 
\def\bA{{\bf A}} 
\def\Av{{\rm A\!v}}

\def\bchi{\boldsymbol{\chi}}

\def\bh{{\bf h}}
\def\bg{{\bf g}}
\def\bV{{\bf V}}
\def\bW{{\bf W}}
\def\bR{{\bf R}}
\def\ovth{{\ov\th}}
\def\ovr{{\ov r}}

\def\cD{{\mathcal D}}
\def\cA{{\mathcal A}}
\def\cC{{\mathcal C}}
 \usepackage{amsbsy} 

\def\Av{{\rm Av}}

\def\j{{\rm j}}
\def\pip{{\pi'}}
\def\mabs#1{\big\vert#1\big\vert}

\def\CS{{\bf (S)}}

\def\CSu{{\bf ($\bf S_1$)}}
\def\CSd{{\bf ($\bf S_2$)}}

\def\bY{\boldsymbol{\mathsf Y}}

\def\fomu{{\rm(FS1)}}
\def\fomd{{\rm(FS2)}}

\def\pom{{\rm(PS)}}
\def\pomu{{\rm(PS1)}}

\def\Ann{{\rm A}}

\def\bA{{\bf A}}
\def\bW{{\bf W}}
\def\br{{\bf r}}
\def\bd{{\bf d}}
\def\bth{{\boldsymbol \th}}
\def\bph{{\boldsymbol \ph}}
\def\dsp{\displaystyle}
\def\fomu{{\rm(FS1)}}
\def\fomd{{\rm(FS2)}}

\def\pom{{\rm(PS)}}
\def\pomu{{\rm(PS1)}}

\def\glu{{\rm(G)}}
\def\Tess{{\rm Tess}}

\begin{document}

\title{\LARGE \vskip-8mm Chains of compact cylinders for cusp-generic\\ 
nearly integrable convex systems on $\A^3$}

\author{Jean-Pierre Marco
\thanks{Universit\'e Paris 6, 
4 Place Jussieu, 75005 Paris cedex 05.
E-mail: jean-pierre.marco@imj-prg.fr
}}
\date{}

\maketitle

\vspace{-6mm}
\begin{abstract} This paper is the first of a series of three dedicated to a proof of
the Arnold diffusion conjecture for perturbations of {convex} integrable Hamiltonian 
systems on $\A^3=\T^3\times\R^3$.  

We consider  systems of the form $H(\th,r)=h(r)+f(\th,r)$, where $h$ is a 
$C^\ka$ strictly convex and superlinear  function on $\R^3$ and $f\in C^\ka(\A^3)$, $\ka\geq2$. 
Given $\e>\Min h$ and a finite family
of arbitrary open sets $O_i$ in $\R^3$ intersecting $h\inv(\e)$, a {\em diffusion orbit}
associated with these data is an orbit of $H$ which intersects each open set 
$\ha O_i=\T^3\times O_i\subset\A^3$. 

The first main result of this paper (Theorem I) states the existence (under cusp-generic conditions 
on $f$ in Mather's terminology) 
of ``chains of compact and normally hyperbolic invariant $3$-dimensional cylinders'' intersecting 
each $\ha O_i$. Diffusion orbits drifting along these chains are then proved to exist 
in \cite{GM,M}. The second main result (Theorem II) consists in 
a precise description of the hyperbolic features of classical systems (sum of  a quadratic kinetic 
energy and a potential) on $\A^2=\T^2\times\R^2$, which is a crucial step to prove Theorem I.

The cylinders are either diffeomorphic to $\T^2\times[0,1]$ or to the product of $\T$ with a 
sphere with three holes. A chain at energy $\e$ for $H$ is a finite family of
such cylinders, which are contained in $H\inv(\e)$ and admit heteroclinic connections between 
them.  The cylinders satisfy additional dynamical properties which ensure the existence,
up to an arbitrarily small perturbation, of orbits of $H$ drifting along them (and 
so along the chain).  

The content of Theorem I is the following.  
Assuming $\ka$ large enough, 
we prove that for every $\bf f$ in an {\em open dense} subset of the
unit sphere in $C^\ka(\A^3)$,  there is a lower semicontinuous threshold
$\beps({\bf f})>0$ for which, when $\eps\in\,]0,\beps(f)[$, the system $H=h+\eps \bf f$ 
admits a chain at energy $\e$ 
which intersects each $\ha O_i$. 

To prove this result we approximate the system $H$ by local normal forms near resonances,
and we  distinguish between  ``strong double resonance'' points 
and ``simple resonance'' curves.
In both cases we first detect normally hyperbolic objects invariant under the normal forms
obtained by averaging with respect to two fast angles (in the simple resonance case)
or a single fast angle (in the double resonance case).

Along simple resonance curves, the approximate systems are  one-parame\-t\-er families of
pendulums on $\A$, while the main role at strong double resonance points is played by classical systems
on $\A^2$, whose study in the generic case is the content
of Theorem II. Given a generic classical system on $\A^2$,
for any integer homology  class~$c$ we first prove the existence of an associated
``chain of heteroclinically connected $2$-annuli realizing~$c$,'' which is asymptotic both to the
critical energy (maximal value of the potential) and to the infinite energy. 
We then prove the existence of a singular annulus, and we finally prove that for any $c$,
the associated chain admits heteroclinic connections with that singular annulus.

Along simple resonance curves, the normalized cylinders are the product of the one-parameter
family of fixed points of the pendulums with the torus $\T^2$ of fast angles, while near the double
resonance points the cylinders are the product of the annuli (or the singular annuli)
in the classical system with the circle $\T$ of the fast angle.  We get the corresponding invariant
objects for $H$ by normally hyperbolic persistence and KAM type results to deal with the 
invariance of the boundaries ($2$-dimensional tori). We finally prove the existence of 
a rich homoclinic and heteroclinic structure for these objects, which gives rise to the chains.

\end{abstract}

\newpage


\begin{center}
{\bf\LARGE  Introduction and main results}
\end{center}

\vskip.5cm
Given $n\geq1$, we denote by $\A^n=\T^n\times\R^n$ the cotangent bundle of the torus $\T^n$,
endowed with its natural angle-action coordinates $(\th,r)$ and its usual symplectic structure. 
This paper is the first of a series of three dedicated to a proof of the Arnold diffusion conjecture 
for nearly integrable Hamiltonian systems on $\A^3$, in the ``convex setting'' which was introduced
by Mather. Two other approaches of the same problem are developped in 
\cite{C,KZ}. 

In this paper we focus on the geometric part of our construction, that is, the existence
of a ``hyperbolic skeleton'' for diffusion, formed by chains of compact invariant normally hyperbolic 
$3$-dimensional cylinders, whose existence is the content of Theorem~I. The proof of the existence 
of diffusion  orbits drifting along chains is the object of \cite{GM,M}. The proof of Theorem~I necessitates
in particular a detailed analysis of the hyperbolic properties of generic classical systems (sum of a quadratic
energy and a potential function) on $\A^2$, which constitutes the second main result of the present 
paper (Theorem~II).


\section{The general setting}

\paraga In \cite{A64}, Arnold introduced the first example of an ``unstable'' family of Hamiltonian systems 
on $\A^3$, namely:
\beq\label{eq:Arnold1}
H_{\eps}(\th,r)=r_1+\pdemi (r_2^2+r_3^2)+\eps(\cos \th_3-1)+\mu(\eps)(\cos \th_3-1) g(\th),
\eeq
where $g$ is a suitably chosen trigonometric polynomial, $\eps>0$ is small enough and $\mu(\eps)<\!\!<\eps$.
The main result of Arnold is the existence of $\eps_0>0$ such that for $0<\eps<\eps_0$, the system
$H_\eps$ admits an ``unstable solution'' $\ga_\eps(t)=\big(\th(t),r(t)\big)$ such that 
\beq\label{eq:unstablesol}
r_2(0)<0,\qquad r_2(T_\eps)>1,
\eeq
for some (large) $T_\eps$. 
In view of this result and the associated constructions, Arnold conjectured (see \cite{A94}) that  for ``typical'' 
systems of the form $H_\eps(\th,r)=h(r)+\eps f(\th,r,\eps)$ on $\A^n$, $n\geq 3$,  the projection in action of
some orbits should visit any element of a prescribed collection of arbitrary open sets intersecting a connected 
component of a level set of $h$. Orbits experiencing this behavior are said to be {\em diffusion orbits}.

This conjecture motivated a number of works, first in a sligthly different context. Namely, setting $\eps=1$ 
in (\ref{eq:Arnold1}) yields a simpler class of systems for which the unperturbed part no longer depends
on the actions only,  but still remains completely integrable (with nondegenerate hyperbolicity).
It became a challenging question to prove the existence of unstable solutions (\ref{eq:unstablesol})
for the slightly more general class of systems
\beq\label{eq:Arnold2}
G_{\mu}(\th,r)=r_1+\pdemi (r_2^2+r_3^2)+(\cos \th_3-1)+\mu g(\th,r),
\eeq
where $g$ belongs to a residual subset of a small enough ball in some appropriate function space 
(finitely or infinitely differentiable, Gevrey, analytic).
This setting (with its natural generalizations) is now called the {\em a priori unstable} case of Arnold diffusion.
In \cite{GM} we set out a geometric framework to deal with such systems, see
\cite{B08,BKZ13,BT99,CY09,DLS00,DLS06a,DLS06b,FM03,GT,GL06,GR13,GLS,M02,T04}
amongst others for different approaches.

\paraga In this paper we focus on the so-called {\em a priori} stable case, that is, we consider perturbations
of integrable systems on $\A^3$ which depend only on the actions.
Our goal is to analyze the hyperbolic structure of such systems (under nondegeneracy conditions) 
and prove the existence of ``many'' 
$3$ and $4$ dimensional hyperbolic invariant submanifolds with a rich homoclinic structure, which in addition
form well-defined ``chains'' (in the spirit of the initial approach of Arnold in \cite{A64}). 
This geometric framework will in turn enable us in \cite{M} to use in the {\em a priori} stable setting
the ``{\em a priori} unstable techniques''  introduced in \cite{GM}, and prove the existence of orbits drifting along 
such chains.

\paraga Let us briefly describe our setting, beginning with the functional spaces.
For $2\leq \ka <+\infty$, we equip  $C^\ka(\A^3):=C^\ka(\A^3,\R)$  with the uniform seminorm 
$$
\norm{f}_\ka=\sum_{k\in\N^6,\ 0\leq \abs{k}\leq \ka}\norm{\d^kf}_{C^0(\A^3)}\leq+\infty
$$
and we set
$
C_b^\ka(\A^3)=\big\{f\in C^\ka(\A^3)\mid \norm{f}_\ka<+\infty\big\},
$
so that $\big(C_b^\ka(\A^3),\norm{\ }_\ka\big)$ is a Banach  algebra.
We consider
systems on $\A^3$, of the form
\beq\label{eq:hampert1}
H(\th,r)=h(r)+ f(\th,r),
\eeq
where  $h:\R^3\to\R$ is $C^\ka$ and the perturbation  $f\in C_b^\ka(\A^3)$ is small 
enough. 

\paraga A first restriction imposed by Mather in \cite{Mat04} in order to use variational methods 
is that the unperturbed part $h$ is strictly convex with superlinear growth at infinity (that is, 
$\lim_{\norm{r}\to+\infty} h(r)/\norm{r}\to+\infty$). Such
Hamiltonians are referred to as {\em Tonelli Hamiltonians}. We will also limit ourselves
to Tonelli Hamiltonians here, since convexity will be necessary in our constructions
in the neighborhood of double resonance points.

\paraga  A usual way to deal with the smallness
condition on $f$, as already illustrated by (\ref{eq:Arnold1}), is to prove the occurrence 
of diffusion orbits  for all systems in ``segments'' originating at $h$, of the form
\beq\label{eq:hampert2}
\big\{H_\eps(\th,r)=h(r)+ \eps f(\th,r)\mid \eps\in\,]0,\eps_0]\big\}
\eeq
where $f$ is a  fixed function. 
This makes it natural that the smallness threshold $\eps_0$ may explicitely depend on $f$,
however this would not be appropriate in our setting since it seems difficult to prove the existence of diffusion 
over {\em whole} segments such as (\ref{eq:hampert2}).
To take this observation into account, following Mather,
we use  a more global framework and introduce ``anisotropic balls''
in which the diffusion phenomenon can be expected to occur generically.
Let $\bS^\ka$ be the unit sphere in $C_b^\ka(\A^3)$. 
Given
$
\beps_0:\bS^\ka\to [0,+\infty[
$
(a ``threshold function''), we define the associated $\beps_0$-ball:
\beq\label{eq:cuspball}
\jB^\ka(\beps_0):=
\big\{\eps {\bf f} \mid {\bf f}\in\bS^\ka,\ \eps\in\,]0,\beps_0({\bf f})[\big\}.
\eeq

\paraga This yields the following version of the diffusion conjecture\footnote{Mather's formulation is indeed
still more precise and involved}, to be compared with \cite{A94}.

\vskip2mm

{\bf Diffusion conjecture in the convex setting.}
{\em 
Consider a $C^\ka$ integrable Tonelli Hamiltonian $h$ on $\A^3$.
Fix an energy~$\e$ larger than $\Min h$ and a finite family of arbitrary open sets 
$O_1,\ldots,O_m$ which intersect $h\inv(\e)$.  
Then for $\ka\geq {\ka_0}$ large enough, there exists a lower semicontinuous  function 
\beq
\beps_0:\bS^\ka\to[0,+\infty[
\eeq
with positive values on an open dense subset of $\bS^\ka$ such that for $f$
in an open and dense subset of $\jB^\ka(\beps_0)$, the system
\begin{equation}
H(\th,r)=h(r)+ f(\th,r)
\end{equation}
admits an orbit which intersects each $\T^3\times O_i$.
}

\paraga The zeros of $\beps_0$ correspond to directions along which diffusion cannot occur.
Simple examples show that such directions exist in general: for instance if $h(r)=\pdemi(r_1^2+r_2^2+r_3^2)$,
each system $H_\eps=h+\eps f$ with $f(\th)=\sin \th_i$ ($i=1,2,3$) is completely integrable and does not
admit diffusion orbits for $\eps$ small enough. Note also that since $\beps_0$ is assumed to be
lower semicontinuous, the associated ball is open in $C^\ka_b(\A^3)$. 
In view of the shape of $\jB^\ka(\beps_0)$, a residual subset in such a ball 
is said to be {\em cusp-residual} and a property which holds on a cusp-residual subset is said to 
be  {\em cusp-generic}.

\begin{figure}[h]
\begin{center}
\begin{pspicture}(0cm,1.8cm)
\psset{xunit=.7,yunit=.7,runit=.7}
\pscircle[linewidth=0.1mm](0,0){3}
\psbezier*[linecolor=lightgray](0,0)(1,-.1)(2,.3)(0,0)
\psbezier*[linecolor=lightgray](0,0)(4,.5)(-1,5)(0,0)
\psbezier*[linecolor=lightgray](0,0)(5,0)(-3,-5)(0,0)
\psbezier*[linecolor=lightgray](0,0)(-.5,3)(-3,3)(0,0)
\psbezier*[linecolor=lightgray](0,0)(-4,4)(-3,-5)(0,0)
\rput(2.5,0){$\bS^\ka$}
\rput(.7,1){$\jB^\ka(\beps_0)$}
\end{pspicture}
\vskip20mm
\caption{A generalized ball}\label{Fig:genball}
\end{center}
\end{figure}
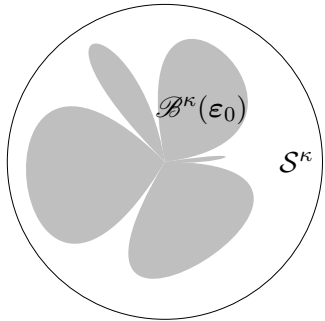
\vskip-5mm

\paraga Our purpose in this paper is to set out a list of nondegeneracy conditions on the perturbation
$f$ which yield the existence of ``a small amount of hyperbolicity'' in the system $H=h+f$, from which 
we can deduce the existence of chains of normally hyperbolic objects which intersect the collection
of open sets $\T^3\times O_i$. We then prove that these conditions are satisfied for {\em any}\footnote{an additional
perturbation will be necessary in order to get the diffusion orbits drifting along the chains, which explains 
the restriction to residual subsets of generalized balls in the previous conjecture}  $f$ in some generalized
ball $\jB^\ka(\beps_0)$, where the threshold function satisfies the conditions of the previous conjecture.
This is the content of Theorem~I stated in Section~\ref{sec:TheoremI} of this Introduction. 
As in Nekhoroshev's approach of exponential stability, our analysis necessitates
to discriminate between ``strong double resonances'' and ``almost simple resonances'' of the unperturbed
Hamiltonian $h$. While the analysis along simple resonances is quite straightforward,
the neighborhood of strong double resonances needs a precise
description of the hyperbolic behavior of generic classical systems  on the annulus~$\A^2$. This is the second main result
of this paper (Theorem~II),  stated  in Section~\ref{sec:TheoremII}
of this Introduction.

%


\setcounter{paraga}{0}
\section{Cylinders, chains and Theorem I}\label{sec:TheoremI}
\paraga Before stating Theorem I we briefly describe the various objects involved in our construction.
More precise definitions are given in Section~\ref{sec:normhyp} of Part I.
Let $X$ be a $C^1$ complete vector field on a smooth manifold $M$, with flow $\Phi$.  Let $p$ be an 
integer $\geq 1$.
\vskip1.5mm
$\bu$ We say that $\jC\subset M$ is a {\em $C^p$ invariant cylinder with boundary} for 
$X$ if $\jC$ is a submanifold  of 
$M$,  $C^p$--diffeomorphic to $\T^2\times [0,1]$, which is invariant under the flow of $X$: 
$\Phi^t(\jC)=\jC$ for all $t\in\R$. 
\vskip1.5mm
$\bu$ We denote by $\bY$ any realization of the two-sphere $S^2$ 
minus three  open discs with nonintersecting closures, so that $\d \bY$ is the union of three circles.
We say that $\jC_\bu\subset M$ is an
{\em invariant singular cylinder} for $X$ if $\jC_\bu$ is a $C^1$ submanifold  of  $M$,  
 diffeomorphic to $\T\times \bY$ and invariant under $\Phi$. The boundary
of a singular cylinder is the disjoint union of three tori.
\vskip1.5mm

Throughout this paper
we will consider vector fields generated by Hamiltonian functions $H\in C^\ka(\A^3)$, $\ka\geq 2$.
The cylinders or singular cylinders will be contained in regular levels of~$H$.

\begin{figure}[h]
\begin{center}
\begin{pspicture}(0cm,2.5cm)
\rput(-2,1.3){
\psset{xunit=.55cm,yunit=.27cm}
\psellipse[linewidth=.3mm](0,0)(1,3)
\pscurve(.1,1)(-.05,0)(.1,-1)
\pscurve(.1,.5)(.15,0)(.1,-.5)
\psellipse[linewidth=.3mm](-6,0)(1,3)
\psframe[fillstyle=solid,fillcolor=white,linecolor=white](-6,-3.3)(-4.8,3.3)
\pscurve(-6,3)(-4,2.8)(-2,2.8)(0,3)
\pscurve(-6,-3)(-4,-2.8)(-2,-2.8)(0,-3)
}
\psset{xunit=.3cm,yunit=.23cm}
\rput(1,1.3){
\rput(16,4){
\psellipse[linewidth=.3mm](0,0)(1,3)
\pscurve(.1,1)(-.05,0)(.1,-1)
\pscurve(.1,.5)(.15,0)(.1,-.5)
}
\rput(6,4){
\psellipse[linewidth=.3mm](-6,0)(1,3)
\psframe[fillstyle=solid,fillcolor=white,linecolor=white](-6,-3.3)(-4.8,3.3)
}
\rput{90}(7.5,6.7){
\psellipse[linewidth=.3mm](0,0)(1,3)
\pscurve(.1,1)(-.05,0)(.1,-1)
\pscurve(.1,.5)(.15,0)(.1,-.5)
}
\pscurve(0,1.02)(4,.2)(8,-.1)(12,.2)(16,1.02)
\pscurve(0,7)(4,6.4)(5.4,7.1)
\pscurve(9.5,7.3)(11.5,6.2)(16.1,6.99)
}
\end{pspicture}
\caption{$3$-dimensional cylinder and singular cylinder}\label{Fig:cylsingcyl}
\end{center}
\end{figure}
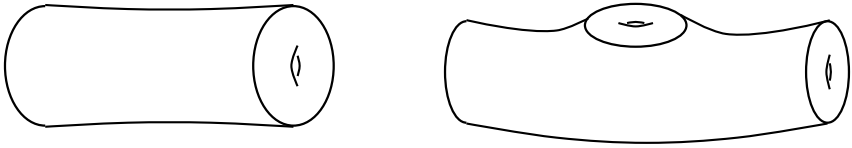

\vspace{-.8cm}
\paraga  The notion of normal hyperbolicity for submanifolds with boundary requires some care.
We introduce in Section~\ref{sec:normhyp} of Part I and Appendix~\ref{app:normhyp}
a simple definition for the normal hyperbolicity of cylinders and singular cylinders,
which coincides with the usual one (see \cite{C04,C08}) but is better 
adapted to our subsequent constructions. In particular,  normally hyperbolic cylinders and singular cylinders admit  
well-defined $4$-dimensional  stable and unstable manifolds, contained in their energy level. 

\paraga In addition to the normal 
hyperbolicity, we will require our cylinders  to admit global  Poincar\'e sections, 
diffeomorphic to $\T\times[0,1]$, whose associated  Poincar\'e maps satisfy a twist condition (with a similar
property for singular cylinders). This enables us to define a particular class of $2$-dimensional invariant tori
contained in these cylinders, which we call {\em essential tori}. Analogous (but slightly more involved)
notions will be defined for singular cylinders. Moreover, we will require specific homoclinic conditions 
to be satisfied by the cylinders, which yields the notion of {\em admissible cylinders}.

\paraga Finally, we will introduce various heteroclinic conditions which will have to be satisfied by pairs of cylinders.
This makes it possible to  define {\em admissible chains},  that is,  
finite families $(\jC_k)_{1\leq k\leq k_*}$ of admissible
cylinders or singular cylinders, in which two consecutive elements satisfy these
heteroclinic conditions.

\paraga The main result of Part~I is the following. 
\vskip2mm
\noindent
{\bf Theorem I. (Cusp-generic existence of admissible chains.)} 
{\it Consider a $C^\ka$ integrable Tonelli Hamiltonian $h$ on $\A^3$.
Fix $\e>\Min h$ and a finite family of open sets $O_1,\ldots,O_m$
which intersect $h\inv(\e)$. Fix $\de>0$.  
Then for $\ka\geq \ka_0$ large enough, there exists a lower semicontinous 
function 
$$
\beps_0:\jS^\ka\to\R^+
$$ 
with positive values on an open dense subset of $\jS^\ka$ such that for $f\in\jB^\ka(\beps_0)$
the system
\begin{equation}\label{eq:hamstatement}
H(\th,r)=h(r)+ f(\th,r)
\end{equation}
admits an admissible chain of cylinders and singular cylinders,
such that each open set $\T^3\times O_k$ contains the $\de$-neighborhood
in $\A^3$ of some essential torus of the chain.}

\vskip3mm

\paraga One can be more precise and localize the previous chain.  Since $h$ is a Tonelli Hamiltonian,
one readily checks that $\om:=\nabla h$ is a diffeomorphism from $\R^3$ onto $\R^3$, and that
the level set $h\inv(\e)$ is diffeomorphic to $S^2$. Given an indivisible vector $k\in\Z^3\setm\{0\}$,
set
$$
\Ga_k=\om\inv(k^\bot)\cap h\inv(\e),
$$
where $k^\bot$ is the plane orthogonal to $k$ for the Euclidean structure of $\R^3$. Then one checks
that $\Ga_k$ is diffeomorphic to a circle, and that if $k\neq k'$ then $\Ga_k$ and $\Ga_{k'}$ intersect
at exactly two points (such intersection points are said to be {\em double resonance points}). 
By projective density, it is possible to choose a family $k_1,\ldots,k_{m-1}$ of indivisible and pairwise independent
vectors of $\Z^3$ such that 
\vskip1.5mm
$\bu$ $\Ga_{k_1}$ intersects $O_1$ and $\Ga_{k_{m-1}}$ intersects $O_{m}$;
\vskip1.5mm
$\bu$ for $2\leq i\leq m-1$, $\Ga_{k_{i-1}}\cap \Ga_{k_{i}}$ contains a point $a_i\in O_i$.
\vskip1.5mm
\noindent Fix $a_1\in \Ga_{k_1}\cap O_1$ and $a_m\in \Ga_{k_{m-1}}\cap O_m$.
Fix an arbitrary orientation on each circle $\Ga_{k_i}$ and let $[a_i,a_{i+1}]_{\Ga_i}$ be the segment of $\Ga_i$
bounded by $a_i$ and $a_{i+1}$ according to this orientation. Set finally
$$
\bGa=\bigcup_{1\leq i\leq m-1}[a_i,a_{i+1}]_{\Ga_i}.
$$

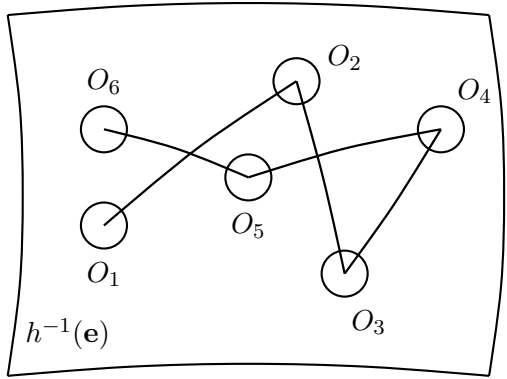
\begin{figure}[h]
\begin{center}
\begin{pspicture}(0cm,2.4cm)
\rput(1.5,0){
\psset{xunit=.8,yunit=.8,runit=.8}
\rput(-5,-2.3){$h\inv({\bf e})$}
\rput(0,6){
\pscurve(-6,-3)(-4,-2.85)(-2,-2.8)(0,-2.85)(2,-3)}
\pscurve(-6,-3)(-4,-2.85)(-2,-2.8)(0,-2.85)(2,-3)
\pscurve(2,3)(2.2,1.5)(2.25,0)(2.2,-1.5)(2,-3)
\rput(-8,0){\pscurve(2,3)(2.2,1.5)(2.25,0)(2.2,-1.5)(2,-3)}
\rput(-2,-.5){
\psset{xunit=.8,yunit=.8,runit=.8}
\pscurve[linewidth=.3mm](-3,0)(-1,1.65)(1,3)
\pscircle(-3,0){.5}
\rput(-3,-1){$O_1$}
\pscurve[linewidth=.3mm](1,3)(1.55,1)(2,-1)
\pscircle(1,3){.5}
\rput(2,3.5){$O_2$}
\pscurve[linewidth=.3mm](2,-1)(3,.4)(4,2)
\pscircle(2,-1){.5}
\rput(2.5,-2){$O_3$}
\pscurve[linewidth=.3mm](4,2)(2,1.6)(0,1)
\pscircle(4,2){.5}
\rput(4.7,2.8){$O_4$}
\pscurve[linewidth=.3mm](0,1)(-1.5,1.6)(-3,2)
\pscircle(0,1){.5}
\rput(0,0){$O_5$}
\pscircle(-3,2){.5}
\rput(-3,3){$O_6$}
}
}
\end{pspicture}
\vskip23mm
\caption{A ``broken line'' $\Ga$ of resonance arcs}\label{Fig:brokenline }
\end{center}
\end{figure}

We will  prove that one can choose $\beps_0$ in Theorem I so that for $f\in \jB(\beps_0)$ the projection to $\R^3$
of the admissible chain is located in a $\rho(f)$-tubular neighborhood of $\bGa$, whose radius $\rho(f)$ tends to
$0$ when $f\to 0$ in $C^\ka(\A^3)$.


\setcounter{paraga}{0}
\section{Generic hyperbolic properties of classical systems on $\A^2$}\label{sec:TheoremII}

A {\em classical system on $\A^2$} is a Hamiltonian of the form 
\begin{equation}\label{eq:classham}
 C(\th,r)=\pdemi T(r)+ U(\th),\qquad (\th,r)\in\A^2
\end{equation}
 where $T$ is a positive definite quadratic form of $\R^2$ and $U$ a $C^\ka$ potential function on $\T^2$,
where $\ka\geq 2$. In the sequel we will require the potential $U$ to admit a single maximum at some $x^0$, which is  nondegenerate
in the sense that the Hessian of $U$ is negative definite. Consequently, the lift of $x^0$ to the zero section of $\A^2$ is a 
hyperbolic fixed point which we denote by $O$. We set $\ov e=\Max U$ and we say that $\ov e$ is the {\em critical energy} for $C$.
Such systems appear (generically), {\em up to a non symplectic rescaling $r-r^0=\sqrt\eps \ov r$} in the 
neighborhood of a double resonance point $r^0$ of the initial system (\ref{eq:hamstatement}), as
the main part of normal forms (we did not change the notation of the variables here). 
The energy of $C$ is not directly related to the initial energy $\e$
of (\ref{eq:hamstatement}), but the difference $e-\ov e$ has rather to be though of as the distance to the double resonance
point (in projection to the action space) rescaled by the factor $\sqrt\eps$.
The aim of Part II is to depict some hyperbolic properties of $C$, 
when {\em $T$ is fixed} and $U$ belongs to a residual  subset of $C^\ka(\T^2)$, $\ka$ large enough.

\paraga  The following definition will be used throughout the paper.

\begin{Def}\label{def:ann}
Let $c\in H_1(\T^2,\Z)$. Let $I\subset\R$ be an interval. 
An {\em annulus for $X^C$  realizing $c$ and defined over $I$} is a $2$-dimensional submanifold $\sA$, contained 
in $C\inv(I)\subset\A^2$, such that
for each $e\in I$,  $\sA\cap C\inv(e)$ is the orbit of a periodic solution $\ga_e$ of $X^C$, 
which is hyperbolic in $C\inv(e)$ and such that the projection 
$\pi\circ\ga_e$ on $\T^2$ realizes $c$. 
We also require the period of the orbits to increase
with the energy and that for each $e\in I$, the periodic orbit $\ga_e$ admits a homoclinic
orbit along which $W^\pm(\ga_e)$ intersect transversely in $C\inv(e)$. Finally, we require the
existence of a finite partition $I=I_1\cup\cdots\cup I_n$ in consecutive intervals such that
the previous homoclinic orbit varies continuously for $e\in I_i$, $1\leq i\leq n$.
\end{Def}

When $I$ is compact, the annulus $\sA$ is clearly normally hyperbolic in the usual sense (the boundary
causes no trouble is this simple setting). The stable and unstable manifolds of $\sA$ are  well-defined, 
as the unions of those of the periodic solutions $\ga_e$. Moreover, $\sA$ can be continued to an annulus
defined over a slightly larger interval $I'\supset I$.

In the aforementioned normalization process of (\ref{eq:hamstatement}) near a double resonance point $r^0$,
the interval over which an annulus is defined  will be crucial for its localization
with respect to~$r^0$.

\paraga  
Note that, due to the reversibility of $C$, the solutions of the vector field $X^C$ occur in ``opposite pairs'',
whose time parametrizations are exchanged by the symmetry $t\mapsto-t$.

\begin{Def}\label{def:singann}
Let $c\in H_1(\T^2,\Z)\setm\{0\}$. A {\em singular annulus for $X^C$  realizing $\pm c$}
 is a $C^1$ compact invariant submanifold $\bY$ of $\A^2$, diffeomorphic to the sphere 
$S^2$ minus three disjoint open discs with disjoint closures 
(so that $\d \bY$ is the disjoint union of three circles),  such that there 
exist constants $e_*<\ov e<e^*$  which satisfy:
\vskip1.5mm
$\bu$ $\bY\cap\, C\inv(\ov e)$  is the union of the hyperbolic  fixed point  $O$ and a pair of opposite 
homoclinic orbits,
\vskip1.5mm
$\bu$  $\bY\cap C\inv(]\ov e,e^*])$ admits two connected components $\bY^+$ and $\bY^-$, 
which are annuli  defined over the interval $]\ov e,e^*]$ and  realizing $c$ and $-c$ respectively,
\vskip1.5mm
$\bu$   $\bY^0=\bY\cap C\inv([e_*,\ov e[)$ is an annulus realizing the null class $0$.
\end{Def}

\begin{figure}[h]
\begin{center}
\begin{pspicture}(0cm,2cm)
\psset{xunit=.5cm,yunit=.4cm}
\rput(-4,0){
\psellipse(8,2.5)(.5,1.5)
\psellipse(0,2.5)(.5,1.5)
\psframe[linestyle=none,fillstyle=solid,fillcolor=white](0,0)(.8,4)
\pscurve(3.5,3)(3.8,2.2)(4.5,3)
\pscurve(3.5,3)(3.8,2.55)(4.05,2.3)
\pscurve(0,4)(2,3.2)(3.5,3)
\pscurve(4.5,3)(6,3.2)(8,4)
\pscurve(4.5,3)(6,3.2)(8,4)
\pscurve(0,1)(2,.2)(4,-.1)(6,.2)(8,1)
\pscircle[fillstyle=solid,fillcolor=black,linecolor=black](4,-.1){.05}
\rput(4,-.7){$O$}
\pscurve[linecolor=black](4,-.1)(3.5,.8)(3.2,1.6)(3.1,2.5)(3.2,3)
\pscurve[linecolor=black](4,-.1)(4.5,.8)(4.8,1.6)(4.9,2.5)(4.8,3)
\rput(8.5,5){$\bY^+$}
\psline[linewidth=.1mm](8,4.6)(7,3)
\rput(-.5,5){$\bY^-$}
\psline[linewidth=.1mm](0,4.4)(1,3)
\rput(4,5){$\bY_0$}
\psline[linewidth=.1mm](4,4.6)(4,1.5)
\rput(1.5,-2){\color{black}$C\inv(\bar e)\cap\bY$}
\psline[linewidth=.1mm](2,-1.3)(3.5,.9)
}
\end{pspicture}
\vskip8mm
\caption{A singular $2$-dimensional annulus}\label{Fig:singcyl}
\end{center}
\end{figure}
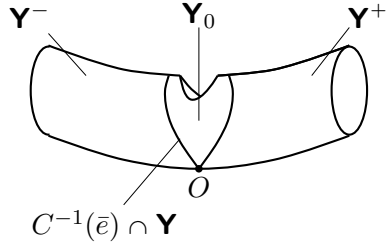

\vskip-2mm

A singular annulus, endowed with its induced dynamics, is essentially the phase space
of  a simple pendulum from which an open neighborhood of the elliptic fixed point has been removed.
According to the remark on the interpretation of the energy of $C$, a singular annulus is to be though
of as located ``at the center'' of the double resonance for the initial system (\ref{eq:hamstatement}).

\paraga We will need the following notion of chains\footnote{we keep the same terminology
as for the cylinders, with a slightly different sense here} of annuli for $C$, from which we will deduce the
existence and properties of the chains of cylinders near the double resonance points.

\begin{Def}\label{def:chains} 
Let $c\in H_1(\T^2,\Z)$. We say  that a family $(I_i)_{1\leq i\leq i_*}$ of  nontrivial intervals, 
contained and closed in the energy interval
$]\ov e,+\infty[$, is {\em ordered} when 
$\Max I_i=\Min I_{i+1}$ for $1\leq i\leq i_*-1$.
A {\em chain of annuli realizing $c$} is a  family $(\sA_i)_{1\leq i\leq i_*}$ of 
annuli realizing $c$, defined over an ordered family $(I_i)_{1\leq i\leq i_*}$,
with the additional property
$$
W^-(\sA_{i})\cap W^+(\sA_{i+1})\neq \emptyset,\qquad 
W^+(\sA_{i})\cap W^-(\sA_{i+1})\neq \emptyset,
$$
for $1\leq i\leq i_*-1$.
\end{Def}

The last condition is equivalent to assuming that the boundary periodic orbits
of $\sA_i$ and $\sA_{i+1}$ at energy $e=\Max I_i=\Min I_{i+1}$ admit heteroclinic orbits\footnote{but
the previous formulation is more appropriate when hyperbolic continuations of the annuli are involved}.

\def\dd{{\bf d}}
\paraga We can  now state the main result of Part II. 
We say that $c\in H_1(\T^2,\Z)\setm\{0\}$ is {\em primitive} when the 
equality $c=mc'$ with $m\in\Z$ implies $m=\pm1$.  We denote by $\H_1(\T^2,\Z)$ the set of primitive 
homology classes, by $\dd$ be the Hausdorff distance for compact subsets of $\R^2$ and  by
$\Pi:\A^2\to\R^2$ the canonical projection. 
\vskip3mm

\noindent{\bf Theorem II.} {\bf (Generic hyperbolic properties of classical systems).}
{\it Let $T$ be a quadratic form on $\R^2$ and for $\ka\ge2$, let $\jU^\ka_0\subset C^\ka(\T^2)$ be the set of 
potentials with a single and nondegenerate maximum. 
Then for $\ka\geq\ka_0$ large enough, there exists a  residual subset 
\beq\label{eq:defUT}
\jU(T)\subset\jU_0^\ka
\eeq 
in $C^\ka(\T^2)$
such that for $U\in \jU$,  the 
associated classical  system $C=\pdemi T+U$  satisfies the following properties.
\begin{enumerate}
\item For each  $c\in \H_1(\T^2,\Z)$ there exists a chain $\bA(c)=(\sA_0,\ldots,\sA_m)$ of
annuli realizing~$c$, defined over ordered intervals $I_0,\ldots,I_m$, such that the first and last intervals 
are of the form
$$
I_0=\,]\Max U,e_m]\quad \textit{and}\quad I_m=[e_P,+\infty[,
$$
for suitable constants $e_m$ and $e_P$ (which we call the Poincar\'e energy).

\item Given two primitive classes $c\neq c'$, there exists $\sig\in\{-1,+1\}$ such that
the  two chains $\bA(c)=(\sA_i)_{0\leq i\leq m}$ and 
$\bA(\sigma c')=(\sA'_i)_{0\leq i\leq m'}$ satisfy
$$
W^-(\sA_0)\cap W^+(\sA'_0)\neq\emptyset
\quad \textit{and}\quad
W^-(\sA'_0)\cap W^+(\sA_0)\neq\emptyset,
$$
both heteroclinic intersections being transverse in $\A^2$.

\item There exists a singular annulus $\bY$ which admits transverse heteroclinic connections with the 
first  annulus of the chain $\bA(c)$, for all $c\in\H_1(\T^2,\Z)$.

\item Under the canonical identification of $H_1(\T^2,\Z)$ with $\Z^2$ and for $e>0$,
let us set, for a given primitive class $c\sim(c_1,c_2)\in\Z^2$:
$$
Y_c(e)=\frac{\sqrt {2e}\,c}{\sqrt{c_1^2+c_2^2}}\in\R^2
$$
Let $\bA(c)=(\sA_0,\ldots,\sA_m)$ be the associated chain and set $\Ga_e=\sA_m\cap C\inv(e)$ for $e\in [e_P,+\infty[$. Then
$$
\lim_{e\to+\infty}\dd\big(\Pi(\Ga_e),\{Y_c(e)\}\big)=0.
$$
\end{enumerate}
}

\vskip2mm

We say that a chain with $I_0$ and $I_m$ as in the first item is {\em biasymptotic 
to $\ov e:=\Max U$ and to $+\infty$}. We will not only consider chains formed by
annuli only, but also ``generalized ones'' in which we will allow one annulus to be singular. 
With this terminology, one can rephrase the content of items 1 and 3 of previous theorem in the following
concise way:
for $U\in\jU$ and for each pair of classes $c,c'\in \H_1(\T^2,\Z)$, there exists a 
generalized chain:
$$
\sA_m\leftrightarrow\cdots\leftrightarrow\sA_1\leftrightarrow\bY\leftrightarrow\sA'_1\leftrightarrow\cdots\leftrightarrow\sA'_{m'}
$$ 
(where $\leftrightarrow$ stands for the heteroclinic connections)
which is biasymptotic to $+\infty$, and realize $c$ and $c'$ respectively. 
This is indeed the main ingredient of our subsequent
constructions, to get the part of the chains of cylinders located in the neighborhood of the ``double resonance
points''. Item 4 will serve us to precisely localize the extremal cylinders, while item 2, which we find interesting in itself, 
will not be used in the construction of our chains.

\vskip2mm

In the $r$--plane, one therefore gets the following symbolic picture for the projection of 6 generalized chains of annuli,
where the annuli  are represented by fat segments, the singular annulus by a fat segment with a circle
and the various heteroclinic connections are represented by  $\leftrightarrow$.

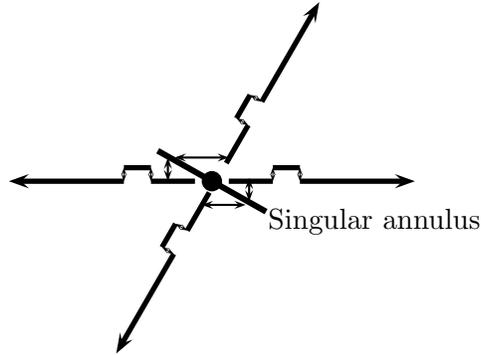
\begin{figure}[h]
\begin{center}
\begin{pspicture}(0cm,2cm)
\psset{unit=.45cm}
\rput(0,-.2){
\psline[linewidth=0.8mm](-1.6,.9)(1.6,-.9)
\pscircle[fillstyle=solid,fillcolor=black](0,0){.3}
}
\rput(4.8,-1.4){Singular annulus}
\psline[linewidth=0.7mm](.5,-0.2)(1.8,-0.2)
\psline[linewidth=0.1mm]{<->}(1.8,0.2)(1.8,-0.2)
\psline[linewidth=0.7mm](1.8,0.2)(2.6,0.2)
\psline[linewidth=0.1mm]{<->}(2.6,0.2)(2.6,-0.2)
\psline[linewidth=0.7mm]{->}(2.6,-0.2)(6,-0.2)
\psline[linewidth=0.7mm](-.5,-0.2)(-1.8,-0.2)
\psline[linewidth=0.1mm]{<->}(-1.8,0.2)(-1.8,-0.2)
\psline[linewidth=0.7mm](-1.8,0.2)(-2.6,0.2)
\psline[linewidth=0.1mm]{<->}(-2.6,0.2)(-2.6,-0.2)
\psline[linewidth=0.7mm]{->}(-2.6,-0.2)(-6,-0.2)
\rput{60}{
\psline[linewidth=0.7mm](.5,-0.2)(1.8,-0.2)
\psline[linewidth=0.1mm]{<->}(1.8,0.2)(1.8,-0.2)
\psline[linewidth=0.7mm](1.8,0.2)(2.6,0.2)
\psline[linewidth=0.1mm]{<->}(2.6,0.2)(2.6,-0.2)
\psline[linewidth=0.7mm]{->}(2.6,-0.2)(6,-0.2)
\psline[linewidth=0.7mm](-.5,-0.2)(-1.8,-0.2)
\psline[linewidth=0.1mm]{<->}(-1.8,0.2)(-1.8,-0.2)
\psline[linewidth=0.7mm](-1.8,0.2)(-2.6,0.2)
\psline[linewidth=0.1mm]{<->}(-2.6,0.2)(-2.6,-0.2)
\psline[linewidth=0.7mm]{->}(-2.6,-0.2)(-6,-0.2)
}
\psline{<->}(1.1,-.1)(1.1,-.8)
\psline{<->}(-1.3,-.2)(-1.3,.5)
\psline{<->}(-1.1,.5)(.5,.5)
\psline{<->}(1,-.9)(-.3,-.9)
\end{pspicture}
\vskip2.1cm
\caption{Projections in action of chains of annuli}\label{fig:classicannuli}
\end{center}
\end{figure}

\vskip-3mm
The projections of the annuli on the action
space are in fact more complicated than lines, they are rather $2$--dimensional submanifolds with boundary, 
which tend to a line when the energy grows to infinity.

\vskip3mm

\section{Outline of the proofs}
\setcounter{paraga}{0}

Part I essentially relies on the result of Part II, which will be described separately.

\subsection{Outline of the proof of Theorem I}

\vskip2mm\noindent
$\bu$ In this description we look at simplified model of the form $H_\eps=h+\eps f$, where we assume
$h(r)=\pdemi(r_1^2+r_2^2+r_3^2)$.  We fix an energy $\e>0$ and consider the 
broken line $\bGa$ defined in Section~\ref{sec:TheoremI}.  Fix an arc $\Ga=\Ga_{k_i}$ from $\bGa$
and assume, again for simplicity, that $k_i=(0,0,1)$, so that $\Ga$ is contained in the plane $r_3=0$, and 
$r\in\Ga$ if and only if:
$$
\om(r)=\nabla h(r)=(r_1,r_2,0),\qquad h(r_1,r_2,0)=\e.
$$
One can assume without loss of generality that the endpoints of $\Ga$ are double resonance points,
that is, the frequency $\ha\om(r):=(r_1,r_2)$ lies on a rational line of $\R^2$. 
To prove the existence of cylinders whose projection in action lies along $\Ga$, we will first average
the perturbation {\em as much as possible} in order to get simplified systems which admit cylinders.
We then use normally hyperbolic persistence to prove that these cylinders give rise to cylinders
in the initial system, provided that the averaged systems are close enough to the initial one.

\vskip2mm\noindent
$\bu$ Given $r^0\in\Ga$, when $\ha\om(r^0)$ is ``sufficiently nonresonant'',
one proves that the system $H_\eps$ is conjugated to the normal form
\beq\label{eq:exnormform}
N_s(\th,r)=h(r)+\eps V(\th_3,r)+R_s(\th,r,\eps)
\eeq
in the neighborhood of $\T^3\times \{r^0\}$, with
\beq
\qquad V(\th_3,r):=\int_{\T^2}f\big((\ha\th,\th_3),r\big)\,d\ha\th,\qquad \ha\th=(\th_1,\th_2),
\eeq
and where $R_s$ is small in some $C^k$ topology. 

\vskip2mm\noindent
$\bu$ When $r$ varies on a small closed segment $S\subset \Ga$ around $r^0$, the truncated normal form
\beq\label{eq:truncnormform}
\pdemi(r_1^2+r_2^2)+\big[\pdemi r_3^2+\eps V(\th_3,r)\big]
\eeq
appears as the skew-product of the unperturbed Hamiltonian $\pdemi(r_1^2+r_2^2)$ 
with a family of ``generalized pendulums'', functions of $(\th_3,r_3)\in \A$, parametrized by $r\in S$
(the fact that $r_3$ itself appears in the parameter is here innocuous).
 For each value of the parameter, the latter pendulums are therefore completely integrable.  
Assume moreover that $V(\,\cdot\,,r)$
admits a single and nondegenerate maximum at some point $\th_3^*(r)$, 
and, for simplicity, that $V\big(\th_3^*(r),r)=0$. 
Then the point $(\th_3^*(r),r_3=0)$ is hyperbolic for the Hamiltonian $\pdemi r_3^2+\eps V(\th_3,r)$
 and one immediately 
gets the existence of a normally hyperbolic cylinder $\cC$ at energy $\e$ for $N_s$ by taking the
the product of the torus $\T^2$ of the angles $\ha\th$ with the curve 
$$
r\in S,\quad  \th_3=\th_3^*(r).
$$
Note that $\cC$ is diffeomorphic to $\T^2\times[0,1]$, so that its boundary is the disjoint union of 
two $2$-dimensional isotropic tori.

\vskip2mm\noindent
$\bu$ When the remainder $R_s$ is small enough in the $C^2$ topology, the previous cylinder persists 
by normal hyperbolicity
{\em provided that its boundary persists}, which will comes from KAM-type results. This necessitates 
both $R_s$ to be small in the $C^k$ topology for $k$ large enough and some frequency to be 
Diophantine, which in turns necessitates a careful choice of the endpoints of the segment $S$. 
One main task in Part I is be to determine {\em maximal} subsegments $S$ of $\Ga$ to which the previous 
description applies.

\vskip2mm\noindent
$\bu$ We will treat the smallness condition of $R_s$ and the KAM conditions separately.
The first remark (see \cite{B10}), is that  under appropriate nondegeneracy conditions 
on $f$, the smallness condition on $R_s$ holds outside a {\em finite set} $D\subset \Ga$ of 
``strong double resonance points''.
Consequently, our first step will be to divide $\Ga$ into ``$s$--segments'' (where $s$ stands for ``purely simple'')
limited by a finite number of consecutive strong double resonance points ($\bigcirc\hskip-2.9mm\bullet$ in the following picture).

\begin{figure}[h]
\begin{center}
\begin{pspicture}(0cm,.5cm)
\psline[linewidth=.3mm](-5,0)(5,0)
\pscircle(-5,0){.2}
\pscircle(-3,0){.2}
\pscircle(0,0){.2}
\pscircle(1.5,0){.2}
\pscircle(4,0){.2}
\pscircle(5,0){.2}
\pscircle[fillstyle=solid,fillcolor=black](-5,0){.1}
\pscircle[fillstyle=solid,fillcolor=black](-3,0){.1}
\pscircle[fillstyle=solid,fillcolor=black](0,0){.1}
\pscircle[fillstyle=solid,fillcolor=black](1.5,0){.1}
\pscircle[fillstyle=solid,fillcolor=black](4,0){.1}
\pscircle[fillstyle=solid,fillcolor=black](5,0){.1}
\end{pspicture}
\caption{The arc $\Ga$ with the strong double resonance points}
\end{center}
\end{figure}
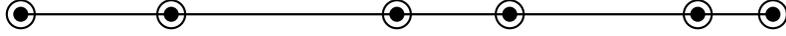

\vskip-5mm

We prove that global normal forms exist along such segments, which enable us to detect ``normally
hyperbolic'' cylinders (without boundary) which are everywhere tangent to the Hamiltonian vector field, 
but not necessarily invariant under its flow. Obviously the notion of normal hyperbolicity has to be 
relaxed beyond its usual sense in this case, which will be done in Section I of Part I. These pseudo
invariant cylinders become genuine normally hyperbolic invariant manifolds once the existence of
$2$ dimensional invariant tori close to their boundaries is proved. We call them the $s$-cylinders.

\vskip2mm\noindent
$\bu$ To prove this existence, and overcome the lack of precise estimates on the size of the remainder
$R_s$, we will begin by proving the existence of genuine invariant cylinders {\em in the neighborhood of the 
double resonance points}. These cylinders will be called
$d$-cylinders in the following.  Thanks to the existence of extremely precise normal forms in the neighborhood
of double resonance points\footnote{in domains whose size tends to $0$ when $\eps\to0$}, their existence
is easy to prove taking Theorem II for granted.  In particular, we will   be able
to prove the existence of many $2$-dimensional persisting tori inside these cylinders.  
We will then turn back to the determination of the maximal segments $S$,
by ``interpolating'' between two $d$-cylinders located near two consecutive double resonance points,
by means of the previous global normal form. This way, the boundaries of the $s$-cylinders will be proved to  
belong to the previous family of  $2$-dimensional tori.

\vskip2mm\noindent
$\bu$ Let us now describe the construction of the $d$-cylinders
in the neighborhood of a double resonance point. Given such a
point $r^0$, for instance $r^0=(1,0,0)$ for simplicity, the first task is to prove the existence of a conjugacy
between the initial system and the normal form
\beq\label{eq:exnormform2}
\begin{array}{lll}
N_d(\th,r)=\pdemi r_1^2+\big[\pdemi(r_2^2+r_3^2)+\eps U(\th_2,\th_3)\big]+R_d(\th,r,\eps),\\[5pt]
\qquad U(\th_2,\th_3):=\dsp\int_{\T}f\Big(\big(\th_1,(\th_2,\th_3)\big),r^0\Big)\,d\th_1,
\end{array}
\eeq
where now the remainder $R_d$ can be proved to be extremely small (in the $C^k$ topology with large $k$) 
over a neighborhood of $r^0$ of diameter $\eps^\nu$, where $\nu$ can be arbitrarily chosen in $]0,\pdemi]$ provided that
$\ka$ is large enough. 

\vskip2mm\noindent
$\bu$ After performing a $\sqrt\eps$ dilatation in action, the main role in (\ref{eq:exnormform2}) will be played by the
classical system 
$$
C(\ov\th,\ov r)=\pdemi(r_2^2+r_3^2)+ U(\th_2,\th_3)
$$
which we will assume to satisfy the genericity conditions of Theorem~II. This will provide us with a large
family of invariant $2$-dimensional annuli for $C$, realizing any primitive integer homology class of $\T^2$,
together with a singular annulus.
They constitute chains and ``generalized chains'' along lines of rational slope in projection to the action
space (see Figure~\ref{fig:classicannuli}).

\vskip2mm\noindent
$\bu$ In the truncated normal form
$$
\pdemi r_1^2+\big[\pdemi(r_2^2+r_3^2)+\eps U(\th_2,\th_3)\big]
$$
each previous annulus $\sA$ of $C$ gives rise (up to a rescaling in action) to a cylinder, product of $\sA$
with the circle $\T$ of the angle $\th_1$. Again, this cylinder is diffeomorphic to $\T^2\times [0,1]$.
Now we can moreover take advantage to the smallness
of $R_d$ to prove the persistence of the boundaries by KAM techniques (we will use here Herman's version
of the invariant curve theorem). This way we prove the existence in the initial system of a $d$-cylinder 
attached to each annulus of $C$, which lies along a simple resonance curve, whose equation is directly
related to the homology class which is relized by $\sA$.

The same method enables us to prove the existence of a singular
cylinder attached to the singular annulus of $C$, and which is located ``at the center of the double resonance''. 
The length of these cylinders is $O(\sqrt\eps)$, due to the rescaling. 

One can also prove the existence of 
heteroclinic orbits between them, as soon as the 
annuli of $C$ from which they are deduced admit heteroclinic connections. 
Finally, one crucial remark is then that the extremal cylinders (attached to the extremal annuli
of $C$) can be {\em continued
to a distance  $O(\eps^\nu)$ from the double resonance point}.
One therefore deduce from 
Figure~\ref{fig:classicannuli} the following symbolic picture, now in the initial system and near $r^0$.

\begin{figure}[h]
\begin{center}
\begin{pspicture}(0cm,3cm)
\rput(0,0){
\psset{xunit=.7,yunit=.7,runit=.7}
\psline[linewidth=.1mm](-5,0)(5,0)
\pscircle(0,0){4}
\pscircle(0,0){1.5}
\psline[linewidth=1mm](-.5,.5)(.5,-.5)
\pscircle[fillstyle=solid,fillcolor=black](0,0){.18}
\psline[linewidth=.8mm](.26,.1)(.5,.1)
\psline[linewidth=.8mm](.46,-.1)(.9,-.1)
\psline[linewidth=.8mm](.8,.1)(4,.1)
\psline[linewidth=.8mm](-.26,.1)(-.5,.1)
\psline[linewidth=.8mm](-.46,-.1)(-.9,-.1)
\psline[linewidth=.8mm](-.8,.1)(-4,.1)
\psline[linewidth=.1mm](-5,0)(5,0)
\rput{45}(0,0){
\psline[linewidth=.8mm](-.26,.1)(-.5,.1)
\psline[linewidth=.8mm](-.46,-.1)(-.9,-.1)
\psline[linewidth=.8mm](-.8,.1)(-4,.1)
\psline[linewidth=.8mm](.26,.1)(.5,.1)
\psline[linewidth=.8mm](.46,-.1)(.9,-.1)
\psline[linewidth=.8mm](.8,.1)(4,.1)
\psline[linewidth=.8mm](-.26,.1)(-.5,.1)
}
\rput(2,4){$\eps^\nu$}
\rput(.4,1.9){$\sqrt\eps$}
}
\end{pspicture}
\vskip2.8cm
\caption{$d$-cylinders and singular cylinder near a double resonance point}
\end{center}
\end{figure}
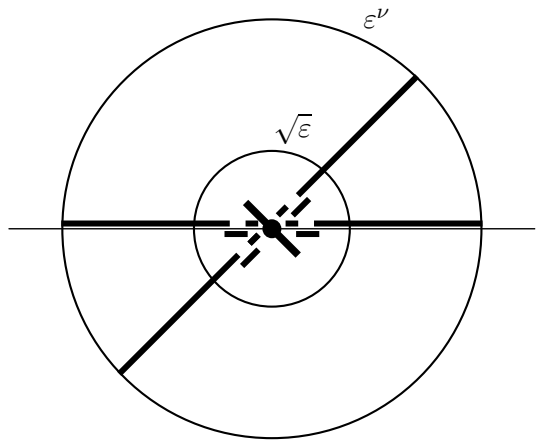

\vskip-4mm
We did not represent the heteroclinic connections since they are immediately deduced from
those of Figure~\ref{fig:classicannuli}. In particular, the four cylinders located close to the singular
cylinder admit heteroclinic connections with it. The chains of cylinders so obtained lie along 
the simple resonance curves getting to the double resonance point, and admit connections with
the singular cylinder. This enables us to ``cross'' the double resonance along a simple resonance curve,
of to ``pass from'' one resonance curve to another one.

\vskip2mm\noindent
$\bu$ Once the existence the $d$-cylinders is proved for each double resonance point 
$\bigcirc\hskip-2.9mm\bullet$\ \
on the segment $\Ga$, we can ``interpolate along $\Ga$''  between the extremal cylinders attached to two consecutive 
such points. These extremal cylinders are those attached to the extremal annuli of the classical systems realizing
the homology corresponding to the resonance curve $\Ga$.
This yields the existence of an $s$--cylinder, 
whose projection in action lies along the segment of $\Ga$ limited by the double resonance points,
and whose ``ends'' moreover ``match'' with both extremal cylinders at these points.

The situation is in fact slighly more complicated, due to the possible generic occurrence of {\em bifurcation points}
for the two-phase averaged systems~(\ref{eq:truncnormform}). These are the points $r\in\Ga$ where
the potential $V(\,\cdot\,,r)$ admits {\em two} nondegenerate global maxima instead of a single one. In the neighborhood
of these points two cylinders coexist, for which we prove the existence of heteroclinic
connections. This yields the following final picture between two double resonance points.

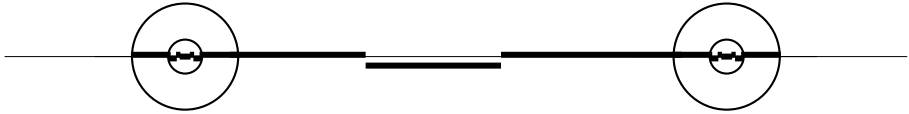
\begin{figure}[h]
\begin{center}
\begin{pspicture}(0cm,.7cm)
\rput(0,0){
\psset{xunit=1.2,yunit=1.2,runit=1.2}
\psline[linewidth=.1mm](-5,0)(5,0)
\psline[linewidth=.8mm](-2.5,.02)(-1,.02)
\psline[linewidth=.8mm](-1,-.1)(.5,-.1)
\psline[linewidth=.8mm](.5,.02)(2.5,.02)
\rput(-3,0){
\psset{xunit=.2,yunit=.2,runit=.2}
\psline[linewidth=.1mm](-5,0)(5,0)
\pscircle(0,0){3}
\pscircle(0,0){1}
\psline[linewidth=.8mm](-.3,0)(.3,0)
\psline[linewidth=.8mm](.26,.1)(.5,.1)
\psline[linewidth=.8mm](.46,-.1)(.9,-.1)
\psline[linewidth=.8mm](.8,.1)(3,.1)
\psline[linewidth=.8mm](-.26,.1)(-.5,.1)
\psline[linewidth=.8mm](-.46,-.1)(-.9,-.1)
\psline[linewidth=.8mm](-.8,.1)(-3,.1)
}
\rput(+3,0){
\psset{xunit=.2,yunit=.2,runit=.2}
\psline[linewidth=.1mm](-5,0)(5,0)
\pscircle(0,0){3}
\pscircle(0,0){1}
\psline[linewidth=.8mm](-.3,0)(.3,0)
\psline[linewidth=.8mm](.26,.1)(.5,.1)
\psline[linewidth=.8mm](.46,-.1)(.9,-.1)
\psline[linewidth=.8mm](.8,.1)(3,.1)
\psline[linewidth=.8mm](-.26,.1)(-.5,.1)
\psline[linewidth=.8mm](-.46,-.1)(-.9,-.1)
\psline[linewidth=.8mm](-.8,.1)(-3,.1)
}
}
\end{pspicture}
\vskip.3cm
\caption{Interpolation between two extremal $d$-cylinders}\label{fig:interpolation}
\end{center}
\end{figure}

\vskip-3mm

\vskip2mm\noindent
$\bu$ This way one obtains a chain of cylinders and singular cylinders along the segment $\Ga$,
by concatenation of the previous chains between consecutive double resonance points. This construction
works for each segment $\Ga_{k_i}$ of the initial broken line. To get a chain along the full broken line
one only has to use the previous description at a double resonance point: the ``incoming chain'' along
$\Ga_{k_i}$ is connected to the ``outgoing chain'' along $\Ga_{k_{i+1}}$ since the singular cylinder
at the point $a_i$ admits heteroclinic connections with the ``initial cylinders'' in both chains.

\begin{figure}[h]
\begin{center}
\begin{pspicture}(0cm,1.4cm)
\psset{xunit=.5,yunit=.5,runit=.5}
\psline[linewidth=.1mm](-5,0)(5,0)
\pscircle(0,0){3}
\pscircle(0,0){1}
\psline[linewidth=1mm](-.5,.5)(.5,-.5)
\pscircle[fillstyle=solid,fillcolor=black](0,0){.18}
\psline[linewidth=.8mm](.26,.1)(.5,.1)
\psline[linewidth=.8mm](.46,-.1)(.9,-.1)
\psline[linewidth=.8mm](.8,.1)(3,.1)
\rput{45}(0,0){
\psline[linewidth=.8mm](-.26,.1)(-.5,.1)
\psline[linewidth=.8mm](-.46,-.1)(-.9,-.1)
\psline[linewidth=.8mm](-.8,.1)(-3,.1)
}
\end{pspicture}
\vskip1.4cm
\caption{Transition between two arcs at a double resonance point}\label{fig:transition}
\end{center}
\end{figure}
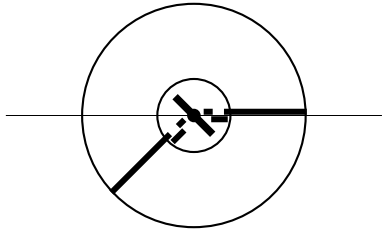

\vskip2mm\noindent
$\bu$ The previous constructions are possible if $f$ is subjected to a list of nondegeneracy conditions,
both along the simple resonance curves involved in the construction of the broken line $\bGa$ and in
the neighborhood of the strong double resonance points (or the intersection points of two distinct curves
in $\bGa)$. The last step is to prove that these conditions are cusp-residual.


\vskip-.5cm
\subsection{Outline of the proof of Theorem II}

In this part we consider a classical system $C(\th,r)=\pdemi T(r)+U(\th)$ on $T^*\T^2$, under the generic
assumption that  $U$ admits a single
and nondegenerate maximum  at $\th_0$. The lift $O=(\th_0,0)$ to the zero section is therefore a hyperbolic
fixed point for the vector field $X_C$. We set $\ov e=\Max U$. 

\vskip2mm\noindent
$\bu$ 
For $e>\ov e$, the so-called Jacobi metric induced 
by $C$ at energy $e$ is defined for $v\in T_\th \T^2$ by 
\begin{equation}
\abs{v}_e=\big(2(e-U(\th))\big)^{\ppdemi}\norm{v},
\end{equation}
where $\norm{\ }$ stands for the norm on $\R^2$ associated with the dual of $T$.
The Jacobi-Maupertuis principle states that, up to reparametrization, the solutions of the Hamiltonian vector field $X^C$
 in $C\inv (e)$ and those of the geodesic vector field $X_e$ induced
by $\abs{\ }_e$ in the unit tangent bundle are in one-to-one correspondence.

\vskip2mm\noindent
$\bu$ Fix a primitive class $c\in H_1(\T^2,\Z)$. By a simple minimization
argument, there exist length-minimizing closed geodesics in the class $c$ for the metric $\abs{\,\cdot\,}_e$.
As a consequence,  for each $e>\ov e$,  there exist periodic orbits of $X^C$ contained in $C\inv(e)$ and realizing $c$,
which we will call minimizing too. 
It turns out that, generically on $U$, these orbits are hyperbolic. Moreover, still generically, there is a discrete
subset $B(c)\subset\,]\ov e,+\infty[$ such that for $e\in \,]\ov e,+\infty[\setm B(c)$, the level $C\inv(e)$ contains
a single minimizing periodic orbit realizing $c$, while $C\inv(e)$ contains exactly two such orbits when $e\in B(c)$.
Finally, Hedlund's theorem proves that when 
$e\in \,]\ov e,+\infty[\setm B(c)$, the corresponding minimizing periodic orbit admit homoclinic orbits, 
while the two minimizing orbits at $e\in B(c)$ are connected by heteroclinic orbits.

\vskip2mm\noindent
$\bu$ Since the orbits are hyperbolic, varying the energy $e$ in the previous description
proves the existence of a (possibly infinite) family of annuli 
$(A_j)_{j\in J}$, defined over the ordered family of intervals $(I_j)_{j\in I}$ limited by
consecutive points of $B(c)$ (the constraint of monotonicity of the periods and 
the existence of continously varying homoclinic orbits in Definition~\ref{def:ann} come from more refined
considerations). 
Moreover, each pair of annuli defined over consecutive intervals
admit heteroclinic connections  by Hedlund's theorem. 

It therefore remains to prove that the chain ``stabilizes''
at both ends, that is, that one can assume $J$ to be finite, of the form $\{1,\ldots,m\}$,
with $I_1=\,]\ov e, e_m[$ and $I_m=\,]e_P,+\infty[$. We refer to the latter as the ``high energy
annulus'' and to the former as the ``low energy annulus''.

\vskip2mm\noindent
$\bu$ {\bf The high energy annuli.} To see that the familly stabilizes at high energies, we   
use the fact that a classical system of the form $C(x,p)=\ppdemi T(p)+U(x)$ at high energy 
appears as a perturbation of the completely integrable system $\ppdemi T$.
The scaling $p=\sqrt\eps \ov p$ reduces the study of $C$ at high energies $e$ to that of 
$$
C_\eps(x,\ov p)=\pdemi T(\ov p)+\eps U(x)
$$
for small $\eps\sim 1/e$.
We canonically identify $H_1(\T^2,\Z)$ with $\Z^2$. Given $c\in\Z^2$,
we define the $c$--averaged potential associated with $U$ as the function
\begin{equation}\label{eq:avpot}
U_c(\ph)=\int_0^1 U\big(\ph+s\,(c_1,c_2)\big)\,ds
\end{equation}
where $\ph$ belongs to the circle  $\T^2/T_c\sim \T$, where  $T_c=\{\la(c_1,c_2)\ [\Z^2]\mid \la\in\R\}$.
Assume that $U_c$ admits a single nondegenerate maximum, which is nondegenerate.
Then  the classical Poincar\'e theorem on the
creation of hyperbolic periodic solutions by perturbation of periodic tori can be applied at each
point $\ov p$ with $\norm{\ov p}\geq \mu_0$, for $\mu_0$ large enough,
 on the simple resonance line $T\inv(\R c)$. As a result, going back to the system $C$ by the inverse scaling, 
 we get an annulus $\sA$ of class $C^\ka$ formed by the union of the rescaled periodic
orbits, which is defined over an interval of the form $]e_P,+\infty[$ and realizes $c$.  One can moreover
prove that these orbits are minimizing in the previous sense.

\vskip2mm\noindent
$\bu$ {\bf The low energy annuli.} The proof of existence of a single low energy annulus realizing a given class
is more involved and requires the study of the symbolic dynamics created by the hyperbolic fixed point $O$
together with its homoclinic orbits (such orbits were proved to exist in \cite{Bo78} and we will give here
a proof close to that of \cite{Be00}, based on discrete weak KAM theory, which enables us to localize them more precisely). 
This requires some additional (generic) nondegeneracy assumptions on the eigenvalues of the fixed point.
We obtain a family (parametrized by the energies 
$e$ slightly larger than $\ov e$) of horseshoes for Poincar\'e sections of the Hamiltonian flow in $C\inv(e)$. This is 
reminiscent of the Shilnikov-Turaev study for hyperbolic fixed points of Hamiltonian systems with homoclinic orbits
which are transverse in their critical energy level, with more precise estimates on the structure and
localization of the horseshoes. The result is the existence of a family of annuli realizing each primitive homology class
and which admit heteroclinic connections between them, provided that some compatibility condition is satisfied.
This will prove the stabilization property at low energy for each class, together that the third item in Theorem II.

\vskip2mm\noindent
$\bu$ {\bf The singular annulus.} We get the existence of (at least) one singular annulus by gluing together 
the annuli corresponding to the homology classes $\pm c$, where $c$ is determined by a minimization condition
on the homoclinic orbits of the hyperbolic fixed points, together with an annulus of periodic orbits realizing the 
zero homology class. This proves the existence of an invariant manifold which contains the fixed point together
with a pair of opposite homoclinic orbits (satisfying special minimization properties), on which a one-parameter
family of null homology periodic orbits accumulates, together with two families of periodic orbits realizing opposite
homology classes. The main point is that the union of the periodic orbits and the homoclinic orbits is a $C^1$
normally hyperbolic manifold, which is due to the nondegeneracy assumptions on the eigenvalues of the
fixed point. The rich heteroclinic structure induced by the family of horseshoes in turn yields the existence
of the heteroclinic connections between the first annuli in each chain and the singular annulus.

\vskip5mm

{\bf Structure of the paper.}  
The paper is split into two parts and seven appendices.  Part I introduces the various notions related to chains of cylinders 
and contains the proof of Theorem~I,  taking for granted the generic properties  of classical systems. 
Part II is dedicated to the various  definitions and statements relative to classical systems and contains 
the proof of Theorem~II.
The first four appendices present technical results related to Part I:
Appendix A recalls basic results on normally hyperbolic manifolds
in our setting, Appendices B and C are devoted to normal forms, and Appendix D
states a finite differentiable version of the invariant curve theorem for twist maps.
The last two appendices are related to Part II: in Appendix E we prove  the
existence of orbits homoclinic to the hyperbolic fixed points for generic classical
systems on $\A^2$,  Appendix F is devoted to a proof of the Hamiltonian Birkhoff-Smale
theorem, while Appendix G recalls some elements of Moser's construction of horseshoes.

\vskip3mm
{\bf Aknowledgements.}  I warmly thank Marc Chaperon, Alain Chenciner, Jacques F\'ejoz and Pierre Lochak for their constant
support and encouragements. I am indebted to Laurent Lazzarini for the proof of the invariant curve theorem 
and for lots of discussions at several stages of the preparation of this work. Cl\'emence Labrousse carefully read
and corrected several parts of this paper, my warmest thanks to her.

\newpage


\setcounter{section}{0}
\begin{center}
{\bf\LARGE  Part I. Cusp-generic chains}
\end{center}

\vskip.5cm

This part is devoted to the proof of Theorem I.
\begin{itemize}
\item In Section~\ref{sec:normhyp} we introduce precise definitions for normally hyperbolic annuli and cylinders. 
\item In  Section~\ref{sec:nondeg} we list the nondegeneracy conditions
imposed to the perturbed systems we consider. 
\item In Section~\ref{sec:cylinders} we introduce the definitions of $d$-cylinders and $s$-cylinders, which 
depend on the resonance zones they are located in.  
We also introduce the twist property  and the twist sections attached to a cylinder.
\item In  Section~\ref{sec:proofscyl} we prove the existence of $d$ and $s$-cylinders with twist sections 
under the nondegeracy conditions of Section~\ref{sec:nondeg}. 
\item In Section~\ref{sec:chains} we describe the homoclinic and heteroclinic intersection conditions which
are satisfied by the cylinders and serve us to define the notion of {\em admissible chains}. We 
prove their existence under the same nondegeracy conditions. 
\item Finally, Section~\ref{sec:cuspgen} proves
the cusp-genericity of our nondegeneracy conditions and ends the proof of Theorem I.
\end{itemize}


\setcounter{paraga}{0}
\section{Normally hyperbolic annuli and cylinders}\label{sec:normhyp}

In this section we first introduce particular definitions for the ``normal hyperbolicity'' of manifolds which
are not necessarily invariant under a vector field, whose occurence is unavoidable in the perturbed
systems we will consider. We then obtain genuine normally hyperbolic manifolds (with boundary)
by considering codimension 1 invariant subsets contained in the previous ones. We refer to \cite{C04,Berg10} for
direct presentations of the normal hyperbolicity of manifolds with boundary.

\paraga In this paper, a {\em $2\ell$-dimensional $C^p$ annulus} will be a $C^p$ manifold $C^p$ 
diffeomorphic to $\A^\ell$.
A {\em singular annulus} will be a $C^1$ manifold $C^1$-diffeomorphic to $\T\times\,]0,1[\,\times \cY$, 
where $\cY$ is (any realization of) the sphere $S^2$ minus three points.
We will have to consider $2$-dimensional annuli embedded in $\A^2$ and $4$-dimensional annuli or singular annuli
embedded in $\A^3$, which we abbreviate in $2$-annuli, $4$-annuli and singular $4$-annuli.

\paraga We now define the main objects under concern in this part, which all are $3$-dimensional manifolds.
\vskip1.5mm
$\bu$ A  {\em $C^p$ cylinder without boundary} is a $C^p$ manifold
$C^p$-diffeomorphic to $\T^2\times \R$.
\vskip1.5mm
$\bu$ A  {\em  $C^p$  cylinder} is a $C^p$ manifold $C^p$-diffeomorphic to $\T^2\times [0,1]$, so that a cylinder
is compact and its boundary has two components diffeomorphic to $\T^2$.
\vskip1.5mm
$\bu$ A  {\em $C^p$  singular cylinder} 
is a $C^p$ manifold $C^p$-diffeomorphic to $\T\times \bY$, where~$\bY$ is
(any realization of) the sphere $S^2$ minus three open discs with nonintersecting
closures. A singular cylinder is compact and its boundary has three components, diffeomorphic to $\T^2$.

\paraga Let $M$ be a $C^\infty$ manifold and $X$ a complete vector field on $M$ with flow $\Phi$. 
A submanifold $N\subset M$ (possibly with boundary) is said to be {\em pseudo invariant} for $X$ when the 
vector field $X$  is tangent to $N$ at each point of $N$. A  submanifold $N$ is said to be 
{\em invariant} when $\Phi(t,N)=N$ for all $t\in\R$. Invariant manifolds are pseudo invariant. When
$N$ is invariant with $\d N\neq\emptyset$, $\d N$ is invariant too.

\paraga We endow now $\A^3$ with its standard symplectic
form $\Om$, and we assume that $X=X_H$ is the vector field generated by $H\in C^\ka(\A^3)$, $\ka\geq 2$.
A pseudo invariant $4$-annulus $\jA\subset A^3$ for $X$ is said to be
{\em pseudo normally hyperbolic in $\A^3$} when there exist 
\vskip1.5mm
$\bu$ 
an open subset $O$ of $\A^3$ containing $\jA$, 
\vskip1.5mm
$\bu$ an embedding $\Psi:O\to \A^2\times \R^2$ whose image has compact closure, such that $\Psi_*\Om$ continues
to a symplectic form $\ov\Om$ on $\A^2\times \R^2$ which satisfies (Appendix~\ref{app:normhyp} (\ref{eq:assumpsymp})),
\vskip1.5mm
$\bu$ a vector field $\jV$ on
$\A^2\times \R^2$ satisfying the assumptions of the normally hyperbolic persistence theorem, in particular~(\ref{eq:addcond}), 
together with those of the symplectic normally hyperbolic theorem (Appendix~\ref{app:normhyp}) for the form $\ov\Om$,
such that, with the notation of this theorem:
\beq
\Psi(\jA)\subset \Ann(\jV)\quad{\rm and}\quad \Psi_*X(x)=\jV(x),\quad \forall x\in O.
\eeq
Such an annulus $\jA$ is therefore of class $C^p$ and symplectic. 
We define similarly pseudo normally hyperbolic singular $4$-annuli, with in this case $p=1$.

\paraga When the previous $4$-annulus $\jA$ is moreover invariant for $X_H$, we say that it is normally hyperbolic.
In this case the image $\Psi(\jA)\subset \Ann(\jV)$ is invariant for $\jV$ and admits
well-defined invariant manifolds $W^\pm\big(\Psi(\jA)\big)$, with center-stable and center-unstable foliations
$\big(W^\pm\big(\Psi(x)\big)\big)_{x\in\jA}$. In this case we define the {\em local} invariant manifolds $W^\pm_{\ell}(\jA)$
for $X$, with respect to $(O,\Psi)$, as the subsets
\beq
\Psi\inv\Big(W_\ell^\pm\big(\Psi(\jA)\big)\Big),
\eeq
where $W_\ell^\pm\big(\Psi(\jA)\big)$ stands for the connected component of $\Ann(\jV)$ in $\Psi(O)\cap W^\pm\big(\Psi(\jA)\big)$.
Similarly, we define the local center stable and unstable manifolds of the points of $\jA$:
\beq
W_\ell^\pm(x)=\Psi\inv\Big(W_\ell^\pm\big(\Psi(x)\big)\Big),\quad x\in\jA.
\eeq
The global manifolds $W^\pm(\jA)$ and $W^\pm(x)$ for $X$ are then defined in the usual way, 
by forward or  backward transport of the corresponding local ones by the flow of $X$.

By compactness of the image of  $\Psi$, one immediately checks that the global manifolds
$W^\pm(\jA)$ and $W^\pm(x)$
are independent of the choice of $(O,\Psi)$. These manifolds are of class $C^p$, coisotropic, and their characteristic
foliations coincide with their center-stable and center-unstable foliations. 

\paraga We define similarly the invariant manifolds of invariant (normally hyperbolic) $4$-singular-annuli,
which are therefore symplectic and whose invariant manifolds satsify the same properties as above.
Observe moreover that, by definition, given an invariant normally hyperbolic singular $4$-annulus $\jA_\bu$,
there exists an open neighborhood $O$ of $\jA_\bu$ in $\A^3$ and a Hamiltonian $H_\circ$ defined on
an open subset $\jO$ containing $O$ such that:
\vskip1mm
$\bu$ $H_\circ$ coincides with $H$ on $O$,
\vskip1mm
$\bu$ $H_\circ$ admits a normally hyperbolic invariant $4$-annulus which contains $\jA_\bu$.

\paraga We still assume that $X=X_H$ is the vector field generated by $H\in C^\ka(\A^3)$, $\ka\geq 2$.
Let $\e$ be a regular value of $H$. 
\vskip1.5mm
$\bu$ 
A pseudo invariant cylinder without boundary $\jC\subset H\inv(\e)$ is {\em pseudo normally hyperbolic in $H\inv(\e)$}
when there exists a pseudo invariant and pseudo normally hyperbolic $4$-annulus $\jA$ for $X_H$ such 
that $\jC\subset\jA\cap H\inv(\e)$. 
\vskip1.5mm
$\bu$ 
An invariant cylinder (with boundary) $\jC\subset H\inv(\e)$  is 
{\em normally hyperbolic in $H\inv(\e)$} when there exists 
an invariant normally hyperbolic $4$-annulus $\jA$ for $X_H$ such that $\jC\subset\jA\cap H\inv(\e)$.
Any such $\jA$ is said to be {\em associated with $\jC$}.
\vskip1.5mm
$\bu$ 
An invariant singular cylinder  $\jC_\bu\subset H\inv(\e)$  is 
{\em normally hyperbolic in $H\inv(\e)$} when there is 
an invariant normally hyperbolic  $4$-singular-annulus $\jA_\bu$ for $X_H$ such that 
$\jC_\bu\subset\jA_\bu\cap H\inv(\e)$.
Any such $\jA_\bu$ is said to be {\em associated with $\jC_\bu$}.
\vskip1.5mm
One immediately sees that normally hyperbolic invariant cylinders or singular cylinders, contained
in $H\inv(\e)$, admit well-defined $4$-dimensional stable and unstable manifolds with boundary, 
also contained in $H\inv(\e)$, together with their center-stable and center-unstable foliations. 
\vskip1.5mm

\paraga From the remark on the singular $4$-annuli, one deduces that given a singular cylinder $\jC_\bu$, there exists an open
neighborhood $O$ of $\jC_\bu$  in $\A^3$ and a Hamiltonian $H_\circ$ defined on
an open subset $\jO$ containing $O$ such that:
\vskip1mm
$\bu$ $H_\circ$ coincides with $H$ on $O$,
\vskip1mm
$\bu$ $H_\circ$ admits a normally hyperbolic cylinder.
\vskip1.5mm
This remark will enable us to deal with singular cylinders in the same way as with usual cylinders in our subsequent
constructions.


\section{Averaged systems, $\de$-double resonances and conditions {\bf(S)}}\label{sec:nondeg}
We first describe the geometry of simple and double resonances at fixed energy of 
a Tonelli Hamiltonian $h\in C^\ka(\R^3)$ and, given a perturbation $f\in C_b^\ka(\A^3)$, we define 
the averaged systems associated with $H=h+f$. We then introduce the set of {\em $\de$-strong double
resonance points} on a resonance circle at fixed energy, where $\de>0$ will be the main control parameter of our 
construction.  This enables us to set out a list of nondegeneracy
conditions {\bf(S)} for the system $H$ along a resonance circle, which will be used throughout Part I and 
yield the ``cusp-generic part'' of Theorem I.


\subsection{Simple and double resonances}

We identify the action space $\R^3$ with its dual, the frequency space, by Euclidean duality. 
We fix a $C^\ka$ Tonelli Hamiltonian $h$ on $\R^3$,  $\ka\geq2$, and set $\om=\nabla h$. Let us first state some
direct geometric consequences of the convexity and superlinearity of $h$.

\paraga The map $\om$ is a $C^{\ka-1}$ diffeomorphism from $\R^3$ onto $\R^3$. Being convex and coercive, 
$h$ admits a single absolute minimum at some point $p$, which  satisfies
$\om(p)=0$. For $\e>h(p)$, the level surface $h\inv(\e)$ bounds a convex domain containing $p$, so
$h\inv(\e)$ is diffeomorphic to $S^2$, and its image  by $\om$ contains $0$ in its ``interior\footnote{the bounded
connected component of its complementary}.''
Moreover,  the map 
$
\vpi\mapsto  \frac{\vpi}{\norm{\vpi}_2}
$
(where $\norm{\ }_2$ stands for the Euclidean norm) defines  a $C^{\ka-1}$ diffeomorphism
from the set $\om\big(h\inv(e)\big)$ onto the sphere $S^2$.

\paraga Let $\pi:\R^3\to\T^3$ be the canonical projection. Fix $\vpi\in\R^3\setm\{0\}$ and consider
the {\em resonance module} $\cM_\vpi= \vpi^\bot\cap\Z^3$ associated with $\vpi$.
Clearly $\pi(\cM^\bot)$ is a subtorus of $\T^3$, which is invariant under the flow generated by the 
constant vector field $\vpi$. This flow is dynamically minimal.

\paraga Given $\vpi\in\R^3\setm\{0\}$, $\cM_\vpi=\vpi^\bot \cap\,\Z^3$ is a submodule of $\Z^3$ whose
rank is the {\em multiplicity of resonance} of $\vpi$. We say that $\vpi$ is a simple resonance frequency when
$\rk\cM_\vpi=1$ and a double resonance frequency when $\rk\cM=2$.
A point $r\in\R^3$ is a simple or double resonance action when $\om(r)$ is a simple or double resonance frequency.

\paraga Given a submodule $\cM$ of $\Z^3$ of rank $m=1$ or $2$, the vector subspace $\cM^\bot$ is 
said to be the {\em resonance subspace associated with $\cM$} (a resonance plane when $m=1$ and a resonance
line when $m=2$).
In the action space, the corresponding resonance 
$
\vpi\inv(\cM^\bot)
$,
is said to be a resonance surface when $m=1$ and a resonance curve when $m=2$. Note that any point on a resonance
curve is a double resonance action, while a point on a resonance surface can be either a simple resonance
action or a double resonance action.

\paraga
Resonance curves and surfaces in the action space are {\em transverse} to the levels 
$h\inv (\e)$ for $\e>\Min h$. As a consequence, the resonance surfaces intersect the energy levels 
along (topological) {\em resonance circles}, while the  resonance curves intersect the levels at isolated {\em double
resonance points}.
Moreover, two independent resonant circles at energy $\e>\Min h$ in the action space
intersect at exactly two double resonance points.

\paraga Recall that a submodule of $\Z^n$ is {\em primitive} when it is not strictly contained in a submodule 
with the same rank. Primitive rank $1$ submodules are generated by indivisible vectors of $\Z^n$.
Note that the resonances can always be defined by primitive submodules,
this will always be the case in the following.

\paraga Given a rank $m$ primitive submodule $\cM$ of $\Z^3$,  there exists 
a $\Z$--basis of $\Z^3$ whose last $m$ vectors form a $\Z$--basis of $\cM$ (see for instance \cite{Art}).
Let $P$ be the  matrix in ${\rm GL_3}(\Z)$ whose $i^{th}$-column is formed by the components of 
the $i^{th}$-vector of this basis. 
Let $\jR=\om\inv(\cM^\bot)$.
The symplectic linear coordinate change in $\A^3$ defined by  
\begin{equation}\label{eq:adcoord}
\th=\,^tP\inv \til\th\ \  [{\rm mod}\ \Z^n],\qquad  r=P \,\til r,
\end{equation}
transforms $h$ into a new Hamiltonian $\til h$  such that, setting $\til\om=(\til\om_1,\til\om_2,\til\om_3)=\nabla\til h$,
the transformed resonance $\til \jR=P\inv\,\jR$ admits the equation
$$
\til\om_{3-m+1}=\cdots=\til\om_3=0
$$
Such coordinates are said to be {\em adapted to $\cM$}.

\begin{notation}\label{not:splitvar}
According to the previous decomposition,  the variables $u$ in $\R^3$ or $\T^3$ will be 
split into $(\ha u,\ov u)=u$, where $\ov u$ is $m$-dimensional and $\ha u$ is $(3-m)$-dimensional.
\end{notation}


\subsection{Averaged systems}\label{ssec:normformepsdep}
\setcounter{paraga}{0}

We consider  a $C^\ka$ Tonelli Hamiltonian $h$ on $\R^3$,  $\ka\geq2$, and set $\om=\nabla h$.
Given $f\in C^\ka(\A^3)$, we set $H=h+f$.

\paraga Fix $r^0\in\R^3$ with $\vpi:=\om(r^0)\neq0$ and let $m=1,2$ be the rank of the resonance module $\cM_\vpi=\vpi^\bot\cap\Z^3$, 
so that the quotient $\T^3/\pi(\cM^\bot)$ is an $m$--dimensional torus. We
denote by $\cT_x\subset\T^3$ the fiber over $x\in\T^3/\pi(\cM^\bot)$, which is therefore a $(3-m)$--dimensional torus invariant
under the flow generated by $h$.

\paraga The $\cM$--averaged system $\Av_{r^0}$ at $r^0$ is defined on the cotangent bundle $T^*[\T^3/\pi(\cM^\bot)]$. 
The cotangent space at $x$ satisfies the natural identifications
$$
\big(T_{x}[\T^3/\pi(\cM^\bot)]\big)^*\simeq (\R^3/\cM^\bot)^*\simeq \langle\cM\rangle,
$$
where $\langle\cM\rangle\subset\R^3$ is the vector subspace generated by $\cM$.

The $\cM$--averaged perturbation is the function  $U_{r^0}:\T^3/\pi(\cM^\bot)\to \R$ defined by
$$
U_{r^0}(x)=\int_{\cT_x}f(\ph,r^0)\,d\mu_x(\ph),
$$
where $\mu_x$ is the induced Haar measure on $\cT_x$.
We are thus led to set
$$
\Av_{r^0}(x,y)=\pdemi\,D^2h(r^0)[y,y]+U_{r^0}(x),\qquad (x,y)\in \big(\T^3/\pi(\cM^\bot)\big)\times \langle\cM\rangle.
$$
Averaged systems are therefore classical systems on $\A^m$.
We say that  $\Av_{r^0}$ is an {\em $s$--averaged system} when $m=1$ and  a 
{\em $d$--averaged} system when $m=2$.

\paraga Fix now an adapted coordinate system $(\th,r)$ at $r^0$. Following Notation~\ref{not:splitvar},
observe that $\ov \th$ and $\ha\th$ define coordinates on the quotient $\T^3/\pi(\cM^\bot_\vpi)$ and on its fibers respectively,
and that $(\ov\th,\ov r)$ are canonically conjugated coordinates on $T^*[\T^3/\pi(\cM^\bot)]$. In these coordinates, 
the averaged system reads 
\beq\label{eq:averagedsyst}
\Av_{r^0}(\ov\th,\ov r)=\pdemi T_{r^0}(\ov r)+U_{r^0}(\ov\th),
\eeq
where $T$ is the restriction of the Hessian $D^2h(r^0)$ to the $\ov r$--space $\R^m$ and $U:\T^m\to\R$ reads
\beq\label{eq:averagedpot}
U_{r^0}(\ov \th)=\int_{\T^m}f\big((\ha\th,\ov\th),r^0\big)\,d\ha\th.
\eeq
Clearly, averaged systems associated to different adapted coordinates are linearly symplectically conjugated.


\subsection{The control parameter for double resonance points on a resonance circle}
We consider now a $C^\ka$ Tonelli Hamiltonian $h$ on $\R^3$ and its frequency map $\om=\nabla h$, together with
$f\in C_b^\ka(\A^3)$, with 
\beq\label{eq:diffmin}
\ka\geq 6, \qquad \norm{f}_\ka\leq 1.
\eeq
We fix $\e>\min h$ and an indivisible vector $k\in\Z^3$, and we set $\Ga=\om\inv(k^\bot)\cap h\inv(\e)$. We
fix a coordinate system $(\th,r)$ adapted to $\cM=\Z k$. 
We  still denote by $f$ the expression of the initial function $f$ in the coordinates $(\th,r)$, so that now $\norm{f}_\ka\leq M$,
where $M$ depends only on $k$.
The aim of this section is to discriminate
between strong and weak double resonance points on $\Ga$ for the system $H=h+f$.


\paraga {\bf The decay of Fourier coefficients of $f$.} 
For $k=(k_j)\in\Z^d$, we use the notation 
\beq
\norm{k}=\max_{1\leq j\leq d}\abs{k_j},
\qquad 
\abs{k}=\sum_{1\leq j\leq d}\abs{k_j}.
\eeq
We adopt the usual convention for multiindices 
and partial derivatives. Let us denote by 
$$
[f]_k(r)=\int_{\T^3}f(\th,r)\,^{-2i\pi\,k\cdot\th}d\th
$$ 
the Fourier coefficient of $f(\,.\,,r)$ of index $k\in\Z^3$ and set $g_k(\th,r)=[f]_k(r) e^{2i\pi k\cdot\th}$. 
Usual estimates yield, for $k\neq0$ and any multiindices $j,\ell\in\Z^3$ such that $\abs{j}+\abs{\ell}<\ka$:
\beq\label{eq:Fouriercoef}
\abs{\d_\th^j\d_r^\ell g(\th,r)}\leq\frac{M}{(2\pi)^{\ka-(\abs{j}+\abs{\ell})}\norm{k}^{\ka-(\abs{j}+\abs{\ell})}}.
\eeq
and in particular the Fourier expansion
$
f(\th,r) =\sum_{k\in\Z^3}[f]_k(r)\,e^{2i\pi\, k\cdot \th}
$
is normally convergent since $\ka>3$. Hence
\begin{equation}
f(\th,r)=\sum_{\ha k\in\Z^2}\phi_{\ha k}(\th_3,r)e^{2i\pi\,\ha k\cdot \ha\th},\quad \textrm{with}\quad
\phi_{\ha k}(\th_3,r)=\sum_{k_3\in\Z}[f]_{(\ha k,k_3)}(r)e^{2i\pi\,k_3\cdot \th_3}.
\end{equation}
Given $K\geq1$ we set 
\begin{equation}\label{eq:function}
f_{> K}(\th,r)=\sum_{\ha k\in\Z^2,\norm{\ha k}> K}\phi_{\ha k}(\th_3,r)\,e^{2i\pi\, \ha k\cdot \ha\th}
\end{equation}

\begin{lemma}\label{lem:choseK}
 Fix an integer $p\in\{2,\ldots,\ka-4\}$ and fix $\de>0$. 
 Then there exists an integer $K:=K(\de)$ such that the function $f_{> K}$ is in $C^2(\A^3)$ and satisfies 
\begin{equation}\label{eq:truncest}
\norm{f_{> K}}_{C^p(\A^3)}\leq \de.
\end{equation}
\end{lemma}

\begin{proof} Since $\norm{f}_{C^\ka}\leq M$ with $\ka\geq6$, by (\ref{eq:Fouriercoef}):
\begin{equation}\label{eq:upperbound}
\abs{\d_\th^j\d_r^\ell g(\th,r)}\leq
\frac{M}{(2\pi)^{4}\norm{k}^{4}}
\end{equation}
as soon as $\abs{j}+\abs{\ell}\leq p$. Let $K(\de)$ be the smallest integer such that 
\begin{equation}
\sum_{k\in\Z^3,\abs{k}> K(\de)}\frac{M}{(2\pi)^4\norm{k}^{4}}\leq\de.
\end{equation}
Hence $f_{> K(\de)}$ is $C^p$ and satisfies (\ref{eq:truncest}) (we do not try to give optimal estimates).
\end{proof}


\paraga {\bf The $\de$-strong double resonance points.}
Since the coordinates $(\th,r)$ are $\cM$-adapted:  $\om(r):=\nabla h(r)=(\ha \om(r),0)\in\R^2\times\R$. 
For $K\in\N$, we set 
\beq
B^*(K)=\big\{\ha k\in\Z^2\setm\{0\}\mid \norm{\ha k}\leq K\big\}.
\eeq

\begin{Def}\label{def:control}
Given a {\em control parameter} $\de>0$, we introduce the set of $\de$-strong double resonance points:
\begin{equation}\label{eq:doubres}
D(\de)=\Big\{r\in \Ga\mid  \exists\, \ha k\in B^*\big( K(\de)\big),\  \ha k \cdot \ha \om (r)=0\Big\},
\end{equation}
where $K(\de)$ was defined in {\rm Lemma~\ref{lem:choseK}}.
\end{Def}

Observe that $D(\de)$ is finite. Indeed, 
if  $\Ga_{\ha k}=h\inv(\e)\cap\om\inv((\ha k,0)^\bot)$ is the simple resonance
at energy $\e$ associated with $(\ha k,0)$, then 
$$
D(\de)=\bigcup_{\ha k\in\bez\big( K(\de)\big)} \Ga\cap \Ga_{\ha k}
$$
and each $\Ga\cap \Ga_{\ha k}$ contains exactly two points, which proves our claim. Note
that $D(\de)$ increases when $\de$ decreases.


\subsection{The nondegeneracy conditions \CS}

We consider  a $C^\ka$ Tonelli Hamiltonian $h$ on $\R^3$,  $\ka\geq2$, and set $\om=\nabla h$.
We fix $\e>\Min h$. Let $k\in\Z^3\setm\{0\}$ be an indivisible vector and set $\Ga_k=\om\inv(k^\bot)\cap h\inv(\e)$.
Given $f\in C_b^\ka(\A^3)$ satisfying~(\ref{eq:diffmin}), we now set out a list of nondegeneracy conditions 
involving the averaged systems attached to $H=h+f$ at the points of $\Ga_k$.
\begin{itemize}
\item \CSu\ {\em  There exists a finite subset $B\subset \Ga_k$ 
such that for $r^0\in\Ga_k\setm B$  the $s$-averaged potential function 
$V_{r^0}:\T\to\R$ admits a single  global maximum,  
which is nondegenerate, and for $r^0\in B$ the function $V_{r^0}$ admits 
exactly two  global maximums, which are nondegenerate.}
\end{itemize}
The nondegeneracy condition on $V_{r^0}$ is to be understood in the Morse sense, 
that is, the second derivative of $V_{r^0}$ at a nondegenerate point is nonzero.
The elements of $B$ will be called {\em bifurcation points}.
To state the next condition, note that each point $r^0$ in $B$ admits a neighborhood $I(r^0)$ in
$\Ga_k$ such that when $r\in I(r^0)$, the averaged potential $V_{r}$ admits two (differentiably varying) 
nondegenerate local maximums $m^*(r)$ and $m^{**}(r)$. 
The second condition is a transversal crossing property at a bifurcation point.
\begin{itemize}
\item \CSd\ {\em  For any $r^0\in B$,  the derivative $\tfrac{d}{dr}\big(m^*(r)-m^{**}(r)\big)$ does not vanish 
at $r^0$.}
\end{itemize}
The next condition focuses on the  double resonance points
contained in $\Ga_k$.
Given such an $r^0$, let
\beq\label{eq:davsyst}
\Av_{r^0}(\ov r,\ov\th)=\pdemi T_{r^0}(\ov r)+U_{r^0}(\ov\th),
\eeq
be the $d$--averaged system at $r^0$ in an adapted coordinate system for the resonance module of $\om(r^0)$.
\begin{itemize}
\item {\bf($\bf S_3$)}\ {\em For every double resonance point $r^0\in\Ga_k$, the potential $U_{r^0}$ belongs to the residual 
set $\jU(T_{r^0})$ of {\rm Theorem II}.}
\end{itemize}

Condition {\bf($\bf S_3$)}  is independent of the choice of the adapted system at $r^0$, by symplectic conjugacy.
We say that $H$ satisfies conditions {\bf(S)} on $\Ga_k$ when it satisfies the previous three conditions.


\section{The cylinders}\label{sec:cylinders}
This section contains definitions and statements only, the proofs are postponed to the next one. 
We fix once and for all a Tonelli  Hamiltonian $h\in C^\ka(\R^3)$,  and an energy $\e>\Min h$, 
together with a resonance circle $\Ga\subset h\inv(\e)$.
We denote  by $\Pi:\A^3\to\R^3$ the natural projection and by $\bd$ the Hausdorff distance between
compact subsets of $\R^3$.
The main result of this section is the following.

\begin{prop}\label{prop:existcylinders}
Fix  $f\in C_b^\ka(\A^3)$ and set $H_\eps=h+\eps f$. Assume that $H:=H_1$ satisfies \CS\ along
$\Ga$. Then for $\ka\geq\ka_0$ large enough,
there exists  $\eps_0>0$ such that for  $0<\eps\leq\eps_0$,  there is a finite sequence
$\big(\jC_k(\eps)\big)_{0\leq k\leq k_*}$ of normally hyperbolic invariant cylinders and singular cylinders
at energy $\e$ for $H_\eps$, whose projection by $\Pi$ satisfies
$$
\bd\Big(\bigcup_{1\leq k\leq k_*}\jC_k(\eps),\Ga\Big)=O(\sqrt\eps).
$$ 
\end{prop}

The cylinders in fact enjoy more stringent ``graph properties'' which will enable us to prove
that they form {\em chains} in Section~\ref{sec:chains}.
In the rest of this section we describe the intermediate steps to prove the previous proposition.
We start with the ``$d$-cylinders'' in the neighborhood of the points of the set $D(\de)$ of $\de$-strong 
double resonance points introduced in Definition~\ref{def:control},  where $\de$ has to be suitably chosen,
and we ``interpolate between them''  with ``$s$-cylinders'' along the complementary arcs of $\Ga$, 
taking the bifurcation points into account.


\subsection{The $d$-cylinders at a double resonance point}\label{sec:mainresd}
In this section we fix an {\em arbitrary} double resonance point $r^0\in\Ga$,  
with $d$-averaged system $C$, and we set out precise definitions for the {\em $d$-cylinders} in the neighborhood of $r^0$.
We will introduce three different families of such $d$-cylinders, according to the way they are constructed. 
Let $(\th,r)$ be adapted coordinates at $r^0$, so that $\om(r^0)=(\om_1,0,0)$ with $\om_1\neq0$.

\vskip1mm
$\bu$ The notion of $2$-dimensional annulus for a classical system was introduced in 
Definition~\ref{def:singann}. Given a compact annulus $\sA$ for $C$, the product of $\sA$ with
``the circle of $\th_1$'' is a normally hyperbolic compact cylinder, which is invariant for a suitable truncation of $H_\eps$.
We will prove that it remains invariant for $H_\eps$, provided $\eps$ is small enough. 
The family of such normally hyperbolic cylinders, attached with all compact annuli of $C$, constitutes our first family of
$d$-cylinders.

\vskip1mm
$\bu$ The construction of the second family is similar to the previous one, but the starting point is a singular annulus
(see Definition~\ref{def:singann}) rather than a compact annulus. The normally hyperbolic objects obtained
this way are singular cylinders.

\vskip1mm
$\bu$ The third family is formed by suitable continuations of the cylinders attached to the annuli
of $C$ which are defined over intervals of the form $[e_P,+\infty[$, we call them {\em extremal cylinders}. 
They will enable us to define the ``$s$-cylinders''  between to two consecutive distinct points of $D(\de)$,
as cylinders containing two suitable extremal cylinders, located in the neighborhood of both
points of $D(\de)$ (see Section~\ref{ssec:scyldef}).

\vskip1mm
Three corresponding existence results are stated, which will be proved in the next section.

\subsubsection{The cylinders attached to a compact $2$-annulus of the $d$-averaged system}
We consider the system $H_\eps=h+\eps f$ and set $H:=H_1$.
We perform a translation in action so that $r^0=0$, without loss of generality.

\paraga We will have to use several
coordinate transformations. To avoid confusion,  we fix an initial coordinate system $(x,y)$ 
adapted to the double resonance point $0$. Hence,  
relatively to these coordinates,
$
\nabla h(0)=(\ha\om,0)\in(\R\setm\{0\})\times\R^2,
$
(where $\ha\om=\om_1$).
With the usual notational convention, the $d$-averaged system associated with $H$ at $0$ reads
\beq\label{eq:classpec}
C(\ov x,\ov y)=\pdemi T(\ov y)+U(\ov x), \qquad (\ov x,\ov y)\in\T^2\times\R^2,
\eeq
where $T$ is the restriction of the Hessian $D^2h(0)$ to the $\ov y$-plane and
\beq\label{eq:quadpot}
U(\ov x)=\int_{\T} f\big((\ha x,\ov x),0\big)\,d\ha x.
\eeq
We also introduce  the {\em complementary part} of the Hessian $D^2h(0)$:
\beq\label{eq:comppart}
Q(y)=\pdemi\big(D^2h(0)y^2-\partial^2_{\ov y}h(0)\ov y^2\big):=y_1L(y),
\eeq
so that $L$ is a linear form on $\R^3$.

\paraga The $d$-cylinders will be conveniently defined relatively to appropriate normalized coordinates, 
that we now introduce. Let us set, for $\eps>0$
$$
\sig_\eps(\th,\sr)=(\th,\sqrt\eps\sr),\qquad (\th,\sr)\in\A^3.
$$

\begin{Def} Fix $d^*>0$, $\sig\in\,]\demi,1[$ and two integers $p,\ell\geq2$. Given $\eps>0$, 
a {\em normalizing diffeomorphism
with parameters $(d^*,\sig,p,\ell)$}
is an analytic embedding 
\beq
\Phi_\eps=\Psi_\eps\circ\sig_\eps: \T^3\times B^3(0,d^*)\to \T^3\times B^3(0,2d^*\sqrt\eps)
\eeq
where $\Psi_\eps: \T^3\times B^3(0,d^*\sqrt\eps)\to \T^3\times B^3(0,2d^*\sqrt\eps)$ is symplectic
and satisfies $\norm{\Phi_\eps-\Id}_{C^0}\leq\eps^\sig$,
such that for $(\th,\sr)\in \T^3\times B^3(0,d^*)$:
\begin{equation}\label{eq:scaleham}
\sN_\eps(\th,\sr):=
\frac{1}{\eps}\, \Big(H_\eps\circ\Psi_\eps(\th,\sr)-\e\Big)
=\frac{\ha\om}{\sqrt\eps}\,\ha\sr+Q(\sr)+C(\ov\th,\ov\sr)
+\sR^0_\eps(\ov\th,\sr)+\sR_\eps(\th,\sr).
\end{equation}
The functions $C$ and $Q$ are defined in (\ref{eq:classpec}) and (\ref{eq:comppart}),
and  $\sR^0_\eps$ and $\sR_\eps$ are $C^p$ functions on $\T^2\times B^3(0,d^*)$ and $\T^3\times B^3(0,d^*)$
respectively, which satisfy
\beq\label{eq:estimrem}
\norm{\sR^0_\eps}_{C^p
}\leq C \sqrt \eps,\qquad 
\norm{\sR_\eps}_{C^p
}\leq \eps^\ell.
\eeq
for a suitable $C>0$.
\end{Def}

We will adopt the notation $(\th,r)$ for the symplectic coordinates such that  $\Psi_\eps(\th,r)=(x,y)$,
so that the nonsymplectic rescaling reads $\sig_\eps(\th,\sr)=(\th,r)$.
The evolution time for the
normal form $\sN_\eps$ has also to be rescaled, which will is here innocuous since we are interested only 
in geometric objects.

\paraga We are now in a position to define the  $d$-cylinder attached to a compact $2$-annulus of $C$.
Fix such an annulus $\sA$, defined over a compact interval $J$. 
Since the periodic orbits in $\sA$ are hyperbolic in their energy level, $\sA$ can be continued
to a slightly larger family of hyperbolic orbits, the union of which we denote by $\sA_*$, and one can moreover 
assume that their period satisfy the same mononicity assumption as for $\sA$. Then, basic angle-action transformations 
prove the existence of an open interval $J^*$ containing $J$
and a symplectic embedding 
\beq\label{eq:embedj}
\j: \T\times J^*\to \A^2,\qquad \j(\T\times J)=\sA, \qquad \j(\T\times J_*)=\sA_*,
\eeq
such that, if $(\ph,\rho)\in\T\times J^*$ are the standard symplectic coordinates:
\beq
C\circ \j(\ph,\rho)=\rho.
\eeq
We say that $(J\subset J^*,\j)$ is a normalizing system for $\sA$. 

\begin{Def}\label{def:scyl}
 Fix an annulus $\sA$ of $C$ with normalizing system $(J\subset J^*,\j)$ and contained in $\T^2\times B^2(0,d^*)$
for some $d^*>0$
\begin{itemize}
\item A {\em $4$-annulus of class $C^p$ attached to $\sA$} for $H_\eps$ is a $C^p$ invariant normally hyperbolic $4$-annulus
$\jA_\eps\subset\A^3$ for the vector field $X_{H_\eps}$, 
such that there exists a $d$-normalizing diffeomorphism $\Phi_\eps$ and  a neighborhood $J'\subset\R$ of $J$
in $J^*$
for which $\Phi_\eps\inv(\jA_\eps)$ contains a graph over the domain 
\beq\label{eq:domain}
\th_1\in\T,\ \sr_1\in ]-d^*\sqrt\eps,d^*\sqrt\eps[,\ \ph\in\T,\ \rho\in J',
\eeq
of the form
\beq\label{eq:graph}
u=U_\eps(\th_1,\sr_1,\ph,\rho), \ s=S_\eps(\th_1,\sr_1,\ph,\rho),
\eeq
where $U_\eps$ and $S_\eps$ are $C^p$ functions which tend to $0$ in the $C^p$-topology when $\eps\to0$.
\item A  $d$-cylinder at energy $\e$ attached to $\sA$ for $H_\eps$ is a (compact and normally hyperbolic) cylinder $\jC_\eps$ 
invariant for the vector field $X_{H_\eps}$, such that there exists a $4$-annulus attached to $\sA$ with 
$\jC_\eps\subset\jA_\eps\cap H_\eps\inv(\e)$, and such that
{\em the projection $\Pi_\rho\big(\Phi_\eps\inv(\jC_\eps)\big)$ on the 
$\rho$-axis contains the interval $J$}.
\item A twist section for such a $d$-cylinder is a global $2$-dimensional transverse section $\Sig\subset \jC_\eps$,
image of a symplectic embedding $\j_\Sig: \T\times [a,b]$, such that the associated Poincar\'e return map is a twist
map in the $\j_\Sig$-induced coordinates on $\T\times [a,b]$.
\end{itemize}
\end{Def}

We refer to Section~\ref{sec:normhyp} for the definition of normally hyperbolic cylinders and associated $4$-annuli.
Note that the constraint 
$\sr_1\in ]-d^*\sqrt\eps,d^*\sqrt\eps[$ yields the localization $r_1\in ]-d^*\eps,d^*\eps[$, which is very stringent.




\paraga Our first existence result is the following.

\begin{lemma}\label{lem:dcyl} Assume $\ka\geq\ka_0$ large enough. Then
for each compact annulus $\sA$ of $C$, there is an $\eps_0>0$ such that for $0<\eps\leq\eps_0$
there exists a $d$-cylinder $\jC_\eps$ at energy $\e$ attached to $\sA$ for $H_\eps$, which
admits a twist section.
\end{lemma}

The proof of Lemma~\ref{lem:dcyl} is in Section~\ref{sec:prooflemdcyl}.

%

\subsubsection{The singular $d$-cylinders at a double resonance point}\label{sec:singularcyl}
\setcounter{paraga}{0}
The definition and existence result for the singular cylinder is very similar to the previous ones.

\begin{Def}\label{def:singann2}
 Let $\sA_\bu$ be a singular annulus for $C$.
\begin{itemize}
\item A {\em singular annulus}Ê attached to $\sA_\bu$ for $H_\eps$ 
is a normally hyperbolic singular $4$-annulus $\jA_\bu(\eps)\subset\A^3$ for the vector field $X_{H_\eps}$, 
such that there exists a $d$-normalizing diffeomorphism $\Phi_\eps$ for which $\Phi_\eps\inv(\jA_\bu(\eps))$ 
tends to the product 
\beq
\bA_\eps:=\big(\T\times\,]-d^*\sqrt\eps,d^*\sqrt\eps[\big)\times \sA_\bu
\eeq
in the $C^1$ topology when $\eps\to0$. More precisely, there exists $\sig\in\,]\pdemi,1]$ and a $C^1$-embedding $\chi_\eps$
defined on $\bA_\eps$, with image $\Phi_\eps\inv(\jA_\bu(\eps))$,
which satisfies
\beq
\norm{\chi_\eps-\chi}_{C^1}\leq \eps^\sig
\eeq
where $\chi$ is the canonical embedding $\bA_\eps\hookrightarrow\A^3$.
\item A  singular $d$-cylinder at energy $\e$ attached to $\sA_\bu$ for $H_\eps$ is a singular cylinder $\jC_\bu(\eps)$ 
for the vector field $X_{H_\eps}$, such that there is a singular annulus $\jA_\bu(\eps)$  with 
$\jC_\bu(\eps)\subset\jA_\bu(\eps)\cap H_\eps\inv(\e)$.
\item A generalized twist section for such a singular $d$-cylinder is a singular $2$-annulus which admits
a continuation to a $2$-annulus, on which the Poincar\'e return map continues to a twist map.
\end{itemize}
\end{Def}

We refer to Section~\ref{sec:normhyp} for the definition of normally hyperbolic singular cylinders and associated singular $4$-annuli.
Again, note that the definition of $\bA_\eps$ and the convergence property
yields a very precise  localization for the singular annuli.

\begin{lemma}\label{lem:singcyl} Assume $\ka\geq\ka_0$ large enough. Then
given a singular 2-annulus $\sA_\bu$ of $C$, there is an $\eps_0>0$ such that for $0<\eps\leq\eps_0$
there exists a singular $d$-cylinder $\jC_\eps$ at energy $\e$ attached to $\sA_\bu$ for $H_\eps$,
which admits a generalized twist section.
\end{lemma}

The proof of Lemma~\ref{lem:singcyl} is in Section~\ref{sec:prooflemsingcyl}.

\subsection{The extremal $d$ cylinders and the interpolating $s$-cylinders}\label{ssec:scyldef}
\setcounter{paraga}{0}
In this section we come back to the initial assumptions of Proposition~\ref{prop:existcylinders}. 
We endow $\Ga$ with an arbitrary orientation and,
given $\de>0$, we fix two consecutive elements $m^0$ and $m^1$ of $D(\de)$ on $\Ga$
according to that orientation. Let $[m^0,m^1]$ be the segment of $\Ga$ they delimit (also
according to that orientation).

\paraga We introduce adapted coordinate systems $(x^i,y^i)$ at $m^i$
relatively to which $\om(m^i)=(\om^i_1,0,0)$ with $\om^i_1\neq0$, and which
both satisfy 
$$
\Ga=\{m\in h\inv(\e)\mid \om^i_3(m)=0\}.
$$
For $m\in\, ]m^0,m^1[$, set
$$
\sig^i=\frac{\om^i_1(m)}{\abs{\om^i_1(m)}}\in\{-1,1\}.
$$ 
Let  $C_i=\pdemi T_{m^i}+U_{m^i}$ be the $d$-averaged systems at $m^i$ relatively to the previous
coordinate systems. Identify $H_1(\T^2,\Z)$ with $\Z^2$ relatively to the same systems, and let
$$
c^i=(\sig^i,0)\in H_1(\T^2,\Z).
$$

\paraga We can now define the extremal $d$-cylinders at the point $m\in\{m^0,m^1\}$ relatively to the
resonance circle $\Ga$.
We  fix the corresponding integer homology class $c\in \{c^0,c^1\}$ and consider a compact 
$2$-annulus $\sA$ of $C$
defined over the interval $J=[e_P,e_\ell]$, where $e_P$ is the Poincar\'e energy for $c$  (see Theorem II).
We want to continue the $d$-cylinders  attached to $\sA$ 
(and the corresponding annuli containing them) ``away from the double resonance and along $\Ga$'', 
to a distance 
$O(\eps^\nu)$ where $\nu\in\,]0,\pdemi[$ can be arbitrarily chosen 
(provided that the regularity $\ka$ is large enough).

To state our result properly, we need to distinguish between the two components of the boundary 
of the $d$-cylinders $\jC_\eps$ attached to $\sA$, introduced in Definition~\ref{def:scyl}. With the notation 
of Lemma~\ref{lem:dcyl}, 
let $(J\subset J^*,\j)$ be a normalizing system for $\sA$ and set $J^*=\,]e_P^*,e_\ell^*[$, so that $e_P^*<e_P$
and $e_\ell^*>e_\ell$.
Since the projection $\Pi_\rho\big(\Phi_\eps\inv(\jC_\eps)\big)$ on the $\rho$-axis contains the 
interval $J$, one can define the {\em inner component} $\d_{in} \jC_\eps$ of $\d \jC_\eps$ as the 
one whose corresponding projection  intersects $]e_P^*,e_P]$, and the {\em outer component} $\d_{out} \jC_\eps$ as 
the one whose corresponding projection  intersects $[e_\ell,e_\ell^*[$.

\begin{figure}[h]
\begin{center}
\begin{pspicture}(0cm,3.2cm)
\rput(3,1.3){
\psset{xunit=1cm,yunit=.43cm}
\psellipse[linewidth=.3mm](0,0)(.5,3)
\pscurve(.1,1)(.05,0)(.1,-1)
\pscurve(.1,.5)(.15,0)(.1,-.5)
\psellipse[linewidth=.3mm](-6,0)(.5,3)
\psframe[fillstyle=solid,fillcolor=white,linecolor=white](-6,-3.3)(-4.8,3.3)
\rput(2,0){
\psellipse[linewidth=.3mm](-6,0)(.5,3)
\psframe[fillstyle=solid,fillcolor=white,linecolor=white](-6,-3.3)(-4.8,3.3)
}
\psline(-6,3)(0,3)
\psline(-6,-3)(0,-3)
\rput(-5,-3.5){$\jC_\eps$}
\rput(-3,3.8){$\jC^{ext}_\eps$}
\rput(-7,0){$\d_{in}\jC_\eps$}
\rput(-3.8,0){$\d_{out}\jC_\eps$}
\psline(-10,-4.5)(2,-4.5)
\rput(-10,-5.2){$r^0$}
\rput(-5,-5.2){$O(\sqrt\eps)$}
\rput(0,-5.2){$\geq C\eps^\nu$}
}
\end{pspicture}
\vskip10mm
\caption{An extremal cylinder}
\end{center}
\end{figure}
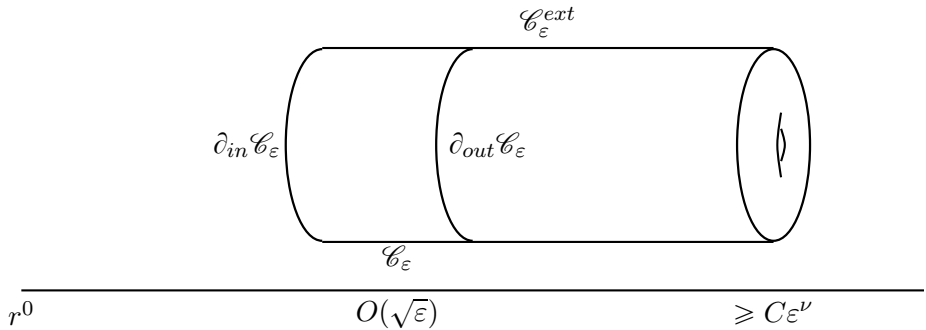

Let $(x,y)$ be the adapted coordinates at $m$. The resonance surface $\om\inv(k^\bot)$ admits the graph
representation
$$
y_3=y_3(\ha y)
$$
and we assume that the $s$-averaged potential
$$
U(\cdot ,y)=\int_{\T^2}f\big((\ha x,\cdot),r)d\ha x\quad :\T\to\R
$$
admits in the neighborhood of $y(m)$ a unique and nondegenerate maximum at $x_3^*(y)$.

\begin{lemma}\label{lem:extcyl} 
Fix $\nu\in\,]0,\pdemi[$ and constants $b>a>0$, $\mu>0$. Then for $\ka\geq\ka_0$ large enough, there exist $\eps_0>0$
such that for $0<\eps_0<\eps$, there exists a cylinder $\jC^{\rm ext}_\eps$
which continues $\jC_\eps$ in the sense that:
\vskip1mm $\bu$ $\jC_\eps\subset \jC^{\rm ext}_\eps$,
\vskip1mm $\bu$ one component of the boundary $\d\jC^{\rm ext}_\eps$
coincide with the inner component $\d_{inn}\jC_\eps$ points
\vskip 1mm $\bu$ the other component of $\d\jC^{\rm ext}_\eps$ is also  a component of the boundary of an {\em invariant} cylinder
contained in $\jC^{\rm ext}_\eps$ and located in the domain
\beq
\ha x\in\T^2,\quad a\eps^\nu\leq \norm{\ha y-\ha y^0}\leq b\eps^\nu,\quad \abs{x_3-x^*_3(m)}\leq \mu,\quad \abs{y_3-y_3^*(m)}\leq\mu\sqrt\eps.
\eeq
\end{lemma}

The proof of Lemma~\ref{lem:extcyl} is in Section~\ref{sec:prooflemextcyl}.

\paraga We can now define in a simple way the $s$-cylinder which ``interpolates'' between
the previous extremal cylinders $\jC^{\rm ext}_\eps(m^0)$ and $\jC^{\rm ext}_\eps(m^1)$.
We say that a cylinder $\jC$ is {\em oriented} when an order on its boundary components has
been fixed, we denote by $\d_\bu\jC$ the first one and by $\jC^\bu$ the second one. 
We say that two cylinders $\jC_0$, $\jC_1$ contained in a cylinder are {\em consecutive} when
$\d^\bu\jC_0=\d^\bu\jC^1$. 

\begin{Def}
An $s$-cylinder at energy $\e$ ``connecting $m^0$ and $m^1$ along $\Ga$'' is a normally cylinder $\jC_\eps$ at
energy $\e$ for $H_\eps$ which contains both extremal cylinders $\jC^{\rm ext}_\eps(m^0)$ and 
$\jC^{\rm ext}_\eps(m^0)$, and whose projection in action is located in a tubular neighborhood of radius 
$O(\sqrt\eps)$ of $\Ga$.

A twist section for an invariant cylinder $\ha\jC\subset\jC_\eps$ is a $2$-dimensional global section $\Sig\subset \ha\jC$,
transverse to $X_H$ in $\jC$, which is the image 
of some exact-symplectic embedding
$\j_{\Sig}:\T\times [a,b]\to \Sig$, such that the Poincar\'e return map associated with $\Sig$ is a twist
map in the $\j_{\Sig}$-induced coordinates on $\T\times [a,b]$.

We say that $\jC_\eps$ satifies the twist property when it admits a finite covering $\ha\jC_1,\ldots,\ha \jC_{\ell(\eps)}$
by consecutive subcylinders, each of which admits a twist section in the previous sense, and whose boundaries
are dynamically minimal.
\end{Def}

\begin{figure}[h]
\begin{center}
\begin{pspicture}(0cm,3.2cm)
\rput(3,1.3){
\psset{xunit=1cm,yunit=.3cm}
\rput(2,0){
\psellipse[linewidth=.3mm](0,0)(.5,3)
\pscurve(.1,1)(.05,0)(.1,-1)
\pscurve(.1,.5)(.15,0)(.1,-.5)
}
\psellipse[linewidth=.3mm](-9,0)(.5,3)
\psframe[fillstyle=solid,fillcolor=white,linecolor=white](-9,-3.3)(-4.8,3.3)
\rput(2,0){
\psellipse[linewidth=.3mm](-10,0)(.5,3)
\psframe[fillstyle=solid,fillcolor=white,linecolor=white](-10,-3.3)(-4.8,3.3)
\psellipse[linewidth=.3mm](-1,0)(.5,3)
\psframe[fillstyle=solid,fillcolor=white,linecolor=white](-1,-3.3)(-0.5,3.3)
}
\psline(-9,3)(2,3)
\psline(-9,-3)(2,-3)
\rput(-3.5,4){$\jC^{s}_\eps$}
\rput(-8.5,4){$\jC^{ext}_\eps(m^0)$}
\rput(1.5,4){$\jC^{ext}_\eps(m^1)$}
\psline(-10,-4.5)(3,-4.5)
\rput(-10,-5.2){$m^0$}
\rput(3,-5.2){$m^1$}
\rput(1,-5.5){$O(\eps^\nu)$}
\psline(1,-4.2)(1,-4.8)
\rput(-8,-5.5){$O(\eps^\nu)$}
\psline(-8,-4.2)(-8,-4.8)
}
\end{pspicture}
\vskip5mm
\caption{An $s$-cylinder}
\end{center}
\end{figure}
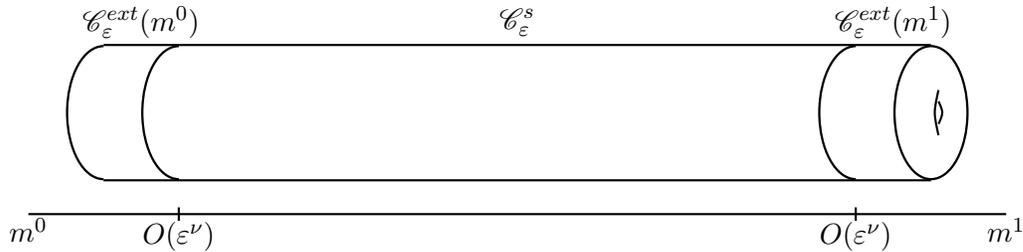

\vskip-3mm
\paraga  The corresponding existence result is the following.

\begin{lemma}\label{lem:scyl} Assume that $\ka\geq\ka_0$ large enough. Then
there is an $\eps_0$ such that there exists a family $(\jC_\eps)_{0<\eps<\eps_0}$ of $s$-cylinders
at energy~$\e$ connecting $m^0$ and $m^1$, which satsifies the twist property.
\end{lemma}

The proof of Lemma~\ref{lem:scyl} is in Section~\ref{sec:prooflemscyl}


\section{Proof of the results of Section~\ref{sec:cylinders}}\label{sec:proofscyl}

We successively prove Lemma~\ref{lem:dcyl}, Lemma~\ref{lem:singcyl}
which rely on the $\eps$-dependent normal forms of Appendix~\ref{app:normformepsdep}. We then deduce
Lemma~\ref{lem:scyl} and  Lemma~\ref{lem:extcyl} from the global normal form of Appendix~\ref{App:globnormforms},
for the former, and from the previous $\eps$-dependent normal forms for the latter. 


\subsection{Proof of Lemma~\ref{lem:dcyl}}\label{sec:prooflemdcyl}
We first prove the existence of normalizing diffeomorphisms. We then prove the existence of $d$--annuli by
the persistence theorem of Appendix~\ref{app:normhyp}, applied to the 
normal form $\sN_\eps$ of equation (\ref{eq:scaleham}). The
intersection of a $d$--annulus with $\sN_\eps\inv(0)$ is a pseudo invariant cylinder which admits a ``twist section''
(which will be naturally defined even in this pseudo invariant context).
The invariant curve theorem (in the version by Herman)  applied to the twist section proves the existence
of a large family of isotropic $2$--dimensional invariant tori inside the pseudo invariant cylinders.
The zone limited by two of them is a compact invariant cylinder and one can choose these tori close enough to
the ``ends'' of the pseudo invariant cylinders to prove our statement.


\subsubsection{Existence of the $d$--normalizing diffeomorphisms}\label{sec:dnormdiff}
In this section we prove the following lemma.

\begin{lemma}\label{lem:dnormalizing}
Fix a set of parameters $(d^*,\sig,p,\ell)$. Then there exists $\ka_0$ such then for $\ka\geq\ka_0$,
there exists a $d$--normalizing diffeomorphism with parameters $(d^*,\sig,p,\ell)$.
\end{lemma}

\begin{proof}
We apply Proposition \ref{prop:normal2} to our Hamiltonian $H_\eps$ at $0$, with $\ell+1$ in place of $\ell$.
We arbitrarily fix $d<\pdemi$ and $\sig<1-d$ (we do not try to get optimal results here). 
Therefore one can choose $\ka\in\N^*$ large enough 
such that, given any $d^*>0$,  there is an $\eps_0>0$ so that,  for $0\leq\eps\leq \eps_0$, there exists an analytic 
symplectic embedding 
$$
\Phi_\eps: \T^3\times B^3(0,d^*\sqrt\eps)\to \T^3\times  B^3(0,2d^*\sqrt\eps)
$$
such that for $(\th,r)\in \T^3\times B^3(0,d^*\sqrt\eps)$, 
$$
N_\eps(\th,r)=H_\eps\circ\Phi_\eps(\th,r)=h(r)+g_\eps(\ov \th,r)+R_\eps(\th,r),
$$
where $g_\eps$ is analytic  on $\T^2\times B(0,d^*\sqrt\eps)$ and  $R_\eps$ is 
$C^\ka$  on $\T^3\times B^3(0,d^*\sqrt\eps)$, with 
\beq\label{eq:estim2}
\norm{g_\eps-\eps[f]}_{C^p\big( \T^2\times B^3(0,d^*\sqrt\eps)\big)}\leq \eps^{2-\sig},\qquad
\norm{R_\eps}_{C^p\big( \T^3\times B^3(0,d^*\sqrt\eps)\big)}\leq \eps^{\ell+1}.
\eeq
The rescaling $r =\sqrt\eps\, \sr$, $t=\frac{1}{\sqrt\eps}\,{\bt}$ yields the new Hamiltonian
\begin{equation}
\sN_\eps(\th,\sr)=\frac{1}{\eps}\, N_\eps(\th,{\sqrt\eps}\,\sr).
\end{equation}
Then, performing a Taylor expansion of $h$ and $[f]$ at $r^0=0$ in $N_\eps$, one gets the normal form:
\begin{equation}
\sN_\eps(\th,\sr)=\frac{1}{\eps}N_\eps(\th,\sqrt\eps\sr)= \frac{\ha \om}{\sqrt\eps}\,\sr_1+\pdemi D^2h(0)\,\sr^2+ U(\ov\th)
+\sR^0_\eps(\ov\th,\sr)+\sR_\eps(\th,\sr)
\end{equation}
where $\sR_\eps(\th,\sr)=\frac{1}{\eps}R_\eps(\th,\sqrt\eps\sr)$ is a $C^\ka$ function on $\T^3\times B^3(0,d^*)$ such that
$$
\norm{\sR_\eps}_{C^p(\T^n\times B(0,d^*))}\leq \eps^\ell.
$$
Moreover
$$
\sR^0_\eps(\ov\th,\sr)=
\frac{1}{\eps}\big(g_\eps(\ov\th,\sqrt\eps\sr)-\eps U(\ov\th)\big)+\frac{\eps^{1/2}}{2}\int_0^1(1-s)^2D^3h(s\sqrt\eps\sr)(\sr^3)\,ds.
$$
Note that
$$
\frac{1}{\eps}\Big(g_\eps(\ov\th,\sqrt\eps\sr)-\eps U(\ov\th)\big)=
\frac{1}{\eps}\big(g_\eps(\ov\th,\sqrt\eps\sr)-\eps [f](\ov\th,\sqrt\eps\sr)\Big)
-\Big([f](\ov\th,\sqrt\eps\sr)-[f](\ov\th,0)\Big),
$$
so that, for a suitable constant $a>0$:
$$
\norm{\sR^0_\eps}_{C^p(\T^2\times B^3(0,d^*))}\leq a\sqrt \eps.
$$ 
This concludes the proof.
\end{proof}


\subsubsection{Existence of the annuli}
\setcounter{paraga}{0} 
In this section we deal with the normal form $\sN_\eps$ of (\ref{eq:scaleham}).
We fix a compact annulus $\sA$ for $C$, defined over $J$.
By hyperbolic continuation of periodic orbits together with the torsion assumptions, $\sA$ can be continued
to an annulus $\sA^*$ (satisfying the same torsion properties as $\sA$), defined over a slightly larger open interval $J^*$. 
 Using the Moser isotopy argument, one then proves the
existence of a neighborhood $\jO_\sA$ of $\sA$ in $\A^2$ and an $\al>0$ such that, setting 
\beq
\jO(J^*,\al):=\T\times J^*\times\,]-\al,\al[^2
\eeq
there exists a ``straightening'' symplectic diffeomorphism 
\beq
\left\vert
\begin{array}{lll}
\phi: \jO(J^*,\al)\longrightarrow \jO_A,\\[3pt]
\phi(\ph,\rho,0,0)=\j(\ph,\rho),\qquad \forall (\ph,\rho)\in\T\times J^*,\\[3pt]
\phi\inv\big(W_{loc}^s(\sA^*)\big)=\{u=0\},\qquad \phi\inv\big(W_{loc}^u(\sA^*)\big)=\{s=0\},\qquad (u,s)\in \,]-\al,\al[^2.\\
\end{array}
\right.
\eeq
so that, in particular, $\phi\big(\T\times J^*\times\{(0,0)\}\big)=\sA^*$.
We say that the symplectic embedding $\phi$ and the coordinates $(\ph,\rho,u,s)$ are adapted to $\sA$ (or $\sA^*$).
The main result of this section is the following, where we set $\T_\eps:=\R/\sqrt\eps\Z$.

\begin{lemma}\label{lem:existann} 
Assume that $\sN_\eps$ satisfies  (\ref{eq:scaleham}) and (\ref{eq:estimrem}), and let $q\geq1$ be a fixed integer and 
 $a$ be a positive constant.
Then if $\ell$ is large enough, for any open interval $J^\bu$ such that
$J\subset J^\bu\subset J^*$, there is an $\eps_0>0$ such that for $0<\eps\leq\eps_0$ there is a symplectic $C^p$ embedding 
\beq\label{eq:embed}
\left\vert
\begin{array}{ll}
\Psi_\eps: \T_\eps\times [-1,1]\times \T\times J^\bu\longrightarrow\A^3\\
\Psi_\eps(\xi,\eta,\ph,\rho)=\Big(\tfrac{1}{\sqrt\eps}\xi,\sqrt\eps\,\eta,\, j_\eps\big(\xi,\eta,\ph,\rho\big)\Big),
\qquad 
\norm{j_\eps-j}_{C^p(\R\times [-1,1]\times\T\times J)}\leq a\sqrt\eps,
\end{array}
\right.
\eeq
whose image is a pseudo invariant hyperbolic annulus for $\sN_\eps$.
Here  $\T_\eps\times [-1,1]\times \T\times J$ is equipped with the symplectic form $d\,\eta\wedge d\xi +d\rho\wedge d\ph$.
The vector field $\Psi_\eps^*X^{\sN_\eps}$ is generated  by the Hamiltonian
\beq
\sM_\eps(\xi,\eta,\ph,\rho)
=\ha\om\,\eta + \bC_0(\rho) + \La^0_\eps(\eta,\rho)+
\La_\eps\Big(\tfrac{1}{\sqrt\eps}\xi,\eta,\ph,\rho\Big),
\eeq
where $\La_\eps$ is $1$--periodic in its first variable and in $\ph$, and
\beq
\norm{\La^0_\eps}_{C^p}\leq a \sqrt\eps,
\qquad \norm{\La_\eps}_{C^p}\leq a \eps^q.
\eeq
\end{lemma}

\begin{proof} We will proceed in three steps to take into account the partial integrability of the system $\sN_\eps$ up to 
the extremely small term $\sR_\eps$.  

\vskip3mm

\noindent$\bu$ {\bf First step.} Using a suitable rescaling in $(\ha\th,\ha\sr)$, we will first prove the existence of a hyperbolic annulus 
for the {\em truncated} normal form $\sN_\eps^0=\sN_\eps-\sR_\eps$ (see (\ref{eq:scaleham})). 

\paraga From the definition of the adapted embedding and coordinates,
it is easy to check that the composed Hamiltonian $\bC=C\circ \phi$ takes the form
\beq\label{eq:hambC}
\bC(\ph,\rho,u,s)=\bC_0(\rho)+\la(\ph,\rho)\,us+\bC_3(\ph,\rho,u,s)
\eeq
with
$
\d^{i+j}_{u^is^j}\bC_3(\ph,\rho,0,0)=0$
for
$
0 \leq i+j\leq 2
$
and
$
\la(\ph,\rho)\geq\la_0>0
$
by compactness of the closure $\ov {J^*}$. 
As a consequence, the vector field $X^\bC=\phi^*X^C_{\vert \jO}$  reads:
\begin{equation}\label{eq:avsys}
\left\vert
\begin{array}{lll}
\ph'=\varpi(\rho)+K^\ph(\ph,\rho,u,s)\\
\rho'=K^\rho(\ph,\rho,u,s)\\
u'=\la(\ph,\rho)\,u+K^u(\ph,\rho,u,s)\\
s'=-\la(\ph,\rho)\,s+K^s(\ph,\rho,u,s),\\
\end{array}
\right.
\end{equation}
where $K:=(K^\ph,K^\rho,K^u,K^s)$ satisfies
$
K(\ph,\rho,0,0)=0
$,
$
\d_uK(\ph,\rho,0,0)=\d_sK(\ph,\rho,0,0)=0
$
for $(\ph,\rho)\in \T\times J^*$.

\paraga We see all the functions and systems as defined over the universal cover of their domains, so that in particular 
$\ha\th\in\R$ and $\ph\in\R$.
We use the same notation $\cO$ and $\jO$ for the initial domains and their covers.
We  introduce the symplectic transformation 
\beq\label{eq:chi}
\left\vert
\begin{array}{lll}
\chi:\R\times\,]-2,2[\,\times \cO\longrightarrow \R\times\,]-2\sqrt\eps,2\sqrt\eps[\,\times \jO\\[4pt]
\chi(\xi,\eta,\ph,\rho,u,s)=\Big(\ha\th=\tfrac{1}{\sqrt\eps}\, \xi,\,\ha\sr=\sqrt\eps\,\eta,\,(\ov\th,\ov\sr)=\phi\big(\ph, \rho,u,s\big)\Big).
\end{array}
\right.
\eeq
We consider the restriction of $\sN^0_\eps$ to the range $\R\times\,]-2\sqrt\eps,2\sqrt\eps[\,\times \jO$
and use bold letters to denote the functions 
$\sN_\eps^0\circ\chi,Q\circ\chi,\sR^0_\eps\circ\chi$, 
so that:
\beq
\bN_\eps^0(z^0)=\ha\om\,\eta+\bQ(z^0)+\bC(\ze)+\bR_\eps^0(z^0),
\eeq
where 
$\ze=(\ph,\rho,u,s)$ and $z^0=(\eta,\ph,\rho,u,s)$.
The symplectic character of $\chi$ yields the following form for the 
Hamiltonian vector field $X^{\bN_\eps^0}$ in the coordinate system $(\xi,\eta,\ph, \rho,u,s)$
\begin{equation}\label{eq:mainsyst}
\left\vert
\begin{array}{lllll}
\xi'=\ha\om&
&\!\!\!+\ \d_{\eta}\bQ(z^0)&
\!\!\!+\ \d_{\eta}\bR^0_\eps(z^0)&\!\!\!
\\
\eta'=0&
&&&\!\!\!
\\
\ph'=\varpi(\rho)&\!\!\!+\ K^\ph(\ze)
&\!\!\!+\ \d_{\rho}\bQ(z^0)&
\!\!\!+\  \d_{\rho}\bR^0_\eps(z^0)&\!\!\!
\\
\rho'=0&\!\!\!+\ K^\rho(\ze)
&\!\!\!-\ \d_{\ph}\bQ(z^0)&\!\!\!-\ \d_{\ph}\bR^0_\eps(z^0)&\!\!\!
\\
u'=\la(\ph,\rho)\,u&\!\!\!
+\ K^u(\ze)&\!\!\!+\ \d_{s}\bQ(z^0)&\!\!\!+\ \d_{s}\bR^0_\eps(z^0)&\!\!\!
\\
s'=-\la(\ph,\rho)\,s&\!\!\!
+\ K^s(\ze)&\!\!\!-\ \d_{u}\bQ(z^0)&\!\!\!-\ \d_{u}\bR^0_\eps(z^0)&\!\!\!
\\
\end{array}
\right.
\end{equation}
Note that $Q(\sr)=\ha\sr\cdot L(\sr)$, where $L:\R^3\to\R$ is a linear map. We set $\bL=L\circ\chi$,
so, by (\ref{eq:chi}): 
\beq\label{eq:estimnew1}
\d_y\bQ(z^0)=2\sqrt\eps\, \eta \cdot \d_y\bL(z^0)\quad\textrm{and}\quad
 \d_{\eta}\bQ(z^0)=2\sqrt\eps\,\bL(z^0)+2\sqrt\eps\, \eta \cdot \d_{\eta}\bL(z^0)
\eeq
where $y$ stands for any variable in the set $\{\ph,\rho,u,s\}$, and
where the derivatives $\d_y\bL$  and $\d_{\eta}\bL$ are bounded in the $C^p$ topology, 
by periodicity and compactness. 
Moreover, clearly
$
\norm{\bR^0_\eps}_{C^p}\leq c_0\sqrt\eps,
$
for a suitable constant $c_0>0$.

\paraga  
To get the  assumptions of the persistence theorem, we fix an interval  $J$ strictly contained in $J^*$
and we introduce the following rescaling, where $\ga_0,\ga_1\in\,]0,1[$ will be chosen below:
\beq\label{eq:tilchi}
\left\vert
\begin{array}{lll}
\til\chi: \cD\longrightarrow \cD_\eps\\[4pt]
\til\chi(\xi,\eta,\til\ph,\til\rho,\til u,\til s)=\Big(\xi,\eta,\ph=\til\ph/\ga_0,\eta=\til\eta,u={\sqrt\eps}\,\til u/{\ga_1} ,\  s={\sqrt\eps}\,\til s/{\ga_1}\Big).
\end{array}
\right.
\eeq
where $\cD_\eps$ and $\cD$ are the subddomains of $\R\times\,]-2,2[\,\times\cO$
defined by $\norm{(u,s)}\leq {2\sqrt\eps/\ga_1}$ and $\norm{(\til u,\til s)}\leq {2/\ga_1}$
respectively. We will restrict the system (\ref{eq:mainsyst}) to the domain $\cD_\eps$.  
On the domain $\cD$, the vector field $\til\chi^*X^{\bN^0_\eps}$ takes the following form
\begin{equation}\label{eq:avsys20}
\left\vert
\begin{array}{llll} 
\xi'=\ha\om + F^\th_\eps(\til z^0)
\\

\eta'=0
\\

\til\ph'=\ga_0\,\varpi(\til \rho) +F^\ph_\eps(\til z^0)
\\

\til\rho'=F^\rho_\eps(\til z^0)
\\

\til u'=\til\la(\til\ph,\til\rho)\,\til u+ F_\eps^{u}(\til z^0)\\

\til s'=-\til\la(\til\ph,\til\rho)\,\til s+ F_\eps^{s}(\til z^0)\\
\end{array}
\right.
\end{equation}
with,  by direct computation and assuming $\eps<1$:
\beq\label{eq:finalest}
\norm{F_\eps}_{C^p}\leq 3 \Max\Big(\frac{\sig\sqrt\eps}{\ga_0^{p}\ga_1^2},\frac{\sig\ga_1}{\ga_0^{p}}\Big).
\eeq

%
\paraga Using a bump flat-top function, one can continue the function 
$F$  to a function defined over $\R^6$, still denoted by $F$, which coincide with the initial one over the domain
$\til\cD/2$ defined by the inequality
$
\norm{(\til u,\til s)}\leq 1,
$
and which moreover vanish outside $\til\cD$ and have the same periodicity properties (relatively to $\xi$ and $\til \ph$)
as the initial one. We can also assume that its $C^p$ norm over $\R^6$ satisfies (\ref{eq:finalest}) up to the choice of a larger
constant $\sig$.  This continuation yields a new vector field $Y_\eps$ on $\R^{6}$.  The persistence theorem applies to $Y_\eps$:
there is a constant $\nu<\!\!<1$ such that, if $\varpi,\ga_0,\ga_1,\eps_0$ satisfy
\beq
\Max(\ga_0\norm{\varpi}_{C^p},\ga_1/\ga_0^p,\sqrt\eps_0/\ga_0^{p}\ga^2_1)<\nu
\eeq
then for $0\leq\eps\leq\eps_0$ the vector field $Y_\eps$ admits a normally hyperbolic invariant manifold, of the form
\beq
\til u=\til U_\eps(\eta,\til\ph,\til\rho),\quad \til s= \til S_\eps(\eta,\til\ph,\til\rho), \qquad \norm{(\til U_\eps,\til S_\eps)}_{C^p}\leq \ov a<1.
\eeq
where $\til U_\eps$ and $\til S_\eps$ are $C^p$ functions
and $\ov a$ is an arbitrary positive constant. Note that $\til U_\eps$ and $\til S_\eps$ are independent of $\xi$.
Clearly, this invariant manifold is contained in $\til\cD/2$, so that the initial system 
(\ref{eq:avsys20}) also admits a normally hyperbolic invariant manifold, with the same equation. 
Moreover, since the system (\ref{eq:avsys20}) is $\ga_0$ periodic in $\til\ph$,
by hyperbolic uniqueness, the functions $\til U_\eps$ and $\til S_\eps$ are 
$\ga_0$ periodic in $\til\ph$  too (and independent of $\xi$), see\cite{B10}.

\paraga As a consequence, the initial system (\ref{eq:mainsyst}) admits a $4$--dimensional pseudo invariant normally hyperbolic
annulus $\ha\cA_\eps^0$ with equation, in the coordinates $(\xi,\eta,\ph,\rho,u,s)$:
\beq
\xi\in \T_\eps,\  \eta\in\,]-1,1[ ,\ \ph\in\T,\  u=\sqrt\eps\, U_\eps(\eta,\ph,\rho),\ s= \sqrt\eps\, S_\eps(\eta,\ph,\rho), 
\qquad \norm{(U_\eps,S_\eps)}_{C^p}\leq \ov a,
\eeq
where $\T_\eps:=\R/\sqrt\eps\Z$ and  $U_\eps(\eta,\ph,\rho):=\til U_\eps(\eta,\ga_0\ph,\ga_1\rho)$ and an analogous definition 
for $S_\eps$.  We will see $\ha\cA_\eps^0$ as the image of the embedding
\beq
\left\vert
\begin{array}{ll}
\ha\Psi_\eps^0 : \T_\eps\times \,]-1,1[\,\times \T\times J\longrightarrow \A^3\\[4pt]
\ha\Psi_\eps^0(\xi,\eta,\ph,\rho)=\big(\xi,\eta,\,j_\eps^0\big(\eta,\ph,\rho\big)\big),\\
\end{array}
\right.
\eeq
with
\beq
j_\eps^0(\eta,\ph,\rho)=\phi\big(\ph,\rho,\sqrt\eps\, U_\eps(\eta,\ph,\rho), \sqrt\eps\, S_\eps(\eta,\ph,\rho)\big).
\eeq

%
%
%

\vskip3mm

\setcounter{paraga}{0}
\noindent$\bu$ {\bf Second step.}    We now take advantage of the integrable structure of $X^{\bN_\eps^0}$
restricted to $\ha\cA_\eps^0$,  stemming both from the fact that, relatively to the $(\xi,\eta,\ph,\rho)$ coordinates,
it is independent of $\xi$ and that the ``complementary part''  is an $\eta$--family of one-degree-of-freedom systems 
in the variables $(\ph,\rho)$. 

\paraga
One immediately checks that the annulus $\ha\cA^0_\eps$ is controllable, so that it is 
a symplectic submanifold of $\A^3$. 
The vector field $(\ha\Psi_\eps^0)^*X^{\bN_\eps}$ is generated by the
Hamiltonian $\ha\bM^0_\eps=\bN^0_\eps\circ \ha\Psi_\eps^0$ relatively to the symplectic form 
\beq\label{eq:inducedform}
\Om_\eps:=(\ha\Psi_\eps^0)^*\Om=d\eta\wedge d\xi+d\rho\wedge d\ph +\eps\, d U_\eps\wedge d S_\eps.
\eeq
 Clearly
\beq\label{eq:hambm}
\begin{array}{lll}
\ha\bM^0_\eps(\xi,\eta,\ph,\rho)=\ha\om\,\eta 
&+& \sqrt\eps\,\eta\,\bL\big(\sqrt\eps\,\eta,\ph,\rho,\sqrt\eps\, U_\eps(\eta,\ph,\rho), \sqrt\eps\, S_\eps(\eta,\ph,\rho)\big)\\
&+& \bC\big(\ph,\rho,\sqrt\eps\, U_\eps(\eta,\ph,\rho), \sqrt\eps\, S_\eps(\eta,\ph,\rho)\big)\\
&+ &\bR_\eps^0\big(\sqrt\eps\,\eta,\ph,\rho,\sqrt\eps\, U_\eps(\eta,\ph,\rho), \sqrt\eps\, S_\eps(\eta,\ph,\rho)\big).
\end{array}
\eeq
Observe that $\ha\bM^0_\eps$ is independent of $\xi$. Moreover, one deduces from (\ref{eq:avsys20})
that the function $\eta$ is a first integral for  $\ha\bM^0_\eps$ relatively to $\Om_\eps$. 
Each level set $\ha\bM^0_\eps=\e$ (for $\abs{\e}$ small enough) is regular since the Hamiltonian vector field 
does not vanish, and is moreover foliated by the invariant subsets
\beq\label{eq:leaves}
\jT(\eta^0,\e)=\big(\T\times\{\eta^0\}\big)\times \big\{(\ph,\rho)\in\T\times J\mid \ha\bM^0_\eps(\eta^0,\ph,\rho)=\e\big\},\qquad \abs{\eta^0}<1.
\eeq
Thanks to (\ref{eq:hambm}), one immediately checks that
\beq
\d_\rho\ha\bM^0_\eps(\eta,\ph,\rho)=\bC'_0(\rho)+O(\sqrt\eps).
\eeq
Hence $\d_\rho\ha\bM^0_\eps(\eta,\ph,\rho)\neq0$ fo $\eps$ small enough, by our monotonicity assumption on the periods
of the periodic orbits of $\sA$ (see Section 3 of the Introduction and Section 1 of Part II). 
Therefore the functions $\ha\bM^0_\eps$ and $\eta$ are independent, and 
$\jT(\eta^0,\e)\subset (\ha\bM^0_\eps)\inv(\e)$  is a $2$--dimensional Liouville torus for $\ha \bM_\eps^0$,
Lagrangian for $\Om_\eps$. 

\paraga By the Liouville-Arnold theorem, there exist
angle-action coordinates adapted to the Lagrangian foliation $(\jT(\eta^0,\e))$,
relatively to which the Hamiltonian is independent of the angles.
The actions are the periods
$
A_j=\int_{\nu_j}\la_\eps
$, $j=1,2$,
of the Liouville form
$$
\la_\eps=(\ha\Psi_\eps^0)^*\la=\eta \,d\xi +  \rho\,d\ph + \eps S_\eps \, dU_\eps
$$ 
(where $\la=\eta\,d\xi+\rho\,d\ph+s\,du$ is the standard Liouville form of $\A^3$)
over a basis $(\nu_j)$  of the homology of the tori $\jT(\eta^0,\e)$. For $\nu_1$, one chooses
the canonical cycle generated by the angle $\xi$, which obviously yields $A_1=\eta$.
As for $\nu_{2}$, one chooses the cycle generated by the angle $\ph$, that is:
$\{(0,\eta^0)\}\times\{(\ph,\rho)\in\T\times J\mid \ha\bM^0_\eps(\eta^0,\ph,\rho)=\e\}$.
Therefore by immediate computation taking (\ref{eq:hambm}) into account to get $\rho$ as an implicit function of 
$\eta,\ph$:
$$
A_{2}=\int_\T\rho(\ph)\,d\ph+\eps\int_{\nu_2} S_\eps \, dU_\eps=\rho+\sqrt\eps \,a^\rho(\eta,\rho,\eps),
$$
where $a^\rho$ is a $C^p$ function.
The associated angles $\al_i$ are computed using a generating function, which easily yields
$$
\al_1=\xi,\qquad \al_{2}=\ph+\sqrt\eps\,a^\ph(\eta,\rho),
$$
where $a^\ph$ is a $C^p$ function.
Given $J^\bu$ such that $J\subset J^\bu\subset J^*$ be a slightly smaller interval, 
the angle-action embedding $\sig_\eps(\al_1,A_1,\al_2,A_2)=(\xi,\eta,\ph,\rho)$ is well defined over
$\T_\eps\times  \,]-1,1[\,\times \T\times J^\bu$ for $\eps$ small enough, with values in 
$\T_\eps\times  \,]-1,1[\,\times \T\times J^{*}$
According to the previous product decomposition,
$\sig_\eps=\Id\times \sig_\eps^{(2)}$ , where $\sig_\eps^{(2)}$ is $\sqrt\eps$--close to the identity in the $C^p$ topology.
Moreover, by construction
$$
\sig_\eps^*\Om_\eps=dA_1\wedge d\al_1+dA_2\wedge d\al_2,
$$
and the transformed Hamiltonian $\bM_\eps^0=\ha\bM_\eps^0\circ \sig_\eps$ depends only on $(A_1,A_2)$.
In the following we set $\Psi_\eps^0=\ha\Psi_\eps^0\circ\sig_\eps$, so that $\bM_\eps^0=\bN_\eps^0\circ\Psi_\eps^0$. Clearly,
using (\ref{eq:hambm}):
%
\beq\label{eq:interham}
\bM_\eps^0(\al_1,A_1,\al_2,A_2)=\ha\om\cdot A_1+\bC_0(A_2)
+\La_\eps^0(A_1,A_2),\qquad
\norm{\La_\eps^0}_{C^p}\leq a^{**}\sqrt\eps
\eeq
where $a^{**}$ is a large enough constant.

\vskip3mm

\setcounter{paraga}{0}
\noindent$\bu$ {\bf Third step.}   We finally use the straightening lemma again
 to get a normal form in the neighborhood of $\cA_\eps^0$,
to which we only have to add the remainder $\sR_\eps$ to get the final annulus $\cA$ 
and the corresponding estimates 
by the persistence theorem. The method is very similar to that of the first step
and we will skip the details.

\paraga By the symplectic normally hyperbolic persistence theorem (Appendix~\ref{app:normhyp}), one can continue the immersion 
$\Psi_\eps^0$ to a straightening symplectic embedding 
$$
\Phi_\eps^0:\frac{\R}{\sqrt\eps\Z}\times  \,]-1,1[\,\times \T\times J\times ]-\al,\al[^2\to\jU
$$ 
where $\jU$ is an open neighborhood of the annulus 
$\cA_\eps^0$ in $\A^3$, such that
$$
\Phi_\eps^0(\xi,\eta,\ph,\rho,0,0)=\Psi_\eps^0(\xi,\eta,\ph,\rho),
$$
and such that the stable and unstable manifolds of the points of $\cA_\eps^0$ are straightened.
We introduce the composed Hamiltonian 
$$
\jN_\eps^0=\bN_\eps^0\circ\Phi_\eps^0=\sN_\eps^0\circ\chi\circ\Phi_\eps^0
$$ 
where $\chi$ was defined in (\ref{eq:chi}). One easily proves that 
$\jN_\eps^0$ admits the following expansion with respect
to the hyperbolic variables $(u,s)$:
$$
\jN_\eps^0(x,u,s)=\bM_\eps^0(x)+\bla(x)us+(\jN_\eps^0)_3(x,u,s,\eps),
$$
where $x=(\xi,\eta,\ph,\rho)$, $\bla(x)\geq\bla_0>0$ and 
\beq\label{eq:vanishbN}
(\jN_\eps^0)_3(x,0,0)=0,\qquad D(\jN_\eps^0)_3(x,0,0)=0,\qquad D^2(\jN_\eps^0)_3(x,0,0)=0.
\eeq
In the following we abbreviate $(\jN_\eps^0)_3$ in $\jN_3$.

\paraga We now go back to the initial Hamiltonian $\sN_\eps$ in the new set of variables, and
set 
$$
\jN_\eps=\sN_\eps\circ\chi\circ\Phi_\eps^0=\jN_\eps^0+\jR_\eps,\qquad \jR_\eps=\sR_\eps\circ\chi\circ\Phi_\eps^0.
$$ 
As above, we perform a cutoff  and limit ourselves to the domain $\jD_\eps$
$$
\jD_\eps:\qquad \xi\in\R,\quad  \abs{\eta}\leq\pdemi,\quad \ph\in\R,\quad \rho\in J_\bu,\quad \norm{(u,s)}\leq2\eps^q,
$$
where again $J_\bu\subset J$ is arbitrarily close to $J$ and $q\geq 1$ is an arbitrary integer. 
We introduce the rescaled variables
$$
\xi=\frac{\til\xi}{\ga_0},\quad \eta=\til\eta,\quad \ph=\frac{\til\ph}{\ga_0},
\quad\rho=\til\rho,\quad u=\eps^q\til u,\quad s=\eps^q\til s.
$$
where $\ga_0$ is to be chosen below, and $\norm{(\til u,\til s)}\leq 2$. 
We use again top-flat bump functions to continue the functions to $\R^6$. The system
associated with $X^{\jN_\eps}$ on $\R^6$ then reads:
\begin{equation}\label{eq:mainsyst2}
\left\vert
\begin{array}{lllll}
\til\xi'=\ga_0\big[\d_\eta \bM_\eps^0
&\!\!\!+\ \eps^{2q}\til{\d_{\eta}\bla}\,\til u\til s&
\!\!\!+\ \til{\d_{\eta}\jN_3}&\!\!\!
+\ \til{\d_{\eta} \jR_\eps}\big]
\\[5pt]
\til\eta'=\hskip8mm0&
&\!\!\!-\ \til{\d_{\xi}\jN_3}&\!\!\!
-\  \til{\d_{\xi} \jR_\eps}
\\[5pt]
\til\ph'=\ga_0\big[\d_\rho \bM_\eps^0
&\!\!\!+\ \eps^{2q}\til{\d_{\rho }\bla}\,\til u\til s&
\!\!\!+\ \til{\d_{\rho }\jN_3}&\!\!\!
+\ \til{\d_{\rho} \jR_\eps}\big]
\\[5pt]
\til\rho'=\hskip8mm0&
&\!\!\!-\ \til{\d_{\ph}\jN_3}&\!\!\!
-\  \til{\d_{\ph} \jR_\eps}
\\[5pt]
\til u'=&\hskip10mm\til{\bla}\,\til u
&\!\!\!+\ {\eps^{-q}}\til{\d_{s}\jN_3}&\!\!\!
+\ {\eps^{-q}} \til{\d_{s} \jR_\eps}
\\[5pt]
\til s'=&\hskip8mm-\til{\bla}\,\til s
&\!\!\!-\ {\eps^{-q}}\til{\d_{u}\jN_3}&\!\!\!
-\  {\eps^{-q}}\til{\d_{u} \jR_\eps}
\\
\end{array}
\right.
\end{equation}
where the $\til{\phantom{A}}$ stands for the composition by the rescaled variables.
Given $v\in\{\xi,\eta,\rho,\ph,u,v\}$,
the following (not optimal) 
estimates are  immediate:
$$
\norm{\til{\d_v\jN_3}}_{C^j}\leq \sig_j\frac{\eps^{2q}}{\ga_0^p},
\qquad \norm{\d_v\jR_\eps}_{C^j}\leq \sig_j\frac{\eps^{\ell-j/2}}{\ga_0^p},
$$
for suitable constants $\sig_j>0$, independent of $\ga_0,\ga_1$ and $\eps$.
We will assume 
$$
\ell>q+p/2,
$$
 a condition which holds as soon as $\ka$ is large enough.
As for the first step of the proof, one readily sees that it is possible to choose 
$\ga_0>0$ and $\eps_0>0$ small
enough so that for $0<\eps<\eps_0$, the system (\ref{eq:mainsyst2})
admits a normally hyperbolic invariant annulus 
with equation
$$
\til\xi\in\R,\quad  \abs{\til\eta}<\pdemi,\quad \til\ph\in\R,\quad \til\rho\in J_\bu,\quad \til u =\til U_\eps(\til x),\quad \til s=\til S_\eps(\til x),
$$
where $\til U_\eps$ and $\til S_\eps$ are $C^p$ functions with $\norm{(\til U_\eps,\til S_\eps)}_{C^p}\leq 1$.
As a consequence, the system $X^{\jN_\eps}$ possesses a pseudo invariant annulus $\cA_\eps$ contained
in the domain $\jD_\eps$, with equation
$$
\xi\in\R,\quad  \abs{\eta}<\pdemi,\quad \ph\in\R,\quad \rho\in J_\bu,\quad u =\eps^q U_\eps(x),\quad s=\eps^qS_\eps(x),
$$
where $U_\eps(\xi,\eta,\ph,\rho)=\til U_\eps(\ga_0\xi,\eta,\ga_0\ph,\rho)$ and a similar definition for $S_\eps$. 
Going back to the initial variables:
\beq
\cA_\eps=\Big\{\Big(\frac{1}{\sqrt{\eps}}\xi,\sqrt\eps\,\eta,\,j_\eps\big(\xi,\eta,\ph,\rho\big)\Big)
\mid \xi\in\frac{\R}{\sqrt\eps\Z},\,\abs{\eta}<\pdemi,\,\ph\in\T,\,\rho\in J_\bu\Big\},
\eeq
with now
$$
\norm{j_\eps-j^0_\eps}_{C^p}\leq \sig \eps^q
$$
for a suitable constant $\sig>0$. The annulus $\cA_\eps$ is symplectic, so the Hamiltonian vector field 
$\Psi_\eps^*X^{\sN_\eps}$ is generated by $\sN_\eps\circ\Psi_\eps$ relatively to the symplectic form
$$
\Om_\eps=d\eta\wedge d\xi+d\rho\wedge d\ph+\eps^{2q}dU_\eps\wedge dS_\eps.
$$
This immediately yields the final Hamiltonian $\sM_\eps$.
\end{proof}


\subsubsection{Twist sections}
We now examine the intersection of the previous pseudo invariant annulus with constant energy levels of the initial
Hamiltonian $\sN_\eps$. 

\begin{lemma}\label{lem:section}
 We keep the notation and assumptions of Lemma~\ref{lem:existann}. 
Then $\cC_\eps=\cA_\eps\cap \sN_\eps\inv(0)$ is a pseudo invariant hyperbolic cylinder for $\sN_\eps$,
on which $(\xi,\ph,\rho)\in\T_\eps\times\T\times J^\bu$ form a chart.
The set $\Sig_\eps$ defined in this chart by the equation $\xi=0$ is a $2$--dimensional transverse section for
the Hamiltonian vector field $X^{\sM_\eps}$ on $\cC_\eps$, on which the coordinates $(\ph,\rho)\in\T\times J$ form an exact symplectic chart.
Relatively to these coordinates, given $J^\bu$ with  $J\subset \ov J^\bu\subset J^\bu$, the Poincar\'e map induced by the 
Hamiltonian flow inside $\cC_\eps$ is well-defined over $\T\times J^\bu$, with values in $\T\times J^\bu$, and reads
\beq\label{eq:formPoinc}
\jP_\eps(\ph,\rho)=\Big(\ph+\eps\varpi(\rho)+\De_\eps^\ph(\ph,\rho),\ 
\rho+\De_\eps^\rho(\ph,\rho)\Big),
\eeq
where $\varpi$, $\De_\eps^\ph$, $\De_\eps^\rho$  are $C^p$ functions such that:
\beq
\varpi'(\rho)\geq \sig,\qquad \norm{\De_\eps^\ph}_{C^p}\leq \eps^q,\qquad \norm{\De_\eps^\rho}_{C^p}\leq \eps^q.
\eeq 
for a suitable $\sig>0$. 
\end{lemma}

\begin{proof} In the chart $(\xi,\eta,\ph,\rho)\in\T_\eps\times\,]-1,1[\,\times\T\times J^\bu$ of $\cA_\eps$ associated with the 
embedding $\Psi_\eps$, the intersection $\cC_\eps$ admits the equation $\sM_\eps=0$, that
is:
$$
\ha\om\,\eta + \bC_0(\rho) + \La^0_\eps(\eta,\rho)+\La_\eps\Big(\tfrac{1}{\sqrt\eps}\xi,\eta,\ph,\rho\Big)=0.
$$
For $\eps$ small enough, from this latter equation one gets the variable $\eta$ as an implicit function of $\xi,\ph,\rho$,
so that $\cC_\eps$ is a $3$--dimensional submanifold of $\cA_\eps$, diffeomorphic to $\T^2\times\,]0,1[$. Moreover
$\cC_\eps$ is pseudo invariant since $\sN_\eps\inv(0)$ is invariant and $\jA_\eps$ is pseudo invariant. 
Since $\cA_\eps$ is normally hyperbolic, $\cC_\eps$ is normally hyperbolic too (see the introduction).
Since $\dot\xi=\ha\om+O(\sqrt\eps)$, $\Sig_\eps$ is a transverse section for $X^{\sM_\eps}$ on $\cC_\eps$ when $\eps$
is small enough. The statement on $\jP_\eps$ is immediate from the expression of $\sM_\eps$.
\end{proof}


\subsubsection{Invariant tori and the boundaries of $d$--cylinders}\label{sec:applicKAM}
It only remains now to apply to $\jP_\eps$ the invariant curve theorem deduced from Herman's presentation, stated
in Proposition~\ref{prop:KAM}, which is possible by chosing $\ell$ large enough. This yields
the existence of $\eps_0$ such that the map $\jP_\eps$ admits an essential curve in each connected component
of $\T\times J^{\bu}\setm\ov J$. As a consequence, the Hamiltonian flow in $\cC_\eps$ admits an invariant
torus in each domain $\T_\eps\times\T\times J^{\bu}\setm\ov J$, which bound in $\cC_\eps$ a compact
normally hyperbolic invariant cylinder. This concludes the proof of Lemma~\ref{lem:dcyl}.


\subsection{Proof of Lemma~\ref{lem:singcyl}}\label{sec:prooflemsingcyl}
Recall that given a singular $2$-annulus $\sA_\bu$ of $C$, there exists a neighborhood $O$ of $\sA_\bu$
in $\A^2$ and a vector field $X_\circ$ on an open set of $\A^2$ such that $X_\circ\equiv X_C$ in $O$ and $X_\circ$
admits a $C^1$ normally hyperbolic $2$-annulus (see Part II and Appendix~\ref{app:normhyp}) 
$\sA_\circ$ which satisfies the properties: 
\vskip1mm $\bu$ 
the components of the boundary of $\sA_\bu$ formed by the opposite periodic orbits coincide
with the components of the boundary of $\sA_\circ$,
\vskip1mm $\bu$ 
the time-one map of $X_\circ$ is a twist map relatively to adapted coordinates.
\vskip1mm $\bu$ 
$X_\circ$ is the Hamiltonian vector field generated by a $C^2$ Hamiltonian $H_\circ$ on $\A^2$.
\vskip1mm
As a consequence, the previous section applied to $H_\circ$ proves the existence of a cylinder attached
to $\sA_\circ$. Now the same argument as above proves the existence of three regular $2$-dimensional
annuli of section at energy $\e$ on which the return map has the form (\ref{eq:formPoinc}) and so admits
invariant circles arbitrarily close to the boundaries. One deduces the existence of $3$ invariant $2$-dimensional 
tori which form the boundary of a singular cylinder, which moreover admits a generalized twist section 
in the sense of Definition~\ref{def:singann2}. This concludes the proof.


\subsection{Proof of Lemma~\ref{lem:scyl}}\label{sec:prooflemscyl}

We will first prove the existence of ``the main part of'' the $s$--cylinder using the global normal form of
Proposition~\ref{prop:globnorm} and hyperbolic persistence. We then prove the existence of the extremal 
cylinders at the double
resonance points. Finally, we get the $s$--cylinders by gluing together
the extremal cylinders at each end of the previous main part.


\subsubsection{Global normal form: the ``main part'' of the $s$--cylinders}\label{ssec:normformhyppersist}
Let $\Ga=\om\inv(k^\bot)\cap h\inv(\e)$. We fix adapted coordinates $(\th,r)$ and write
$r_3=r_3^*(\ha r)$ for the equation of the resonance surface $\om\inv(k^\bot)$.
We set 
$$
\ell(\ha r)=\big(\ha r, r_3^*(\ha r)\big).
$$
As in Appendix~\ref{App:globnormforms}, we 
fix two consecutive points $r'$ and $r''$ in $D(\de)$, fix $\rho<\!<\dist_{\Ga}(r',r'')$ and set 
$$
\Ga_\rho:=[r^*,r^{**}]_\Ga\subset [r',r'']_\Ga,
$$
where $r^*,r^{**}$ are defined by the equalities
$$
\dist_\Ga(r^*,r')=\dist_\Ga(r^{**},r'')=\rho.
$$
We set $\ha\Ga_\rho=\Pi(\Ga_\rho)$ where $\Pi$ is the projection $r\mapsto\ha r$.  For $c>0$, we then set
\begin{equation}
\jU_{c,\rho}=\{r\in\R^3\mid \dist_\infty(r,\Ga_\rho)<c\rho\},
\qquad \jD_{c,\rho}=\Pi\big(\jU_{c,\rho}\big),
\qquad \jW_{c,\rho}=\T^3\times \jU_{c,\rho}.
\end{equation}

Our starting point is the following consequence of the global normal form in Proposition~\ref{prop:globnorm}.

\begin{lemma}\label{lem:mainpart}
Consider the system
\begin{equation}
N(\th,r)=h(r)+\eps V(\th_3,r)+\eps W_0(\th,r)+\eps W_1(\th,r)+\eps^2 W_2(\th,r),
\end{equation}
where 
\begin{equation}
V(\th_3,r)=\int_{\T^2}f(\th,r)\,d\th_1d\th_2,
\end{equation}
and where the functions $W_0\in C^p(\A^3)$, $W_1\in C^{\ka-1}(\jW_{c\rho})$, $W_2\in C^\ka(\jW_{c\rho})$ satisfy
\begin{equation}
\begin{array}{lll}
\norm{W_0}_{C^p(\jW_{c\rho})}\leq \de,\\[5pt]
\norm{W_1}_{C^2(\jW_{c\rho})}\leq c_1\, \rho^{-3} \\[5pt]
\norm{W_2}_{C^2(\jW_{c\rho})}\leq c_2\,\rho^{-6},
\end{array}
\end{equation}
for suitable constants $c_1,c_2>0$. Assume that for $\ha r\in \jD_{c\rho}$ the function
\beq
\begin{array}{lll}
V\big(\cdot,\ell(\ha r)\big): \T\to\R
\end{array}
\eeq
admits a single maximum at $\th_3^*(\ha r)$, which is nondegenerate.
Then for $0<\eps<\eps_0$, the system $N_\eps$ admits a pseudo invariant cylinder $\cC_\eps$ 
of the form
\beq
\cC_\eps=\Big\{\big(\ha\th,\ha r, \th_3=\Th_3(\ha\th,\ha r), r_3=R_3(\ha\th,\ha r)\big)
\mid \ha\th\in\T^2,\ \ha r\in \jD_{c\rho}\Big\}\cap N_\eps\inv(\e),
\eeq
where $c>0$ is small enough, and where
\beq
\norm{(\Th_3-\th_3^*(\ha r))}_{C^0}\leq c\sqrt\de,\qquad \norm{(R_3-r_3^*(\ha r))}_{C^0}\leq c\sqrt\de\sqrt\eps,
\eeq
for a suitable $C>0$ independent of $\rho$. Moreover, there exists $\mu>0$ such that any invariant set
which is contained in a domain of the form
\beq\label{eq:uniqueloc}
\Big\{(\ha\th,\ha r,\th_3,r_3)\mid \ha\th\in\T^2,\ha r\in\jD_{c\rho},\abs{\th_3-\th_3(\ha r)},\abs{r_3-r_3^*(\ha r)}\Big\}\cap N_\eps\inv(0)
\eeq
is contained in $\cC_\eps$.
\end{lemma}

\begin{proof} 
We first work in the universal covering $\R^3\times\R^3$ of $\A^3$ and use the same notation for the elements of 
$\A^3$ and their lifts.  

\vskip2mm

$\bullet$ The differential system associated with $X^{N_\eps}$ reads
\begin{equation}\label{eq:vectfield0}
\left\vert
\begin{array}{lllll}
\ha \th'= 
\ha\om(r)&\!\!\!
+\ \eps\d_{\ha r}V(\th_3,r)&\!\!\!
+\ \eps\,\partial_{\ha r}W_0(\th,r)&\!\!\!
+\ \eps\,\partial_{\ha r}W_1(\th,r)&\!\!\!
+\ \eps^2\,\partial_{\ha r}W_2(\th,r)\\[5pt]
\ha r'=                                   &&\!\!\!
-\ \eps\,\partial_{\ha \th}W_0(\th,r)&\!\!\!
-\ \eps\,\partial_{\ha \th}W_1(\th,r)&\!\!\!
-\ \eps^2\,\partial_{\ha \th}W_2(\th,r)\\[5pt]
\th_3'=
\om_3(r)&\!\!\!
+\ \eps\d_{r_3}V(\th_3,r)&\!\!\!
+\ \eps\,\partial_{r_3}W_0(\th,r)&\!\!\!
+\ \eps\,\partial_{r_3}W_1(\th,r)&\!\!\!
+\ \eps^2\,\partial_{r_3}W_2(\th,r)\\[5pt]
r_3'=                                   &\!\!\!
-\ \eps\d_{\th_3}V(\th_3,r)&\!\!\!
-\ \eps\,\partial_{\th_3}W_0(\th,r)&\!\!\!
-\ \eps\,\partial_{\th_3}W_1(\th,r)&\!\!\!
-\ \eps^2\,\partial_{\th_3}W_2(\th,r).\\
\end{array}
\right.
\end{equation}

\vskip2mm

$\bullet$ We will first estimate the various terms of the previous system in the domain $\sD_\eps$ defined by
\beq\label{eq:domainwork}
\sD_\eps=\Big\{
(\th,r)\in\R^3\times\R^3\mid 
\ha\th\in \R^2,\  
\ha r\in \jD_{\rho},\  
\abs{\th_3-\th_3^*\big(\ell(\ha r)\big)}\leq \sqrt\de, \  
\abs{r_3-r^*_3(\ha r)}\leq\sqrt\eps
\Big\}.
\eeq
We first set
\beq
\bth_3=\th_3-\th^*_3\big(\ell(\ha r)\big),\qquad \br_3=r_3-r^*_3(\ha r),
\eeq
so that 
\beq\label{eq:expansion}
\left\vert
\begin{array}{rll}
\om_3(r)&=&a(\ha r)\,\br_3+\chi(\ha r,\br_3)\,\br_3^2\\[5pt]
-\d_{\th_3}V(\th_3,r)&=&b(\ha r)\,\bth_3+\chi(\ha r,\bth_3,\br_3)\,\bth_3^2+\chi(\ha r,\bth_3,\br_3)\,\bth_3\,\br_3+\chi(\ha r,\bth_3,\br_3)\,\br_3^2.\\
\end{array}
\right.
\eeq
Observe that $a(\ha r)\geq a>0$ and $b(\ha r)\geq b>0$.
To diagonalize the hyperbolic part, we set 
\beq
\ze(\ha r)=a(\ha r)^{\frac{1}{4}}b(\ha r)^{-\frac{1}{4}}
\eeq
and
\beq
u=\ze(\ha r)\,\br_3+\sqrt\eps\, \ze(\ha r)\inv\,\bth_3,\qquad s=\ze(\ha r)\,\br_3-\sqrt\eps\,\ze(\ha r)\inv\,\bth_3,
\eeq
The inverse transformation reads
\beq\label{eq:inverse}
\th_3\big(I,(u-s)/\sqrt\eps\big):=\th_3^*(I)+\ze(I)\frac{u-s}{2\sqrt\eps},\qquad 
r_3\big(I,(u+s)\big):=r_3^*(I)+\frac{u+s}{2\ze(I)}.
\eeq
We finally complete the change of variables and time by setting
\beq
\ph= \frac{1}{\ga\sqrt\eps}\,\ha\th,\qquad I=\ha r,\qquad \dot{{u}}=\frac{1}{\sqrt\eps} u'.
\eeq
In the following we denote by $M$ a universal constant, independent of $\eps$ and $\de$.

\vskip2mm
$\bu$ {\bf Estimates for $\dot\ph$.}  Observe that
\beq
\begin{array}{lll}
\dot\ph=\ga\ha\th'=\ga\, \Om(I,u,s)&+&\eps \d_{\ha r}V\Big(\th_3\big(I,(u-s)/\sqrt\eps\big),I,r_3\big(I,(u+s)\big)\Big)\\
&+&\eps\d_{\ha r}W_0\Big(\ph/(\ga\sqrt\eps), I, \th_3\big(I,(u-s)/\sqrt\eps\big),r_3\big(I,(u+s)\big)\Big)\\
&+&\eps\d_{\ha r}W_1\Big(\ph/(\ga\sqrt\eps), I, \th_3\big(I,(u-s)/\sqrt\eps\big),r_3\big(I,(u+s)\big)\Big)\\
&+&\eps^2\d_{\ha r}W_2\Big(\ph/(\ga\sqrt\eps), I, \th_3\big(I,(u-s)/\sqrt\eps\big),r_3\big(I,(u+s)\big)\Big)
\end{array}
\eeq
where
\beq
\Om(I,u,s)=\ha\om\big(I,\, r_3^*(I)+\pdemi (u+s)(\ze(I))\inv\big),
\eeq
Forgetting about the variables to avoid cumbersome notations, and using the estimates on the various fonctions, one gets:
\beq
\begin{array}{lll}
\norm{\eps \d_{\ha r}V}_{C^0}\leq M\eps,&\norm{\eps \d_{\ha r}V}_{C^1}\leq M\sqrt\eps,\\[5pt]
\norm{\eps\d_{\ha r}W_0}_{C^0}\leq \eps\de,& \norm{\eps\d_{\ha r}W_0}_{C^1}\leq \sqrt\eps\de,\\[5pt]
\norm{\eps\d_{\ha r}W_1}_{C^0}\leq M\eps^{3/2}\rho^{-2},& \norm{\eps\d_{\ha r}W_1}_{C^1}\leq M\eps\rho^{-3},\\[5pt]
\norm{\eps^2\d_{\ha r}W_2}_{C^0}\leq M\eps^2\rho^{-5},&\norm{\eps^2\d_{\ha r}W_2}_{C^1}\leq M\eps^{3/2}\rho^{-6}.\\
\end{array}
\eeq
\vskip2mm
$\bu$ {\bf Estimates for $\dot I$.}  In the same way
\beq
\begin{array}{lll}
\dot I=
&-&\sqrt\eps\d_{\ha \th}W_0\Big(\ph/(\ga\sqrt\eps), I, \th_3\big(I,(u-s)/\sqrt\eps\big),r_3\big(I,(u+s)\big)\Big)\\[5pt]
&-&\sqrt\eps\d_{\ha \th}W_1\Big(\ph/(\ga\sqrt\eps), I, \th_3\big(I,(u-s)/\sqrt\eps\big),r_3\big(I,(u+s)\big)\Big)\\[5pt]
&-&\eps^{3/2}\d_{\ha \th}W_2\Big(\ph/(\ga\sqrt\eps), I, \th_3\big(I,(u-s)/\sqrt\eps\big),r_3\big(I,(u+s)\big)\Big),
\end{array}
\eeq
which yields
\beq
\begin{array}{lll}
\norm{\sqrt\eps\d_{\ha r}W_0}_{C^0}\leq \sqrt\eps\de,& \norm{\sqrt\eps\d_{\ha r}W_0}_{C^1}\leq \de,\\[5pt]
\norm{\sqrt\eps\d_{\ha r}W_1}_{C^0}\leq M\eps\rho^{-2},& \norm{\sqrt\eps\d_{\ha r}W_1}_{C^1}\leq M\sqrt\eps\rho^{-3},\\[5pt]
\norm{\eps^{3/2}\d_{\ha r}W_2}_{C^0}\leq M\eps^{3/2}\rho^{-5},&\norm{\eps^{3/2}\d_{\ha r}W_2}_{C^1}\leq M\eps\rho^{-6}.\\
\end{array}
\eeq
\vskip2mm
$\bu$ {\bf Estimates for $\dot u$ and $\dot s$.} We will give the details for $\dot u$ only, the case of $\dot s$ being
exactly similar. First note that
\beq
\dot u=\frac{1}{\sqrt\eps}\Bigg[\ze(I)r_3'+\frac{\sqrt\eps}{\ze(I)}\th_3'
+I'\Big(\ze'(I)\big(\br_3-\frac{\sqrt\eps}{\ze(I)^2}\ze'(I)\bth_3\big)
+\ze(I)(r_3^*)'(I)-\frac{\sqrt\eps}{\ze(I)}(\th_3^*)'(I)\Big)\Bigg]
\eeq
(where $'$ stands both for the initial time derivative and the usual derivative of functions). We first focus on the  part of
\beq
\frac{1}{\sqrt\eps}\Big(\ze(I)r_3'+\frac{\sqrt\eps}{\ze(I)}\th_3'\Big)
\eeq
involving the functions $\om_3$, $\d_{\th_3}V$ and $\d_{r_3}V$ only.
A straightforward computation, using in particular (\ref{eq:expansion}) proves that
\beq
\frac{1}{\sqrt\eps}\Big(-\eps\ze\d_{\th_3}V+\frac{\sqrt\eps}{\ze}\big(\om_3+\d_{r_3}V\big)\Big)=\la(I)u+\chi(I,u,s,\eps)
\eeq
with
\beq
\la(I)=\sqrt{a(I)b(I)},
\eeq
and, assuming $\eps\leq \de$:
\beq
\norm{\chi}_{C^0}\leq M\sqrt\eps\de,\qquad \norm{\chi}_{C^1}\leq M\de.
\eeq
As for the contribution of the functions $W_0,W_1,W_2$, one gets
\beq
\begin{array}{ll}
\chi_0:=-\ze\sqrt\eps\d_{\th_3}W_0+\dsp\frac{\eps}{\ze}\d_{r_3}W_0,\\[7pt]
\chi_1:=-\ze\sqrt\eps\d_{\th_3}W_1+\dsp\frac{\eps}{\ze}\d_{r_3}W_1,\\[7pt]
\chi_2:=-\ze\eps^{3/2}\d_{\th_3}W_2+\dsp\frac{\eps^2}{\ze}\d_{r_3}W_2,\\
\end{array}
\eeq
so that
\beq
\begin{array}{ll}
\norm{\chi_0}_{C^0}\leq \de\sqrt\eps,&\qquad \norm{\chi_0}_{C^1}\leq \de,\\[5pt]
\norm{\chi_1}_{C^0}\leq M\Big(\dsp\frac{\eps}{\rho}+\frac{\eps^{3/2}}{\rho^2}\Big),
&\qquad \norm{\chi_1}_{C^1}\leq M\Big(\dsp\frac{\sqrt\eps}{\rho^2}+\frac{\eps}{\rho^3}\Big),\\[5pt]
\norm{\chi_2}_{C^0}\leq \dsp \frac{\eps^{3/2}}{\rho^5},&\qquad \norm{\chi_2}_{C^1}\leq \dsp\frac{\eps}{\rho^6}.
\end{array}
\eeq
Finally, the estimates for the remaining term
\beq
\frac{1}{\sqrt\eps}I'\Big(\ze'(I)\big(\br_3-\frac{\sqrt\eps}{\ze(I)^2}\ze'(I)\bth_3\big)
+\ze(I)(r_3^*)'(I)-\frac{\sqrt\eps}{\ze(I)}(\th_3^*)'(I)\Big)
\eeq
are the clearly same as those of $\dot I$. 

\vskip2mm

$\bu$ Once this preliminary work is done in the domain (\ref{eq:domainwork}) one easily extends the system to 
$\R^6$ by using bump functions. Let $\eta:\R\to[0,1]$ be a $C^\infty$ function with support in $[-2,2]$, which is equal to $1$
in $[-1,1]$. Then the functions
\beq 
\mu_\eps(x)=\eta(x/\sqrt\eps),\qquad \nu_\de(x)=\eta(x/\sqrt\de)
\eeq
satisfy
\beq
\norm{\mu_\eps}_{C^0}=1,\qquad \norm{\mu_\eps}_{C^1}=1/\sqrt\eps,
\qquad
\norm{\nu_\de}_{C^0}=1,\qquad \norm{\nu_\de}_{C^1}=1/\sqrt\de.
\eeq
The new system obtained by replacing the various factors in (\ref{eq:vectfield0}) with their product by
\beq
\mu_\eps(\br_3)\nu_\de(\bth_3)
\eeq
admits the same estimates as the previous ones, to the cost of changing $M$ to a larger constant. 

\vskip2mm

$\bu$ As a consequence, with respect to the new variables and time, the new system in $\R^6$ reads
\begin{equation}\label{eq:presystem}
\left\vert
\begin{array}{lllll}
\dot \ph= \ga\, \Om(I,u,s)&+&F_\ph(\ph,I,u,s,\eps)\\
\dot I=0&+&F_I(\ph,I,u,s,\eps)\\
\dot u=\la(I)\,u&+&F_u(\ph,I,u,s,\eps)\\
\dot s=-\la(I)\,s&+&F_s(\ph,I,u,s,\eps),\\
\end{array}
\right.
\end{equation}
with, setting $F=(F_\ph,F_I,F_u,F_s)$ and assuming
\beq
\eps=\rho^7,\qquad \eps^{1/7}<\de,
\eeq
\beq
\norm{F}_{C^0}\leq M\sqrt\de\sqrt\eps,\qquad \norm{F}_{C^1}\leq M\sqrt\de.
\eeq

\vskip2mm

$\bullet$ Since $\la(I)$ is bounded from  below by a positive constant $\la$. Therefore the persistence theorem applies when
the constant $\de$  is small enough. This yields
the existence of a normally hyperbolic invariant manifold $\jA^*_\eps$ for (\ref{eq:presystem}), of the form 
$$
\Big\{\big(\ph,I,U(\ph,I),S(\ph,I)\big)\mid (\ph,I)\in\R^4\Big\},
$$
where $(U,S):\R^4\to\R^2$ is a $C^p$ map which satisfies
\begin{equation}\label{eq:finest}
\norm{(U,S)}_{C^0(\R^4)}\leq \frac{2M}{\la}\,\de\,\sqrt\eps,\qquad \norm{(U,S)}_{C^1(\R^4)}\leq C\,\de.
\end{equation}
One immediately checks these estimates are consistent with the definition of the domain (\ref{eq:domain})
so that $\jA^*_\eps\subset \sD_\eps$.

 \vskip2mm

$\bullet$ To go back to the initial coordinates, one has to apply the inverse change (\ref{eq:inverse}) is contained in the domain
$$
\abs{\bth_3}\leq c\sqrt\de,\qquad \abs{\br_3}\leq c\sqrt\de\sqrt\eps,
$$
for a suitable $c>0$.
One also gets the periodicity in the angular variables by the same uniqueness argument as in \cite{B10}.  
The local maximality property (\ref{eq:uniqueloc}) is also an immediate consequence of the normally hyperbolic
persistence theorem.
This concludes the proof of Lemma~\ref{lem:mainpart}.
\end{proof}

 \vskip2mm

It remains now to prove the existence of a section and the twist property. In this section we will content ourselves
with a non-connected section relatively to which the return map admits the twist property, we will in fact
prove the existence of a covering of $\cC_\eps$ with consecutive cylinders with connected twist sections
in the next Section.
We begin with a lemma which is a direct consequence of the normal forms in Lemma~\ref{lem:mainpart}
and describes the Hamiltonian vector field in restriction to the annuli $\cA_\eps$.

\begin{lemma}\label{lem:normformvectfield}
With the notation and assumptions of {\rm Lemma~\ref{lem:mainpart}}, the vector field on the invariant compact 
$s$-annulus $\cA_\eps$ admits the following normal form relatively to the coordinates $(\ha\th,\ha r)$:
\beq
\left\vert
\begin{array}{llcll}
\ha \th'&=&
\ha\om\big(\ha r, r_3^*(\ha r)+R_3(\ha r)\big)&
+&\chi_{\ha\th}(\ha\th,\ha r)\\[5pt]
\ha r'&=&  0                                &
+&\chi_{\ha r}(\ha\th,\ha r)\\[5pt]
\end{array}
\right.
\eeq
where
\beq
\begin{array}{lll}
\norm{R_3}_{C^0}\leq M\de\sqrt\eps,&\qquad \norm{R_3}_{C^1}\leq M\de\\
\norm{\chi_{\ha\th}}_{C^0}\leq M\sqrt\de \sqrt\eps,&\qquad  \norm{\chi_{\ha\th}}_{C^1}\leq M\sqrt\de \\
\norm{\chi_{\ha r}}_{C^0}\leq M\sqrt\de\,\eps,&\qquad  \norm{\chi_{\ha r}}_{C^1}\leq M\sqrt\de\sqrt\eps.
\end{array}
\eeq
\end{lemma}

One can cover the annulus $\cA_\eps$ with a a finite number of (overlapping) open subsets over 
which $\om_1\big (\ell(\ha r)\big)\neq0$
or $\om_2\big (\ell(\ha r)\big)\neq0$.
To simplify the following, we will assume that $\om_1\big (\ell(\ha r)\big)\neq0$ in
the neighborhood of  $\cA_\eps$, the general case being easily deduced from the latter (since
we allow for non-connected sections).

\begin{lemma}
With the notation and assumptions of {\rm Lemma~\ref{lem:normformvectfield}}, and assuming moreover
that 
\beq
\om_1\big (\ell(\ha r)\big)\geq\varpi_0>0
\eeq
in the neighborhood of $\cA_\eps$, then the submanifold
\beq
\Sig=\Big\{(\ha \th,\ha r)\in\T^2\times\jD_{c\eps^{1/7}}\mid \th_1=0\Big\}
\eeq
is a transverse section for the vector field on $\cA_\eps$. The intersection $\Sig\cap \cC_\eps$ is 
a global section for the flow on $\cC_\eps$ which admits $(\th_2,r_2)$ as a global exact-symplectic chart.
Relatively to these coordinates, the flow-induced return map attached to $\Sig\cap \cC_\eps$ is a twist 
map.
\end{lemma}

\begin{proof} Recall that  $h$ is a Tonelli Hamiltonian on $\R^3$ and that  
$(\th,r)$ are adapted coordinates, relatively to which $\Ga=\{\om_3=0\}\cap h\inv(\e)$. 
Let $\Pi:\R^3\to\R^2$ be the projection on the $\ha r$--plane.

The submanifold $\{\om_3=0\}$ is a graph over its projection $\pi(\{\om_3=0\})$ since
$\d_{r_3}\om_3=\d^2_{r_3}h>0$. Observe also that $\Ga$ is the ``apparent contour'' of the
level $h\inv(\e)$ with respect to the direction $r_3$, so that $\Pi(\Ga)$ bounds the projection
$C_\e:=\Pi\big(h\inv(]-\infty,\e])\big)$, which is strictly convex. By the implicit function theorem, this 
proves also that $\Pi\big(\{\om_3=0\}\cap h\inv(]-\infty,\e])\big)=C_\e$ and therefore that $\{\om_3=0\}$
is a graph over the whole $\ha r$--plane.

Fix $\th_3^*\in\T$ and set $S:=\{\th_3=\th_3^*\}\times\{\om_3=0\}\subset \T^3\times\R^3$.
Then clearly $S$ is invariant under the Hamiltonian flow generated by $h$, and is moreover symplectic.
Taking $(\ha \th,\ha r)$ as a chart on $S$, the induced Liouville form reads $r_2d\th_2+r_3d\th_3$.
The restriction of the Hamiltonian flow to $S$ is generated by the restriction $\ha h$ of $h$ to $S$. The
sublevels of this restriction are the sets $C_\e$, so that $\ha h$ is quasi-convex. 
Given $\e>\Min h$, relatively to the coordinates $(\ha\th,\ha r)$: 
$$
\ha h\inv(\e)=\T^2\times \ha\Ga.
$$ 
Each torus $\T^2\times\{\ha r\}$ on this level is invariant under the flow, with rotation vector 
$\varpi(\ha r)\in\R^2$. Since $\ha h$ is quasi-convex, the map $\Ga\to P\R^2$ which associates to $\ha r\in\ha\Ga$
the projective line generated by $\varpi(\ha r)$ is a local diffeomorphism. 

The submanifold $\Sig$ is clearly a transverse section for the vector field on $\cA_\eps$. The previous
property shows that the unperturbed return map associated with the vector field
\beq
\left\vert
\begin{array}{llcll}
\ha \th'&=&
\ha\om\big(\ha r, r_3^*(\ha r)+R_3(\ha r)\big)&
&\\[5pt]
\ha r'&=&  0                                &
&\\[5pt]
\end{array}
\right.
\eeq
is a twist map relatively to the coordinates $(\th_2,r_2)$. Since the complete map is a $\sqrt\de$ perturbation in the $C^1$
topology of the unperturbed one, it still admits the twist property when $\eps$ is small enough.
\end{proof}

We finally go back to the initial system by applying the inverse normalization
introduced in Appendix~\ref{App:globnormforms}, Proposition~\ref{prop:globnorm}.

\begin{lemma} Given $\mu>0$,
there exists $\eps_0$ such that for $0<\eps<\eps_0$ the Hamiltonian system $H_\eps$ admits a pseudo
invariant and pseudo normally hyperbolic cylinder $\jC_\eps$ at energy $\e$, which contains any invariant
set contained in the domains
$$
\begin{array}{ll}
D^0=\Big\{(\ha\th,\ha r,\th_3,r_3)\mid \ha\th\in\T^2,\ a\eps^{1/7}\leq \norm{\ha r-\ha r^0}\leq b\eps^{1/7},\ 
\abs{\th_3-\th_3^*(\ha r)}\leq \mu,\ \abs{r_3-r_3^*(\ha r)}\leq \mu\sqrt\eps\Big\},\\[5pt]
D^1=\Big\{(\ha\th,\ha r,\th_3,r_3)\mid \ha\th\in\T^2,\ a\eps^{1/7}\leq \norm{\ha r-\ha r^1}\leq b\eps^{1/7},\ 
\abs{\th_3-\th_3^*(\ha r)}\leq \mu,\ \abs{r_3-r_3^*(\ha r)}\leq \mu\sqrt\eps\Big\}.
\end{array}
$$
The cylinder $\jC_\eps$ admits a (non-connected) twist section.
\end{lemma}

\begin{proof} Recall that, by Proposition~\ref{prop:globnorm}:
$$
H_\eps=N_\eps\circ \Phi_\eps\inv,
$$ 
where, setting $\Phi_\eps=(\Phi_\eps^{\th},\Phi_\eps^{r})$:
\begin{equation}
\norm{\Phi_\eps^{\th}-\Id}_{C^0(\jW_{c\rho})}\leq c_\Phi\,\eps\,\rho^{-2}\leq c_\Phi\eps^{5/7},\qquad 
\norm{\Phi_\eps^{r}-\Id}_{C^0(\jW_{c\rho})}\leq c_\Phi\,\eps\,\rho^{-1}\leq c_\Phi\eps^{6/7}.
\end{equation}
The inverse image $\jC_\eps=\Phi_\eps\inv(\cC_\eps)$ is therefore a pseudo invariant and normally
hyperbolic cylinder for $H_\eps$, which contains any invariant set contained in $D^0\cup D^1$ (up to
the choice of suitable constants). The inverse image $\Phi\inv(\Sig)$ is a section for the Hamiltonian 
flow in $\jC_\eps$, which satisfies the twist condition since its return map is a small $C^1$ perturbation 
of that of $\Sig$, which admits a nondegenerate torsion.
\end{proof}


\subsection{Proof of Lemma~\ref{lem:extcyl}}\label{sec:prooflemextcyl}
\setcounter{paraga}{0}
We assume without loss of generality that $m^0=0$ and $\e=0$. We introduce adapted coordinates $(x,y)$
at $0$, in which the equation of the resonance $\Ga$ moreover reads $\om_3=0$. We set
\beq
[f](\ov x,y)=\int_\T f\big((x_1,\ov x),y\big)dx_1.
\eeq

\paraga 
Our starting point is the normal form of Proposition~\ref{prop:normal2}. 
We set, for $(\th,r)\in\T^3\times B(0,\eps^d)$:
\beq\label{eq:rednormform}
N_\eps(\th,r):=H_\eps\circ\Phi_\eps(\th,r)=h(r)+  g_\eps(\ov \th,r)+R_\eps(\th,r),\qquad
\eeq
where $g_\eps$ and $R_\eps$ are $C^p$ functions such that 
\beq
\label{eq:estimg}
\norm{g_\eps-\eps[f]}_{C^p\big( \T^2\times B(0,\eps^d)\big)}\leq \eps^{1+\sig},\qquad
\norm{R_\eps}_{C^p\big( \T^3\times B(0,\eps^d)\big)}\leq \eps^\ell.
\eeq
and $\Phi_\eps$ is $\eps^{\sig}$--close to the identity in the $C^p$--topology. We set
$$
V_\eps(\th_3,r)=\frac{1}{\eps}\int_{\T}g_\eps\big((\th_2,\th_3),r\big)\,d\th_2.
$$

\paraga We assumed that for each point $r$ of $\Ga$ in the neighborhood of $0$,
 the $s$--averaged potential 
 \beq
 <f>(\th_3,r)=\int_\T[f](\ov \th,r)d\th_2
 \eeq
  admits a single and nondegenerate 
maximum at some point $\th_3^*(r)$, which in turns yields the following result.

\begin{lemma} For each point $r\in\Ga$ in the neighborhood of $0$,
the function $V_\eps(\cdot,r)$ admits a unique and nondegenerate maximum on 
$\T$ at some point
$\th_3^{**}(r)$. Moreover, there is a constant $c>0$ such that
\beq\label{eq:locmax}
\abs{\th_3^{**}(r)-\th_3^{*}(r)}<c\,\eps^{\sig/2}.
\eeq
\end{lemma}

\begin{proof} By (\ref{eq:estimg}):
$$
\norm{V_\eps-<f>}_{C^p\big( \T\times B(0,\eps^d)\big)}\leq \eps^{\sig}
$$
with $p\geq 2$. The claim then immediately follows from the nondegeneracy of the maximum $\th_3^*(r)$.
\end{proof}

\paraga We introduce now the truncated normal form
$$
N_\eps^0(\th,r):=h(r)+  g_\eps(\ov \th,r).
$$
The main observation now is that  $N_\eps^0$ is independent
of $\th_1$, so that $r_1$ is a first integral and the system can be reduced to its level sets.
The total system is then recovered by taking the product with the circle $\T$ of $\th_1$.
So we fix $r_1$ and set
$$
\bh(\ov r)=h(r_1,\ov r),\qquad 
\bg_\eps(\ov \th,\ov r)=g_\eps\big(\ov \th,(r_1,\ov r)\big),\qquad 
\bN_\eps^0(\ov\th,\ov r)=\bh(\ov r)+\bg_\eps(\ov\th,\ov r).
$$
Observe that the function $\bh$ is convex and superlinear over $\R^2$. Therefore, setting $\bom=\nabla\bh$,
 the resonance
$\bGa$ defined by $\bom_3(\ov r)=0$ is a graph over the $r_2$--axis, with equation
$$
\bGa:\quad r_3=r_3^{**}(r_2;r_1).
$$

\paraga We set $\bV_\eps(\th_3,\ov r)=V_\eps\big(\th_3,(r_1,\ov r)\big)$. 
Fix a point 
$$
\ov r^0\in\bGa\cap B(0,\eps^d).
$$
 Let $S_\eps:\T^2\to\R$ be the solution of the homological equation 
$$
\bom(\ov r^0)\d_{\ov\th}S_\eps(\ov\th)=\bg_\eps(\ov\th,\ov r)-\eps\bV_\eps(\th_3,\ov r)
$$
that is, since $\bom_3(\ov r^0)=0$:
$$
S_\eps(\ov\th)=\frac{1}{\bom_2(\ov r^0)}\int_0^{\th_2}\big(\bg_\eps(\ov\th,\ov r)-\eps\bV_\eps(\th_3,\ov r)\big)\,d\th_2.
$$
In particular, there is a $\mu>0$ such that 
$$
\norm{S_\eps}_{C^p}\leq \frac{\mu}{\norm{\ovr^0}}.
$$
We introduce the the symplectic change
$$
\phi(\ov\th,\ov r)=\big(\ov\th,\ov r^0+\ov r-\eps\d_{\ov\th}S_\eps(\ov\th)\big)
$$

\paraga By immediate computation, one checks that
$$
\bN_\eps^0\circ\phi(\ov\th,\ov r)=\bh(\ov r^0+\ov r)
+\eps\bV_\eps(\ov \th_3,\ov r^0)
+\bW_0(\ovth,\ovr)
+\eps\bW_1(\ovth,\ovr)
+\eps^2\bW_1(\ovth,\ovr),
$$
with
$$
\bW_0(\ovth,\ovr)=\big(\bg_\eps(\ovth,\ovr^0+\ovr)-\bg_\eps(\ovth,\ovr^0)\big),
$$
$$
\bW_1(\ovth,\ovr)=\big(\bom(\ovr^0+\ovr)-\bom(\ovr^0)\big)\d_\ovth S
-
\int_0^1\d_\ovr\bg_\eps(\ovth,\ovr^0+\ovr-\sig\eps\d_\ovth S)(\d_\ovth S)\,d\sig,
$$
$$
\bW_2(\ovth,\ovr)=\int_0^1(1-\sig)D^2\bh(\ovth,\ovr^0+\ovr-\sig\eps\d_\ovth S)(\d_\ovth S)^2\,d\sig.
$$
Hence, there is an $M>0$ such that
$$
\begin{array}{lll}
\norm{\bW_0}_{C^2}\leq M\eps,\qquad \norm{\bW_0(\cdot,\ov r)}_{C^2}\leq M\eps\norm{\ov r},\\[6pt]
\norm{\eps\bW_1}_{C^2}\leq \dsp  M\frac{\eps}{\norm{\ovr^0}},\\[6pt]
\norm{\eps^2\bW_2}_{C^2}\leq \dsp M\frac{\eps^2}{\norm{\ovr^0}^2}.
\end{array}
$$

\paraga The resulting Hamiltonian differential equations read
\begin{equation}
\left\vert
\begin{array}{lllll}
\th_2'= 
\bom_2(\ov r)&\!\!\!
+&\!\!\!
+\ \partial_{r_2}\bW_0&\!\!\!
+\ \eps\,\partial_{r_2}\bW_1&\!\!\!
+\ \eps^2\,\partial_{r_2}\bW_2\\[5pt]
r_2'=                                   &&\!\!\!
-\ \partial_{\th_2}\bW_0&\!\!\!
-\ \eps\,\partial_{\th_2}\bW_1&\!\!\!
-\ \eps^2\,\partial_{\th_2}\bW_2\\[5pt]
\th_3'=
\bom_3(\ov r)&\!\!\!
+&\!\!\!
+\ \partial_{r_3}\bW_0&\!\!\!
+\ \eps\,\partial_{r_3}\bW_1&\!\!\!
+\ \eps^2\,\partial_{r_3}\bW_2\\[5pt]
r_3'=                                   &\!\!\!
-\ \eps\,V'(\th_3)&\!\!\!
-\ \partial_{\th_3}\bW_0&\!\!\!
-\ \eps\,\partial_{\th_3}\bW_1&\!\!\!
-\ \eps^2\,\partial_{\th_3}\bW_2.\\
\end{array}
\right.
\end{equation}

\paraga We now follow exactly the same lines as in the proof of Lemma~\ref{lem:mainpart}, in the much simpler
present framework. The previous estimates prove
the existence in the system $\bN_\eps^0$ of a ``local'' annulus
of hyperbolic periodic orbits $\sA(\ovr^0)$ ``centered at $\ovr^0$'', which is a graph of the form
\beq\label{eq:graphform}
\th_3=\Th_3(\th_2,r_2),\quad r_3=R_3(\th_2,r_2),\qquad 
(\th_2,r_2)\in \T\times \,\Big]r^0_2-c\norm{\ovr^0},r^0_2+c\norm{\ovr^0}\Big[.
\eeq
provided that 
$$
\norm{\ov r^0}\geq C\sqrt\eps
$$
where $C$ is large enough, with also $c<C$ smal enough.

\paraga Now, by immediate hyperbolic maximality for periodic orbits,
the union of all these ``local annuli'' is an annulus $\sA^{ext}$, which is a graph of the previous form with 
$$
(\th_2,r_2)\in \T\times\,]c''\sqrt\eps, c'''\eps^\nu[
$$ 
for suitable constants $c'',c'''>0$. Moreover, for any prescribed
constant $\mu>0$, one can choose $c''$ large enough and $c'''$ small enough so that
\beq\label{eq:localization}
\abs{\Th_3(\th_2,r_2)-\th_3^{**}\big(r_1,r_2,r_3^{**}(r_1,r_2)\big)}\leq \mu,\quad \abs{R_3(\th_2,r_2)-r^{**}_3(r_1,r_2)}\leq \mu\sqrt\eps,
\eeq
(recall that the value of $r_1$ was fixed at the beginning of this part).

\paraga Going back to the total truncated normal form $N_\eps^0$ in (\ref{eq:rednormform}) and varying the variable $r_1$, the
previous family of annuli give rise to a $4$--dimensional invariant hyperbolic annulus $\cA^0_\eps$, which is now a
graph of the form 
$$
\th_3=\Th_3(\ha\th,\ha r),\quad r_3=R_3(\ha\th,\ha r),\qquad 
(\th_1,\th_2)\in \T^2,\ r_1\in \,]-\ha c\eps^\nu,\ha c\eps^\nu[,\ r_2\in\,]\ha c\sqrt\eps,\ha c\eps^\nu[,
$$
for a suitable $\ha c>0$ and for $0<\eps<\eps_0$ small enough. 
Note that the restriction to $\cA^0_\eps$ of the Hamiltonian flow generated by $N_\eps^0$ is 
completely integrable, since $\cA^0_\eps$ is foliated by the invariant $2$--tori obtained by taking the product of the circle
of $\th_1$ with the periodic orbits foliating the annulus $\sA^{ext}$.

\paraga We now follow the same process as for the $d$--cylinders (Lemma~\ref{lem:existann}).
The previous integrability property yields symplectic angle-action coordinates on $\cA_\eps$.
This enable us to use the smallness of the complementary term $R$ in (\ref{eq:rednormform}) 
and get a perturbed $4$--dimensional pseudo invariant annulus $\cA_\eps$ and $3$--dimensional 
pseudo invariant cylinder $\cC^*_\eps=\cA_\eps\cap N_\eps\inv(0)$
at energy $0$ equipped with a global section  $\Sig$ with coordinates $(\th_1,r_1)$. Relatively to these 
coordinates, the Poincar\'e return maps takes the form described in Lemma~\ref{lem:section},
and we deduce as in Section~\ref{sec:applicKAM} the existence of Birkhoff-Herman tori arbitrarily close to 
the boundary of the cylinder. 

\paraga It remains now to prove that the pullback $\jC^*_\eps=\Phi\inv_\eps(\cC^*_\eps)$ of the previous cylinder
``continues'' the cylinder $\jC_\eps$ attached with  the annulus $\sA_\ell$ of the averaged system $C$. 
This will be done by proving that one (well-chosen) end of the cylinder $\jC^*_\eps$ is contained in $\jC_\eps$. 
Recall that the cylinder $\jC_\eps$ is maximal in some neighborhood of the form
$$
\jU(\jC_\eps):\quad \abs{\th_3-\th_3^*}\leq a,\quad \abs{r_3-r_3^*}\leq a\sqrt\eps,
$$
where $a$ is an arbitrary constant.
It is therefore enough to prove that at least two Birkhoff-Herman tori contained in the end of $\jC^*_\eps$ are contained
in $\jU(\jC_\eps)$. We will state the problem in the variables of $N_\eps$. Just as in the proof of Lemma~\ref{lem:dnormalizing}
the rescaling $r =\sqrt\eps\, \sr$, $t=\frac{1}{\sqrt\eps}\,{\bt}$ yields the new Hamiltonian
\begin{equation}
\sN_\eps(\th,\sr)=\frac{1}{\eps}\, N_\eps(\th,{\sqrt\eps}\,\sr).
\end{equation}
Then, performing a Taylor expansion of $h$ and $[f]$ at $r^0=0$ in $N_\eps$, one gets the normal form:
\begin{equation}
\sN_\eps(\th,\sr)=\frac{1}{\eps}N_\eps(\th,\sqrt\eps\sr)= \frac{\ha \om}{\sqrt\eps}\,\sr_1+\pdemi D^2h(0)\,\sr^2+ U(\ov\th)
+\sR^0_\eps(\ov\th,\sr)+\sR_\eps(\th,\sr)
\end{equation}
where 
$$
\norm{\sR^0_\eps}_{C^p(\T^2\times B^3(0,d^*))}\leq a\sqrt \eps,\qquad
\norm{\sR_\eps}_{C^p(\T^n\times B(0,d^*))}\leq \eps^\ell.
$$ 
Our statement will be an immediate consequence of (\ref{eq:locmax}) and (\ref{eq:localization}),
taking into account the following expansion:
\beq
\begin{array}{lll}
\frac{1}{\eps} \sN(\th,\sr)&=&\demi T(\sr)+U(\ov\th,0)+\big[\frac{1}{\eps}g_\eps(\ov\th,r)-U(\ov\th,0)\big]
+\big[U(\ov\th,0)-U(\ov\th,0)\big]+\sqrt\eps\ha h(\sr)\\
&=&C(\ov\th,\sr)+\sqrt\eps \ha C(\ov\th,\sr),
\end{array}
\eeq
where $\ha h$ stands for the third order term in the Taylor expansion of $h$ at $0$.

\paraga Finally, as for the other boundary torus of $\jC^{ext}_\eps$, one only has to apply the localization
statement on the periodic orbits of the extremal annulus $\cA^0_\eps$ together with the results of 
Lemma~\ref{lem:section}
and Section~\ref{sec:applicKAM}. This concludes the proof of Lemma~\ref{lem:extcyl}.


\subsubsection{Bifurcation points}\label{Sec:bifurcation1}
In this part we examine the case where bifurcation points may exist between the resonant points $m^0$ and $m^1$.
Three types of intervals of $\Ga$ have to be considered: $[m^0,b]$, $[b^0,b^1]$ and $[b,m^1]$. We will limit ourselves
to the first one, the other two being essentially equivalent (and simpler for the second one).

\begin{lemma}\label{lem:bifurcation1}
Assume that the interval $[m^0,b]$ contains no other bifurcation points than $b$. Then  for $\eps_0>$ small enough there
exists a family $(\jC_\eps)_{0<\eps<\eps_0}$ of (invariant and normally hyperbolic) $s$-cylinders along $[m^0,b]$, in the sense that:
\begin{itemize}
\item $\jC_\eps$ contains the extremal cylinder $\jC_\eps^{ext}(m^0)$ of {\rm Lemma~\ref{lem:extcyl}},
\item the projection in action of the other boundary of $\jC_\eps$ is located in a ball $B(b_+,\sqrt\eps)$, where 
$b_+$ is such that $[m^0,b]\subset [m^0,b_+[$,
\item the projection in action of $\jC_\eps$ is located in an $O(\sqrt\eps)$ tubular neighborhood of $\Ga$.
\end{itemize}
\end{lemma}

\begin{proof} The global normal form used in the proof of Lemma~\ref{lem:mainpart} is still valid here and results in the
existence of a pseudo invariant cylinder whose ``left extremity'' contains the extremal cylinder $\jC_\eps^{ext}(m^0)$,
and which moreover admits a twist section.  It remains to study the behavior of the cylinder in the neihborhood of the
point $b$. 

We will use the $\eps$-dependent normal form of Appendix~\ref{app:normformepsdep}. For this, we fix two actions $b_0,b_1$ on
the resonant circle $\Ga$, very close to one another, such that the bifurcation point in the small interval they delimit. We
moreover assume that  the frequency vectors $\om(b_1)$, $\om(b_2)$ are $2$-Diophantine and that the $s$-averaged potential
staill admits a nondegenerate maximum in a neighborhood of $b_1$ and $b_2$. We will prove the statement for the
first cyllinder, the other one being similar.
We set
$$
[f](\th_3,r)=\int_{\T^{2}}f\big((\ha\th,\th_3),r\big)\,d\ha\th.
$$
Given two integers $p,\ell\geq 2$ and two constants $d>0$ and $\de<1$ with $1-\de>d$,
then, if $\ka$ is large enough, there is an $\eps_0>0$ such that for $0<\eps<\eps_0$, there exists an analytic symplectic embedding 
$$
\Phi_\eps: \T^3\times B(b_1,\eps^d)\to \T^3\times B(b_1,2\eps^d)
$$
such that 
$$
N_\eps(\th,r)=H_\eps\circ\Phi_\eps(\th,r)=h(r)+  g_\eps(\th_3,r)+R_\eps(\th,r),
$$
where $g_\eps$ and $R_\eps$ are $C^p$ functions such that 
\beq
\norm{g_\eps-\eps[f]}_{C^p\big( \T\times B(b_1,\eps^d)\big)}\leq \eps^{2-\de},\qquad
\norm{R_\eps}_{C^p\big( \T^n\times B(b_1\eps^d)\big)}\leq \eps^\ell.
\eeq
Moreover, $\Phi_\eps$ is close to the identity, in the sense that
\beq
\norm{\Phi_\eps-\Id}_{C^p\big( \T^3\times B(b_1,\eps^d)\big)}\leq \eps^{1-\de}.
\eeq
One proves as in the previous section that the function $g_\eps(\cdot,r):\T\to\R$
admits for $r$ close to $b_1$ a unique and nondegenerate maximum at some $\th_3(r)$. Now the differential system
generated by the truncated system $N_\eps^0(\th,r)=h(r)+  g_\eps(\th_3,r)$ reads
$$
\left\vert
\begin{array}{lllll}
\ha \th'= 
\ha\om(r)&\!\!\!
+\ \eps\d_{\ha r}g_\eps(\th_3,r)\\[5pt]
\ha r'=  0                                 &\\[5pt]
\th_3'=
\om_3(r)&\!\!\!
+\ \eps\d_{r_3}g_\eps(\th_3,r)\\[5pt]
r_3'=                                   &\!\!\!
-\ \eps\d_{\th_3}g_\eps(\th_3,r).\\
\end{array}
\right.
$$
One proves exactly as in the same way as in the previous sections the existence of a pseudo invariant
normally hyperbolic $4$ annulus of the form
$$
\ha\th\in\T^2,\quad \norm{\ha r-\ha b_1}\leq d\sqrt\eps,\quad \abs{\th_3-\th_3(r)}\leq \de,\quad \abs{\th_3-\th_3(r)}\leq\de\sqrt\eps.
$$
Moreover, this annulus carries a completely integrable Hamiltonian flow and is foliated by the invariant tori
$\ha r= cte$. Then we proceed as in the case of the $d$-cylinders to prove the existence of a section
$\Sig$ for this flow, whose attached return map admits nondegenerate torsion. Finally the existence of
invariant circles for this return map is proved by the Herman theorem. 
\end{proof}


\section{Intersection conditions and chains}\label{sec:chains}
This section is devoted to the precise description of homoclinic and heteroclinic properties of the cylinders
we got in the previous one. These properties will be of crucial use in \cite{M} in order to prove the generic
existence of orbits shadowing the chains of cylinders.

\setcounter{paraga}{0}
\subsection{\bf Intersection conditions, gluing condition, and admissible chains}
Let $H$ be a proper $C^2$ Hamiltonian function on $\A^3$ and fix a regular value~$\e$. 

\paraga{\bf Oriented cylinders.} We say that a cylinder $\jC$ is {\em oriented} when an order
is prescribed on the two components of its boundary. We denote the first one by $\d_\bu\jC$
and the second one by $\d^\bu\jC$.


\paraga {\bf The homoclinic condition \fomu.} 
A compact invariant cylinder $\jC\subset H\inv(\e)$ with twist section $\Sig$ and associated
invariant symplectic $4$-annulus $\jA$ satisfies condition \fomu\ when there exists a 
$5$-dimensional submanifold $\De\subset \A^3$, transverse to $X_H$ such that,  
\begin{itemize}
\item  there exist $4$-dimensional submanifolds $\jA^\pm\subset W^\pm(\jA)\cap\De$ such that the
restrictions to $\jA^\pm$ of the characteristic projections $\Pi^\pm:W^\pm(\jA)\to\jA$ are
diffeomorphisms on $\jA$, whose inverses we denote by $j^\pm:\jA\to\jA^\pm$;
\item there exists a continuation $\jC_*$ of $\jC$ such that 
 the $3$-dimensional manifolds $\jC^\pm_*=j^\pm(\jC_*)$ 
have a nonempty intersection, transverse in the $4$-dimensional manifold 
$
\De_\e:=\De\cap H\inv(\e),
$;
\item the projections $\Pi^\pm(\jI)\subset\jC$ are $2$-dimensional transverse sections of the vector field $X_H$ 
restricted to $\jC$, and the associated Poincar\'e maps $P^\pm: \Pi^\pm(\jI)\to\Sig_k$ are diffeomorphisms.
\end{itemize}

Note that $\jC_*^+\cap\jC_*^-$ is a $2$-dimensional submanifold of $\De_\e$, whose role will be crucial in
\cite{M}.

\paraga {\bf The heteroclinic condition \fomd.} 
A pair  $(\jC_0,\jC_1)$ of compact invariant oriented cylinders with twist sections 
$\Sig_0$, $\Sig_1$ and associated invariant symplectic $4$-annuli $(\jA_0,\jA_1)$
satisfies  condition \fomd\ when there exists a $5$-dimensional submanifold 
$\De\subset \A^3$, transverse to $X_H$ such that:
\begin{itemize}
\item  there exist $4$-dimensional submanifolds 
$\til \jA_0^-\subset W^-(\jA_0)\cap\De$ 
and
$\til \jA_1^+\subset W^+(\jA_1)\cap\De$ 
such that 
$\Pi_0^-{\vert \til\jA_0^-}$
and
$\Pi_1^+{\vert \til\jA_1^+}$
are diffeomorphisms on their images $\til \jA_0$, $\til \jA_1$, which we require to be neighbohoods of the 
boundaries $\d^\bu\jC_0$ and $\d_\bu\jC_1$ in $\jA_0$ and $\jA_1$ respectively,
we denote their inverses by $j_0^-$ and $j_1^+$;
\item there exist neighborhoods $\til\jC_0$ and $\til\jC_1$ of $\d^\bu\jC^0$ and $\d_\bu\jC^1$ in continuations
of the initial cylinders, such that
$\til \jC_0^-=j_0^-(\til\jC_0)$ and $\til\jC_1^+=j_1^+(\til\jC_0)$  intersect transversely 
in the $4$-dimensional manifold 
$
\De_\e:=\De\cap H\inv(\e),
$, let $\jI_*$ be this intersection;
\item the projections $\Pi_0^-(\jI_*)\subset\jC$ and $\Pi_1^+(\jI_*)\subset\jC$ are $2$-dimensional 
transverse sections of the vector field $X_H$ 
restricted to $\til \jC_0$ and $\til\jC_1$, and the Poincar\'e maps $P_0: \Pi_0^-(\jI_*)\to\Sig_0$ and $P_1: \Pi_1^-(\jI_*)\to\Sig_1$ 
are diffeomorphisms (where $\Sig_I$ stands for Poincar\'e sections in the neighborhoods $\til\jC_i$).
\end{itemize}


\paraga {\bf The homoclinic condition \pomu.}  
Consider an invariant cylinder $\jC\subset H\inv(\e)$ with twist section $\Sig$ and attached Poincar\'e
return map $\ph$, so that $\Sig=j_\Sig(\T\times[a,b])$, where $j_\Sig$ is exact-symplectic.
Define $\Tess(\jC)$ as the set of all invariant tori generated by the previous
circles under the action on the Hamiltonian flow (so each element of $\Tess(\jC)$ is a Lispchitzian Lagrangian
torus contained in $\jC$). The elements of $\Tess(\jC)$ are said to be {\em essential tori}.
 
\vskip2mm

We say that an invariant cylinder $\jC$ with  associated invariant symplectic $4$-annulus $\jA$ satisfies the 
{\em partial section property~\pom} when there exists a $5$-dimensional submanifold $\De\subset \A^3$, 
transverse to $X_H$ such that:
\begin{itemize}
\item  there exist $4$-dimensional submanifolds $\jA^\pm\subset W^\pm(\jA)\cap\De$ such that the
restrictions to $\jA^\pm$ of the characteristic projections $\Pi^\pm:W^\pm(\jA)\to\jA$ are
diffeomorphisms, whose inverses we denote by $j^\pm:\jA\to\jA^\pm$;
\item there exist conformal exact-symplectic diffeomorphisms 
\beq
\Psi^{\rm ann}:\jO^{\rm ann}\to \jA,\qquad \Psi^{\rm sec}:\jO^{\rm sec}\to\De_\e:=\De\cap H\inv(\e)
\eeq
where $\jO^{\rm ann}$ and $\jO^{\rm sec}$ are neighborhoods of the zero section in $T^*\T^2$ endowed with
the conformal Liouville form $a\la$ for a suitable $a>0$;
\item each  torus $\jT\in\Tess(\jC)$ is contained in some $\jC$ and the image $\Psi^{\rm ann}(\jT)$ is a Lipschitz
graph over the base $\T^2$;
\item for each such torus $\jT$, setting $\jT^\pm:=j^\pm(\jT)\subset \De_\e$, the images
$\Psi^{\rm sec}(\jT^\pm)$ are  Lipschitz graphs over the base $\T^2$.
\end{itemize}


\paraga {\bf Bifurcation condition.}  For bifurcations points we just rephrase our
nondegeneracy condition \CSd: for any $r^0\in B$,  the derivative $\tfrac{d}{dr}\big(m^*(r)-m^{**}(r)\big)$ 
does not vanish. This will immediately yield transverse heteroclinic intersection properties for the intersections
of the corresponding cylinders.


\paraga {\bf The gluing condition \glu.}  A pair  $(\jC_0,\jC_1)$ of compact invariant oriented cylinders 
satisfies  condition \glu\ when they are contained in a invariant cylinder and satisfy 
\begin{itemize}
\item  $\d^\bu\jC_0=\d_\bu\jC_1$ is a dynamically minimal invariant torus that we denote by $\jT$,
\item $W^-(\jT)$ and $W^+(\jT)$ intersect transversely in $H\inv(\e)$.
\end{itemize}


\paraga {\bf Admissible chains.}
A finite family of compact invariant oriented cylinders $(\jC_k)_{1\leq k\leq k_*}$ is an {\em admissible chain} 
when each cylinder satisfies either $\fomu$ or $\pomu$ and, for  $k\in\{1,\ldots,k_*-1\}$,
the pair $(\jC_k,\jC_{k+1})$ satisfies either $\fomd$ or $\glu$, or correspond to a bifurcation point.


\paraga The main result of this section is the following. Recall that $\Pi:\A^3\to\R^3$ is the second projection
and $\bd$ the Hausdorff distance between compact subsets of $\R^3$.

\begin{prop}\label{prop:intersection} Fix a $C^\ka$ Tonelli Hamiltonian $h$,
let  $f\in C_b^\ka(\A^3)$ and set $H_\eps=h+\eps f$. Assume that $H$ satisfies \CS\ along
$\Ga$. Then there is an  $\eps_0>0$ such that for $0<\eps\leq\eps_0(\de)$, 
there exists an admissible
chain $\big(\jC_k(\eps)\big)_{0\leq k\leq k_*}$ of cylinders and singular cylinders at energy 
$\e$ for $H_\eps$,  whose projection by $\Pi$ satisfies
$$
\bD\Big(\bigcup_{1\leq k\leq k^*},\Ga\Big)=O(\sqrt\eps).
$$
\end{prop}

Note that $k^*$ is {\rm independent} of $\eps$.
The rest of this section is dedicated to the proof of the previous proposition


\subsection{Proof of Proposition~\ref{prop:intersection}}


\subsubsection{Conditions \fomu\ and \fomd\ for $d$-cylinders}
In this section we prove the following lemma.

\begin{lemma}\label{lem:dcondFS}
 Let $r^0$ be a double resonance point of $D$, with $d$-averaged system $C$.
\begin{itemize}
\item Fix a compact annulus $\sA$ of $C$, defined over $I$, with continuation $\sA^*$ defined over $I^*$. 
We assume that for each energy $e\in I^*$, there exists a homoclinic solution $\ze_e$ for $\ga_e$, 
continuously depending on $e$, such that  the stable and unstable manifolds of $\ga_e$ transversely 
intersect in $C\inv(e)$. Then there is an $\eps_0>0$ such that for $0<\eps\leq\eps_0$ the cylinder
$\jC_\eps$ satisfies the homoclinic condition \fomu. 
\item Fix two annuli $\sA$ and $\sA'$ of $C$ defined over adjacent intervals $I$ and $I'$ and
such that the boundary orbits at $e$ admit transverse heteroclinic connections. 
Then there is an $\eps_0>0$ such that for $0<\eps\leq\eps_0$ the associated cylinders
$\jC_\eps$ and $\jC'_\eps$ satisfies the hetroclinic condition \fomd. 
\end{itemize}
\end{lemma}

\begin{proof} We begin with the case of a single annulus $\sA$, with continuation $\sA^*$.

\paraga Let $\pi^\pm:W^\pm(\sA^*)\to \sA^*$ be the characteristic projections. 
Fix a $C^\ka$ arc $w:I^*\to \A^2$ such that 
$w(e)\in \ze_e\setm\ga_e$ for $e\in I^*$.
Since the manifolds $W^+(\ga_e)\trans W^-(\ga_e)$ in $C\inv(e)$ and since
the vector field $X^C$ is transverse to the stable and unstable foliations of $W^\pm(\ga_e)$
inside these manifolds, the tangent vectors to the stable and unstable leaves at $w(e)$ together
with $X^C(\nu(e))$ form a basis of $T_{w(e)}C\inv(e)$.
The $C^{\ka-1}$  curves
\beq
\sig^\pm=\pi^\pm\big(w(I^*)\big)\subset \sA^*.
\eeq
are clearly global sections
of the Hamiltonian flow of $X^C$ on the annulus $\sA$.
Finally, since $w(I^*)$ is contractible, one can find a $3$--dimensional $C^\ka$ submanifold 
$S\subset\A^2$ containing $w(I^*)$ and transverse to
$X^C$.

\paraga Consider now the Hamiltonian $\sN_\eps$ on $\A^3$. To ged rid of the artificial singular term $\ha\om\,\ha \sr/\sqrt\eps$,
we perform the symplectic change $\bchi=\chi\times\Id:\A_\eps\times\A^2\to\A^3$, where  
$\chi(\xi,\eta)=(\ha\th=\xi/\sqrt\eps,\ha\sr=\sqrt\eps\eta)$. 
Hence $\sN_\eps\circ\bchi$ is now defined on $\T_\eps\times\,]-1,1[\,\times\T^2\times B^2(0,d^*)$, with explicit expression
\beq\label{eq:snchi}
\sN_\eps\circ\bchi(\xi,\eta,\ov \th,\ov\sr)=\ha\om\,\eta+Q(\sqrt\eps\eta,\ha\sr)+C(\ov\th,\ov\sr)+\sR_\eps^0(\sqrt\eps\eta,\ha\th,\ha\sr)
+\sR_\eps(\tfrac{1}{\sqrt\eps}\xi,\sqrt\eps\eta,\ha\th,\ha\sr).
\eeq

\paraga The system (\ref{eq:snchi}) is an $O(\sqrt\eps)$--perturbation in the $C^p$ topology of the truncated form
$$
\Av(\xi,\eta,\ov \th,\ov\sr)=\ha\om\,\eta+C(\ov\th,\ov\sr),
$$
which admits the product $\Sig:=\T_\eps\times\,]-1,1[\,\times\sA^*$ as an invariant annulus, with
stable and unstable manifolds $W^\pm(\Sig)=\T_\eps\times\,]-1,1[\,\times W^\pm(\sA^*)$. The characteristic
projections also inherits the same product structure: for any $(\xi,\eta,\ov \th,\ov\sr)\in W^\pm(\Sig)$
$$
\Pi^\pm(\xi,\eta,\ov \th,\ov\sr)=\big(\xi,\eta,\pi^\pm(\ov \th,\ov\sr)\big).
$$
Moreover, $W^\pm(\Sig)$ transversely intersect one another along 
$\T_\eps\times\,]-1,1[\,\times \big(W^+(\sA^*)\cap W^-(\sA^*)\big)$ in $\A^3$, and both manifolds transversely
intersect the section $\bS:=\T_\eps\times\,]-1,1[\,\times S$ in $\A^3$. Let 
$$
\La_\eps:=\bS\cap W_{loc}^+(\Sig)\cap W_{loc}^-(\Sig),
$$
then $\Pi^\pm(\La)=\T_\eps\times\,]-1,1[\,\times \sig^\pm$ are global transverse sections for the unperturbed
flow on $\Sig$.

\paraga We proved in the previous section that the annulus $\Sig$ persists in the system $\sN_\eps\circ\bchi$
and gives rise to an annulus $\cA_\eps$ which is $O(\sqrt\eps)$ close to $\Sig$ in the $C^p$ topology.
By the nomally hyperbolic persistence theorem, the local parts of the stable and unstable manifolds limited by $\Sig$ and $\bS$
also persist in the perturbed system.
The corresponding parts $W_{loc}^\pm(\cA_\eps)$ are $O(\sqrt\eps)$ close to the unperturbed
ones in the $C^p$ topology. As a consequence, for $\eps$ small enough, these manifolds transversely intersect
in $\A^3$ and they both transversely intersect $\bS$ in $\A^3$. We set
$$
\La_\eps:=\bS\cap W_{loc}^+(\cA_\eps)\cap W_{loc}^-(\cA_\eps).
$$
The characteristic foliations on $W_{loc}^\pm(\cA_\eps)$ are $O(\sqrt\eps)$--perturbations of those of $W_{loc}^\pm(\Sig)$
in the $C^{p-1}$ topology. As a consequence the characteristic projections $\Pi^\pm_\eps$ are also 
$O(\sqrt\eps)$--perturbations of $\Pi^\pm$ and their restriction to $\La_\eps$ are embeddings into $\cA_\eps$.
Finally their images $\Sig_\eps^\pm$ are clearly transverse sections for the restriction to $\cA_\eps$ of
Hamiltonian vector field generated by $\sN_\eps\circ\bchi$ for $\eps$ small enough. This concludes the proof
of the first part of the lemma.

\vskip3mm

\setcounter{paraga}{0}
We now fix two compact annuli $\sA$ and $\sA'$ of $C$, defined over $I,I'$ such that $I\cap I'\neq\emptyset$
and such there exists $e_0\in I\cap I'$ with $W^-(\ga_{e_0})\trans W^+(\ga'_{e_0})$ inside $C\inv(e)$. 
We fix an interval $I_\circ\subset I\cap I'$ over which the periodic solutions $\ga_e$ and $\ga'_e$ admit
a heteroclinic solution $\ze_e$ continuously depending on $e$.
We let $\sA_\circ$ and $\sA'_\circ$ be the corresponding annuli for $C$
We finally fix  coherent families  ($\jC_{\circ\eps}$, $\jA_{\circ\eps}$),  ($\jC'_{\circ\eps}$, $\jA'_{\circ\eps}$) 
attached to $\sA_\circ$ and $\sA'_\circ$. The proof essentially follows the same lines as the previous one.

\paraga Let $\pi^-:W^\pm(\sA)\to \sA$ and $(\pip)^-:W^\pm(\sA')\to \sA'$ be the characteristic projections. 
Fix a $C^\ka$ arc $w:I\to \A^2$ such that 
$w(e)\in \ze_e\setm{\ga_e\cap \ga'_e}$ for $e\in I$.
The tangent vectors to the stable and unstable leaves at $w(e)$ together
with $X^C(\nu(e))$ form a basis of $T_{w(e)}C\inv(e)$.
The $C^{\ka-1}$  curves
\beq
\sig^-=\pi^\pm\big(w(I^*)\big),\quad (\sig')^+=(\pip)^+\big(w(I^*)\big).
\eeq
are global sections of the Hamiltonian flow of $X^C$ on the annulus $\sA$.
Let $S\subset\A^2$ be a transverse section containing $w(I^*)$.
We keep the same convention as above for $\sN_\eps$ and $\sN_\eps\circ\bchi$.

\paraga The system 
$$
\Av(\xi,\eta,\ov \th,\ov\sr)=\ha\om\,\eta+C(\ov\th,\ov\sr),
$$
admits the products $\Sig:=\T_\eps\times\,]-1,1[\,\times\sA$ and  $\Sig':=\T_\eps\times\,]-1,1[\,\times\sA'$ as invariant annuli, 
whose stable and unstable manifolds and characteristic
projections have same product structure as in the previous section.
Now, $W^\pm(\Sig)$ transversely intersect one another along 
$\T_\eps\times\,]-1,1[\,\times \big(W^-(\sA)\cap W^+(\sA')\big)$ in $\A^3$ and both manifolds transversely
intersect the section $\bS:=\T_\eps\times\,]-1,1[\,\times S$ in $\A^3$. We set 
$$
\La:=\bS\cap W_{loc}^-(\Sig)\cap W_{loc}^+(\Sig'),
$$
then $\Pi^-(\La)$ and $(\Pi')^+(\La)$ are global transverse sections for the unperturbed
flows on $\Sig$ and $\Sig'$.

\paraga By the same perturbative argument as above,
$$
\La_\eps:=\bS\cap W_{loc}^+(\cA_\eps)\cap W_{loc}^-(\cA'_\eps).
$$
satsifies the conditions of our claim. This concludes the proof
of Lemma~\ref{lem:dcondFS}.
\end{proof}


\subsubsection{Condition \pomu\ for $s$-cylinders}
\setcounter{paraga}{0}

We end the proof of the existence of a covering of an $s$-cylinders with subcylinders which admits a twist
section, and by the same token we prove the graph properties of Condition \pomu.

\paraga {\bf The section $\De$ and the transition diffeomorphisms.}
 We want now to go back to the normal forms of Lemma~\ref{lem:mainpart}
which we localize over domains of diameter $\sqrt\eps$ in the $r$ variable, which will enable us to get a covering of a neighborhood
of the annulus $\cA_\eps$ with domains in which we can control the behavior of its center-stable and center-unstable foliations
(when properly defined). 
The main preliminary observation is the following well-known one, whose proof can be easily deduced 
from \cite{Po93}.

\begin{lemma}
With the assumptions of {\rm Lemma~\ref{lem:mainpart}}, there exists a constant $a>0$ such that the annulus $\cA_\eps$
admits a covering by {\em invariant} subannuli $(\cA_\eps^{m})_{1\leq m\leq m_*(\eps)}$, with diameter $\leq a\sqrt\eps$. 
\end{lemma}

The main result of this part is the following.

\begin{lemma}
With the assumptions of {\rm Lemma~\ref{lem:mainpart}}, assume moreover that $\abs{\th_3^*(\ha r)}\leq 1/4$.
Fix $\de>0$. Then the section
\beq
\De=\Big\{(\th,r)\in\A^3\mid \th_3=\demi\Big\}
\eeq
is transverse to the Hamiltonian vector field and intersects the manifolds $W^\pm(\cA_\eps)$ transversely in $\A^3$.
Moreover, there exist 
conformal exact-symplectic diffeomorphisms 
\beq
\Psi^{\rm ann}:\jO^{\rm ann}\to \jA,\qquad \Psi^{\rm sec}:\jO^{\rm sec}\to\De_\e:=\De\cap H\inv(\e)
\eeq
where $\jO^{\rm ann}$ and $\jO^{\rm sec}$ are neighborhoods of the zero section in $T^*\T^2$ endowed with
the conformal structure $\sqrt\eps \sum r_id\th_ i$ in suitable coordinates. 
Relatively to the induced coordinates, the characteristic transition diffeomorphisms $j^\pm$ are $\de$-Lipschitz
when $\eps$ is small enough.
\end{lemma}

\begin{proof}
The differential system associated with $X^{N_\eps}$ reads
\begin{equation}\label{eq:vectfield2}
\left\vert
\begin{array}{lllll}
\ha \th'= 
\ha\om(r)&\!\!\!
+\ \eps\d_{\ha r}V(\th_3,r)&\!\!\!
+\ \eps\,\partial_{\ha r}W_0(\th,r)&\!\!\!
+\ \eps\,\partial_{\ha r}W_1(\th,r)&\!\!\!
+\ \eps^2\,\partial_{\ha r}W_2(\th,r)\\[5pt]
\ha r'=                                   &&\!\!\!
-\ \eps\,\partial_{\ha \th}W_0(\th,r)&\!\!\!
-\ \eps\,\partial_{\ha \th}W_1(\th,r)&\!\!\!
-\ \eps^2\,\partial_{\ha \th}W_2(\th,r)\\[5pt]
\th_3'=
\om_3(r)&\!\!\!
+\ \eps\d_{r_3}V(\th_3,r)&\!\!\!
+\ \eps\,\partial_{r_3}W_0(\th,r)&\!\!\!
+\ \eps\,\partial_{r_3}W_1(\th,r)&\!\!\!
+\ \eps^2\,\partial_{r_3}W_2(\th,r)\\[5pt]
r_3'=                                   &\!\!\!
-\ \eps\d_{\th_3}V(\th_3,r)&\!\!\!
-\ \eps\,\partial_{\th_3}W_0(\th,r)&\!\!\!
-\ \eps\,\partial_{\th_3}W_1(\th,r)&\!\!\!
-\ \eps^2\,\partial_{\th_3}W_2(\th,r).\\
\end{array}
\right.
\end{equation}

We fix $r^0\in \Ga^*$ and we localize the study to the domain 
\beq
\sD_\eps=\Big\{(\th,r)\in\A^3\mid \norm{r-r^0}\leq \al\sqrt\eps,\ \abs{\th_3-\th_3^*(\ha r^0)}\leq\sqrt\de\Big\}
\eeq
where $\al>0$ is a fixed constant. We perform the conformally symplectic change of variables
\beq
\ph=\ga\th,\qquad r-r^0=\ga\inv\sqrt\eps\br,
\eeq
and we set
\beq
\bph_3=\ph_3-\ph^*_3(\ha I).
\eeq
This yields the new system
\begin{equation}\label{eq:vectfield3}
\left\vert
\begin{array}{lccll}
\dot{\ha \ph}= 
\ga\dsp\frac{\ha\om(r^0)}{\sqrt\eps}+
\ga D\ha\om(r^0)(\br)&+&
\sqrt\eps\, G_{\ha \th}(\ph,\br)\\[5pt]
\dot{\ha\br}=0                              
&+&
\de\, G_{\ha r}(\ph,\br)\\[5pt]
\dot\bph_3=
a(\ha r )\br_3&+&
\sqrt\eps\, G_{\ha \th}(\ph,\br)\\[5pt]
\dot \br_3=b(\ha r)\bph_3&+&
\de\, G_{\ha r}(\ph,\br).\\[5pt]
\end{array}
\right.
\end{equation}
with
\beq
a(\ha r)=\ga\d_{r_3}\om_3\big(\ha r,r^*_3(\ha r)\big),\qquad b(\ha r)=-\ga\inv\d^2_{\th_3^2}V(\th_3,\ell(\ha r)).
\eeq
To diagonalize the hyperbolic part, we set
\beq
u=\ze(\ha r)\br_3+\ze(\ha r)\inv\bth_3,\qquad s=\ze(\ha r)\br_3+\ze(\ha r)\inv\bth_3,\qquad \ze=a(\ha r)^{1/4}b(\ha r)^{1/4},
\eeq
This in turn yields the system
\begin{equation}\label{eq:vectfield1}
\left\vert
\begin{array}{lccll}
\dot{\ha \ph}= 
\ga\dsp\frac{\ha\om(r^0)}{\sqrt\eps}+
\ga D\ha\om(r^0)(\br)&+&
F_{\ha \ph}(\ph,\br)\\[5pt]
\dot{\ha\br}=0                              
&+&
F_{\ha r}(\ph,\br)\\[5pt]
\dot u=
\la(\br) u&+&
F_{u}(\ph,\br)\\[5pt]
\dot s=-\la(\br)s&+&
F_{s}(\ph,\br).\\[5pt]
\end{array}
\right.
\end{equation}
where
$
\la(\br)=\sqrt{a(\br)b(\br)}
$
and 
\beq
\norm{\big(F_{\ha \ph},F_{\ha \br},F_{u},F_{s}\big)}_{C^2}\leq \de.
\eeq
The main interest of the previous change is that now the time-one map of the unperturbed flow
is $C^0$ and $C^1$ bounded by a constant which is {\em independent of $\eps$}. 
By the persistence theorem and the covering argument, this proves that the stable and unstable manifolds
of the annulus $\cA_\eps\cap \sD_\eps$ admit the equations
\beq
\br_3=R_3^\pm(\ha\ph,\ha \br,\ha\bph_3),\qquad \norm{R^\pm_3}_{C^2}\leq M\de.
\eeq
Their characteristic vector fields are therefore $\de$-small in the $C^1$ topology, which proves the
existence of a small constant $\ell$ such that the section
\beq
\ov\De=\{\abs{\th_3-\th_3^*(\ha r)} =\ell\}
\eeq
which intersects $W^\pm(\cA_\eps)$ transversely, and such that the transition diffeomorphisms induced
by the characteristic flow are $\de$-close to the identity in the $C^1$ topology. 

Setting finally 
\beq
\ov\De=\{\th_3=\pdemi\}
\eeq
one immediately sees from the system (\ref{eq:vectfield2}) that the transition diffeomorphisms 
induced by the characteristic flows on $W^\pm(\jA_\eps)$ between $\jA$ and the section $\ov\De$ 
are also $\de$-close to the identity in the $C^1$ topology, relatively to the coordinates $(\ha\th,\ha\br)$,
when $\eps$ is small enough.

This proves in particular that  the image by  the transition diffeomorphisms of an essential torus contained 
in $\jA_\eps$, which is a graph in the coordinates $(\ha\th,\ha\br)$,  remains a graph over the base 
$\T^2$ in the previous coordinates, provided that its Lipschitz constant is small enough. But one
can choose this constant arbitrarily small by assuming $\eps$ small enough, by usual theorems on
twist maps, since the return map on the section $\Sig$ is a perturbation of the integrable twist.
\end{proof}


\section{From nondegeneracy to cusp-genericity}\label{sec:cuspgen}

\setcounter{paraga}{0}

\paraga We keep the assumptions and notation of Theorem~I. We select a finite set of simple resonance circles 
$\Ga_1,\ldots,\Ga_\ell$ at energy $\e$,
whose union contains a broken line of resonance arcs $\bGa$ with intersections in the open sets $O_i$, as depicted
in Section~\ref{sec:TheoremI} of the Introduction.
We first prove that Conditions {\bf S} hold for each resonance $\Ga_i$ 
for a residual set of functions $f$ in $\jS^\ka$. 

\paraga We introduce the averaging operator $\jI$ along the resonance curve $\Ga_k$. We arbitrarily choose  
a $C^\infty$ parametrization $\tau:\T\to\Ga_k$ and we fix adapted coordinates $(\th,r)$ to $\Ga_k$.  We define
$$
\begin{array}{rll}
\jI_k : C_b^\ka(\A^3,\R)&\longrightarrow& C^\ka(\T\times \T,\R)\\
f&\longmapsto&g(\ph,s)=\dsp\int_{\T^2} f\big((\ha \th, \ph),\tau(s)\big)\,d\ha \th.
\end{array}
$$
Then $\jI_k$ is clearly linear, continuous and surjective, so it is an open mapping. 

\paraga By classical Morse theory, the subset $\jV\subset C^\ka(\T\times\T,\R)$ of all functions $V(\ph,s)$ such that 
$V(\cdot,s)$ admits a single and nondegenerate global maximum for $\ph\in\T\subset B$, where $B$ is a finite subset of $\T$,
and exactly two nondegenerate  global maximums with ``transverse crossing'' at the points of $B$, is open dense in  
$C^\ka(\T\times\T,\R)$.

\paraga The inverse image of $\jV$ by  $\jI_k$ is therefore open dense in $C^\ka(\A^3)$ and so is its intersection with the unit 
sphere $\jS^\ka$, by linearity. This is precisely the set of functions $f\in\jS^\ka$
which satisfy the nondegeneracy assumptions ${\bf (S_1)}$ and $\bf (S_2)$ for $\Ga_k$. 
One gets the opennes and density of the same conditions for $1\leq k\leq \ell$ by finite intersection.

\paraga  As for condition ${\bf (S_3)}$, we follow a similar method and introduce the averaging operator attached to 
a given double resonance point $r^0\in \Ga_{k}$, that is, the function  
$$
\begin{array}{rll}
\jJ: C^k(\A^3,\R)&\longrightarrow& C^k(\T^2,\R)\\
f&\longmapsto&g(\ov\th)=\dsp\int_\T f\big((\th_1,\ov \th),r^0\big)\,d\th_1.
\end{array}
$$
As above, this is a linear, surjective and open mapping. So, if $T$ is the quadratic part of the $d$--averaged system at $r^0$,
the set of functions $f$ such that $\jJ(f)$ is in the open dense set $\jU(\T)$ of Theorem~II is residual in 
$C^\ka(\A^3,\R)$, and so also in $\jS^\ka$. Since the set of double resonant points in the broken line which have to be
taken into account is finite, Condition ${\bf (S_3)}$
for $\Ga_k$, $1\leq k\leq \ell$ is open and dense in $\jS^\ka$.

\paraga We denote by $\jO$ the open dense subset of $\jS^\ka$ for which Conditions $\bf(S)$ are satisfied 
for $\Ga_k$, $1\leq k\leq \ell$. Given $f\in\jO$, then there exists a largest $\eps_0(f)$ which 
satisfies all the threshold conditions involved in all the constructions of the
proof of existence of admissible chains. For $0<\eps<\eps_0(f)$, the system $h+\eps f$ admits an admissible
chain located in the $O(\sqrt\eps)$ neighborhood of $\T^3\times\bGa$. Given $\de>0$ small enough (independent
of $f$), there exists a largest $\beps_0(f)<\eps_0(f)$ such that for $0<\eps<\beps_0(f)$ each $O_i$ contains the 
$\de$-neighborhood of some essential torus in the chain. 

\paraga Given $f$ in the previous open dense subset of $\jS^\ka$, we define $\eps_0(f)$ as the largest 
positive real number which satisfies all the threshold conditions involved in all the constructions of the
proof of existence of admissible chains. Since all the threshold can be chosen lower semicontinuous
by construction, $\eps_0$ is itself lower-semicontinuous.

\newpage


\setcounter{section}{0}
\begin{center}
{\LARGE B. Hyperbolic properties of classical systems on $\A^2$}
\end{center}

The proof of Theorem II will necessitate the introduction of a set of explicit nondegeneracy conditions (Conditions $(D)$), 
under which our result  holds true. We will gather these conditions in Section 1, together with some basic 
definitions and results. The rest of
the paper is organized as follows.
\begin{itemize}

\item  In Section 2 we first recall basic facts on the Jacobi-Maupertuis correspondence 
between classical systems and geodesic flows,  and then very classical results on the dynamics of the geodesic
flows on the torus $\T^2$ (mainly those of Morse and Hedlund), as well as the hyperbolicity theorem of Poincar\'e on
globally minimizing orbits on surfaces. 
Together with Conditions $(D)$, this enables us to prove  the existence 
of (possibly infinite)  chains of annuli, starting from the critical energy $e=\ov e$ and asymptotic to $e=+\infty$.

\item In Section 3 we prove that under Conditions $(D)$ the previous chains admit
a {\em single} annulus asymptotic to $+\infty$. 

\item In Section 4 we analyze the behavior of the system in the neighborhood of the critical energy $\ov e$, assuming 
the existence of suitable homoclinic orbits for the hyperbolic fixed point at the critical energy. 
A Hamiltonian version of the Birkhoff-Smale theorem proves that the previous chains are indeed {\em finite},
with a single annulus asymptotic to the critical energy. Moreover, the same theorem
proves the existence of (at least) one singular annulus, together with the existence of the family of heteroclinic orbits between
the various objects. At this point, Theorem II will be proved except for the existence of homoclinic orbits for the 
hyperbolic fixed point and the genericity of the constraints induced by Conditions $(D)$.

\item Section 5 is devoted to the proof of the genericity of Conditions $(D)$ and its consequences. 

\item In Appendix~\ref{sec:hom} we prove the existence of the homoclinic orbits  for the 
hyperbolic fixed point of (generic) classical systems.

\item In Appendix~\ref{sec:proofBS},  we give an extensive proof of the Hamiltonian Birkhoff-Smale theorem.

\item Finally we recall in Appendix~\ref{sec:horses} some results on horseshoes in the plane, in a simple and well adapted form due to Moser.
\end{itemize}

We tried to make all proofs basic and self-contained, so we certainly do not give the shortest possible 
(neither the most elegant) ones.
Some results are rather well-known, in particular the existence of homoclinic orbits, since the works of Bolotin 
\cite{Bo78,Bo83,BB03,BR02}  (see also \cite{Be00,Mat10} and the results of weak KAM theory).  However
we need here to make them more precise and, in particular, to analyze the 
convergence of periodic orbits to the (poly)homoclinic orbits. 
This is also the case for the Hamiltonian Birkhoff-Smale theorem, for
which we need a very accurate formulation. We could have more fully exploited the genericity results on
the boundary of the unit ball of the stable norms for classical systems (see \cite{C}) but we chosed to
limit ourselves to what is strictly necessary in view of giving a complete geometrical description of the
diffusion mechanism.


\section{Basic definitions and the nondegeneracy conditions}\label{sec:nondeg2}


\subsection{Basic definitions}
In this section we fix a classical system $C$ of the form (\ref{eq:classham}). 

\paraga We denote by $\rho:\A^2\to\A^2$ the involution defined by 
\beq\label{eq:natinv}
\rho(\th,r)=(\th,-r),
\eeq
 which reverses the
symplectic form $\Om$. Since $C$ is invariant under $\rho$, its solutions arise in opposite pairs $\ga^*$ and $\ga^{**}$,
which satisfy $\ga^*(t)=\rho\circ\ga^{**}(-t)$.

\paraga We set $\ov e=\Max U$ and for $e>\ov e$  we introduce the so-called Jacobi metric induced 
by $C$ at energy $e$, defined for $v\in T_\th\A^2$ by 
\begin{equation}\label{eq:riemmet}
\abs{v}_e=\big(2(e-U(\th))\big)^{\ppdemi}\norm{v},
\end{equation}
where $\norm{\ }$ stands for the norm on $\R^2$ associated with the dual of $T$.

\paraga Given a $\tau$--periodic solution of a vector field $X$ on a manifold, the associated {characteristic exponents} 
are the eigenvalues of the derivative of the  time-$\tau$ flow $\Phi^{\tau X}$ at any point $m$ of the orbit.  
In the case where $X$ is Hamiltonian and the periodic solution is nontrivial,  $D_m\Phi^{\tau X}$ always admits~$1$ as an
eigenvalue, with multiplicity at least $2$ (due to the invariance by time shifts and the preservation of energy). 
We will say that a nontrivial periodic solution of a Hamiltonian system  is {\em nondegenerate} when the multiplicity of 
the eigenvalue~$1$  is exactly $2$. This amounts to saying that the Poincar\'e map associated with any section
does not admit $1$ as an eigenvalue. A periodic solution of a Hamiltonian vector field $X$ is said to be hyperbolic when 
it is nondegenerate and when moreover the eigenvalues of $D_m\Phi^{\tau X}$ which are different from $1$ are not located 
on the unit circle (which is equivalent to the usual hyperbolicity of the Poincar\'e return map).


\paraga Let $\abs{\ }$  be a Riemannian metric on $\T^2$. 
The length of a piecewise $C^1$  arc $\ze:[a,b]\to \T^2$ is defined as the integral
$
\ell(\ze)=\int_a^b \abs{\ze'(t))}\,dt.
$
If an arc $\ze$ satisfies
$\ze'(t)\neq 0$ for $t\in[a,b]$, one defines a new arc $\xi$ (its arc-length parametrization), defined  by
$
\xi=\ze\circ\sig\inv
$,
where 
$
\sig(t)=\int_a^t\abs{\ze'(\tau)}\,d\tau.
$
Consequently,  the domain of $\xi$ is  $[0,\ell(\ze)]$, where $\ell(\ze)$ is
the Riemannian length of the curve~$\ze$, and $\abs{\xi'(s)}=1$ for $s\in[0,\ell(\ze)]$.

\paraga Let us recall the so-called Jacobi-Maupertuis principle in the setting of our classical system $C$.
 We denote by $\jL_C:T\T^2\to T^*\T^2$ the Legendre diffeomorphism associated with $C$.
We say that a trajectory has energy $e$ when  the corresponding orbit is contained in $C\inv(e)$.

\vskip2mm
\noindent{\bf Lemma (Jacobi-Maupertuis).} {\em Let $e>\ov e$ be fixed. Fix a $C^1$ curve $\ze$ on $\T^2$ such that 
$\jL_C(\ze,\ze')$ takes its values in $C\inv(e)$. Then  $\ze$  is a trajectory with energy $e$ of $C$ if and only if
the arc-length parametrization of $\ze$ relative to the Jacobi metric $\abs{\ }_e$ is a geodesic of this metric.
} 
\vskip2mm

So, up to reparametrization, the solutions of the vector field $X^C$ in $C\inv (e)$ and those of the geodesic vector field $X_e$ induced
by $\abs{\ }_e$ in the unit
tangent bundle are in one-to-one correspondence. In particular, the reparametrization of a periodic solution of $X^C$ is a periodic solution
of $X_e$ and the characteristic exponents of both solutions are related by the reparametrization. In particular, both are simultaneously 
nondegenerate or hyperbolic.

\paraga Let $\abs{\ }$  be a Riemannian metric on $\T^2$.
We say that a closed curve on $\T^2$ is length-minimizing in some class $c\in H^1(\T^2,\Z)$ when it belongs to this class
and when moreover its length is minimal among the lengths of all closed piecewise $C^1$  curves belonging to $c$. 
Turning back to our classical system $C$, we say that a periodic solution of $X^C$ with energy $e>\ov e$ is minimizing when 
its projection on $\T^2$ is length-minimizing  in its homology class for the Jacobi metric $\abs{\ }_e$.  These definitions can 
be related to the minimization of the Lagrangian action, see the next section and \cite{DC95,Mat10}.

\paraga We canonically identify $H_1(\T^2,\Z)$ with $\Z^2$,
so that a primary class $c$ is an indivisible pair of integers (that is, $c=(c_1,c_2)\in\Z^2$ with $c_1\wedge c_2=1$).  
We define the $c$--averaged potential associated with $U$ as the function
\begin{equation}\label{eq:avpot2}
U_c(\ph)=\int_0^1 U\big(\ph+s\,(c_1,c_2)\big)\,ds
\end{equation}
where $\ph$ belongs to the circle  $\T^2/T_c$, where  $T_c=\{\la(c_1,c_2)\ [\Z^2]\mid \la\in\R\}$ (note that $T_c$ is  also a circle).

\paraga One checks that if $\th^0$ is a nondegenerate maximum of $U$, then its lift $O=(\th^0,0)$ to
the zero section of $\A^2$ is a hyperbolic fixed point for $X^C$, with real eigenvalues. 
We say that $O$ admits a {\em proper conjugacy neighborhood} when exists there exists a symplectic coordinate 
system  $(u_1,u_2,s_1,s_2)$ in a  neighbohood of $O$,  relatively to which $O=(0,0,0,0)$ and
the Hamiltonian $C$ takes the normal form
\begin{equation}
\la_1 u_1s_1+\la_2 u_2s_2+R(u_1s_1,u_2s_2),
\end{equation}
with $\la_1>\la_2>0$, $R(0,0)=0$ and $D_{(0,0)}R=0$. We moreover require  the coordinates to satisfy the equivariance
condition
\beq\label{eq:equivariance}
\rho(u,s)=(-s,u)
\eeq
where $\rho$ was defined in (\ref{eq:natinv}).  We denote by $W_\ell^s=\{u=0\}$ and $W_\ell^u=\{s=0\}$ the local
stable and unstable manifolds of $O$. The strongly stable and unstable
manifolds of $O$ read, in the previous coordinates
$W^{ss}=\{u=0,\, s_2=0\}$,  $W^{uu}=\{s=0,\, u_2=0\}$.
We also introduce the subsets
$W^{sc}=\{u=0,\, s_1=0\}$, $W^{uc}=\{s=0,\, u_1=0\}$.
We will see  that these subsets may be intrinsically defined. We define the {\em exceptional set}
as the union 
\begin{equation}
\jE=W^{ss}\cup W^{sc} \cup W^{uu} \cup W^{uc}.
\end{equation}

\paraga We introduce the amended Lagrangian $\til L$ associated with $C$:
\begin{equation}
\til L(\th,v)=\pdemi \norm{v}^2-(U-\Max U),\qquad \forall (\th,v)\in T\T^2.
\end{equation}
Given a solution $\om$ of $C$ homoclinic to the fixed point $O$, we define its amended action as the integral
\begin{equation}
\int_{-\infty}^{+\infty} \til L(\ze(t),\ze'(t))\,dt
\end{equation}
where $\ze=\pi\circ\om$. This integral is immediately proved to be convergent, due to vanishing of $\til L$ at $O$
and the exponential convergence rate in the neighborhood of $O$.


\subsection{The nondegeneracy conditions}  
In this section we fix a positive definite quadratic form $T$ and, for $U\in C^\ka(\T^2)$, we denote the associated classical system
by $C_U(\th,r)=\pdemi T(r)+U(\th)$.  We now introduce our main nondegeneracy conditions, {\em which are to be understood as
conditions on $U$ only, even if the dynamical behavior of the system $C_U$ is involved}.

\paraga {\bf Conditions on the fixed point and its homoclinic orbits.}

\begin{itemize} {\em 
\item {$(D_1)$} The  potential function $U$ admits a single global maximum at some $\th^0\in\T^2$, which is nondegenerate.

\item $(D_2)$ The fixed point $O=(\th^0,0)$ of $C_U$ admits a proper conjugacy neighborood.

\item $(D_3)$ No orbit homoclinic to $O$ intersects the exceptional set $\jE$.

\item $(D_{4})$ There exists a pair of opposite solutions homoclinic to $O$ whose amended action is smaller than the 
amended action of the other homoclinic solutions.
}\end{itemize}

\paraga {\bf Conditions on the periodic solutions.} Here e fix  $c\in \H^1(\T^2,\Z)$.
\begin{itemize} {\em
\item $(D_5(c))$ Each periodic solution of $X^{C_U}$ with energy $>\ov e$ is nondegenerate.

\item $(D_6(c))$ There exists a  subset $B(c)$ of  $]\bar e,+\infty[$  such that  
\begin{itemize}
\item for  $e\in \, ]\bar e,+\infty[\,\setm B(c)$, there exists 
exactly one length-minimizing  periodic trajectory in the class $c$ for the Jacobi  metric $\abs{\ }_e$,
\item for $e_0\in B(c)$, 
there exist exactly two length-minimizing periodic trajectories in $c$ for $\abs{\ }_{e_0}$.
\end{itemize}

\item $(D_7(c))$ Given $e_0\in B(c)$, the lengths $\ell^*(e)$ and $\ell^{**}(e)$ 
of the continuations of the two length-minimizing periodic orbits at $e_0$ satisfy
$
(\ell^*-\ell^{**})'(e_0)\neq0.
$
 
\item $(D_8(c))$ The stable and unstable manifolds  of two minimizing periodic solutions of $X^{C_U}$ with the same
energy transversely intersect inside their energy level.
}\end{itemize}

\paraga {\bf Conditions on the averaged potential.}
\begin{itemize} {\em 
\item $(D_9(c))$ The averaged potential  $U_c$  admits a single maximum, which is nondegenerate.
}\end{itemize}

In the rest of the paper we say for short that $U$ satisfies Conditions $(D_5)-(D_9)$ when $U$ satisfies $(D_5(c))-(D_9(c))$ 
for each  $c\in \H_1(\T^2,\Z)$.


\setcounter{paraga}{0}

\section{Conditions $(D_5)-(D_8)$ and the chains of annuli}\label{sec:annuli''}
In this section we fix a classical system $C$ of class $C^2$ of the form (\ref{eq:classham}), and we prove
the existence of {\em possibly infinite} chains of annuli realizing each primary homology class as soon as  
Conditions $(D_5)-(D_8)$ are satisfied.  Only minimizing periodic orbits  
will be involved in our construction.

\paraga  Given a discrete subset $B$ of $]\ov e,+\infty[$, the subset $]\ov e,+\infty[\setm B$ admits  a finite or countable
familly of connected components  $(I_k)_{k\in Z}$, where  $Z$ is an interval of $\Z$,   which can be ordered so that  
$\Sup {I_k}=\Inf I_{k+1}$ when $k,k+1\in Z$. Since $Z$ is obviously unique up to translation,
we call $(I_k)_{k\in Z}$ {\em the} family of connected components of $]\ov e,+\infty[\setm B$.

\begin{prop}\label{mainprop} 
Let $c\in \H_1(\T^2,\Z)$ be fixed and assume that  $U$ satisfies Conditions $(D_5(c))-(D_8(c))$. 
Then the set $B(c)$ is discrete and the system $C$ possesses a chain $\jA(c)$ of annuli realizing $c$, defined over 
the ordered family $(\ov I_{k})_{k\in Z}$, where $(I_k)_{k\in Z}$ is the ordered family of connected components 
of $]\ov e,+\infty[\setm B(c)$ and where $B(c)$ was introduced in Condition $(D_6(c))$. Moreover, each annulus
satisfies the twist property and each periodic orbit in it admits homoclinic orbits.
\end{prop}

The closure $\ov I_k$ is to be understood relatively to the space $]\ov e,+\infty[$.

\paraga The proof  relies on some classical results that we now recall. 
Let $\abs{\ }$ be a Riemannian metric on $\T^2$ and let us keep the same notation for its lift to $\R^2$.
We say that a piecewise $C^1$  curve $\ze:[a,b]\to \R^2$ is {\em length-minimizing} between $a$ and $b$ if its length
is minimal among those of all piecewise $C^1$ curves defined on $[a,b]$ and taking the same values as $\ze$ at
$a$ and~$b$.   
We then say that a piecewise $C^1$  curve $\ze:\R\to \R^2$ is {\em fully length-minimizing} when, for any compact 
interval $[a,b]$, the restriction $\ze_{\vert[a,b]}$ is length-minimizing in the previous sense. Of course length-minimizing 
curves  are geodesics  of the metric.

\vskip1mm
 
We say that $\ze:\R\to\R^2$ is $\Z^2$--periodic with period $T>0$ when there exists $m\in\Z^2$ such that
 $$
\ze(t+T)= \tau_m\circ \ze(t),\quad \forall t\in\R,
$$
where $\tau_m$ is the translation of the vector $m$ in $\R^2$. In this case we say that $m$ is a rotation vector for 
$\ze$. The period and rotation vector are not unique, but one obviously has ``minimal'' ones. 

\vskip2mm
 
A classical theorem of Morse proves the existence of fully minimizing $\Z^2$--periodic geodesics for any (minimal) rotation
vector $m\neq0$. It is not difficult to prove that such a geodesic  is a graph over the line $\R.m$ (relatively to a suitable coordinate
system). Moreover, one proves that  two fully minimizing $\Z^2$--periodic geodesics with the same rotation vector are either disjoint 
or equal.  Therefore, two such  disjoint  geodesics form the boundary of a well-defined {\em strip} in $\R^2$. 
The following result will be crucial in our construction.
 
\vskip3mm
\noindent{\bf Theorem (Hedlund).}
{\em
 Assume that the strip defined by two fully minimizing $\Z^2$--periodic geodesics $\ze_1$ and $\ze_2$ with the
 same minimal rotation vector does not contain any 
 other $\Z^2$--periodic fully minimizing geodesic. Then there exist a fully minimizing geodesic which is
 $\al$--asymptotic to $\ze_1$ and $\om$--asymptotic to $\ze_2$,  and a fully minimizing geodesic  which 
 is $\al$--asymptotic to $\ze_2$ and $\om$--asymptotic to $\ze_1$.
}

\vskip3mm
See for instance \cite{Ba} for a complete proof. 
The following result for geodesic systems on $\T^2$ is also classical and easily deduced from \cite{Ba}.

\vskip3mm

\noindent{\bf Lemma.} \label{prop:equivmin} {\em
 Let $\ze:\R\to\R^2$ be a piecewise $C^1$ $\Z^2$--periodic curve with period $T$ and let $\pi:\R^2\to\T^2$ be the canonical projection. 
 Then the following three properties are equivalent.
 \begin{itemize}
 \item The restriction $\ze_{\vert[0,T]}$ is length-minimizing between $0$ and $T$.
 \item The curve $\ze$ is fully minimizing.
 \item The projection $\pi\circ\ze$ is length-minimizing in its homology class (see the previous section).
 \end{itemize}
}

Let us finally state in our setting the hyperbolicity theorem of Poincar\'e, whose proof relies on the previous statement 
(see \cite{P} (vol. 3) for the original one). Recall that we say that a periodic solution of $X^C$ with energy $e>\ov e$ is 
minimizing when its projection
on $\T^2$ is length-minimizing in its homology class.

\vskip3mm
\noindent{\bf Theorem (Poincar\'e).}
{\em Let $C$ be a classical system of the form (\ref{eq:classham}).
A minimizing periodic solution with energy $e>\ov e$ which is nondegenerate is hyperbolic.}

\vskip3mm

\paraga {\bf Proof of Proposition \ref{mainprop}.}
We first prove the existence of the chain, and we then study the twist property of the annuli.
The fact that $B(c)$  is discrete is an immediate consequence of $(D_6(c))$ and $(D_7(c))$. 
Let $I=\,]e_0,e_1[$ be a component of  $]\ov e,+\infty[\,\setm B(c)$, so  $e_0,e_1\in B(c)$. 

\vskip1mm

Assume first that $e_0 >\ov e$.
For each energy $e\in I$, the Jacobi metric   $\abs{\ }_e$ admits a single length-minimizing periodic geodesic 
$\xi_e$ in the class $c$.  The Jacobi-Maupertuis lemma associates with this geodesic a unique reparametrized periodic 
solution  $\ga_e$ of $X^C$, with energy $e$, which is minimizing by definition. Since it is nondegenerate by 
Condition $(D_5(c))$, it is hyperbolic  by the Poincar\'e theorem. 
This and the uniqueness of $\xi_e$ prove that $(\ga_e)_{e\in I}$  is a differentiable family.

The Poincar\'e theorem still applies to each of the two minimizing solutions which exist at  the limit point $e_1$.
The same is true for $e_0$ since $e_0\neq \ov e$. 
By uniqueness and hyperbolicity, this proves that the previous family $(\ga_e)_{e\in I}$ 
can be differentiably continued over a slightly larger interval $\ha I\supset \ov I=[e_0,e_1]$.
%
This proves the union $\sA$ of the orbits of the solutions $\ga_e$, $e\in \ov I$, is an annulus defined over $I$
and realizing $c$.

\vskip1mm
In the case where $e_0=\ov e$, the proof is even simpler since $e_0$ no longer belongs to the interval over which the
annulus is defined. So the above reasoning proves that $\sA$ is an annulus over $]\ov e,e_1]$.

\vskip1mm

This way one gets a family $(\sA_k)_{k\in Z}$ of annuli defined over the ordered intervals $(\ov I_k)_{k\in Z}$.
Consider  two  intervals $I_k$ and $I_{k+1}$ of $]\ov e,+\infty[\setm B(c)$, 
with $e_0=\Max I_k=\Min  I_{k+1}$, so  $e_0\in B(c)$.
Let  $\Ga_{e_0}^k$ and $\Ga_{e_0}^{k+1}$ be the minimizing periodic orbits at $e_0$. 
By the Hedlund theorem applied to the associated geodesic solutions as above, 
there exist  in the level $C\inv (e_0)$ a heteroclinic orbit connecting 
$\Ga^k_{e_0}$ and  $\Ga^{k+1}_{e_0}$ and another one connecting $\Ga^{k+1}_{e_0}$ and $\Ga^k_{e_0}$.  
By Condition $(D_8(c))$ the invariant manifolds of  $\Ga^k_{e_0}$ and  $\Ga^{k+1}_{e_0}$ transversely intersect 
along these solutions.  This proves that the family $\jA(c):=(\sA_k)_{k\in Z}$ is a chain of annuli. 

\vskip1mm

Hedlund's theorem applied at energies in the interior of $I_k$ also proves that the corresponding periodic 
solutions admit at least two homoclinic solutions.

\vskip1mm 

It only remains to prove that each $\sA_k$ satisfies the twist property, and for this we will use some basic facts
from Mather's variational theory.  By Mather's graph property, one easily sees that when $e\in I_k^\circ$,  
the unique minimizing periodic orbit $\Ga(e)$ realizing $c$ coincides with the Mather set $\jM_\om$,
where $\om\in H^1(\T^2)$ belongs to the subderivative of the Mather $\beta$ function at $\rho=c/T(e)$, where $T(e)$
is the period of $\Ga(e)$. 
Moreover, $e=\al(\om)$, where $\al$ stands for the Mather $\al$ function
(see \cite{DC95,Fa09,Mat10}). 

For each $T\in \{T(e)\mid e\in I_k^\circ\}$, fix $\om(T)$ in the subderivative of 
$\beta$  at $\rho=c/T$. Then
$$
\al(\om(T(e))=e,
$$
which shows that $e\mapsto T(e)$ is injective. Since we already know that $T$ is continuous,
this proves that $T$ is monotone, which concludes the proof.
\hfill $\square$



\vskip3mm

Note that the Hedlund theorem in fact provides us with more homoclinic or heteroclinic connections
than what is required in our definitions.



\section{Condition $(D_9)$ and the high energy annuli}\label{sec:highann}
In this section we prove that,  under  the additional condition  
$(D_9(c))$, the set  $B(c)$ of Condition $(D_6(c))$ is bounded above.

\begin{prop}\label{prop:highen}
Let $C$ be a classical system of the form (1), with $U\in C^2(\T^2)$. Fix $c\in \H_1(\T^2,\Z)$. Assume 
that Condition $(D_9(c))$ is satisfied.
Then there is $e(c)\geq\ov e$ such that for $e > e(c)$, there exists a unique length-minimizing geodesic in the class $c$ 
for the Jacobi metric $\abs{\ }_e$.  Moreover, the corresponding solution of $X^C$ is hyperbolic in $C\inv(e)$.
\end{prop}

Note that we do not assume that Conditions $(D_4(c)-D_8(c))$ are satisfied in the previous proposition. Our proof
here is perturbative rather than based on minimizing arguments. Indeed,  we will crucially use the fact that when the 
energy $e$ is large enough, the system $C$ can be viewed as a perturbation
of the integrable system $\pdemi T(r)$.   We in fact essentially reprove the Poincar\'e theorem on the destruction of 
resonant tori and the birth of hyperbolic periodic orbits,  paying moreover attention to their minimizing properties.

\begin{proof} We canonically identify $H_1(\T^2,\Z)$ with $\Z^2$. Using a standard linear symplectic change of variable, one
can assume that $c=(1,0)$.

\vskip2mm
$\bullet$  Let $M=(m_{ij})$ be the matrix of $T$,  so that the frequency map associated with  $\pdemi T(r)$ is 
$\varpi(r)=M r$. 
We will  examine the properties of $C$ in the neighborhood of the  resonance $\jR$  of equation $\varpi_2(r)=0$, that is,  the line
$$
m_{12}r_1+m_{22}r_2=0
$$
in the action space. Note that $m_{22}\neq0$, so that $r_1$ is a natural parameter on $\jR$ and 
\beq\label{eq:energy}
C(\th,r)\sim_{\abs{r_1}\to\infty} \al r_1^2,\qquad \al>0,
\eeq
in a small enough neighborhood of $\jR$.
The averaged potential at a point $r^0\in\jR$ is the function
\beq\label{eq:avclasspot}
U_c(\th_2)=\int_\T U(\th_1,\th_2)\,d\th_1.
\eeq
Note that $U_c$ is independent of $r^0$.
By $(D_9(c))$, $U_c$ admits a nondegenerate maximum at some point $\th_2^*\in\T$. 

\vskip2mm
$\bullet$ We fix $r^0=(r_1^0,r_2^0)\in\jR$ and introduce the homological equation
$$
\varpi(r^0)\,\d_\th S(\th)=\big(U(\th)-U_c(\th_2)\big),
$$
whose solution is immediate: up to constants
$$
S(\th)=\frac{1}{\varpi_1(r^0)}\int_0^{\th_1}\big(U(s,\th_2)-U_c(\th_2)\big)\,ds.
$$
Therefore
$$
\norm{S}_{C^{\ka}(\T^2)}\leq \frac{2\norm{U}_{C^\ka(\T^2)}}{\abs{\varpi_1(r^0)}}\leq \frac{\mu_0}{\abs{r^0_1}},
$$
for a suitable constant $\mu_0>0$.
\vskip2mm
$\bullet$ We perform the symplectic change
$$
\Phi(\th,r)=\big(\th,\ r^0+r-\d_{\th} S(\th)\big), 
$$
so that:
$$
C\circ\Phi(\th,r)=e^0+\varpi_1(r^0)\, r_1+\pdemi T(r) +U_c(\th_2)+ R(\th,r),
$$
with $e^0=\pdemi T(r^0)$ and
$$
R(\th,r)=\pdemi T\big(\d_\th S(\th)\big) -\varpi(r)\d_\th S(\th).
$$
Given $\rho>0$ large enough, to be chosen below, one therefore gets
\beq\label{eq:remainder}
\norm{R}_{C^1(\T^2\times \ov B(0,\rho))}\leq \frac{\mu_1}{\abs{r_1^0}},
\eeq
for a suitable $\mu_1>0$ (which depends  on $\rho$ but not on $r^0$).
\vskip2mm
$\bullet$ The Hamiltonian vector field generated by $C\circ\Phi$ reads
\begin{equation}\label{eq:vectfield1'}
\left\vert
\begin{array}{llll}
\dot \th_1=\varpi_1(r^0) \ \ + & [m_{11}r_1+m_{12}r_2] &+ &\d_{r_1} R(\th,r) \\
\dot \th_2= &[m_{12}r_1+m_{22}r_2] &+&\d_{r_2} R(\th,r) \\
\dot r_1= & &-& \d_{\th_1} R(\th,r) \\
\dot r_2= -U'_c(\th_2)& &-&\d_{\th_2} R(\th,r).\\
\end{array}
\right.
\end{equation}
Let $\ell$ be the integer part of $\varpi_1(r^0)$. To avoid the classical degeneracy problem when $\abs{r_1}\to\infty$,
we will consider this system as defined on the covering $\ha \A^2$ of $\A^2$ induced by the covering 
$\R/(\ell\Z)\to\R/\Z$ for the first factor, so that we consider the various
functions as being $\ell$-periodic with respect to $\th_1$. We denote by $\ha \T^2$ the coresponding
covering of $\T^2$.

\vskip2mm 
$\bullet$  Let $\Phi_0$ denote the flow of the unperturbed vector field obtained by setting $R\equiv0$.
Setting $r_2^*=\frac{m_{12}}{m_{22}} r_1+r_2$, the unperturbed vector field reads
\begin{equation}\label{eq:vectfield2'}
\left\vert
\begin{array}{llll}
\dot \th_1=\varpi_1(r^0) + [m^*_{11}r_1+m_{12}r^*_2]\\
\dot r_1= 0\\
\dot \th_2=m_{22}r_2^*\\
\dot r_2^*= -U'_c(\th_2),\\
\end{array}
\right.
\end{equation}
with $m_{11}^*=m_{11}-\frac{m^2_{12}}{m_{22}}$. The last two equations  are induced by the ``pendulum-like'' Hamiltonian
$$
H(\th_2,r_2^*)=\pdemi m_{22} (r_2^*)^2+U_c(\th_2)
$$ 
for the usual symplectic structure,
and therefore immediately integrated. The complete integration easily follows. 
In particular, the hyperbolic fixed point $(\th_2^*,0)$
for $X^H$ gives rise to an family of periodic hyperbolic solution of (\ref{eq:vectfield2'}),
parametrized by the energy.

\vskip2mm 
$\bullet$
On the compact set $[0,2]\times \T^2\times \ov B(0,\rho)$, by (\ref{eq:remainder}), the flow $\Phi$ of the 
system (\ref{eq:vectfield1'})  clearly satisfies
\beq\label{eq:approxflow}
\norm{\Phi-\Phi_0}_{C^1}\leq \frac{\mu_2}{\abs{r_1^0}}
\eeq
for a suitable $\mu_2>0$. 
We assume that $\rho$ is small enough so that the derivative $\dot\th_1$ in~(\ref{eq:vectfield1'}) 
satisfies $\mabs{\dot \th_1} \geq \pdemi \abs{\varpi_1(r^0)}>0$ on the set 
$\ha \T^2\times \ov B(0,\rho)$.
Therefore the $2$-dimensional submanifold 
$$
\Sig=\{\th_1=0\}\cap (C\circ\Phi)\inv(\{e^0\})
$$
is a symplectic section for the flow of (\ref{eq:vectfield1'}), on which the coordinates $(\th_2,r^*_2)$ define a  chart.  
By (\ref{eq:approxflow}), it is easy to check that the associated return map to $\Sig$ takes the form
$$
P(\th_2,r^*_2)=P_0(\th_2,r^*_2)+\ov P(\th_2,r^*_2)
$$
where $P_0$ is  
the time--$(\ell/\abs{\varpi_1(r^0)})$  flow of the system $H$, and where the remainder term $\ov P$ tends to $0$ 
in the $C^1$ topology when $\abs{r_1^0}\to\infty$. Now 
$$
1-\frac{1}{\abs{\varpi_1(r^0)}}\leq \frac{\ell}{\abs{\varpi_1(r^0)}} \leq 1
$$
so for $r^0_1$ large enough $P$ has a hyperbolic fixed point arbitrarily close to $(\th_2^*,0)$. Moreover, this point
admits stable and unstable manifolds that are graphs over a subset of the form $\abs{\th_2-\th_2^*}<1$ 
in the covering $\R^2$ of $\A$.

\vskip2mm

$\bullet$ Now, as a consequence of the $\Z^2$--periodicity of (\ref{eq:vectfield1'}) and by local uniqueness of the hyperbolic
solutions, the  periodic orbit $\ha\Ga$ of (\ref{eq:vectfield1'})  on $\ha \A^2$ is an $\ell$--covering of a unique periodic orbit 
$\Ga$ for the system on $\A^2$, and this is also the case for the previous stable and unstable manifolds. From their graph property,
we deduce that suitable parts of them are graphs of weak KAM solutions of the system $C$ and, as a consequence, that $\Ga$
is the corresponding Mather set. So the evenly distributed measure on $\Ga$ minimizes the Lagrangian action and 
by \cite{DC95}, its time reparametrization minimizes the action of the Jacobi metric at energy $e^0$. This proves that
the projection of $\Ga$ on $\T^2$ minimizes the $\abs{\ }_{e^0}$ length. The definition of a Mather set also proves the
uniqueness of the minimizing solution. 

\vskip2mm

$\bullet$ The previous results hold true as soon as $\abs{r^0}$ is large enough. 
Our claim then easily follows from (\ref{eq:energy}). 
\end{proof}

\begin{cor}\label{cor:highen}
With the same assumptions as in Proposition~\ref{prop:highen}, suppose now that  
Conditions $(D_4(c))-(D_8(c))$ are satisfied. Then $B(c)$ is bounded above.
\end{cor}

\begin{proof}
This is an immediate consequence of Proposition~\ref{prop:highen}, since $e(c)$ is an upper bound for $B(c)$.
\end{proof} 

Note that in general $e(c)>\Sup B(c)$. 


\section{The low-energy annuli}\label{sec:lowann}
In this section we admit two technical results on the existence of polyhomoclinic
orbits and horseshoes, to be proved in Appendix~\ref{sec:hom} and Appendix~\ref{sec:proofBS}. 
We fix a potential $U$ satisfying Conditions $(D)$.

\vskip1mm

Without loss of generality, we assume $\ov e=\Max U=0$. 

%


\subsection{The horseshoes in the neighborhood of the critical energy}
In this section we state two results to be proved in Appendix~\ref{sec:proofBS}, which constitute our 
version of the Hamiltonian Birkhoff-Smale theorem.
By Condition $(D_2)$, there exists a symplectic coordinate 
system  $(u_1,u_2,s_1,s_2)$ in a  neighbohood of $O$,  in which
the Hamiltonian $C$ takes the normal form
\begin{equation}\label{eq:normform4}
\la_1 u_1s_1+\la_2 u_2s_2+R(u_1s_1,u_2s_2),
\end{equation}
where the remainder $R$ is flat at order $1$ at $0$. The coordinates moreover satisfy the equivariance
condition (\ref{eq:equiv}).  We still denote by $C$ the classical system 
in the normalizing coordinates. 


\paraga We first have to introduce particular sections in the neighborhood of $O$.  We refer to Appendix~\ref{sec:proofBS} for more
information on their definition and construction.
We denote by $B(\eps)$ the ball of $\R^4$
centered at $0$ with radius $\eps$ for the Max norm.  For $\eps>0$ small enough and for $\sig\in\{-1,+1\}$, we introduce 
the (pieces of) hyperplanes
\beq\label{eq:sections1}
\Sig_{\sig}^u[\eps]=\{(u,s)\in \ov B(\eps)\mid u_2=\sig\eps\},\qquad
\Sig_{\sig}^s[\eps]=\{(u,s)\in \ov B(\eps)\mid s_2=\sig\eps\},
\eeq
which obviously are transverse sections  for the  flow, due to (\ref{eq:normform4}). 
We say that $\sig$ is the {\em sign} of the section 
$\Sig_{\sig}^{u,s}[\eps]$. Moreover, the subsets
\beq\label{eq:sections2}
\Sig_{\sig}^u[\eps,e]=\Sig_{\sig}^u[\eps]\cap C\inv(e),\qquad
\Sig_{\sig}^s[\eps,e]=\Sig_{\sig}^s[\eps]\cap C_{\vert B}\inv(e)
\eeq
are (for $\abs{e}$ small enough) two--dimensional submanifolds of $C\inv(e)$ which are transverse to the flow inside 
$C\inv(e)$.


\paraga A {\em polyhomoclinic orbit} for $O$ is a finite ordered  family $\Om=(\Om_1,\ldots,\Om_\ell)$ of orbits
$\Om_i$ of $X^C$ which are homoclinic to $O$.  Polyhomoclinic solutions are defined in the same way. The order
of a polyhomoclinic orbit is to be understood as a cyclic order, so we consider any shifted sequence
$(\Om_k,\ldots,\Om_1,\ldots,\Om_{k-1})$ as defining the same polyhomoclinic orbit as$\Om$, the context being always
clear in the following. Note finally that the energy of each homoclinic orbit is $0$.

 Given a polyhomoclinic orbit $\Om=(\Om_1,\ldots,\Om_\ell)$, we will see in Section \ref{sec:proofBS} that the
(first and last)  intersections $a_i$ and $b_i$ of each $\Om_i$ with the sections $\Sig^u[\eps]$ and $\Sig^s[\eps]$ 
respectively are well defined when $\eps$
is small enough.  Consequently, with each $\Om_i$ are associated the entrance sign $\sig_{ent}(\Om_i)$ 
(that is, the sign $\sig$ such that
$a_i\in\Sig^u_\sig[\eps]$) and the exit sign $\sig_{ex}(\Om_i)$ (that is, the sign $\sig$ such that $b_i\in\Sig^s_\sig[\eps]$).
Such an $\eps$ will be said to be {\em adapted} to $\Om$.
 Let us now introduce a central definition.

\begin{Def}\label{def:compa}
We say that a polyhomoclinic orbit $\Om=(\Om_1,\ldots,\Om_\ell)$ is {\em compatible}
when, for $1\leq i\leq\ell$:
$$
\sig_{ent}(\Om_i)=\sig_{ex}(\Om_{i+1}),
$$ 
with the usual convention $\ell+1=1$.
\end{Def}

Given such a compatible $\Om$, the ordered associated sequence of signs $(\sig_{ex}(\Om_i))_{1\leq i\leq \ell}$ therefore completely 
characterizes the entrance and exit data of the sequence $(\Om_i)$.


\paraga
Given a  polyhomoclinic orbit $\Om=(\Om_1,\ldots,\Om_{\ell})$, the Birkhoff-Smale theorem states the existence
of a family of horseshoes for the Poincar\'e map induced by the flow. 
The basic rectangles of the horsehoes are contained in the sections $\Sig_{\sig}^u[\eps,e]$ (and so are two-dimensional)
and the family is therefore parametrized by the energy. The rectangles are located around the exit points of the
homoclinic orbits, so that each homoclinic orbit gives rise to such a rectangle.
In turns out that the behavior of the horseshoes crucially depends on the sign of the energy; we will therefore
state two different results, for which we refer to the Appendix for the
basic notions.

\begin{thm}\label{thm:hypdyn1} 
Fix a compatible polyhomoclinic orbit $\Om=(\Om_1,\ldots,\Om_{\ell})$, a {\em small enough} adapted $\eps>0$,
and denote by $a_i$ the exit point of   $\Om_i$ relatively to the section $\Sig^u[\eps]$. Then, 
the following properties hold true.

\begin{enumerate}

\item There exists $e_0>0$ and, for $1\leq i\leq \ell$, there exist  neighborhoods 
$\cR_{i}$  in $\Sig^u[\eps]$ of the exit points 
$a_{i}$,  such that for $\abs{e}\leq e_0$, the intersections 
$$
R_{i}(e)=\cR_{i}\cap C\inv(e)
$$
are rectangles in the section $\Sig^u[\eps,e]$ (relatively to suitable coordinates). 

\item For $e\in\,]0,e_0[$, the family $\big(R_{i}(e)\big)_{1\leq i\leq \ell}$
is a horseshoe for the  Poincar\'e  map $\Phi$ associated with the section $\Sig^u[\eps,e]$  in $C\inv(e)$, whose 
transition matrix  $A=\big(\al(i,j)\big)$ satisfies
\beq\label{eq:transmat}
\al(i,j)=1\quad\textit{when}\quad \sig_{ent}(\Om_i)=\sig_{ex}(\Om_{j}).
\eeq
In particular, since $\Om$ is compatible:
\beq\label{eq:transmat1}
\al(i,i+1)=1\quad\textit{for}\quad1\leq i\leq {\ell}.
\eeq

\item As a consequence, for $e\in\,]0,e_0[$, the periodic coding sequence
$$
\cdots(1,2,\ldots,\ell)\cdots
$$
is admissible and yields a hyperbolic periodic point $m(e)$ in $R^0(e)$ for the Poincar\'e return map. 
Let $\Phi_{out}$ be the Poincar\'e map 
between $\Sig^u$ and $\Sig^s$ along the homoclinic orbit $\Om^0_1$. Then, when $e\to 0$:
\begin{itemize}
\item  the unstable manifold of $m(e)$ converges to $\Sig^u[\eps,0]\cap W^u(O)$,
\item  the stable manifold of $m(e)$ converges to $\Phi_{out}\inv\big(\Sig^s[\eps,0]\cap W^s(O)\big)$,
\end{itemize}
in the $C^1$ compact-open topology.
\end{enumerate}
\end{thm}

Again, we used the convention $\ell+1=1$ in the previous statement on the transition matrix.
Our second result  focuses on the case where the polyhomoclinic orbit is the concatenation of two compatible
homoclinic orbits with different exit signs, so that it is no longer compatible.

\begin{thm}\label{thm:hypdyn2} 
Fix two compatible homoclinic orbits $\Om_0$ and $\Om_1$, fix a {\em small enough} adapted $\eps$ and assume that
\beq\label{eq:condsign}
\sig_{ex}(\Om_0)\neq \sig_{ex}(\Om_1).
\eeq
Then, denoting by $a_\nu$ the exit point of the homoclinic orbit $\Om_\nu$
there exists $e_0<0$ and two neighborhoods 
$\cR_0,\cR_1$  in $\Sig^u[\eps]$ of 
$a_0$ and  $a_1$ respectively,  such that for $e_0\leq e<0$, the intersections 
$$
R_0(e)=\cR_0\cap C\inv(e)
\quad\textit{and}\quad
R_1(e)=\cR_1\cap C\inv(e)
$$
are rectangles in the section $\Sig^u[\eps,e]$ (relatively to suitable coordinates). Moreover, the pair $\big(R_0(e),R_1(e)\big)$ 
is a horseshoe for the  Poincar\'e  map $\Phi$ associated with the section $\Sig^u[\eps,e]$  in $C\inv(e)$, with
transition matrix  
$$
A=\left[
\begin{array}{lll}
0&1\\
1&0\\
\end{array}
\right].
$$
The periodic coding sequence
$$
\cdots(0,1)\cdots
$$
defines a unique hyperbolic periodic point $m(e)$ in the rectangle $R_0(e)$ for $e_0<e<0$. Let $\Phi_{out}$ be the Poincar\'e map 
induced by the flow between $\Sig^u$ and $\Sig^s$ along the homoclinic orbit $\Om_0$. Then, when $e\to 0$:
\begin{itemize}
\item  the unstable manifold of $m(e)$ converges to $\Sig^u[\eps,0]\cap W^u(O)$
\item  the stable manifold of $m(e)$ converges to $\Phi_{out}\inv\big(\Sig^s[\eps,0]\cap W^s(O)\big)$
\end{itemize}
in the $C^1$ topology.
\end{thm}

The proofs of these theorems are postponed to Section \ref{sec:proofBS}, to which we refer for the coordinates
on the sections $\Sig^u$ and $\Sig^s$. The constraints on the size of $\eps$ will be made explicit in Section \ref{sec:proofBS}.


\subsection{Convergence of periodic orbit to polyhomoclinic orbits}\label{ssec:convergence}
We now introduce a specific definition for the convergence of periodic orbits. 

\begin{Def} \label{def:conv} Let $\Om=(\Om_1,\ldots,\Om_p)$ be a polyhomoclinic orbit and we fix  $\eps>0$ as above.
We say that a sequence 
$(\Ga_n)_{n\in{\N^*}}$ of periodic orbits of $X^C$ {\em converges to $\Om$} when 
\begin{itemize}
\item for $n\geq n_0$, 
$
\Ga_n\cap\Sig^u[\eps]=\{a_1^{n},\ldots,a_p^{n}\}
\  \ \textit{and}\  \ 
\Ga_n\cap\Sig^s[\eps]=\{b_1^{n},\ldots,b_p^{n}\},
$
with the following cyclic order
$$
a_1^{n}<b_1^{n}<a_2^{n}<b_2^{n}<\cdots<a_p^{n}<b_p^{n},
$$
according to the orientation on $\Ga_n$ induced by the flow;
\item  $\lim_{n\to\infty} a_i^{n}=a_i$ and  $\lim_{n\to\infty} b_i^{n}=b_i$, where 
$a_i$ and $b_i$ are the exit and entrance points of $\Om_i$.
\end{itemize}
\end{Def}

One easily sees that this definition makes sense since the convergence property is clearly independent of the choice 
of $\eps$ (small enough).

\begin{Def}\label{def:positive} We say that a polyhomoclinic orbit $\Om=(\Om_1,\ldots,\Om_\ell)$ is {\em positive} when there 
exists a sequence  of {\em minimizing} periodic orbits with {\em positive energy} of the system $C$ which converges 
to  $\Om$. 
\end{Def}

One of the main interest of the notion comes from the following result, which will be proved in Section~\ref{sec:proofBS}.

\begin{lemma}\label{lem:poscomp}
A positive polyhomoclinic orbit is compatible.
\end{lemma}

We will also need the following result for opposite polyhomoclinic 
orbits. Recall that given a solution $\ga$ of $C$, the function $\ha \ga : t\mapsto \rho\circ\ga(-t)$, where $\rho(\th,r)=(\th,-r)$,
is another solution of $C$, which we call opposite to $\ga$. We adopt the same terminology an notation for the orbits.

\begin{lemma}\label{lem:opp}
Let $\Om=(\Om_1,\ldots,\Om_\ell)$ be a positive polyhomoclinic orbit of the system $C$. Then 
$\ha\Om=(\ha\Om_1,\ldots,\ha\Om_\ell)$ is also a positive polyhomoclinic orbit of $C$. Moreover,
for $1\leq i\leq\ell$
$$
\sig_{ex}(\Om_i)=-\sig_{ex}(\ha\Om_{i}).
$$ 
\end{lemma}

\begin{proof} 
This is an immediate consequence of the definition of the sections, the invariance of $C$ and the equivariance 
property (\ref{eq:equivariance}).
\end{proof}

The whole construction of the initial annuli in the next section will be based on the previous two theorems and the
following results, which will be proved in Section \ref{sec:hom}.

\begin{prop}\label{prop:polyhom} Let $c\in \H_1(\T^2,\Z)$. Then  
there exists a {\em positive} polyhomoclinic solution $\om=(\om_1,\ldots,\om_\ell)$ 
such that the concatenation 
$$
(\pi\circ\om_\ell)*\cdots*(\pi\circ\om_1)
$$ 
realizes the class $c$, where $\pi:\A^2\to \T^2$
is the canonical projection.  For each primitive class $c$ we choose once and for all such a polyhomoclinic solution, 
which we denote by $\om(c)$, and we write $\Om(c)$ for the corresponding polyhomoclinic orbit.
\end{prop}

We conclude this part with a last lemma, whose proof is postponed to Section \ref{sec:hom} and which 
will be crucial for proving the existence of singular cylinders.

\begin{lemma}\label{lem:simppos} Assume that Condition $(D_4)$ is satisfied. Then
there exists a simple homoclinic orbit to $O$ which is positive.
\end{lemma}

Of course, by simple homoclinic orbit we mean here a polyhomoclinic orbit containing a single element.

%


\subsection{Existence of annuli asymptotic to polyhomoclinic orbits}
In this part, we prove the existence of ``Birkhoff-Smale'' annuli, which are asymptotic to the polyhomoclinic orbits $\Om(c)$.

\begin{lemma}\label{lem:anasympt} Fix a classical system of the form (\ref{eq:classham}) and assume that Conditions $(D)$
are satisfied.
Fix  a primitive homology class $c\in \H_1(\T^2,\Z)$. Then there exists an annulus $\sA_{BS}(c)$
defined over  an interval of the form $]0,e_1(c)]$ (recall that $\ov e=0$),  which is ``asymptotic'' to the   polyhomoclinic orbit 
$\Om(c)$ when $e\to 0$, in the sense that $\Om(c)\in\ov{\sA_{BS}(c)}$. 
Moreover, $\sA_{BS}(c)$ satisfies the transverse homoclinic property and the twist property, and the period of the orbit at
energy $e$ on $\sA_{BS}(c)$ tend to $+\infty$ when $e\to0$.
\end{lemma}

Note that the periodic orbits in $\sA_{BS}(c)$ need {\em not} be minimizing.

\vskip1mm

\begin{proof}  We will of course apply Theorem~\ref{thm:hypdyn1} to $\Om(c):=\big(\Om_1(c),\ldots,\Om_{\ell^0}(c)\big)$, but, 
in order to prove that this annulus admits the transverse homoclinic property, we will need 
to introduce another polyhomoclinic orbit $\Om(\ha c\,):=\big(\Om_1(\ha c\,),\ldots,\Om_{\ell^1}(\ha c)\big)$, 
with $\ha c \in \H_1(\T^2,\Z)$, $\ha c\neq c$,
and to apply Theorem~\ref{thm:hypdyn1} to the concatenation
$\Om(c)*\Om(\ha c\,)$.  For this we need the sign condition
\beq\label{eq:signcond}
\sig_{ex}(\Om_1(c))=\sig_{ex}(\Om_1(\ha c\,))
\eeq
to be satisfied. The existence of such an $\om(\ha c\,)$ is immediate by Proposition~\ref{prop:polyhom} and 
Lemma~\ref{lem:opp}. Note that $\ha c\neq-c$.

\vskip1mm$\bu$
We set  $\Om(c):=\Om^0=(\Om^0_1,\ldots,\Om^0_{\ell^0})$ and 
$\Om(\ha c\,):=\Om^1=(\Om^1_1,\ldots,\Om^1_{\ell^1})$ and we write
$\{1^0,2^0,\ldots,\ell^0,1^1,2^2,\ldots,\ell^1\}$ for the associated set of indices.

\vskip1mm$\bu$ Note that, by (\ref{eq:signcond}), the polyhomoclinic
$$
\Om^*=\big(\Om^0_1,\ldots,\Om^0_{\ell^0},\Om^1_1,\ldots,\Om^1_{\ell^1}\big)
$$
is compatible, so that one can apply Theorem \ref{thm:hypdyn1}.
There exists  an energy $e_1(c)>0$ and, for $1^\nu\leq i^\nu\leq \ell^\nu$,  there exist  neighborhoods 
$\cR^\nu_{i^\nu}$  in $\Sig^u[\eps]$ of the exit points 
$a_{i^\nu}$,  such that for $\abs{e}\leq e_0$, the intersections 
$
R_{i^\nu}^\nu(e)=\cR_{i^\nu}^\nu\cap C\inv(e)
$
are rectangles in the section $\Sig^u[\eps,e]$. 
These rectangles form a horseshoe for the  Poincar\'e  map $\Phi$ associated with the section $\Sig^u(e)$  in $C\inv(e)$, whose 
transition matrix  $A=\big(\al(i,j)\big)$ satisfies (\ref{eq:transmat}).

\vskip1mm$\bu$
For $e\in \,]0,e_1(c)]$ we denote by
$m(e)\in R^0_{1^0}(e)$ the periodic point associated with the periodic coding 
\beq\label{eq:percoding1}
\cdots\,(1^0,2^0,\ldots,\ell^0)\,\cdots
\eeq
which is admissible relatively to $A$ since $\Om^0$ is compatible.
Let $\Ga_e$ be the corresponding periodic orbit  for the Hamiltonian flow. We will prove that the union
$$
\sA_{BS}(c)=\bigcup_{e\in\,]0,e_1(c)]}\Ga_e
$$
is an annulus defined over $]0,e_1(c)]$. 

\vskip1mm$\bu$ Note first that
$\Ga_e$ is hyperbolic, as $m(e)$ is. One then has to prove that the projection on $\T^2$ of the corresponding
solution realizes $c$.
For this, the crucial point is that $\cR_{i^0}^0$ is a neighborhood of $a_{i^0}^0$ {\em in $\Sig^u$}. By compatibility,
 there exists a sequence 
$(\ov \Ga(e_n))$ of minimizing periodic orbits, with $\ov \Ga(e_n)\subset C\inv(e_n)$, which converge to the 
polyhomoclinic orbit  $\Om_0=\Om^{(c)}$ (so $e_n\to0$ when $n\to\infty$).
So, for $n$ large enough the orbits $\ov\Ga(e_n)$ intersect the section $\Sig^u$ at points $m_i^n\in \cR_{i^0}^0$,
which are ordered in the following (cyclic) way
$$
m_1^n\prec m_2^n\cdots\prec m_\ell^n.
$$
By periodicity, the point $m_1^n$ is in the maximal invariant set defined by the horseshoe and admits the coding 
(\ref{eq:percoding1}). It therefore coincides with $m(e_n)$ by uniqueness. As a consequence the orbits
$\Ga_{e_n}$ and $\ov\Ga(e_n)$ coincide. 
Now all the orbits in $\sA_{BS}(c)$ are homotopic in $\A^2$ and the orbits $\ov\Ga_n$ realize $c$,
this proves in particular that the annulus $\sA_{BS}(c)$ realizes $c$. Note moreover that each
orbit $\Ga(e)$ is homotopic in $\A^2$ to the concatenation $\Om_1*\cdots*\Om_\ell$.

\vskip1mm$\bu$ Let us prove the existence of transverse homoclinic orbits for each $\Ga(e)$. 
In fact, there exists an infinite set of such orbits, which come from the application of Theorem~\ref{thm:hypdyn1} to the
polyhomoclinic orbit $\Om^*$.
Given any  finite sequence $[a_1,\ldots,a_p]$ which is not a concatenation of the sequence $1^0,2^0,\ldots,\ell^0$ and which is 
admissible according to the transition matrix $A$,  each coding of the form 
$$
\ldots,1^0,2^0,\ldots,\ell^0, [a_1,\ldots,a_p],1^0,2^0,\ldots,\ell^0,\ldots
$$
gives rise to a nontrivial orbit homoclinic to the periodic point $m(e)$.
Now sequences such as $[a_1,\ldots,a_p]$ exist due to the presence of symbols from the second polyhomoclinic
orbit $\Om^1$. The resulting homoclinic orbits are obviously transverse inside their energy level by construction 
of the horseshoe.

\vskip1mm$\bu$ It only remains to prove the twist property. For this first remark that the period of $\Ga(e)$
is equivalent to $\ell^0\tau(e)$, where $\tau(e)$ is the transition time between the entrance and exit sections
near the fixed point introduced in Theorem~\ref{thm:hypdyn1}. Now, by Lemma~\ref{lem:phiin}  one immediately 
checks that 
$$
\tau(e)=-\frac{1}{\la_2}\Log(e)+\tau_r(e),
$$
where $\tau_r$ is $C^1$ bounded. This proves that $\tau'(e)<0$ and that moreover $\tau(e)\to-\infty$ when $e\to0$.
Reducing $e_1(c)$ if necessary, this proves our claim for the restricted annulus.
\end{proof}


\subsection{The singular annulus}
In this section we in fact prove the existence of a singular annulus attached to each pair of opposite {\em simple} 
positive homoclinic orbits. 

\begin{lemma}\label{lem:singann}
Fix a classical system $C$ of the form (\ref{eq:classham}) and assume that Conditions $(D)$
are satisfied. Then there exists a singular annulus $\sA^\bu$ for $C$, which admits 
heteroclinic connections with each initial annulus of the chains
of Lemma~\ref{lem:inannhetconn}, which are transverse in their energy 
levels. Moreover, $\sA^\bu$ admits a neighborhood $O$ in $\A^2$ such that there exists 
a Hamiltonian $C_\circ$, defined on an open set $\jO\subset \A^2$ containing $O$,
whose Hamiltonian vector field coincides with $X_C$ on $O$, and which admits 
a normally hyperbolic $2$-dimensional annulus on which the its time-one map 
is a twist map in suitable symplectic coordinates.
\end{lemma}

\begin{proof} We know by Lemma \ref{lem:simppos} that there exist a simple positive homoclinic orbit $\Om$,
and, by Lemma~\ref{lem:opp}, its opposite orbit $\ha\Om$ is also positive and satifies 
$\sig_{ex}(\ha\Om\,)=-\sig_{ex}(\Om)$. We denote by $c$ and $-c$ the (necessarily primitive) homology 
classes corresponding to $\Om$ and $\ha \Om$ respectively.
We will apply Theorem~\ref{thm:hypdyn1} (and, more precisely, Lemma~\ref{lem:anasympt}) to {\em each} 
homoclinic orbit $\Om$ and $\ha \Om$, and 
Theorem~\ref{thm:hypdyn2} to the pair $(\Om,\ha\Om)$, which obviously satisfies the sign condition 
of this theorem.

\vskip1mm$\bu$ By Lemma~\ref{lem:anasympt}, there exists an interval $I^*=\,]0,e^*]$ and two annuli 
$$
\sA^\pm=\sA_{BS}(\pm c)
$$ 
realizing $\pm c$ and defined over  $I^*$, such that
$
\Om\in\ov{\sA^+},\qquad \ha\Om\in\ov{\sA^-}.
$
By Theorem~\ref{thm:hypdyn2}, there exists $e^0<0$ auch that for $e\in\,]e^0,0[$, the periodic sequence 
$\cdots(0,1)\cdots$ defines a hyperbolic periodic  point $m^0(e)$ for the Poincar\'e map associated with $\Sig[\eps,e]$.
Let $\Ga^0(e)$ be the associated hyperbolic periodic orbit for the flow, then,
as in Lemma~\ref{lem:anasympt}, the union 
$$
\sA^0=\bigcup_{e\in\,]e^0,0[} \Ga^0(e)
$$
is an annulus, which, by construction, contains the union $\Om\cup\ha\Om$ in its closure. Now we define
our singular annulus as the union
$$
\sA^\bu=\sA^+\cup \sA^-\cup \sA^0\cup \Om\cup\ha\Om\cup\{O\}.
$$

\vskip1mm$\bu$ Let us now prove that $\sA^\bu$ is a $C^1$ normally hyperbolic submanifold of $\A^2$. 
There exist several ways for doing this, the most elegant one being to use the ``block theory'' of \cite{C08}. 
We only need to exhibit a neighborhood of $\sA^\bu$ which satisfies the expansion and contraction 
conditions of \cite{C08}. For this we can use the (singular) foliation 
$$
\sA^\bu=\bigcup_{e\in[e^0,e^*]}\sA^\bu\cap C\inv(e)
$$ 
and construct a suitable $3$--dimensional block around each leaf of this foliation, contained in the corresponding
energy level and continuously varying with the energy. This amounts to finding
stable and unstable bundles for the $C^0$ manifold $\sA^\bu$, fibered by the energy,  and proving that these 
bundles are $C^0$.

This is obvious for the bundles over the union $\sA^+\cup \sA^-\cup \sA^0$, which is (regularly) foliated by 
hyperbolic periodic orbits: the stable and unstable bundles are just the unions of the stable and unstable bundles
over each orbit, these latter ones being defined as those of the map $\Phi^{TC}$, where $T$ is the period of
the corresponding orbit. 

Now, going back to Theorem~\ref{thm:hypdyn1} applied to the simple homoclinic orbit $\Om$, 
denote by $m^+(e)$ the periodic point corresponding to the coding sequence $\cdots (0,0) \cdots$. Therefore 
$$
m^+(e)=\Sig^u[\eps,e]\cap \sA^+.
$$
On the other hand 
$$
m^0(e)=\Sig^u[\eps,e]\cap \sA^0.
$$
Now $m^+(e)$ and $m^0(e)$ lie at the intersection of their stable and unstable manifolds, and by Theorem~\ref{thm:hypdyn1}
and Theorem~\ref{thm:hypdyn2}, when $e\to 0$:
$$
W^u(m^+(e))\to \Sig[\eps,0]\cap W_{loc}^u(O),\qquad  W^u(m^0(e))\to \Sig[\eps,0]\cap W_{loc}^u(O)
$$
and,  {\em  if $\Phi_{out}$ is the Poincar\'e map along $\Om$ between $\Sig^u[\eps]$ and $\Sig^s[\eps]$}, then
$$
W^s(m^+(e))\to \Phi_{out}\inv\big(\Sig[\eps,0]\cap W_{loc}^s(O)\big),\qquad  
W^s(m^0(e))\to \Phi_{out}\inv\big(\Sig[\eps,0]\cap W_{loc}^s(O)\big),
$$
(the convergence being understood in the $C^1$ compact open topology). This proves that one can define the
stable and unstable bundles along $\Om$ by continuously continuating those of the orbits in $\sA^+$ and $\sA^0$.
The same argument holds for $\ha \Om$.

As for $O$, of course the energy manifolds becomes singular. However, observe that, due to the form of the flow 
on $W^u(O)$ and $W^s(O)$ (see Section \ref{sec:proofBS}), the ``transverse space''
at $O$ (in $C\inv(0)$) is necessarily the plane $W$ of equation
$$
u_1=0,\ s_1=0,
$$
that is, the plane generated by the weak directions. In this plane, the stable direction is of course
$u_1=0$, while the unstable one is $s_1=0$. Finally, the strong $\lambda$--lemma (see for instance
\cite{D89}) proves that the stable and unstable bundles are continuous at $O$.

Now, considering normalized generating vectors for the stable and unstable bundles, one can construct a tubular 
neighborhood $N_e$ (with small radius $\de>0$) of each leave in its energy level. The union of these blocks satisfies
the expansion and contraction condition of \cite{C08} for a large enough iterate $\Phi^{\tau C}$, which proves that
$\sA^\bu$ is a Lipschitz manifold.

Finally, one sees from Section~\ref{sec:proofBS} that the ratio between the outer Lipschitz constants and the inner ones
is lower bounded by $\la_1/\la_2-\rho$, where $\rho$ can be made arbitrarily small by taking $\abs{e^0}$ and $e^*$ small
enough, which proves by \cite{C08} that $\sA^\bu$ is in fact of class $C^{1+\de}$ for some suitable $\de>0$.

The last assertion on the continuation of $\sA_\bu$ comes directly from the possibility of gluing
symplectically a disc whose boundary is a periodic orbit with zero homology, and continuing the vector field $X_C$
to this disc in such a way that it is foliated by periodic orbits surrounding an elliptic point. 
Using the relation energy/period in the neighborhood of the hyperbolic fixed point on the annulus $\sA_\bu$,
one can moreover control the continuation in such a way that the time-one map of the flow satisfies a twist
condition (as in the case of the standard pendulum $\pdemi r^2+ a\cos\th$ when $a$ is small enough.) Finally
the normal hyperbolicity is got by taking a trivial product of the disc with a hyperbolic point and smoothing
in the neighborhood of the gluing zone.
\end{proof}


\subsection{Initial annuli and heteroclinic connections}
In this section we show how to modify the chains $\jA(c)$ obtained in Proposition~\ref{mainprop} in order for 
them to be finite, with initial annuli admitting  heteroclinic connections with the singular annulus.
We also show how to continue them to obtain the third statement of Theorem~II.

\begin{lemma}\label{lem:inannhetconn}
Fix a classical system of the form (\ref{eq:classham}) and assume that Conditions $(D)$
are satisfied. Then:
\begin{enumerate}
\item for each $c\in \H(\T^2,\Z)$ there exists a chain $\bA(c)=\big(\sA_1(c),\ldots,\sA_\ell(c)\big)$,
where $\sA_1(c)$ is defined over $]0,e_1(c)]$ and 
$\sA_\ell(c)$ is defined over $[e_\ell,\ldots,+\infty[$;

\item given $c,c'\in\H_1(\T^2,\Z)$, there exists $\sig\in\{0,1\}$ such that $\sA_1(c)$ and $\sA_1(\sig c')$
satisfy
$$
W^u(\sA_1(c))\cap W^s(\sA_1(\sig c'))\neq \emptyset,\qquad W^s(\sA_1(c))\cap W^u(\sA_1(\sig c')))\neq \emptyset
$$
both intersections being transverse in $\A^2$;

\item moreover, for each $c\in\H_1(\T^2,\Z)$, $\sA_1(c)$ admits transverse
heteroclinic connections as above with $\sA^\bu$.
\end{enumerate}
\end{lemma}

\begin{proof}
Recall that, given a class $c\in \H_1(\T^2,\Z)$, we proved in  Proposition~\ref{mainprop} and Corollary~\ref{cor:highen}  
the existence of a chain $\jA(c)=(\sA_k)_{k\in Z}$ of annuli realizing $c$ and  defined over a sequence of  consecutive intervals of 
the form  $(I_k)_{k\in Z}$,  where $Z$ is an upper bounded interval of $\Z$. 

\vskip1mm$\bu$ If $Z$ is finite we choose $\bA(c)=\jA(c)$.
However, in order to prove our claim on the heteroclinic  connections, we have to make precise the relation between 
the first annulus $\sA_1$ of this chain and the ``Birkhoff-Smale'' 
annulus $\sA_{BS}(c)$ of Lemma~\ref{lem:anasympt}. 
Let $I_1$ and  $I_{BS}$  be the intervals associated with $\sA_1$ and $\sA_{BS}$ respectively.

By construction and Condition $(D_6(c))$, there exists a unique minimizing periodic orbit in the class $c$ for each energy 
$e$ in $I_1$, which is precisely the intersection $\ov\Ga(e)=\sA_1\cap C\inv(e)$.
Now, since the polyhomoclinic orbit $\Om(c)$ is positive,  there exists a sequence  $(e_n)$ in $I_1$, with 
$e_n\to 0$ when $n\to\infty$,  such that the associated orbit $\ov\Ga(e_n)$ converges to $\Om(c)$. 
Therefore, for $n$ large enough, for the same reason as in Lemma~\ref{lem:anasympt} (hyperbolic maximality),
the orbit $\ov\Ga(e_n)$ necessarily coincides with  the orbit $\Ga(e_n)=\sA(c)\cap C\inv(e_n)$. 
As a consequence $\sA(c)\cap \sA_0\neq\emptyset$. But this intersection is 
closed in $C\inv(\R^{*+})$ and it is also open by uniqueness of the continuation of hyperbolic periodic orbits. 
Therefore both annuli coincide over the intersection $I_1\cap I_{BS}$.

\vskip1mm$\bu$ Assume now that $Z$ is infinite. In this case, by the same arguments as above,
there exists a sequence $n_k\to-\infty$, such that
each annulus $\sA_{n_k}$ contains a minimizing orbit $\Ga(e_{n_k})$ and 
the sequence $\big(\Ga(e_{n_k})\big)$ converges to $\Om(c)$. 
Again, the same arguments as above prove that  $\Ga(e_{n_k})\subset\sA(c)$ for $k\geq k_0$ large enough,
and that $\sA_{n_k}$ is contained in $\sA(c)$ for $k\geq k_0$. 
In this case, we set $\bA(c)=(\sA'_1,\ldots,\sA'_\ell)$, with
$$
\sA'_1=\sA_{BS}(c)\cap C\inv(]0,\max I_{n_{k_0}}]),
$$
and
$$
\sA'_2:=\sA_{k_0+1},\ldots,\  \sA'_\ell:=\sA_{\max Z}.
$$

\vskip1mm$\bu$ It remains to prove the existence of heteroclinic connections. 
Note first that by Lemma \ref{lem:opp}, given two primitive classes $c$ and $c'$, there exists $\sig\in\{0,1\}$ such that
 $$
 \sig_{ex}\big(\Om^0_1\big)=\sig_{ex}\big(\Om^1_1\big),
$$
where, as usual, $\Om(c)=(\Om^0_1,\ldots,\Om^0_{\ell^0})$ and 
$\Om(\sig c')=(\Om^0_1,\ldots,\Om^0_{\ell^0})$.
We will prove that the initial annuli of $\bA(c)$ and $\bA(\sig c')$ admit heteroclinic connections. 
For this we will apply Theorem \ref{thm:hypdyn1} to the polyhomoclinic orbit 
$\Om(c)*\Om(\sig c')$, which is compatible by our choice of $\sig$.
For $0<e<e_0$, this yields the existence of orbits for the Poincar\'e map with coding sequences of the form
\beq\label{eq:hetcoding}
\ldots,(1^0,2^0,\ldots,\ell^0),[a_1,\ldots,a_p],(1^1,2^1,\ldots,\ell^1),\ldots
\eeq
where $[\ell^0,a_1,\ldots,a_p,1^1]$ is any finite sequence admissible relatively to the transition matrix. 
Such sequences obviously exist (for instance $[\ell^0,1^0,\ldots,\ell^0,1^1]$, thanks to (\ref{eq:transmat1})).
Now, for each energy $e\in\,]0,e_0]$, the coding (\ref{eq:hetcoding}) induces a heteroclinic orbit for the Poincar\'e map
between the periodic point $m(e)$ with periodic coding $(1^0,2^0,\ldots,\ell^0)$ and the periodic point $m'(e)$ with 
periodic coding $(1^1,2^1,\ldots,\ell^1)$. Therefore the associated orbits $\Ga(e)$ and $\Ga'(e)$ admit heteroclinic
connections at energy $e$. These connections are transverse in their energy level, by construction of the horseshoe. 
This immediately proves that the initial annuli of the chains $\bA(c)$ and $\bA(\sig c)$ admit transverse 
heteroclinic connections, by our previous construction of these annuli. 

\vskip1mm$\bu$ Our last statement is then obvious, by construction of the singular annulus $\sA^\bu$, since
it contains $\sA_{BS}(\pm c^\bu)$ for the corresponding $c^\bu$.
\end{proof}


\section{The set $\jU$ is residual}\label{sec:gen}
We fix as usual a positive definite quadratic form $T$ on $\R^2$ and for each $U\in C^\ka(\T^2)$, $\ka\geq 2$, we denote
by $C_U$ the associated classical system on $\A^2$. In this section we complete the proof of Theorem II, that is,
we show that the set $\jU$ of potentials $U\in C^\ka(\T^2)$ such that $C_U$ satisfies the three items of this theorem is residual in 
$C^\ka(\T^2)$. The main ingredient of our proof is the parametrized genericity theorem of Abraham, which we
recall here in an adapted form for the convenience of the reader.

\vskip2mm

\noindent {\bf Theorem (\cite{AR}).} {\em Fix $1\leq k<+\infty$. Let $\jA$ be a $C^k$ and second-countable Banach manifold.
Let  $X$ and $Y$ be finite dimensional $C^k$ manifolds. 
Let $\chi : \jA\to C^k(X,Y)$ be a map such that the
associated evaluation 
$$
\ev_\chi: \jA\times X\to Y,\qquad \ev_\chi(A,x)=\big(\chi(A)\big)(x)
$$
is $C^k$ for the natural structures.
Fix a be a submanifold $\De$ of $Y$ such that
$$
k>\dim X-\codim  \De
$$
and assume that $\ev_\chi$ is transverse to $\De$. 
Then the set $\jA_\De$ of $A\in\jA$ such that $\chi(A)$ is transverse to $\De$
is residual in $\jA$.}

\vskip2mm

As always, the spaces $C^\ka(\T^m)$ are endowed with their usual $C^\ka$ norms for $1\leq \ka <\infty$, which make them
Banach spaces, and $C^\infty(\T^m)$ is equipped with its usual Fr\'echet structure.
We will try to use everywhere abstract arguments; however direct proofs based on explicit constructions would also
often be possible.


\subsection{Large energies}
Let us first recall an easy result.

\begin{lemma}\label{lem:dens1}
Given $m\geq 1$, the set $\jM^\ka(\T^m)$ of functions of $C^\ka(\T^m)$ which admit a unique maximum, which is 
nondegenerate, is open and dense in $C^\ka(\T^m)$ for $2\leq \ka\leq +\infty$.
\end{lemma}

\begin{proof} 
This is a standard result in Morse theory. The fact that $\jM^\ka(\T^m)$ is open is obvious and its density can easily be proved by 
adding a suitable small enough $C^\infty$ bump function with a unique nondegenerate maximum to any given function in 
$C^\ka(\T^m)$. 
\end{proof} 

\begin{lemma}
Fix $2\leq \ka\leq +\infty$. Given $c\in\H_1(\T^2,\Z)$, the set $\jU_9(c)$ of potentials in $C^\ka(\T^2)$ 
such that Condition $(D_9(c))$ is satisfied is open and dense $C^\ka(\T^2)$. As a consequence, the set $\jU_9$ 
of potentials in $C^\ka(\T^2)$ such that Condition $(D_9)$ is satisfied is residual in $C^\ka(\T^2)$.
\end{lemma}

\begin{proof} Up to the standard coordinate change, one can assume that $c\sim (1,0)$ in $\Z^2$.
The averaged potential then reads
$$
U_c(\th_2)=\int_\T U(\th_1,\th_2)\,d\th_1.
$$
Consider the (obviously well-defined) map $\jI_c:C^\ka(\T^2)\to C^\ka(\T)$ such that
$$
\jI_c(U)=U_c.
$$
Clearly $\jI_c$ is linear and continuous ($\norm{\jI_c}\leq 1$). It is also clearly surjective (a given function $U_c$ admits
the function $U(\th_1,\th_2)=U_c(\th_2)$ as a preimage). So by the usual Open Mapping Theorem for Banach spaces the
map $\jI_c$ is open for $2\leq \ka <+\infty$. This is still an open mapping for $\ka=+\infty$, by the Fr\'echet version of the previous
theorem (see for instance \cite{Ru}). Now by Lemma \ref{lem:dens1} the subset  $\jM^\ka(\T)\subset C^\ka(\T)$ is open and dense, 
so its inverse image $\jU_9(c)=\jI_c\inv(\jM^\ka(\T))$ is open and dense in $C^\ka(\T^2)$. The second claim is immediate by countable
intersection, since $C^\ka(\T^m)$ is complete for the usual norm.
\end{proof}


\subsection{The neighborhood of the critical energy}
In this part we prove that our conditions $(D_1)-(D_4)$ are generic in $C^\infty(\T^2)$ and postpone the study of the
$C^\ka$ regularity to the last section, where the global structure of the systems will be used explicitely.


\paraga One can even be more general for $(D_1)$. By Lemma \ref{lem:dens1}, 
the set $\jU_1^\ka:=\jM^\ka(\T^2)$ of all potentials $U$ such 
that $(D_1)$ is satisfied is open and dense in $C^\ka(\T^2)$, for $2\leq \ka\leq +\infty$.


\paraga We will now have to use the Sternberg conjucacy theorem and we limit ourselves to $C^\infty$ potentials.
Given $U\in \jU_1^\infty$, we denote by $O_U$ the hyperbolic fixed point of $C_U$ associated with the maximum 
$\th^0_U$  of $U$. 

\begin{lemma}
The set $\jU_2$ of potentials $U\in\jU_1^\infty$ such such that $O_U$ admits a proper conjugacy neighborhood is residual
in $C^\infty(\T^2)$. 
\end{lemma}

\begin{proof} Following \cite{B10}, recall that if $A$ is the matrix of $T$, 
if $B_U=-\d^2U(\th_0)$ and if
$$
L_U=\big(A^{-1/2}(A^{1/2}B_U A^{1/2})^{1/2}A^{-1/2}\big)^{1/2},
$$
then the change of variables
$\ov u=\frac{1}{\sqrt 2}(Lx+L\inv y)$,
$\ov s=\frac{1}{\sqrt 2}(Lx-L\inv y)$,
is symplectic  and reduces the quadratic part of the system $C_U$ to the form
$\pdemi \langle D(U)\ov u, \ov s\rangle$,
with
\beq\label{eq:mapD}
D(U)=L_UAL_U \in \S.
\eeq
Here $\langle\,\,,\,\rangle$ stands for the Euclidean scalar product and where $\S\subset M_2(\R)$ is the cone of positive definite 
symmetric matrices. This shows that the positive eigenvalues of $O_U$ are the eigenvalues of $D(U)$. Finally,
one easily checks that  the map $D:\jU_1^\infty\to \S$ defined by~(\ref{eq:mapD}) is continuous and open. 

\vskip1mm

To apply the Sternberg theorem and get our conjugacy result,  we need the system $C_U$ to be formally conjugated to a normal form 
$$
N(u,s)=\la_1u_1s_1+\la_2u_2s_2+R(u_1s_1,u_2s_2),
$$
where $R$ is $C^1$ flat at $(0,0)$. For this, a sufficient condition is that  the positive eigenvalues $\la_i$ of $O_U$  satisfy 
the nonresonance conditions
\beq\label{eq:nonres}
\la_1k_1+\la_2k_2\neq 0,\qquad \forall (k_1,k_2)\in\Z^2\setm\{(0,0)\}.
\eeq
The subset $\S^*$ of positive symmetric matrices whose eigenvalues satisfy (\ref{eq:nonres}) is cleary residual in $\S$, so that 
the inverse image $\jU_2:=D\inv(\S^*)\subset \jU_1^\infty$ is also residual in $\jU_1^\infty$, by the previous property.

\vskip1mm

Now,   the equivariant symplectic Sternberg theorem (see \cite{C85} for a general exposition and 
\cite{BK02} for a recent proof in our setting) applies in a neighborhood of $O_U$ for every $U\in \jU_2$, 
and yields a proper conjugacy neighborhood. This proves our statement. 
\end{proof} 

Note that the previous result does not hold in finitely differentiable classes, since the minimal regularity of the system in order to get a 
conjugacy tends to $+\infty$ when the ratio  $\la_1/\la_2$ tends to $1$.


\paraga  Recall that given $U\in\jU_2$, the exceptional set $\jE_U\subset W^s(O_U)\cup W^u(O_U)$ attached to $O_U$ is  
intrinsically defined and continuously depend on $U$.

\begin{lemma}
The set $\jU_3$ of potentials in $\jU_2$ such that condition $(D_3)$ is satisfied is residual in $C^\infty(\T^2)$. 
\end{lemma}

\begin{proof} We will prove that for each integer $N$, the set $\jU_3(N)$ of all $U\in \jB^\infty(0,N)\cap \jU_2$
such that $(D_3)$ is satisfied for $C_U$ is residual in $\jB^\infty(0,N)\cap \jU_2$. Our claim easily follows, since
$$
\bigcap_{N\in\N}\Big(\jU_3(N)\cup (\jU_2\setm \ov \jB^\infty(0,N)\Big)
$$
is residual in $\jU_2$ and $(D_3)$ is satisfied for any $U$ in this subset. 

\vskip1mm

Observe first that given $U\in \jB^\infty(0,N)$,
there exists a Euclidean disk $B_{\de(N)}\subset \T^2$ centered at $\th^0_{U}$ with radius $\de(N)>0$, such that  the local 
stable and unstable manifolds $W_{loc}^\pm(O_U)$ of $O_U$ are graphs over $B_{\de(N)}$.  From now on we fix $N$
and drop it from the notation.
 
\vskip1mm

We denote by $W^\pm(U,\de)$ the parts of $W_{loc}^\pm(O_U)$ located above $B_\de$.
For $n\in\N$, we set
$$
W_n^-(U)=\Phi^{C_U}([0,n],W^u(U,\de)).
$$ 
We will first prove that the set of $U\in \jB^\infty(0,N)$ such that $W_n^-(U)$ transversely intersects $W^s(U,\de)$ and 
satisfies $W_n^-\cap \jE_U =\emptyset$ is open and dense in $\jB^\infty(0,N)$. Since the proof of the transversality property is classical,
we will only give a sketch of proof, a more detailed version can be found in [LM].

\vskip2mm 

The following lemma can be easily proved following the lines of [CP]. Here $B(0,r)$ will be the ball in $\R^2$
centered at $0$ with radius $r$ relatively to the Max norm.

\begin{lemma}\label{lem:pertlag} Let $(M,\om)$ be a $4$--dimensional symplectic manifold, let $H\in C^2(M,\R)$ and let $L$ be a Lagrangian
submanifold contained in some level $H\inv(e)$ (not necessarily regular). Assume that $z\in L$ satisfies $X^H(z)\neq0$. 
Then one can find a neighborhood $N$ of $z$ in $M$ and a symplectic coordinate system 
$(x,y)\in B(0,2)\times B(0,\eps)$ on $N$ such that
$$
L\cap N=\{y=0\},\qquad z=\big((0,1),(0,0)\big), \qquad X_{\vert L}^H=\frac{\d}{\d x_1}.
$$
Assume that $L'$ is another Lagrangian manifold contained in $H\inv(e)$ such that $z\in L\cap L'$.  
Let $\Sig=\{x_1=1\}\cap H\inv(e)$, so that $\Sig$ is a symplectic section for $X^H$ (assuming $N$ small enough).
Then there exists a neighborhood $\til\Sig$ of $z$ in $\Sig$ such that given $\ze\in \til\Sig\cap L'$, there exists a Lagrangian submanifold 
$\jL$ of $M$ such that
\begin{itemize}
\item $\jL\cap N$ is the graph of some function $y=\phi(x)$ over $B(0,2)$;
\item $\jL\cap L' \cap \Sig=\{\ze\}$;
\item $\jL\cap \Sig$ and $L'\cap \Sig$ transversely intersect at $\ze$;
\item the part of $\jL$ over $B(0,2)\setm B(0,1)$ is contained in $H\inv(e)$;
\item the part of $\jL$ over $\{\abs{x_2}\geq 1\}\cup \{x_1\leq -1\}$ coincides with $L$;
\item for $1\leq k\leq +\infty$, given $\eps>0$, one can choose $\til \Sig$ small enough so that the previous conditions
are realized with $d_{C^k}(0,\phi)<\eps$.
\end{itemize}
\end{lemma}

\vskip2mm

Let now $\d W^s(U,\de)$ be the part of $W^s_{loc}(O_U)$ located over the circle $\d B_\de$. By compactness of 
$\ov W^u_n$, for every $z\in \d W^s(U,\de)\cap W^u_n$, one easily proves the existence of an arbitrarily small $\tau(z)>0$,
continuously depending on $z$, such that, if 
$z'=\Phi^{C_U}(z,\tau(z))$, 
$$
(*) \hskip2cm \pi\inv(z')\cap \Phi^{C_U}(\R^-,z)=\emptyset\hskip2cm
$$
(that is, the projection of the semiorbit $\Phi^{C_U}(\R^-,z)$ has no selfintersection at $z'$).
Moreover, using the classical twist property for Lagrangian spaces along an orbit of $C_U$, one can assume that:
\vskip1mm
\item  \hskip3cm $(**)$ \hskip1cm the  restriction of $\pi$ to $W^u_n$ is regular at $z'$. 
\vskip2mm

 Applying this process to each point of $\d W^s(U,\de)\cap W^u_n$, one gets a 
continuous curve $\sig_U\subset W^s(U,\de)$, surrounding $O_U$, such that every point $z\in\sig_U$ satisfies both
properties $(*)$ and $(**)$.  

\vskip1mm

The previous lemma will be applied to the case where $L=W_n^-$, $L'=W^s(U,\de)$ 
and where $z$ is a point of $\sig\cap W_n^-$. We fix $\eps>0$.
By compactness of  the semiorbit $\Phi^{C_U}(\R^-,z)$ and the previous assumption, one can assume the neighborhood 
$N$ of Lemma~\ref{lem:pertlag} small enough so that the restriction of $\pi$ to the subset $\Phi^{C_U}(\R^-,W_n^-\cap N)$ is injective and its 
restriction to $L\cap N$ is regular. With the previous notation, one immediately checks that one can find a point $\ze\in\til\Sig\cap W^s(U,\de)$ 
{\em not contained in the exceptional set $\jE_U$} such that the conclusions of the lemma hold true for
$\ze$ and the new manifold $\jL$,  and moreover such that $\jL$ is the graph of some function $\psi$
over $\pi(L\cap N)$ with $d_{C^\infty}(0,\psi)<\eps$.

\vskip1mm

We now define a new potential function $\til U$,  which coincides with $U$ outside the projection $\pi(B(0,1)\times\{0\})$
(with the notation of the lemma for the coordinates in $N$), and such that for $\th\in\pi(B(0,1)$: 
$$
\til U(\th)=\ov e_U -\pdemi T(\psi(x)),
$$
where $\ov e_U=\Max U$ is the critical energy for $C_U$ (so that the associated level contains $W^\pm(O_U)$).
By construction, $\jL\subset C_{\til U}\inv(\ov e_U)$, and one immediately checks that  the fixed point $O_{\til U}$ of
new system $C_{\til U}$ coincides with $O_U$ and has the same energy. 
Moreover,  the new manifolds $\til W_n^-$ and $\til W_{loc}^+(O_U)$, 
defined in the same way as above for the new system $C_{\til U}$ and $O_{\til U}$, now transversely intersect 
at $\ze$ in $C_{\til U}\inv(\til e_{\til U})$. 

\vskip1mm

Taking Lemma \ref{lem:pertlag} into account, this yields a neighborhood $N_z$ of $z$ inside which the intersection of $\til W_n^-$
and $\til W_{loc}^+(O_U)$ is transverse. The main remark is that the size of $N_z$ is by construction independent of $\eps$.
Moreover, clearly $\til U\to U$ in the $C^\infty$ topology when $\eps\to 0$.

\vskip1mm

Observe now that the subset $\sig\cap W_n^-$ is compact (this is an easy consequence of the fact that $\sig_U$
is a transverse section for the flow of $C_U$ on $W^s_{loc}$). It is therefore possible to cover $\sig\cap W_n^-$
with a finite number of neighborhoods $(N_i)$ of the previous form.  Using the fact that transversality is an open property,
one can then choose a (sufficiently decreasing) finite sequence $(\eps_i)$ in such a way that such  that the above process applied 
to each $N_i$ with the perturbation parameter $\eps_i$ yields a new potential $\til U$ for which $\til W_n^-$ and $\til W_{loc}^+(O_{\til U})$
transversely intersect in their energy level, and moreover such that the intersection points do not belong to the exceptional 
set $\jE_{\til V}$.

\vskip1mm

As a consequence, the subset $\jU_3(n,N)$ of all $U\in \jB^\infty(0,N)\cap \jU_2$ such that $W_n^-(U)$ and $W^s(U,\de)$ 
transversely intersect in their energy level outside the exceptional set is dense in $\jB^\infty(0,N)\cap \jU_2$. Since the exceptional set and the
transverse intersections continuously vary with $U$, this is also an open subset. Therefore the intersection
$$
\bigcup_{n\in\N} \jU_3(n,N)
$$
is residual in $\jB^\infty(0,N)\cap \jU_2$, which concludes the proof, according to our first remark.
\end{proof}


\paraga  Recall now that
$O_U$ admits a set of homoclinic orbits whose projection on $\T^2$ generate $\pi_1(\T^2,\th^0_U)$. 
We defined the amended potential $U^*=U-\max_{\T^2}U$, together with the associated amended Hamiltonian, Lagrangian and action.
The amended action of a homoclinic orbit is by definition the amended action of the two opposite solutions associated with it.

\begin{lemma}
The set $\jU_4$ of potentials in $\jU_3$ such that condition $(D_4)$ is satisfied is residual in $C^\infty(\T^2)$. 
\end{lemma}

\begin{proof} We fix a lift of $\T^2$ to $\R^2$ and given $V\in\jV(U)$, we write $x_V$ for the corresponding lift of $\th^0_V$.
We will first prove that there is $M>0$, independent of $V$,  such that any trajectory homoclinic  to $\th^0_V$ with minimal 
amended action lifts to a trajectory contained in the ball $B(x_V,M)\subset\R^2$. Observe first that in the zero energy level of the amended 
Hamiltonian $C_{V^*}$, the velocity $\norm{\dot \th}$ is bounded above by some constant $\mu_1>0$ (independent of $V\in\jV$).
Moreover,  
$$
\Max_{\T^2\setm B} V^*= -\mu_2 <0
$$
and moreover, assuming $\de$ small enough, for any $m\in\Z^2\setm\{0\}$ any curve from $x_V$ to $x_V+m$ intersects 
$\R^2\setm\Pi\inv(B)$ along segments of curves of total length at least $\norm{m}/2$. Hence if the lift $\eta$ of a homoclinic trajectory
(starting from $x_V$) is not contained in $B_\infty(x_V,M)$, it lies in $\R^2\setm\Pi\inv(B)$ during at least $M/2\mu$ and, taking 
the conservation of energy into account,  its
action satisfies
$$
\int L^*(\eta,\dot\eta)\geq \frac{M\mu_2}{\mu_1},
$$
which proves our claim.

\vskip1mm

The previous remark proves that any minimizing homoclinic orbit lies in the intersection 
$$
W^u(O_V,M)\cap W^s(O_V,M)
$$
where 
$$
W^\pm(O_V,M)=\Pi \Big[W^\pm(X_V)\cap (B_\infty (x_V,M)\times\R^2)\Big].
$$ 
One easily checks that $W^\pm(O_V,M)$ are compact. Therefore, using standard arguments of transversality and the graph property
of the stable and unstable manifolds over $B$ (see for instance \cite{O08} and references therein), one proves that the subset $\til \jV$
of potentials $V$ such that $W^\pm(O_V,M)$ intersect transversely in $C_{V^*}\inv(0)$ is open and dense in $\jV$. In particular, for
$V\in\til\jV$, the number of corresponding homoclinic orbits is finite. 

\vskip1mm

Fix now $V\in \til \jV$ and select arbitrarily one homoclinic orbit with minimal amended action. Pick out some point $\th$ on its 
trajectory and fix a bump function $\eta:\T^2\to \R$ whose support is centered at $\th$ and do not intersect any other minimizing homoclinic
trajectory. Then, by transversality, for $\mu$ small enough the system $C_{V+\mu\eta}$ still admits a homoclinic orbit in the neighborhood
of the initial one, whose amended action is smaller that that of the initial one. This proves that this orbit is strictly minimizing, and so
the subset of potentials with this property is dense. The openness easily follows from the finiteness of the number of initial minimizing 
homoclinic orbits. 

\vskip1mm

It only remains to prove that one can perturb the system so that the strictly minimizing orbit does not intersect the exceptional orbits,
which can be easily proved by using adapted bump functions.
\end{proof}


\subsection{Intermediate energies}
We will now use the parametric transversality theorem after introducing adapted representations, for which we need to work
again in the $C^\ka$ topology with $\ka<+\infty$.


\paraga  Let us begin with the nondegeneracy of periodic orbits and prove the following result.

\begin{lemma}
The set $\jU_5$ of potentials $U\in\jU_2$ such such each periodic solution of $C_U$ contained in $C_U\inv(]\Max U,+\infty[)$
 is nondegenerate is residual in  $C^\ka(\T^2)$.
\end{lemma}

A similar result is proved in \cite{O08} for the restriction of classical systems to regular energy levels. The proof here is slighly
more difficult since we have to deal with one parameter families of energy levels.
Again, more general results with full details will appear in [LM].

\begin{proof} Let us set $M=\A^2$ for the sake of clarity. Let $X=M\times [0,T]$, where $T>0$ is a fixed parameter and
let $Y=M^2\times \R$.  From now on we fix $k\geq 2$. Given $N\in \N^*$, we set
$$
\jB_N=\jU_1\cap B_{C^k(\T^2)}(0,N).
$$
Fix $E>0$ (large). 
It is not difficult to prove that there exists $\tau(N,E)>0$ such that for each $U\in \jB_N$, any nontrivial periodic solution of $C_U$
with energy in $[\Max U, E]$ has a (minimal) period larger than $\tau(N,E)$. Obviously $\tau(N,E)$ tends to $0$ when $(N,E)\to\infty$.

\vskip2mm

We introduce the map $\chi: \jB_N\to C^{k-1}(X,Y)$ defined by
$$
\chi_U(z,t)=\Big(z,\Phi^{C_U}_t(z),C_U(z)-\Max U\Big)
$$
(where $\chi_U$ stands for $\chi(U)$). Finally, we set
$$
\De=\Big\{(z,z,e)\mid z\in M,\ e\in\,]\al,\be[\Big\}
$$
where $0<\al<\be$ are two fixed parameters. Therefore the preimage $\chi_U(\De)$ is the set of $(z,t)$ such that  
$z$ is $t$--periodic (note that $t$ need not be the minimal period of $z$) with energy $C_U(z)-\Max U\in\,]\al,\be[$.

\vskip2mm

One easily checks that if $\chi_U\trans_{(z,t)} \De$, then $z$ is nondegenerate $t$--periodic, in the sense that the
Poincar\'e return map relatively to a transverse section inside the energy level of $z$ does not admit $1$ as an eigenvalue.

\vskip2mm 

Now one deduces from Takens perturbability theorem (see  \cite{Ta83}) that there exists an open dense subset
$\jO_N(\al,\be)\subset \jB_N$ such that $\ev_\chi=\jO_N\times X\to Y$ is transverse to $\De$. Indeed, since there
exists a minimal period $\tau$ for all the periodic orbits contained in $\chi_U\inv(\De)$, it is enough to ensure that
the Poincar\'e maps of all periodic points do not admit any root of unity of order $\leq T\tau$ as an eigenvalue,
which is an immediate consequence of Takens results. 

\vskip2mm 

Finally Abraham's transversality theorem applies in our setting and yields the existence of a residual subset 
\beq\label{eq:defR}
\jR_N(\al,\be)\subset\jO_N(\al,\be)
\eeq
such that for $U\in \jR:=\jR_N(\al,\be)$, any periodic solution of $C_U$ contained in 
$$
C_U\inv(]\al+\Max U,\be+\Max U[)
$$
is nondegenerate. Since this is clearly an open property, the set $\jR$ is in fact open and dense, which will
be used in our subsequent constructions. Now our claim is easily obtained by considering sequences $\al_n\to0$,
$\be_n\to+\infty$ and the corresponding intersections of subsets, and finally considering the intersections over $N$.
\end{proof}


\paraga  Conditions $(D_6)$ and $(D_7)$ are to be examined simultaneously.
Here again we assume that $2\leq \ka<+\infty$ is fixed.

\begin{lemma}
The set $\jU_{6,7}$ of potentials $U\in\jU_1^\ka$ such that conditions $(D_6)$ and $(D_7)$ hold true is residual in 
$C^\ka(\T^2)$.
\end{lemma}

\begin{proof} We will first prove that the set of potentials for which two periodic orbits of the same length 
and the same energy satisfy Condition $(D_7)$ is generic, and then that there generically exist at most two orbits with the
same length at the same energy, from which Condition $(D_6)$ will follow. 

\vskip2mm

$\bullet$ Here we use the same techniques as above as start with the open subset $\jR_N(\al,\be)$ defined in
(\ref{eq:defR}). We fix a parameter $\tau_{\rm max}>0$ and introduce the set 
$$
D(\taum)=\Big\{(z_1,z_2)\in M^2\mid z_2\notin \Phi^{C_U}\big([0,\taum],z_1\big)\Big\}.
$$
So $D(\taum)$ is clearly open and when $z_1$ and $z_2$ are periodic points with period $\leq \taum$, their orbits are disjoint.
Now we set
$$
X=D(\taum)\times [0,\taum]^2\times (\R^+)^2,\qquad Y=M^2\times M^2\times \R\times\R.
$$
Given $z\in M$ such that $C_U(z)>\Max U$ and $t\in\R^+$, we denote by $L_U(z,t)$ the length of the projection
$$
\pi\big(\Phi^{C_U}([0,t],z\big)
$$
relatively to the Jacobi-Maupertuis metric at energy $C_U(z)$. We also denote by $\Psi^{C_U}$ the gradient flow of 
$C_U$ relatively to the Euclidean metric on $M$.
We can now introduce the map 
$$
\chi:\jR_N(\al,\be)\to C^{k-1}(X,Y)
$$
such that for $(z_1,z_2,t_1,t_2,s_1,s_2)\in X$, 
$$
\chi_U(z_1,z_2,t_1,t_2,s_1,s_2)=
$$ 
$$
\Big(
\Psi^{C_U}_{s_1}(z_1),\Phi^{C_U}_{t_1}(z_1),\Psi^{C_U}_{s_2}(z_2),\Phi^{C_U}_{t_2}(z_2),
C_U(z_1)-C_U(z_2),L_U(z_1,t_1)-L_U(z_2,t_2).
\Big)
$$
We finally set
$$
\De=\Big\{(z_1,z_1,z_2,z_2,0,0,)\mid z_1\in M,\ z_2\in M\Big\}.
$$
Using the fact that the gradient vector field does not vanish on $C_U>\Max U$, 
one easily checks that the preimage $\chi_U\inv(\De)$ is the set of points $x=(z_1,z_2,t_1,t_2,s_1,s_2)\in X$ such that
\begin{itemize}
\item $(z_1,z_2)\in D(\taum)$, $z_1$ is $t_1$--periodic and $z_2$ is $t_2$--periodic,
\item $s_1=s_2=0$,
\item $C_U(z_1)=C_U(z_2)$ and  $L_U(z_1)=L_U(z_2)$.
\end{itemize}
Note that $\chi_U\inv(\De)$ is invariant under the diagonal action of $\Phi^{C_U}$ on the first two factors.
As a consequence, 
$$
\dim \Big(T_x\chi_U(X) \cap T_{\chi_U(x)}\De\Big)\geq 2.
$$
Now 
$$
\dim X = 2 \dim M + 4 = \codim \De +2.
$$
Note also that by construction of $\jR_N(\al,\be)$, the points $z_i$ are nondegenerate. Therefore they can be 
continued in one-parameter families $(z_i(e))$ of periodic points when the energy varies (in an essentially 
unique way) for $e$ in a neighborhood of $e_0=C_U(z_1)=C_U(z_2)$. Assume that 
$$
\frac{dL_U(z_1(e))}{de}_{\vert e=e_0}=\frac{dL_U(z_2(e))}{de}_{\vert e=e_0},
$$
then clearly
$$
\dim \Big(T_x\chi_U(T_xX) \cap T_{\chi_U(x)}\De\Big)\geq 3
$$
and $\chi_U$ cannot be transverse to $\De$ at $x$. This proves that if $\chi_U\trans_x \De$, then the lengths
of the corresponding periodic orbits have a transverse crossing at the energy $e_0$.

\vskip2mm

Therefore, to prove that Condition $(D_7)$ is generic, one only has to prove that $\chi_U$ is generically transverse
to $\De$, for which we will use Abraham's theorem. We are therefore reduced to prove the transversality of the
map $\chi$ with $\De$. We will use the decomposition
$$
T_{\chi_U(z)}Y=T_{(z_1,z_1)}M^2\times T_{(z_2,z_2)}M^2\times\R\times\R.
$$
Since the points $z_i$ are nondegenerate, and for the same reason as above, 
for a fixed $U\in\jR_N(\al,\be)$ the projection of the image 
$
T_x\chi_U(T_xX)
$
on the factor $T_{(z_1,z_1)}M^2\times T_{(z_2,z_2)}M^2$ is transverse to the space 
$T_{(z_1,z_1)}\De_M\times T_{(z_2,z_2)}\De_M$, where $\De_M=\{(z,z)\mid z\in M\}$. It only remains to prove that
varying $U$ enables one to control independently the last two terms in the decomposition. 

\vskip2mm

Using standard straightening theorems, one sees that given a point $w\neq z_1$ on the orbit of $z_1$ under $\Phi^{C_U}$,
one can add an arbitrarily small $C^\infty$ bump function (with controlled support around $w$) $\eta$, chosen so that $z_1$ 
is still a $t_\eta$--periodic for $C_{U+\eta}$, with the same energy,  and such that $L_{U+\eta}(z_1,t_\eta)>L_U(z_1,t_1)$.
This way one easily proves that the image 
$
T_x\chi_U(T_xX)
$
contains vectors of the form $(0,0,0,0,0,1)$, according to the previous decomposition.

\vskip2mm

Now the same ideas applied with bump functions $\eta$ centered at the point $z$ enables one to vary the energy of 
the point $z_1$ alone, so that the image 
$
T_x\chi_U(T_xX)
$
also contains vectors of the form $(0,0,0,0,1,u)$. This proves that $\chi$ is transverse to $\De$.

\vskip2mm

Applying Abraham's theorem now proves that the set of potentials $U\in \jR_N(\al,\be)$ such that $(D_7)$ holds true is 
residual, which concludes the first part of the proof.

\vskip2mm

$\bullet$ As for Condition $(D_6)$, we now introduce the open subset 
$$
D(\taum,U)=\Big\{(z_1,z_2,z_3)\in M^3\mid z_i\notin \Phi^{C_U}([0,\taum],z_j),\ i\neq j\Big\}
$$
and the manifolds
$$
X=D(\taum)\times [0,\taum]^3\times (\R^+)^3,\qquad Y=M^2\times M^2\times M^2\times \R^3\times\R^3,
$$
together with the map $\chi:\jR_N(\al,\be)\to C^\ka(X,Y)$ defined by
$$
\chi_U((z_i,z_i),t_i,s_i)=\Big(\big(\Psi^{C_U}(z_i),\Phi^{C_U}(z_i)\big),\big(C_U(z_i)\big),\big(L_U(z_i,t_i)\big)\Big),\qquad i\in\{1,2,3\}.
$$
Exactly as above, one proves that $\chi$ is transverse to the submanifold
$$
\De=\Big\{\big((z_i,z_i),(e,e,e),(\ell,\ell,\ell)\big)\mid i\in\{1,2,3\}, e\in\R,\ell\in\R\Big\}.
$$
Again,
$$
\dim X=3\dim M+6=\codim \De+2
$$
but now,  if $x\in \chi_U\inv(\De)$
$$
\dim \Big(T_x\chi_U(T_xX) \cap T_{\chi_U(x)}\De\Big)\geq 3.
$$
Therefore if $\chi_U\trans_x \De$, then $\chi_U(x)\notin \De$. Therefore if $\chi_U\trans \De$, then there exist at most 
two periodic orbits with the same length on the same energy level. Since $\chi$ is transverse to $\De$, this property
is residual for the potentials in $\jR_N(\al,\be)$. Finally, the transverse crossing condition $(D_7)$ proves that Condition $(D_6)$
holds over a residual subset in $\jR_N(\al,\be)$, so also in $C^\ka(\T^2)$ by countable intersection.
\end{proof}


\paraga  One cannot directly apply Oliveira's work to get the Kupka-Smale theorem in the complete phase space,
due to the presence of homoclinic tangencies.
Nevertheless it is enough to obtain the transversality of heteroclinic intersections beetwen annuli at bifurcation points
and suitable coverings of the various annuli.

\begin{lemma}
{The set $\jU_8$ of $U\in\jU_{6,7}$ such that two distinct minimizing orbits in the same energy level
of $C_U$ admit transverse heteroclinic connections is residual in $C^\infty(\T^2)$. Morever, each annulus
admits a covering by subannuli for which there exist a continuous family of transverse homoclinic orbits 
attached to each periodic orbit.}
\end{lemma}

\begin{proof}
This is an immediate consequence of Oliveira's work \cite{O08} by slightly perturbing the homoclinic tangencies
in order to create new homoclinic orbits in a neighborhood, with transverse homoclinic intersections.
\end{proof}


\subsection{End of proof of the Theorem II} 
We now have to glue the previous result together, taking into account that the required regularity for the nearly critical
energies is $\ka=+\infty$. 


\paraga {\bf Existence of the singular annulus.} Given a potential $U$ in $\jU_4\subset C^\infty(\T^2)$, one can apply the
results of Section 4 to the associated classical system $C_U$. In particular, Proposition 5.4 proves the existence of
a singular $\sA^\bu$ annulus defined over $]-e_0,e_*[$, with $\ov e\in \,]-e_0,e_*[$, which realizes some suitable
homology classes. This is a $C^1$ normally  hyperbolic manifold, to which the usual persistence results under 
$C^1$ perturbations apply. Therefore,  given a small $\de>0$,  for $V$ close enough to $U$ in the $C^\ka$ topology 
with $\ka\geq 2$, the system $C_V$ also admits a singular annulus $\sA^\bu_V$ defined over $]-e_0+\de,e_*-\de[$ and 
realizing the same homology classes as $\sA^\bu$. 
Moreover, for each $c\in\H_1(\T^2,\Z)$, $C_V$ also admits an annulus $\sA_{BC}(c)$, as constructed in 
Lemma~\ref{lem:anasympt}.


\paraga {\bf Chains and heteroclinic connections with the singular annulus.} Fix now $c\in\H_1(\T^2,\Z)$.
Consider a potential $U_0\in C^\ka(\T^2)$,  with $\ka\geq \ka_0$, and fix $c\in\H_1(\T^2,\Z)$. Given $\eps>0$, there exists 
$U\in \jU_4$ with $\norm{U-U_0}_{C^\ka(\T^2)}<\eps$ such that $C_U$ admits a singular annulus as above. 
Moreover, given a small $\de>0$,  there exists a small neighborhood $\jO$ of $U$ in the $C^\ka$ topology 
such that for $V\in\jO$, in addition to the singular annulus described above, the system $C_V$ admits a chain 
$\sA_1,\ldots,\sA_m$ of annuli,  defined over $I_1=\,]\ov e+\de, e_*],\ldots, I_m=[e_\infty,+\infty[$ and realizing $c$, 
such that $\sA_1$ admits transverse heteroclinic connections with $\sA^\bu$. The annuli moreover satisfy the
transverse homoclinic property as well as the twist property.

\vskip1mm 

From this, one easily deduce that, given $N\geq N_0$ and $c\in \H_1(\T^2,\Z)$, the subset of $C^\ka(\T^2)$
of all potentials for which there exists a chain $\sA_1,\ldots,\sA_m$ of annuli, defined over 
$I_1=\,]\ov e+1/N, e_*],\ldots, I_m=[e_\infty,+\infty[$, with a heteroclinic connection between $\sA_1(U)$ and 
$\sA^\bu(U)$, is dense in $C^\ka(\T^2)$. Again, easy persistence results for hyperbolic orbits together with our 
construction of the high-energy annulus prove that this set is open in $C^\ka(\T^2)$. Therefore the set of potentials
for which exists a chain realizing $c$ is residual in $C^\ka(\T^2)$ and our claim follows by countable intersection
over $c$. 

\vskip1mm

The same type of arguments also prove the existence of connections between annuli $\sA_1(c)$ and 
$\sA_1(\sig c')$. 

\vskip1mm

The estimates of 4) in Theorem II are straightforward computations, using the fact that a classical system
at high energy is a perturbation of a flat metric on $\T^2$.

\vskip1mm

Finally, the existence of a chain being an open property, the set of $\jU$ for which the conclusions of Theorem II
hold is open and dense.


\appendix

\section{Normal hyperbolicity and symplectic geometry}\label{app:normhyp}
\setcounter{paraga}{0}

We refer to \cite{Berg10,BB13,C04,C08,HPS} for the references on normal hyperbolicity, see also \cite{Y} for a setting close 
to ours, in the spirit of Fenichel's approach. Here we limit ourselves to a very simple class of systems which
admit a normally hyperbolic invariant (non compact) submanifold, which serves us as a model from which
all other definitions and properties will be deduced.

\paraga  The following statement is a simple version of the persistence theorem for normally 
hyperbolic manifolds well-adapted to our setting, whose germ can be found in \cite{B10} and whose proof
can be deduced from the previous references.

\vskip3mm

\noindent {\bf The normally hyperbolic persistence theorem.} 
{\em Fix $m\geq 1$ and consider a vector field on $\R^{m+2}$ of the form 
$\jV=\jV_0+\jF$,
with $\jV_0$ and $\jF$ of class $C^1$ and reads
\begin{equation}\label{eq:formV0}
\dot x=X(x,u,s),\qquad \dot u=\la_u(x)\, u,\qquad \dot s =-\la_s(x)\,s,
\end{equation}
for $(x,u,s)\in \R^{m+2}$. Assume moreover that there exists $\la>0$ such that
the inequalities
\beq\label{eq:ineg}
\la_u(x)\geq \la,\quad{\it and}\quad \la_s(x)\geq \la,\qquad  x\in\R^m.
\eeq
hold. Fix a constant $\mu>0$.
Then there exists a constant $\de_*>0$ such that if 
\begin{equation}\label{eq:condder}
\norm{\partial _xX}_{C^0(\R^{m+2})}\leq \de_*,\qquad 
\norm{\jF}_{C^1(\R^{m+2})}\leq \de_*,
\end{equation}
the following assertions hold.
\begin{itemize}
\item The maximal invariant set for $\jV$  contained in $O=\big\{(x,u,s)\in\R^{m+2}\mid \norm{(u,s)}\leq \mu\big\}$
is an $m$-dimensional manifold $\Ann(\jV)$ which admits the graph representation:
$$
\Ann(\jV)=\big\{\big(x,u=U(x), s=S(x)\big)\mid x\in\R^m\big\},
$$ 
where $U$ and $S$ are $C^1$ maps $\R^m\to\R$  
such that \begin{equation}\label{eq:loc}
\norm{(U,S)}_{C^0(\R^m)}\leq \frac{2}{\la}\,\norm{\jF}_{C^0}.
\end{equation}
\item The maximal positively invariant set for $\jV$  contained in  $O$ 
is an $(m+1)$-dimensional manifold $W^+\big(\Ann(\jV)\big)$ which admits the graph representation:
$$
W^+\big(\Ann(\jV)\big)=\big\{\big(x,u=U^+(x,s), s\big)\mid x\in\R^m\ s\in[-\mu,\mu]\big\},
$$ 
where $U^+$ is a $C^1$ map $\R^m\times[-1,1]\to\R$  
such that \begin{equation}\label{eq:loc2}
\norm{U^+}_{C^0(\R^m)}\leq c_+\,\norm{\jF}_{C^0}.
\end{equation}
for a suitable $c_+>0$.
Moreover, there exists $C>0$ such that for $w\in W^+\big(\Ann(\jV)\big)$, 
\beq
\dist\big(\Phi^t(w),\Ann(\jV)\big)\leq C\exp(-\la t),\qquad t\geq0.
\eeq
\item The maximal negatively invariant set for $\jV$  contained in $O$
is an $(m+1)$-dimensional manifold $W^-\big(\Ann(\jV)\big)$ which admits the graph representation:
$$
W^-\big(\Ann(\jV)\big)=\big\{\big(x,u, s=S^-(x,u)\big)\mid x\in\R^m,\ u\in[-\mu,\mu]\big\},
$$ 
where $S^-$ is a $C^1$ map $\R^m\times[-1,1]\to\R$  
such that \begin{equation}\label{eq:loc3}
\norm{S^-}_{C^0(\R^m)}\leq c_-\,\norm{\jF}_{C^0}.
\end{equation}
for a suitable $c_->0$.
Moreover, there exists $C>0$ such that for $w\in W^-\big(\Ann(\jV)\big)$, 
\beq
\dist\big(\Phi^t(w),\Ann(\jV)\big)\leq C\exp(\la t),\qquad t\leq0.
\eeq
\item
The manifolds $W^\pm\big(\Ann(\jV)\big)$ admit $C^0$ foliations $\big(W^\pm(x)\big)_{x\in \Ann(\jV)}$ such that 
for $w\in W^\pm(x)$
\beq
\dist\big(\Phi^t(w),\Phi^t(x)\big)\leq C\exp(\pm\la t),\qquad t\geq0.
\eeq
\item If  moreover $\jV_0$ and $\jF$ are of class $C^p$, $p\geq1$, and if in addition of the 
previous conditions the domination inequality, the condition
\beq\label{eq:addcond}
p\,\norm{\partial_x X}_{C^0(\R^m)}\leq \la
\eeq
holds,  then the functions $U$, $S$, $U^+$, $S^-$ are of class $C^p$ and 
\beq
\norm{(U,S)}_{C^p(\R^m)}\leq C_p \norm{\jF}_{C^p(\R^{m+2})}.
\eeq
for a suitable constant $C_p>0$.
\item Assume moreover that the vector fields $\jV_0,\jV$ are $R$-periodic in $x$, where $R$ is a lattice in $\R^m$.
Then their flows and the manifolds $\Ann(\jV)$ and $W^\pm\big(\Ann(\jV)\big)$ pass to the quotient $(\R^m/R)\times \R^2$
Assume that the time-one map of $\jV_0$ on $\R^m/R\times\{0\}$ is $C^0$ bounded by a constant $M$.
Then, with the previous assumptions, the constant $C_p$ depends only on $p$, $\la$ and $M$.
\end{itemize}
}

The last statement will be applied in the case where $m=2\ell$ and $R=c\Z^\ell\times\{0\}$, where $c$ is a positive constant,
so that the quotient $\R^{2\ell}/R$ is diffeomorphic to the annulus $\A^\ell$.

\paraga The following result describes the symplectic geometry of our system in the case where $\jV$
is a Hamiltonian vector field. We keep the notation of the previous theorem.

\vskip3mm

\noindent {\bf The symplectic normally hyperbolic persistence theorem.}  {\it Endow $\R^{2m+2}$ with a symplectic form
$\Om$ such that there exists a constant $C>0$ such that  for all $z\in O$
\beq\label{eq:assumpsymp}
\abs{\Om(v,w)}\leq C\norm{v}\norm{w},\qquad \forall v,w \in T_z M.
\eeq
Let $\jH_0$ be a $C^2$ Hamiltonian on $\R^{2m+2}$ whose Hamiltonian vector field $\jV_0$ satisfies (\ref{eq:formV0})
with conditions (\ref{eq:ineg}), and consider a Hamiltonian $\jH=\jH_0+\jP$.
Then there exists a constant $\de_*>0$ such that if 
\begin{equation}\label{eq:condder2}
\norm{\partial _xX}_{C^0(\R^{m+2})}\leq \de_*,\qquad 
\norm{\jP}_{C^2(\R^{m+2})}\leq \de_*,
\end{equation}
the following properties hold.
\begin{itemize}
\item The manifold $\Ann(\jV)$ is $\Om$-symplectic.
\item The manifolds $W^\pm\big(\Ann(\jV)\big)$ are coisotropic and the center-stable and center-unstable
foliations $\big(W^\pm(x)\big)_{x\in \Ann(\jV)}$ coincide with the characteristic foliations 
of $W^\pm\big(\Ann(\jV)\big)$.
\item If $\jH$ is $C^{p+1}$ and condition (\ref{eq:addcond}) is satisfied, then
$W^\pm\big(\Ann(\jV)\big)$ are of class $C^p$ and the foliations $\big(W^\pm(x)\big)_{x\in \Ann(\jV)}$
are of class $C^{p-1}$.
\item There exists a neighborhood $\jO$ of $\Ann(\jV)$ and a symplectic straightening symplectic diffeomorphism 
$\Psi:\jO\to O$ such that
\beq
\begin{array}{lll}
\Psi\big(\Ann(\jV)\big)=\A^{\ell}\times\{(0,0)\};\\[4pt]
\Psi\big(W^-\big(\Ann(\jV)\big)\big)\subset\A^{\ell}\times\big(\R\times\{0\}\big),\qquad
\Psi\big(W^-\big(\Ann(\jV)\big)\big)\subset\A^{\ell}\times\big(\{0\}\times\R\big);\\[4pt]
\Psi\big(W^-(x)\big)\subset\{\Psi(x)\}\times\big(\R\times\{0\}\big),\qquad
\Psi\big(W^+(x)\big)\subset\{\Psi(x)\}\times\big(\{0\}\times\R\big).\\
\end{array}
\eeq
\end{itemize}
}

\begin{proof} Let $f$ be the time-one flow of $\jV$. By the domination condition,
one can assume $\de_*$ small enough so that there exist two positive 
constants $\chi$ and $\mu$ verifiying $\chi\mu<1$, such that
\beq
\begin{array}{lll}
\forall z\in W^+\big(\Ann(\jV)\big)\cap O,&\forall  v\in T_zW^+\big(\Ann(\jV)\big), \  \norm{T_zf(v)}\leq \mu \norm{v},\spa\\
&\forall v\in T_z W^+(\Pi^+(z)),\    \norm{T_zf(v)}\leq \chi\norm{v},\spa
\end{array}
\eeq
\beq
\begin{array}{lll}
\forall z\in W^-\big(\Ann(\jV)\big)\cap O,&\forall v\in T_z W^-\big(\Ann(\jV)\big),\    \norm{T_zf\inv(v)}\leq \mu\norm{v},\spa\\
&\forall v\in T_z W^-(\Pi^-(z)),\    \norm{T_zf\inv(v)}\leq \chi\norm{v}.\spa\\
\end{array}
\eeq
Fix $z\in W^+\big(\Ann(\jV)\big)\cap O$.
Then if $v\in T_z\big(W^+\big(\Ann(\jV)\big)\big)$ and $w\in T_z\big(W^+(\Pi^+(z))\big)$, since $f$ is symplectic
$$
\abs{\Om(v,w)}=\abs{\Om\big(T_zf^m(v),T_zf^m(w)\big)}\leq C\norm{T_zf^m(v)}\,\norm{T_zf^m(w)}
\leq C\,(\chi\mu)^m\norm{v}\,\norm{w}
$$
for $m\in O$, so passing to the limit shows that $\Om(v,w)=0$. Therefore
\beq\label{eq:inclu}
T_z\big(W^+(\Pi^+(z),f)\big)\subset\big(T_z(W^+\big(\Ann(\jV)\big))\big)^{\bot_{\Om}}.
\eeq
This  proves in particular that the manifolds $W^+(x)$, $x\in V$,  are  isotropic. One obviously gets a similar result for $W^-(x)$.

\vskip2mm

-- Let $n^0$ be the dimension of $V$, so that by normal hyperbolicity $2n=n_0+n_++n_-$ and $\dim W^\pm\big(\Ann(\jV)\big)=n_0+n_\pm$.
Since $\Om$ is nondegenerate:
$$
\dim \big(T_z(W^+\big(\Ann(\jV)\big))\big)^{\bot_{\Om}}=2n-(n^0+n_+)=n_-,
$$
and by (\ref{eq:inclu}) $n_+\leq n_-$. By symmetry one gets the equality $n_+=n_-$.

\vskip2mm

-- Moreover, this equality proves that 
$$
T_z\big(W^+(\Pi^+(z),f)\big)\subset\big(T_z(W^+\big(\Ann(\jV)\big))\big)^{\bot_{\Om}},
$$
and $W^+\big(\Ann(\jV)\big)$ is coisotropic. This also proves that its characteristic foliation is the family $(W^+(x))_{x\in \Ann(\jV)}$.
The analogous statements for the unstable manifolds are obvious. 

\vskip2mm 

-- The manifold $\Ann(\jV)$ is therefore the quotient of the manifolds $W^\pm\big(\Ann(\jV)\big)$ by its characteristic foliation,
this immediately implies that the projections $\Pi^\pm$ are symplectic sumbmersion, that is 
$$
(\Pi^\pm)^*\Om_{\vert \Ann(\jV)}=\Om_{\vert W^\pm\big(\Ann(\jV)\big)}.
$$

\vskip2mm

-- Finally, the manifold $\Ann(\jV)$ is the intersection of the two coisotropic manifolds 
$W^\pm\big(\Ann(\jV)\big)$, and at each point $x\in \Ann(\jV)$ 
$$
(T_x\Ann(\jV))^{\bot_\Om}=\big(T_x(W^+\big(\Ann(\jV)\big))\big)^{\bot_{\Om}}+\big(T_x(W^-\big(\Ann(\jV)\big))\big)^{\bot_{\Om}}
=E_x^++E_x^-
$$
so by definition 
$(T_x\Ann(\jV))^{\bot_\Om}\cap T_x\Ann(\jV)=\{0\}$ and $\Ann(\jV)$ is symplectic. 

\vskip2mm 

-- The last statement is a direct consequence of the symplectic tubular neighborhood theorem and the Moser
isotopy argument (see \cite{LM} for more details).
\end{proof}

\section{Global normal forms along arcs of simple resonances}\label{App:globnormforms}
We consider perturbed systems of the form $H_\eps(\th,r)=h(r)+\eps f(\th,r)$, where $h$ is a $C^\ka$ Tonelli 
Hamiltonian on $\R^3$ and $f$ an element of the unit ball $\jB^\ka$ of $C^\ka(\A^3)$. We fix a simple 
resonance $\Ga$ at energy $\e>\Min h$ for $h$ and assume the coordinates $(\th,r)$ to be adapted to 
$\Ga$, that is,  $\Ga=\{r\in h\inv(\e)\mid \om_3(r)=0\}$. Relatively to these new coordinates, 
$\norm{f}_{C^\ka}\leq M$. We split the variables $x=(x_1,x_2,x_3)$ into the 
fast part $\ha x=(x_1,x_2)$ and the slow part $\ov x= x_3$.


\subsection{The global normal form}

Given a subset $\Ga_\rho$ of $\Ga$, for $\rho>0$, we introduce the tubular neighborhood
\begin{equation}\label{eq:tube}
\jW_{\rho}(\Ga_\rho)=\T^3\times \{r\in\R^3\mid \dist(r,\Ga_\rho)<\rho\}.
\end{equation}
Recall that we say that a connected subset of $\Ga$ is an interval. Given a control parameter $\de>0$,
we denote by $D(\de)$ the set of $\de$-strong double resonance points of $\Ga$, introduced in Definition~\ref{def:control}.

\begin{prop}\label{prop:globnorm}   
Fix an integer $p\in\{2,\ldots,\ka-4\}$ and a control parameter $\de>0$.
Fix two consecutive points $r'$ and $r''$ in $D(\de)$, fix $\rho<\!<\dist_{\Ga}(r',r'')$ and set 
$$
\Ga_\rho:=[r^*,r^{**}]_\Ga\subset [r',r'']_\Ga,
$$
where $r^*,r^{**}$ are defined by the equalities
$$
\dist_\Ga(r^*,r')=\dist_\Ga(r^{**},r'')=\rho.
$$
Then there exists $c\in\,]0,1[$ such that for $0<\eps<c\rho^4$ there exists a symplectic analytic embedding 
$\Phi_{\eps}: \jW_{c\rho}\to\jW_\rho$  which satisfies
\begin{equation}\label{eq:normform}
N(\th,r)=H\circ \Phi_{\eps}(\th,r)=h(r)+\eps V(\th_3,r)+\eps W_0(\th,r)+\eps W_1(\th,r)+\eps^2 W_2(\th,r),
\end{equation}
where 
\begin{equation}
V(\th_3,r)=\int_{\T^2}f(\th,r)\,d\th_1d\th_2,
\end{equation}
and where the functions $W_0\in C^p(\A^3)$, $W_1\in C^{\ka-1}(\jW_{c\rho})$, $W_2\in C^\ka(\jW_{c\rho})$ satisfy
\begin{equation}\label{eq:estimates}
\begin{array}{lll}
\norm{W_0}_{C^p(\jW_{c\rho})}\leq \de,\\[5pt]
\norm{W_1}_{C^2(\jW_{c\rho})}\leq c_1\, \rho^{-3} \\[5pt]
\norm{W_2}_{C^2(\jW_{c\rho})}\leq c_2\,\rho^{-6},
\end{array}
\end{equation}
for suitable constants $c_1,c_2>0$.
Moreover, there exists $c_\Phi>0$
such that, if $\Phi_\eps=(\Phi_\eps^{\th},\Phi_\eps^{r})$,
\begin{equation}\label{eq:approxphi}
\norm{\Phi_\eps^{\th}-\Id}_{C^0(\jW_{c\rho})}\leq c_\Phi\,\eps\,\rho^{-2},\qquad  \norm{\Phi_\eps^{r}-\Id}_{C^0(\jW_{c\rho})}\leq c_\Phi\,\eps\,\rho^{-1}.
\end{equation}
The constants $c,c_1,c_2,c_\Phi$ do not depend on $\rho$ and $\eps$.
\end{prop}


\subsection{Proof of Proposition~\ref{prop:globnorm}}
Let $p\in\{2,\ldots,\ka-4\}$ and  $\de>0$  be fixed, and let $K(\de)$ be as in Lemma~\ref{lem:choseK}.


\paraga We begin with a geometric lemma which enables us to control the size of the small denominators which appear
in the averaging process.

\begin{lemma}\label{lem:estdist} With the notation of {\rm Proposition~\ref{prop:globnorm}},
given $\rho>0$, set
\beq
\jU_{\rho}(\Ga^*)=\{r\in\R^3\mid \dist(r,\Ga_\rho)<\rho\}.
\eeq
Then there exist constants $c_0,C_0>0$  such that for every $r\in\jU_{c_0\rho}(\Ga^*)$
\beq\label{eq:controlgeom}
\Min_{\ha k\in B^*\big( K(\de)\big)}\abs{\ha k\cdot \ha \om(r)}\geq C_0\rho,
\eeq
where $B^*\big( K(\de)\big)=\big\{\ha k\in\Z^2\setm\{0\}\mid \norm{k}\leq K(\de)\big\}$ and where $K(\de)$ was
defined in {\rm Lemma~\ref{lem:choseK}}.
\end{lemma}

\begin{proof} Choose $\ha k',\ha k''\in\bez\big(K(\de)\big)$ {\em with minimal norm}
such that  $\ha k'\cdot \ha \om(r')=0$ and $\ha k''\cdot \ha \om(r'')=0$. 
As a consequence, if $\ha k\in\Z^2$ satisfies $\ha k\cdot \ha \om(r')=0$ or $\ha k\cdot \ha \om(r'')=0$, 
then $\ha k\in\Z\,\ha k'\cup \Z\,\ha k''$.
\vskip1mm
The resonance surfaces $\om\inv\big((\ha k',0)\big)$ and $\om\inv\big((\ha k'',0)\big)$ are transverse to 
$\Ga$ at $r'$ and $r''$ respectively (in $\R^3$). As a consequence, there exist constants $c_0,C_0>0$
such that
\beq
\forall r\in \jU_{c_0\rho},\qquad \abs{\om(r)\cdot \ha k'}\geq C_0\rho,\quad \abs{\om(r)\cdot \ha k''}\geq C_0\rho.
\eeq
Hence the previous inequalities also hold for the vectors $\ha k\in \Z\,\ha k'\cup \Z\,\ha k''$.
Now if 
$$
\ha k\in B^*\big( K(\de)\big)\setm (\Z\,\ha k'\cup \Z\,\ha k''),
$$
 $\ha\om(r)\cdot \ha k\neq 0$ for $r\in[r',r'']_\Ga$.
As a consequence, reducing $c_0$ and $C_0$ if necessary, (\ref{eq:controlgeom}) holds true.
\end{proof}


\paraga {\bf Averaging and proof of Proposition~\ref{prop:globnorm}.}
Recall that 
\begin{equation}
f(\th,r)=\sum_{\ha k\in\Z^2}\phi_{\ha k}(\th_3,r)e^{2i\pi\,\ha k\cdot \ha\th}\quad \textrm{with}\quad
\phi_{\ha k}(\th_3,r)=\sum_{k_3\in\Z}[f]_{(\ha k,k_3)}(r)e^{2i\pi\,k_3\cdot \th_3}
\end{equation}
and
\begin{equation}
f_{> K}(\th,r)=\sum_{\ha k\in\Z^2,\norm{\ha k}> K}\phi_{\ha k}(\th_3,r)\,e^{2i\pi\, k_3\cdot \th_3}.
\end{equation}
We use the classical  Lie transform method to produce a diffeomorphism which cancels the harmonics 
$\phi_{\ha k}$  for $1\leq \norm{\ha k}\leq K:=K(\de)$. 

\vskip2mm

$\bullet$ We first solve the homological equation
\begin{equation}\label{eq:homol}
\ha\om(r)\cdot\partial_{\ha\th} S (\th,r)=f(\th,r)-\phi_0(\th_3,r)-f_{> K}(\th,r).
\end{equation}
Up to constants, the solution of~(\ref{eq:homol}) reads
\begin{equation}\label{eq:genfonct}
S(\th,r)=\sum_{\ha k\in\Z^2\setm\{0\},\norm{\ha k}\leq K} 
\Frac{\phi_{\ha k}(\th_3,r)}{2i\pi\,\ha k\cdot\ha \om(r)}e^{2i\pi\,\ha k\cdot\ha \th}.
\end{equation}
By Lemma~\ref{lem:estdist}, it is therefore well-defined an analytic in the domain $\jW_{c_0\rho}$, provided 
that $c_0>0$ is small enough. Moreover, by direct computation
for $i,j$ in $\N^3$ and $0\leq \abs{i},\abs{j}\leq \ell$ :
\begin{equation}\label{eq:estimS}
\norm{\partial ^i_\th\partial ^j_rS}_{C^0(\jW_{c\rho})}\leq \ov c(\ell)\rho^{-(1+\abs{j})}
\end{equation}
for a constant $\ov c(\ell)>0$.

\vskip2mm

$\bullet$ We now consider the time-one diffeomorphism $\Phi_\eps:=\Phi^{\eps S}$ of the Hamiltonian flow generated
by the function $\eps S$, defined on the set $\jW_{c\rho}$ with $c<c_0$.
The Taylor expansion at order 2 of the transformed Hamiltonian $H_\eps$ reads
\begin{equation}
H_\eps\circ \Phi^{\eps S}(\th,r)=H(\th,r)+\eps\{H,S\}(\th,r)+
\eps ^2\int_0^1(1-\sig)\big\{\{H,S\},S\big\}\big(\Phi^{\sig\,\eps S}(\th,r)\big)\,d\sig,
\end{equation}
with Poisson bracket  $\{u,v\}=\partial_\th u\partial_r v-\partial_r u\partial_\th v$. The
new Hamiltonian reads
$$
H\circ \Phi^{\eps S}(\th,r)=h(r)+\eps V(\th_3,r)+\eps W_0(\th,r)+\eps W_1(\th,r)+\eps^2 W_2(\th,r),
$$
where
\begin{equation}\label{eq:expform}
\begin{array}{lll}
V(\th_3,r)&\!\!=&\!\!\dst\phi_0(\th_3,r)\,=\,\int_{\T^2} f(\th,r)\,d\th_1d\th_2,\\[5pt]
W_0(\th,r)&\!\!=&\!\!\dst f_{> K}(\th,r),\\[5pt]
W_1(\th,r)&\!\!=&\om_3(r)\d_{\th_3}S(\th,r),\\
W_2(\th,r)&\!\!=&\!\!\dst \{f,S\}(\th,r)+\int_0^1(1-\sig)\big\{\{H,S\},S\big\}\big(\Phi^{\sig\,\eps S}(\th,r)\big)\,d\sig,\\
\end{array}
\end{equation}
which proves (\ref{eq:normform}), together with the estimate on $W_0$ in (\ref{eq:estimates}) by Lemma~\ref{lem:choseK}.

\vskip2mm

$\bullet$ It remains to estimate the size of the various functions. 
To prove (\ref{eq:approxphi}), we use the same method as in \cite{Bou10}, based on
(\cite{DH09}, Lemma 3.15). We introduce the weighted norm on $\R^3\times\R^3=T_{(\th,r)}\A^3$:
$$
\abs{(u_\th,u_r)}=\Max\big(\rho\norm{u_\th},\norm{u_r}\big)
$$
for which, by (\ref{eq:estimS}):
$$
\abs{X^{\eps S}}_{C^0(\jW_{c_0\rho})}\leq \ov c(1)\eps.
$$
Therefore, provided that $\eps\rho^{-1}$ is small enough to ensure that $\Phi^{\eps S}(\jW_{c\rho})\subset \jW_{c_0\rho}$,   
there exists $c_\Phi>0$ such that
$$
\norm{\Phi_\eps^{\th}-\Id}_{C^0(\jW_{c\rho})}\leq c_\Phi\,\eps\,\rho^{-2},\qquad 
\norm{\Phi_\eps^{r}-\Id}_{C^0(\jW_{c\rho})}\leq c_\Phi\,\eps\,\rho^{-1},
$$
which proves (\ref{eq:approxphi}). Finally, for the same reason and provided that $\eps\rho^{-4}$ is small enough
$$
\norm{\Phi^{\sig\,\eps S}}_{C^2(\jW_{c\rho})}\leq \norm{\Id}_{C^2(\jW_{c\rho})}+\norm{\Phi^{\sig\,\eps S}-\Id}_{C^2(\jW_{c\rho})}
\leq 2,
$$
Note also that 
$$
\norm{\{f,S\}}_{C^2(\jW_\rho)}\leq c_*\,\rho^{-4},\qquad \norm{\big\{\{H,S\},S\big\}}_{C^2(\jW_\rho)}\leq c_{**}\,\rho^{-6}.
$$
for some $c_*,c_{**}>0$.
Therefore,  using the Faa di Bruno formula as  in \cite{Bou10} to estimate the second term of $W_2$,
one immediately gets the last estimate in (\ref{eq:estimates}).


\section{Normal forms over $\eps$--dependent domains}\label{app:normformepsdep}

As usual $\T^n=\R^n/\Z^n$ and $\A^n=T^*\T^n$.
In this section we will exceptionally work with Hamiltonian systems on $\A^n$, $n\geq2$, since our result is 
in fact easier to state in its full generality. Moreover, we will no longer assume any convexity or superlinearity
condition for the unperturbed part $h$. We will construct normal forms for perturbed Hamiltonians 
$H_\eps=h+\eps f$ of class $C^\ka$ on $\A^n$, 
in $\eps$--dependent neighborhoods of partially resonant and partially Diophantine actions. 

\subsection{Setting and main result}
\setcounter{paraga}{0}

For $2\leq p\leq \infty$, the $L^p$ norm 
on $\R^n$ or $\C^n$ will be denoted by $\norm{\cdot}_p$, 
while we will write $\abs{\,\cdot\,}$ when
$p=1$.

\paraga  Given $\tau>0$ and a submodule $\cM$ of $\Z^n$ of rank $m\geq0$, we say that a vector $\om\in\Z^n$ is 
{\em $\cM$--resonant and $\tau$--Diophantine} if 
$\om^\bot\cap\Z^n=\cM$ and for any submodule $\cM'$ of $\Z^n$ such that $\cM\oplus\cM'=\Z^n$, 
there exists a constant $\ga>0$ (depending on $\cM'$) such that
\beq\label{eq:inegdioph}
\forall k\in \cM'\setm\{0\}, \qquad \abs{\om\cdot k}\geq \frac{\ga}{\abs{k}^\tau}.
\eeq
Clearly, (\ref{eq:inegdioph}) is satisfied for any complementary submodule $\cM'$ if and only if it is satisfied for a 
single one. We say that $\om$ is {\em $m$--resonant and $\tau$--Diophantine} when there exists a rank $m$ submodule $\cM$
such that $\om$ is $\cM$--resonant and $\tau$--Diophantine.
The set of $m$--resonant and $\tau$--Diophantine vectors has full measure as soon as
$\tau>n-m-1$ (and is residual when $\tau=n-m-1$). 
When $m=0$ one recovers the usual Diophantine case, and we will assume $m\geq1$ in the following.
The case $m=n-1$ is particular since 
(\ref{eq:inegdioph}) is trivially satisfied for any nonzero $(n-1)$--resonant vector (for a suitable $\ga$) as soon as $\tau\geq0$. 
In the following we will not make an explicit distinction between the case $m=n-1$ and the case $1\leq m\leq n-2$, even thought
the proofs are slightly different.

\paraga 
Recall that given a submodule $\cM$ of $\Z^n$ of rank $m$, 
there exists  a $\Z$--basis of $\Z^n$ whose last $m$ vectors form a $\Z$--basis of $\cM$.
Given the  matrix $P$ in ${\bf Gl}_n(\Z)$ whose $i^{th}$-column is formed by the components of the $i^{th}$-vector of
this basis, one defines a symplectic linear coordinate change in $\A^n$ by  setting
\begin{equation}\label{eq:adcoord2}
\th=\,^tP\inv \til\th\ \  [{\rm mod}\ \Z^n],\qquad  r=P \,\til r.
\end{equation}

\paraga Let $h$ be an integrable Hamiltonian on $\R^n$ fix  $r^0$ such that
$\big((\nabla h)(r^0)\big)^\bot\cap\Z^n=\cM$. The change (\ref{eq:adcoord2})
transforms $h$ into a new Hamiltonian $\til h$ such that the last $m$ coordinates of the frequency vector 
$\nabla\til h(\til r^0)$ vanish, while the first $n-m$ ones are nonresonant. Such coordinates $(\til\th,\til r)$ will
be said {\em adapted to $\til r^0$}.
We say that $r^0$ is $m$--resonant and $\tau$--Diophantine for $h$ 
when its associated frequency vector $\nabla h(r^0)$ is. One easily checks that this is the case
 if and only if there exists adapted coordinates of the form (\ref{eq:adcoord2}), 
relatively to which the frequency vector satisfies
$$
(\ha\om,0)\in\R^{n-m}\times\R^m
$$
where the vector $\ha\om$ is $\tau$--Diophantine in the usual sense. Once such adapted coordinates
are chosen, we accordingly split all variables $x$ into $(\ha x,\ov x)$, where $\ha x$ stands for the first  $n-m$  components
of $x$ and $\ov x$ stands for the last $m$ ones.

\paraga We can now state our result. We fix $n\geq 3$ and $1\leq m \leq n-1$. We define the $C^p$ norm of a function on a 
fixed domain as the upper bound of the partial derivatives of order $\leq k$ on the domain.

\begin{prop}\label{prop:normal2} 
Consider an unperturbed Hamiltonian $h$ of class $C^\ka$ on $\R^n$, fix a perturbation $f$ in the unit ball of $C^\ka(\A^n)$ and  set
as usual $H_\eps=h+\eps f$. Fix two integers $p,\ell\geq 2$ and two constants $d>0$ and $\de<1$ with $1-\de>d$.
Fix an  $m$--resonant and  $\tau$--Diophantine action $r^0$ for $h$ and assume the coordinates $(\th,r)$ to be adapted to $r^0$.
Set
$$
[f](\ov\th,r)=\int_{\T^{n-m}}f\big((\ha\th,\ov\th),r\big)\,d\ha\th.
$$
Then, if $\ka$ is large enough, there is an $\eps_0>0$ such that for $0<\eps<\eps_0$, there exists an analytic symplectic embedding 
$$
\Phi_\eps: \T^n\times B(r^0,\eps^d)\to \T^n\times B(r^0,2\eps^d)
$$
such that 
$$
H_\eps\circ\Phi_\eps(\th,r)=h(r)+  g_\eps(\ov \th,r)+R_\eps(\th,r),
$$
where $g_\eps$ and $R_\eps$ are $C^p$ functions such that 
\beq\label{eq:estimnormform}
\norm{g_\eps-\eps[f]}_{C^p\big( \T^{n-m}\times B(r^0,\eps^d)\big)}\leq \eps^{2-\de},\qquad
\norm{R_\eps}_{C^p\big( \T^n\times B(r^0,\eps^d)\big)}\leq \eps^\ell.
\eeq
Moreover, $\Phi_\eps$ is close to the identity, in the sense that
\beq\label{eq:estimphi}
\norm{\Phi_\eps-\Id}_{C^p\big( \T^n\times B(r^0,\eps^d)\big)}\leq \eps^{1-\de}.
\eeq
\end{prop}


\subsection{Proof of Proposition \ref{prop:normal2}}

Proposition \ref{prop:normal2} will be an easy consequence of the resonant normal forms for analytic systems
derived in  \cite{Po93}, together with  classical analytic smoothing results for which we refer for instance to \cite{Ze76}.
Another and more direct technique was introduced in
\cite{Bou10} to prove Nekoroshev-type results in the finitely differentiable case. For the sake of simplicity we adopt here
the convention of  \cite{Po93} and set $\T^n=\R^n/(2\pi\Z^n)$, one immediately recovers our usual setting by a linear
change of variables which will not affect the estimates in Proposition \ref{prop:normal2}.


\subsubsection{P\"oschel's normal form}

\paraga Given a subset $D\subset \R^n$, for any function $u: \T^n\times  D\to \C$ such that 
$u(\cdot,r)\in L^1(\T^n)$ for $r\in D$, we write
$$
[u]_k(r)=\int_{\T^n}u(\th,r)e^{ik\cdot \th}\,d\th
$$
for the Fourier coefficient of order $k\in\Z^n$.

\paraga Given $\sig>0$, $\rho>0$, we set 
$$
U_\sig\T^n=\{\th\in\C^n\mid \abs{\Im \th} <\sig\},\qquad  V_\rho D=\{r\in\C^n\mid \dist (r,D) < \rho\},
$$
where $\dist$ is the metric associated with $\norm{\cdot}_2$.
As in \cite{Po93}, for $u$ analytic in  $U_{\sig'}\T^n\times V_\rho D$ with $\sig'>\sig$, with Fourier expansion
$$
u(\th,r)=\sum_{k\in\Z^n} [u]_k(r)\,e^{ik\cdot\th}
$$
we set 
$$
\norm{u}_{D,\sig,\rho}=\Sup_{r\in V_\rho D}\sum_{k\in\Z^n}\abs{u_k(r)}\,e^{\abs{k}\sig}<+\infty.
$$
One easily gets the following inequalities
\beq\label{eq:inegnorm}
\norm{u}_{C^0(U_{\sig}\T^n\times V_\rho D)}\leq \norm{u}_{D,\sig,\rho}\leq ({\rm coth}^na) \norm{u}_{C^0(U_{\sig+a}\T^n\times V_\rho D)},
\eeq
for $0< a <\sig'-\sig$.

\paraga In this section we consider a nearly integrable Hamiltonian of the form 
$$
\sH_\eps(\th,r)=\sh(r)+\sf_\eps(\th,r)
$$
where $\sh$ and $\sf_\eps$ are analytic on the complex domain $U_{\sig_0}\T^n\times V_{\rho_0} P$, 
where $\sig_0>0$, $\rho_0>0$ are fixed and where $P$ is some
domain in $\R^n$.
We denote by $\varpi$ the frequency map associated with $\sh$.

\paraga Let $\al$ be a (small) constant and $K$ be a (large) constant, which will eventually depend on the parameter $\eps$.
Fix a submodule $\cM$ of rank $m$ of $\Z^n$. 
Following \cite{N77}, we say that a domain $D^*$ in the frequency space $\R^n$ is {\em $(\al,K)$--nonresonant modulo $\cM$}
when for all $\om\in D^*$,
$$
\abs{\om\cdot k}\geq\al\  \textrm{for\ all}\ k\in\Z^n\setm\cM\ \textrm{such\ that}\  \abs{k}\leq K.
$$
We then say that a domain $D$ in the action space is $(\al,K)$--nonresonant modulo $\cM$ for the unperturbed Hamiltonian $\sh$
when $\varpi(D)$ is.

\paraga We assume now that $\cM=\{0\}\times\Z^m$  and we use the corresponding
decomposition $x=(\ha x,\ov x)$ for the variables. 
The main ingredient of our proof is the following result by P\"oschel.

\vskip2mm

\noindent{\bf Theorem \cite{Po93}.} {\em Let $D\subset P$ be a domain which is 
$(\al,K)$--nonresonant modulo $\cM$ for~$\sh$.  Let
$$
\mu(\eps):=\norm{f_\eps}_{D,\sig_0,\rho_0}.
$$
Let
$$
Z_0(\ov\th,r)
=\!\!
\sum_{\ov k\in\Z^m,\norm{\ov k}\leq K} [f_\eps]_{(0,\ov k)}(r)\,e^{2i\pi \ov k\cdot\ov \th}.
$$
Then there are positive constants $c,c',c''$ depending only on the $C^2$ norm of $\sh$ such that 
for any triple $(\mu,\sig,\rho)$ which satisfies 
\beq\label{const1}
0\leq \mu\leq c \,\frac{\al}{K}\, \rho,\qquad \rho\leq \min\big(c'\,\frac{\al}{K},\rho_0\big),\qquad \frac{6}{K}\leq \sig\leq\sig_0,
\eeq
then, when $\mu(\eps)\leq\mu$, there exists a symplectic embedding
$$
\Phi_\eps : U_{\rho_*}\T^n\times V_{\sig_*} D\to U_{\rho}\T^n\times V_{\sig}D,
$$
where $\rho_*=\rho/2$ and $\sig_*=\sig/6$
such that
$$
\sH_\eps\circ\Phi_\eps(\th,r)=\sh(r)+Z_\eps(\ov\th,r)+M_\eps(\th,r)
$$
where 
\beq\label{maj1}
\norm{Z_\eps-Z_0}_{D,\sig_*,\rho_*}\leq c''\frac{K}{\al\rho}\mu^2,\qquad \norm{M_\eps}_{D,\sig_*,\rho_*}\leq  e^{- K\sig/6}\mu. 
\eeq
Moreover, the $\th$--component $\Phi_\eps^\th$ and the $r$--component $\Phi_\eps^r$ of $\Phi_\eps$ are close to the identity, 
in the sense that
\beq\label{maj2}
\begin{array}{lll}
&\displaystyle\phantom{\int^{\int}}
\norm{\Phi_\eps^\th(\th,r)-\th}\leq c''\frac{K}{\al}\frac{\sig}{\rho}\mu,\qquad \forall (\th,r)\in U_{\sig_*}\T^n\times V_{\rho_*} D,\\ 
&\displaystyle\phantom{\int^{\int^\int}}
\norm{\Phi_\eps^r(\th,r)-r}\leq c''\frac{K}{\al}\mu,\qquad \forall (\th,r)\in U_{\sig_*}\T^n\times V_{\rho_*} D,\\
\end{array}
\eeq
(here $\norm{\cdot}$ stands for an arbitrary norm over $\C^n$ and the first equality is to be understood on the lift of $U_{\sig_*}\T^n$).
}


\subsubsection{Proof of Proposition \ref{prop:normal2}}
We keep the notation of Proposition \ref{prop:normal2}, in particular $h$ is a $C^\ka$ Hamiltonian on $\R^n$, and we assume 
$\ka\geq 2$, so that the frequency map $\om=\nabla h$ is at least $C^1$.

\paraga The proof will rely on the following easy result. 

\begin{lemma}\label{lem:nonres}  Consider an $m$--resonant and $\tau$--Diophantine action $r^0$ for $h$ 
and assume the coordinates $(\th,r)$
to be adapted to $r^0$. Therefore
$\om(r^0)=(\ha \om,0)$
with $
\ha\om\in\R^{n-m}$
such that 
\begin{itemize}
\item $\ha\om\neq 0$ in the case $m=n-1$, 
\item $\displaystyle\abs{\ha k\cdot \ha\om}\geq \frac{\ga}{\gabs{\ha k}^\tau},\quad \forall \ha k\in\Z^{n-m}\setm\{0\}$
for some $\ga>0$ when $1\leq m\leq n-2$.  
\end{itemize}
Let $\cM:=\{0\}\times\Z^m$. Then the following properties hold true.

\begin{itemize}
\item If $m=n-1$, let $\al=\abs{\ha\om}/2$. Then there is a constant $\la>0$ such that for $K>0$ large enough,
the ball $B(r^0,\la/K)$ is $(\al,K)$ nonresonant modulo $\cM$.

\item If $1\leq m\leq n-2$, let $\nu>1+\tau$. Then,  for $K$ large enough,
the ball $B(r^0,K^{-\nu})$ is $(\al,K)$ nonresonant modulo $\cM$, with  
$$
\al =\frac{\ga}{2}\,K^{-\tau}.
$$ 
\end{itemize}
\end{lemma}

\begin{proof} Assume first that $m=n-1$. Then there is a  $\la>0$ such that for $K$ large enough,
$\norm{\om(r)-\om(r^0)}_\infty\leq \al/K$ when $\norm{r-r^0}\leq \la/K$ and the result easily follows from the inequality
$$
\abs{\om(r)\cdot  k}=\gabs{\om(r^0)\cdot k+\big(\om(r)-\om(r^0)\big)\cdot k}\geq 
\gabs{\ha\om\,\ha k}-K \frac{\al}{K}\geq 2\al-\al=\al,
$$
if $\norm{r-r^0}\leq\la/K$, $k=(\ha k,\ov k)\notin\cM$ (and so $\gabs{\ha k}\geq 1$) and $\abs{k}\leq K$.

\vskip2mm 

Assume now that $1\leq m\leq n-2$ and
observe that for $k=(\ha k,\ov k)\notin \cM$ with $\abs{k}\leq K$, then $\ha k\neq 0$ and $\gabs{\ha k}\leq K$,
so that 
$$
\abs{\om(r^0)\cdot  k}=\gabs{\ha\om\cdot \ha k}\geq \frac{\ga}{K^\tau}.
$$
Moreover,  there exists $C>0$ such that for $K$ large enough, 
$\norm{\om(r)-\om(r^0)}_\infty\leq  C\, K^{-\nu}$ when $\norm{r-r^0}\leq K^{-\nu}$.
Therefore, if $k\notin \cM$ and $\norm{r-r^0}\leq K^{-\nu}$
$$
\abs{\om(r)\cdot  k}=\gabs{\om(r^0)\cdot k+\big(\om(r)-\om(r^0)\big)\cdot k}\geq \frac{\ga}{K^\tau}-C\, K^{-\nu} K
$$
and the result easily follows since $\nu-1> \tau$.
\end{proof}

\paraga Let us now recall the following analytic smoothing result. 

\vskip2mm

\noindent{\bf Theorem \cite{Ze76}.} {\em Let $\ka$ be a fixed nonnegative integer, let $r^0\in\R^n$ and for $R>0$ set
$A_R:=\T^n\times \ov B^n(r^0,R)$. Fix $R>0$.
Then there are constants $s_0>0, c_0>0$ such that for $0< s < s_0$,
for any function $f\in C^\ka(A_{2R},\R)$, there exists a function $\ell_s(f)$, 
analytic in $\jB_s=U_s\T^n\times V_s B^n(r^0,R)$, such that $\big(\ell_s(f)\big)(A_R)\subset \R$ and
\beq\label{eq:smooth1}
\norm{\ell_s(f)-f}_{C^p(A_R)}\leq c_0\, s^{\ka-p}\norm{f}_{C^\ka(A_R)},\qquad 0\leq p\leq\ka;
\eeq
\beq\label{eq:smooth2}
\gabs{\ell_s(f)}_{C^0(\jB_s)}\leq c_0\,\norm{f}_{C^\ka(A_R)}.
\eeq
Moreover, the map $f\mapsto \ell_s(f)$ is linear and $\ell_s(f)(\th,r)$ is independent of $\th$ when $f(\th,r)$ is.}

\paraga We are given an $m$--resonant and $\tau$--Diophantine action $r^0$ for $h$.
We will apply P\"oschel's theorem to the  analytic Hamiltonian
$$
\sH_\eps=\ell_{2\sig(\eps)}(H_\eps)=\ell_{2\sig(\eps)}(h)+\eps\ell_{2\sig(\eps)}(f)=\sh_\eps+\sf_\eps,
$$
where the smoothing operator $\ell$ is defined relatively to the domain $A_1=\T^n\times\ov B^n(r^0,1)$,
and where $\sig(\eps)\to0$ when $\eps\to0$ (see below the explicit form of $\sig$).
Note first that the setting is slightly different from that of P\"oschel, since the unperturbed Hamiltonian
$\sh_\eps$ depends on $\eps$. To control this dependence, we will chose $\ka$ so that by (\ref{eq:smooth1})
the frequency vector $\nabla \sh_\eps$ is close enough to $\nabla h$, which will allow us to use 
Lemma~\ref{lem:nonres} to obtain nonresonant domains for $\sh_\eps$. Moreover, which is crucial,
the $C^2$ norm of $\sh_\eps$ is bounded independently of $\eps$ thanks to (\ref{eq:smooth1}), so that 
P\"oschel's theorem can be applied to $\sH_\eps$ with {\em uniform} constants $c,c',c''$. 

\paraga Our domain will have the following form:
$$
U_{2\sig(\eps)}\T^n\times V_{2\sig(\eps)}B(r^0,\eps^d)
$$
where the exponent $d$ will be chosen below in order for $B(r^0,\eps^d)$ to be $(\al(\eps),K(\eps))$--nonresonant 
modulo $\cM$ for $\sh_\eps$.
The regularity $\ka$ will be chosen according to (\ref{eq:smooth1}), in order to satisfy a number
of constraints.

\paraga The main point is that, by (\ref{eq:inegnorm}), 
$$
\norm{\sf_\eps}_{B(r^0,\eps^d),\sig(\eps),\sig(\eps)}
\leq \eps\, ({\rm coth}^n\sig(\eps))\abs{\ell_{2\sig(\eps)}(f)}_{C^0\big(U_{2\sig(\eps)}\T^n\times V_{2\sig(\eps)}B(r^0,\eps^d)\big)},
$$
so that, by  (\ref{eq:smooth2}), when $\eps$ is small enough
\beq\label{eq:estimmu}
\norm{\sf_\eps}_{B(r^0,\eps^d),\sig(\eps),\sig(\eps)}\leq 2c_0\,\eps\, \big(\sig(\eps)\big)^{-n}
\eeq
since $\norm{f}_{C^\ka(A_1)}\leq 1$.

\paraga Forgetting first about the $\eps$--dependence of the constants, let us describe our construction. 
By P\"oschel's theorem, there exists a symplectic embedding
$$
\Phi_\eps : U_{\rho_*}\T^n\times V_{\sig_*} B(r^0,\eps^d)\to U_{\rho}\T^n\times V_{\sig}B(r^0,\eps^d),
$$
such that
$$
\sH_\eps\circ\Phi_\eps(\th,r)=\sh_\eps(r)+Z_\eps(\ov\th,r)+M_\eps(\th,r).
$$
As a consequence, for $(\th,r)\in \T^n\times B(r^0,\eps^d)$:
$$
\begin{array}{lll}
H_\eps\circ\Phi_\eps(\th,r)&=&\sH_\eps\circ\Phi_\eps(\th,r)+(H_\eps-\sH_\eps)\circ\Phi_\eps(\th,r)\\
&=&\sh_\eps(r)+Z_\eps(\ov\th,r)+\Big[M_\eps(\th,r)+(H_\eps-\sH_\eps)\circ\Phi_\eps(\th,r)\Big]\\
&=& h(r)+Z_\eps(\ov\th,r)+\Big[M_\eps(\th,r)+(\sh_\eps-h)(r)+(H_\eps-\sH_\eps)\circ\Phi_\eps(\th,r)\Big].\\
\end{array}
$$
To get our final result we will therefore set
$$
g_\eps=Z_\eps,\qquad R_\eps(\th,r)=M_\eps(\th,r)+(\sh_\eps-h)(r)+(H_\eps-\sH_\eps)\circ\Phi_\eps(\th,r),
$$
and estimate the $C^p$ norms of these functions. This will be  an easy consequence of the Cauchy inequalities 
once the size of the domains are properly determined.

\paraga {\bf The case $m=n-1$.} 
To make explicit the dependence of the domains and constants with respect to $\eps$, 
let us fix three  constants $a,b,c$ which satisfy the following
inequalities
\beq\label{eq:inegconst}
0<a<b<c,\qquad b<d \qquad 2na+b+(p+1)c<\de,
\eeq
and choose the regularity $\ka$ large enough so as to satisfy
\beq\label{eq:inegka}
\ka> \Max\Big(p+\frac{\ell}{a},p+\frac{(n+p)b+1}{a},2p+n+\frac{2}{b}\Big),
\eeq
where $\de,d$ and $p,\ell$ were introduced in Proposition~\ref{prop:normal2}. 

\vskip1mm$\bu$ Let $\om(r^0)=(\ha\om,0)$. We first fix the width 
$$
\sig(\eps)=\eps^a
$$
of the smoothing process.
We will apply P\"oschel's theorem to the Hamiltonian $\sH_\eps=\sh_\eps+\sf_\eps$ on the domain $D:=B(r^0,\eps^d)$,
which is $(\al/2=\abs{\ha\om}/4,K(\eps))$ nonresonant modulo $\cM$ for $\sh_\eps$ with
$$
K(\eps)=\eps^{-b},
$$
for $\eps$ small enough. To see this, observe that by (\ref{eq:smooth1}) applied to each component of
 $\varpi_\eps=\nabla h_\eps$ and $\om=\nabla h$:
$$
\norm{\varpi_\eps-\om}_{C^0}\leq c_0\big(\sig(\eps)\big)^{\ka-1}\norm{h}_{C^\ka(B(r^0,1))}=C_0\eps^{(\ka-1)a}
$$
for $\eps$ small enough. So, for $r\in B(r^0,\eps^d)$ and $\abs{k}\leq K(\eps)$, since $b<d$, by Lemma~\ref{lem:nonres}
$$
\abs{\varpi_\eps(r)\cdot k}\geq\abs{\om(r)\cdot k}-\abs{\big(\varpi_\eps(r)-\om(r)\big)\cdot k}\geq \al-C_0\eps^{(\ka-1)a-b},
$$
and the claim immediately follows for $\eps$ small enough, since $(\ka-1)a-b>0$ by  (\ref{eq:inegka}).

\vskip1mm$\bu$ With our choice of $\sig(\eps)$, equation (\ref{eq:estimmu}) yields 
$$
\mu(\eps):=\norm{\sf_\eps}_{B(r^0,\eps^d),\sig(\eps),\sig(\eps)}\leq 2c_0\,\eps^{1-na}
$$
We finally set
$$
\rho(\eps)=\eps^c
$$
so that we can apply P\"oschel's theorem with $\rho_0=\sig(\eps)$, $\sig_0=\sig(\eps)$ and 
the triple 
$$
(\mu,\sig,\rho)=\big(\mu(\eps),\sig(\eps),\rho(\eps)\big),
$$ 
since the three constraints of
equation~(\ref{const1}) are satisfied for $\eps$ small enough, by equation~(\ref{eq:inegconst}).

\vskip1mm$\bu$ Then by (\ref{maj1}) and the Cauchy inequalities, taking the inequality
$\rho<\sig$  into account, one gets for a suitable $C>0$ and for $\eps$ small enough
\beq\label{ineg1}
\norm{Z_\eps-Z_0}_{C^p}\leq c'' \frac{K}{\al\rho}\mu^2\frac{1}{\rho^p}\leq C \eps^{2-2na-b-(p+1)c},
\eeq
\beq\label{ineg2}
\phantom{\int^\int}\norm{M_\eps}_{C^p}\leq e^{-K\sig/6}\frac{1}{\rho^p}\mu\leq \demi\eps^\ell,\phantom{\int^\int}
\eeq
and
\beq\label{ineg3}
\norm{\Phi_\eps^\th-\Id}_{C^p},\ \norm{\Phi_\eps^\th-\Id}_{C^p}\leq c''\frac{K}{\al}\frac{\sig}{\rho}\frac{1}{\rho^p}\mu
\leq C\eps^{1-(n-1)a-b-(p+1)c}.
\eeq

\vskip1mm$\bu$ The proof of (\ref{eq:estimphi}) is now immediate from (\ref{ineg3}) and (\ref{eq:inegconst}).

\vskip1mm$\bu$ To prove the second inequality of (\ref{eq:estimnormform}) note that on the one hand
$$
\norm{h-\sh_\eps}_{C^p}\leq c_0\, (2\sig)^{\ka-p}\norm{h}_{C^\ka},\quad
\norm{H_\eps-\sH_\eps}_{C^\ka}\leq c_0\, (2\sig)^{\ka-p}\norm{H_\eps}_{C^p}\leq c_0\, (2\sig)^{\ka-p}\big(\norm{h}_{C^\ka}+1\big)
$$
for $\eps$ small enough, which yields by the Faa-di-Bruno formula, for a suitable $C>0$
$$
\norm{(h-\sh_\eps)+(H_\eps-\sH_\eps)\circ\Phi_\eps}_{C^p}\leq C \sig^{\ka-p}\leq \eps^{a(\ka-p)}\leq \demi\eps^\ell,
$$
by the first inequality of (\ref{eq:inegka}). The conclusion then readily follows from (\ref{ineg2}).

\vskip2mm$\bu$ Finally, to prove the first inequality of (\ref{eq:estimnormform}), note that
$$
g_\eps-\eps[f]=(Z_\eps-Z_0)+(Z_0-\eps[f]).
$$
The first term is conveniently controlled by (\ref{ineg1}):
\beq\label{eq:intermed}
\norm{Z_\eps-Z_0}_{C^p}\leq \demi \eps^{2-\de},
\eeq
for $\eps$ small enough, thanks to (\ref{eq:inegconst}). Moreover
$$
Z_0(\ov\th,r)-\eps[f](\ov\th,r)=\De_1(\ov\th,r)-\De_2(\ov\th,r)
$$
with
$$
\De_1(\ov\th,r)=\!\!\!\!\sum_{\ov k\in\Z^m,\abs{\ov k}\leq K}\!\!\!\!\big([\sf_\eps]_{(0,\ov k)}(r)-\eps[f]_{(0,\ov k)}(r)\big)e^{2i\pi\ov k\cdot\ov \th}
$$
$$
\De_2(\ov\th,r)\!\!\!\!\sum_{\ov k\in\Z^m,\abs{\ov k}>K}\!\!\!\!\eps[f]_{(0,\ov k)}(r)e^{2i\pi\ov k\cdot\ov \th}.\phantom{\int^\int}\\
$$
Now, 
$$
\norm{\sf_\eps-\eps f}_{C^p}\leq c_0\,(2\sig)^{\ka-p}\,\eps
$$
since $f$ has unit norm in $C^\ka(\A^n)$. Therefore
$$
\norm{\De_1}_{C^p}\leq C\,K^{n+p} \norm{\sf_\eps-\eps f}_{C^p}\leq C' \eps^{1+(\ka-p)a-(n+p)b}.
$$
Then, by usual integration by parts for Fourier coefficients, one gets:
$$
\norm{\De_2}_{C^p}\leq CK^p\sum_{k\in\Z^n,\abs{k}>K} \frac{1}{\abs{k}^{\ka-p}}\leq \frac{C'}{K^{\ka-2p-n}}=C'\,\eps^{b(\ka-2p-n)}.
$$
From these two estimates one finally deduces the inequality
$$
\norm{Z_0-\eps[f]}_{C^p}\leq \eps^{2-\de}
$$
from (\ref{eq:intermed}) and the last two inequalities of (\ref{eq:inegka}). Observe finally that 
$\Phi_\eps\big(\T^n\times B(r^0,\eps^d)\big)\subset \T^n\times B(r^0,2\eps^d)$  for $\eps$ small
enough, thanks to (\ref{eq:estimphi}) since $d<1-\de$, which concludes the proof.

\paraga {\bf The case $1\leq m\leq n-2$.}  The proof is very similar to the previous one, up to minor changes for the
definition of the nonresonant domain. With the notation of Lemma~\ref{lem:nonres}, we now require the following 
inequalities for our constants (chosing $\nu=1+2\tau$):
$$
0<a<b<c,\qquad (1+2\tau b)<d,\qquad 
2na +(1+\tau) b+(p+1)c<\de,
$$
and we still assume that $\ka$ satisfies (\ref{eq:inegka}). The proof then exactly follows the same lines as above.


\section{The invariant curve theorem}\label{app:invcurve}
\setcounter{paraga}{0}
For the sake of completeness we reproduce here the statement and proofs from \cite{LM}.
Let  $J^*$ be an open interval of $\R$. We consider a map $\jP_\eps:\T\times J^*\to \A$ of class $C^5$, of the form
\begin{equation}
\label{eq:Poincare}
\jP_\varepsilon(\varphi,\rho)=\bigl(\varphi+\varepsilon \varpi(\rho)+\Delta^\ph_\varepsilon(\varphi,\rho),\rho+\Delta_\varepsilon^\rho(\varphi,\rho)\bigr),
\end{equation}
with $\norm{\varpi}_{C^5}<+\infty$, 
and we moreover assume
\begin{equation}\label{eq:constraints1}
\varpi'(\rho)\geq \sig>0,\quad \norm{\Delta_\varepsilon^\varphi}_{C^5}\leq\varepsilon^7,\quad 
\norm{\Delta_\varepsilon^\rho}_{C^5}\leq\varepsilon^7.
\end{equation}


\begin{prop}\label{prop:KAM} Let $J\subset J^*$ be a nonempty open interval. Then there exists $\eps_0>0$, depending
only on the length of $J$, $\sig$ and $\norm{\varpi}_{C^5}$, such that for $0\leq\eps\leq\eps_0$, the map 
$\jP_\eps$ admits an essential invariant circle contained in $\T\times J$.
\end{prop}

The proof will be based on the translated curve theorem of Herman
(see VII.11.3 and VII.11.11.A.1 in~\cite{Herman}), which we first recall in a form 
adapted to our setting. Given $\delta>0$, we set $\A_\delta=\T\times [-\delta,\delta]$.
A map $F:\A_\de\to\A$ is said to satisfy the {\em intersection property}
provided that for each essential curve $\jC\subset \A_\delta$,
$F(\jC)\cap \jC\neq \varnothing$.


\begin{thm} [Herman] Fix $\de>0$.
Fix $\gamma\in\R$ such that there exists  $\Gamma>0$ satisfying
\begin{equation}\label{eq:Markoff}
\abs{\gamma-\frac{m}{n}}>\frac{\Ga}{n^2},\qquad \forall n\geq 1,\ \forall m\in\Z.
\end{equation}
Assume moreover $\Ga\leq 10\,\de$. Consider an embedding 
$F:\A_\delta\to\T\times\R$ of the form 
\beq\label{eq:formF}
F(\varphi,r)=\big(\varphi+\gamma+r,\ r+\ze(\varphi,r)\big),
\eeq
with $\ze\in C^4(\A_\delta)$, which satisfies the intersection property and 
\[
\Max_{1\leq i+j\leq 4}\norm{\partial_r^i\partial_\varphi^j \ze }_{C^0(\A_\delta)}\leq \Gamma^2.
\]
Then there is a continuous map $\psi : \T\to [-\delta,\delta]$ and a diffeomorphism $f\in {\mathrm{Diff}}^1(\T)$ 
with rotation number $\gamma$  such that $F(\varphi,\psi(\varphi))=f\big(\varphi,\psi(f(\varphi))\big)$ and 
\[
\norm{\psi}_{C^0(\T)}\leq \Gamma^{-1}\Max_{1\leq i+j\leq 4}
\norm{\partial_r^i\partial_\varphi^j \ze }_{C^0(\A_\delta)}.
\] 
\end{thm}

A real number $\gamma$ satisfying Condition~\eqref{eq:Markoff} is said to be of constant type,
with Markoff constant  $\Gamma$. 
Note that we do {\em not} require that $\Gamma$ is the best possible constant.
A more comprehensive exposition of the previous theorem (with better constants) is presented in~\cite{LMS}. 
We will also need the following result (see IV.3.5 in~\cite{Herman}).

\begin{lemma}[Herman] 
\label{lem:Herman-constant-type}There exists a constant $\tau\in\,]0,1[$ such that for any $0<\eta<1/2$, any interval 
of $\R$ with length $\geq\eta$ contains infinitely many real numbers
of constant type with Markoff constant at least $\tau\eta$.
\end{lemma}

\begin{proof}[Proof of Proposition~\ref{prop:KAM}]
We will first conjugate $\jP_\varepsilon$ to  a map of the form~(\ref{eq:formF}).
We set 
\[
\delta_\varepsilon(\varphi,\rho)=\frac{1}{\varepsilon}\Delta_\varepsilon^\rho(\varphi,\rho),\qquad
\Phi_\varepsilon(\varphi,\rho)=\bigl(\varphi,\varpi(\rho)+\delta_\varepsilon(\varphi,\rho)\bigr),\qquad (\ph,\rho)\in\T\times J.
\]
Let $\alpha,\varepsilon>0$ and $\rho_0\in J$ satisfy $[\rho_0-2\alpha,\rho_0+2\alpha]\subset J$ and 
\begin{equation}\label{eq:inv-eps}
\varepsilon^{6}\leq\sig/2. 
\end{equation}
By (\ref{eq:inv-eps}), $\Phi_\varepsilon$   properly embeds $\T\times J$
into $\A$ and, setting $\varpi_0=\varpi(\rho_0)$:
\beq\label{eq:incfond}
\Phi\inv_\varepsilon\bigl(\T\times [\varpi_0-\alpha\sig/2,\varpi_0+\alpha\sig/2]\bigr)
\subset \T\times[\rho_0-\alpha,\rho_0+\alpha].
\eeq
If moreover 
\begin{equation}\label{eq:p-eps}
\varepsilon^7\leq\alpha
\end{equation}
then the estimates on $\jP_\varepsilon$ and $\delta_\varepsilon$ show  that 
\[
\jP_\varepsilon\circ \Phi^{-1}_\varepsilon\bigl( \T\times [\varpi_0-\alpha\sig/2,\varpi_0+\alpha\sig/2]\bigr)
\subset \T\times 
[\rho_0-2\alpha,\rho_0+2\alpha]\subset \T\times J.
\]
Therefore, assuming~\eqref{eq:inv-eps} and~\eqref{eq:p-eps}, the  map 
$\widetilde\jP_\varepsilon=\Phi_\varepsilon\circ\jP_\varepsilon\circ \Phi^{-1}_\varepsilon$ is well defined 
over $ \T\times [\varpi_0-\alpha\sig/2,\varpi_0+\alpha\sig/2]$.
We write 
$$
R(\ph,\rho)=\varpi(\rho)+\delta_\varepsilon(\varphi,\rho),\qquad (\ph,\rho)\in\T\times J,
$$
for the second component of $\Phi_\eps$.
By straightforward computation:
\beq\label{eq:mapjP}
\widetilde\jP_\varepsilon\big(\varphi,R(\ph,\rho)\big)=\big(\varphi+\varepsilon R(\ph,\rho), R'(\ph,\rho)\big),
\eeq
where 
\beq\label{eq:remR'}
R'(\ph,\rho)=\varpi\big(\rho+\Delta_\varepsilon^\rho(\varphi,\rho)\big)
+\delta_\varepsilon\bigl(\varphi+\varepsilon R(\ph,\rho),\rho+\Delta_\varepsilon^\rho(\varphi,\rho)\bigr).
\eeq
Therefore,  by (\ref{eq:mapjP}) and (\ref{eq:remR'}),
after expanding $R'$, the map  $\widetilde\jP_\varepsilon$ takes the form:
\beq
\widetilde\jP_\varepsilon(\varphi,R)=\bigl(\varphi+\varepsilon R,R+\Delta_\varepsilon^R(\varphi,R)\bigr),
\qquad
(\varphi,R)\in\T\times[\varpi_0-\alpha\sig/2,\varpi_0+\alpha\sig/2],
\eeq
with
\beq
\norm{\Delta_\varepsilon^R}_{C^4}\leq \nu \varepsilon^7
\eeq
where the  constant $\nu>0$ depends only on $\sig$ and $\norm{\varpi}_{C^5}$.
We finally fix 
\beq\label{eq:intI}
R_0\in I:= [\varpi_0- \alpha\sig/4, \varpi_0+ \alpha\sig/4]
\eeq
and set 
\beq
\gamma_\eps=\varepsilon R_0,
\qquad
\phi_{R_0,\eps}(\varphi,r)=(\varphi,R_0+\tfrac{1}{\eps} r).
\eeq
The map 
$$
F_{R_0,\varepsilon}=\phi_{R_0,\eps}^{-1}\circ \widetilde\jP_\varepsilon\circ\phi_{R_0,\eps}
$$
is well defined over $\A_{{\alpha\varepsilon\sig}/{4}}$  and takes the required form
$$
F_{R_0,\varepsilon}(\varphi,r)=\big(\varphi+\gamma_\eps+r,\ r+\ze_\eps(\varphi,r)\big),
$$
with 
$
\ze_\eps(\varphi,r)=\varepsilon\Delta_\varepsilon^R\bigl(\varphi,R_0+\tfrac{1}{\eps} r\bigr).
$
In particular
$$
\norm{\ze_\eps}_{C^4}\leq  \nu \varepsilon^{4}.
$$


We can now apply  the translated curve theorem to $F_{R_0,\varepsilon}$ restricted to a suitable subdomain 
of $\A_{{\alpha\varepsilon\sig}/{4}}$.  We 
assume that $\varepsilon$ is small enough so that
\begin{equation}
\varepsilon\alpha\sig<1.
\end{equation}
Thus Lemma~\ref{lem:Herman-constant-type}  applied to the interval $I_\eps=\eps I$ (where $I$ was introduced
in (\ref{eq:intI})), with $\eta=\varepsilon\alpha\sig/2$,
shows that there exists $\gamma_\eps=\varepsilon R_0\in I_\eps$ of constant type, with
Markoff constant 
$$
\Gamma_\eps=\tau\varepsilon\alpha\sig/2.
$$
So $R_0\in I$ and the map $F_{\varepsilon,R_0}$ is well defined on $\A_{{\alpha\varepsilon\sig}/{4}}$.
We will apply the translated curve theorem to $F_{\varepsilon,R_0}$ on $\A_{\de_\eps}$, with 
$$
\delta_\eps=\tau\varepsilon\alpha\sig/20<{\alpha\varepsilon\sig}/{4},
$$
so that 
$
\Gamma_\eps=10\,\de_\eps.
$
Thus, assuming
\begin{equation}
\label{eq:estimation-reste-herman}
\nu\eps^4\leq \frac{1}{4}\tau^2\al^2\sig^2\varepsilon^2,
\end{equation}
there exists a continuous map $\psi : \T\to [-\delta_\eps,\delta_\eps]$ whose graph 
$C$ is an invariant essential circle for $F_{\varepsilon,R_0}$. Since $C\subset \A_{{\alpha\varepsilon\sig}/{4}}$
$$
\phi_{R_0,\eps} (C)\subset \T\times [R_0-{\alpha\sig}/{4},R_0+{\alpha\sig}/{4}]\subset \T\times [\varpi_0-{\alpha\sig}/{2},\varpi_0+{\alpha\sig}/{2}].
$$
Therefore, by (\ref{eq:incfond}) 
$$
\jC=\Phi\inv\circ\phi_{R_0,\eps} (C)\subset \T\times[\rho_0-\alpha,\rho_0+\alpha]
$$
is an essential invariant circle for $\jP_\eps$, contained in $\T\times J$, which exists as soon as
$$
0\leq\eps\leq\eps_0:=\Min(\tfrac{1}{\al\sig},\tfrac{\tau\al\sig}{2\sqrt\nu}).
$$
This concludes the proof.
\end{proof}


\section{Proof of the existence of homoclinic orbits}\label{sec:hom}
For the sake of completeness, in this section we prove the following proposition.

\begin{prop} Let $C$ be of the form (\ref{eq:classham}) and
satisfy  Conditions~$(D)$. Let $c\in\H_1(\T^2,\Z)$. Then  
there exists a sequence $(\ga_n)_{n\in{\N^*}}$ of minimizing periodic 
solutions of $X^C$  {\em  with positive energies}, whose projections on $\T^2$ belong to $c$ and 
whose orbits converge to a polyhomoclinic orbit for the hyperbolic fixed point.
\end{prop}

Since we need to make precise the convergence process for periodic orbits to the polyhomoclinic 
orbits, we will give an extensive, though not original, proof.  Here we will closely follow 
the simple proofs in  \cite{BK,Be00} and use the discrete setting, which immediately 
yields finite dimensional spaces and easy compactness results. Only at the very end we will have to 
adapt this approach to recover the convergence notion in the continuous setting.

\subsection{The discrete setting}
\setcounter{paraga}{0}

Here we fix $C$ as in (\ref{eq:classham}) .
We write $x$ for points in $\R^2$ and $\th$ for points in $\T^2$.
We denote by $\ha C(r,x)=\pdemi T(r)+\ha U(x)$  the lift of $C$ to $T^*\R^2$ and by $\ha L(x,v)=\pdemi T_\bu(v)-\ha U(x)$
the associated Lagragian.
We denote by $\jL:T\R^2\to T^*\R^2$ the Legendre  diffeomorphism associated with $\ha L$ and $\ha C$.

\paraga
Thanks to the particular form of $\ha C$ (a rescaling in action yields a perturbation of the 
convex integrable Hamiltonian $\pdemi T(r)$) one easily  proves  that there exists a 
constant $\tau_0>0$ such that  for $0<\tau\leq\tau_0$,  $\Phi^{\tau \ha C}$
admits a generating function $\ha S_\tau$  on $(\R^2)^2$. This function is characterized by
 the following equivalence
\begin{equation}\label{eq:equiv}
\Big[(x',r')=\Phi^{\tau \ha C}(x,r)\Big]\Longleftrightarrow 
\Big[r=-\partial_x \ha S_\tau(x,x')\ \textrm{and}\ r'=\partial_{x'} \ha S_\tau(x,x')\Big].
\end{equation}
 for any pairs of elements $(x,r)$ and $(x',r')$ of  $T^*\R^2$.
The function $\ha S_\tau$ is nothing but the action integral
$
\ha S_\tau(x,x')=\int_0^\tau \ha L\big(\eta(t)\big)\,dt
$,
where  $\eta$ is the pullback by $\jL$ of the solution of $X^{\ha C}$ with initial condition $(x,r)$.
In particular, $\ha S_\tau$  satisfies the following periodicity property
\begin{equation}\label{eq:periodS}
\ha S_\tau(x+m,x'+m)=\ha S_\tau(x,x'),\qquad \forall m\in\Z^2,\quad \forall (x,x')\in (\R^2)^2.
\end{equation}
Moreover, one easily sees that
\begin{equation}\label{eq:quadS}
 \ha S_\tau (x,x')\sim\frac{1}{\tau}\,T_\bu(x-x')\qquad\textrm{when}\qquad \norm{x-x'}_2\to+\infty
\end{equation}
uniformly with respect to $\norm{x-x'}_2$. 
Note finally that since the Hamiltonian flow is $C^1$, by transversality the action $\ha S_\tau$ is  $C^1$ in the variables 
$(x,x',\tau)$.

\paraga There no longer exists a generating function in the usual sense on $\T^2$. However one can still introduce
a generalized one, defined for $(\th,\th')\in(\T^2)^2$ by
$$
S_\tau(\th,\th')=\Min\big\{\ha S_\tau(x,x')\mid \pi(x)=\th,\ \pi(x')=\th'\big\}
$$
where $\pi$ stands for the projection $\R^2\to\T^2$. Note that, by (\ref{eq:periodS}) and (\ref{eq:quadS}), there exists
$\rho>0$ such that for each pair $(\th,\th')\in(\T^2)^2$, there exists a pair $(x,x')\in(\R^2)^2$ with $\norm{x}\leq 1$,
$\norm{x'}\leq \rho$ and  $S_\tau(\th,\th')=\ha S_\tau(x,x')$. We say that $(x,x')$ is a minimizing lift for $(\th,\th')$.

 The function $S_\tau$ is not differentiable in general, but it is easy to see that it is Lipschitzian for the natural product 
distance $d$ on $\T^2$.
Indeed, consider  two pairs  $(\th_1,\th'_1)$, $(\th_2,\th'_2)$ on $(\T^2)^2$, fix a minimizing lift $(x_1,x_1')$  
for  $(\th_1,\th'_1)$ and choose 
the lifts  $x_2,x'_2$ of $\th_2,\th'_2$ which are the closest ones to the points $x_1,x'_1$ (so $(x_2,x'_2)$ is not
necessarily minimizing for $(\th_2,\th'_2)$). 
Therefore, setting $K=\Lip_{\ov B(0,2)\times \ov B(0,\rho+1)} \ha S_\tau$,
$$
\begin{array}{ll}
S_\tau(\th_2,\th'_2)-S_\tau(\th_1,\th'_1)&\leq  \ha S_\tau(x_2,x'_2)-\ha S_\tau(x_1,x'_1)\\
&\leq K\norm{(x_2,x'_2)-(x_1,x'_1)}
= K d((\th_2,\th'_2),(\th_1,\th'_1)).
\end{array}
$$
Interverting the pairs $(\th_1,\th'_1)$ and $(\th_2,\th'_2)$  then yields
$$
\abs{S_\tau(\th_2,\th'_2)-S_\tau(\th_1,\th'_1)}\leq K d((\th_2,\th'_2),(\th_1,\th'_1)),
$$
which proves our claim.

\paraga Given $\tau\in\,]0,\tau_0]$, a sequence  of $\R^2$  is called a (discrete) {\em trajectory} of the 
system $\Phi^{\tau \ha C}$ when its points
are the projection on $\R^2$ of those of an orbit of  $\Phi^{\tau \ha C}$ in $T^*\R^2$. So a sequence
$(x_k)_{k\in\Z}$ is a trajectory if and only if
$$
\partial_{x'} \ha S_\tau(x_{k-1},x_{k})+\partial_{x} \ha S_\tau(x_k,x_{k+1})=0,\qquad \forall k\in\Z,
$$
since it is then the projection of the orbit $(x_{k},r_{k})_{k\in\Z}$ with $r_k=-\d_x \ha S_\tau(x_k,x_{k+1})$.

We finally define a (discrete) trajectory for $\Phi^{\tau C}$ on $\T^2$ as the projection on $\T^2$ of a 
discrete trajectory for $\ha C$ on $\R^2$.

\paraga A finite sequence will be called a segment. The following lemma is immediate.

\begin{lemma} \label{lem:rel} Let $i<j$ be fixed integers and let $\tau\in\,]0,\tau_0]$.
\begin{itemize}
\item Fix a segment $(x_i,\ldots,x_j)$ in $(\R^2)^{j-i+1}$ and assume that there exists  
a sequence 
$$
(x_i^n,\ldots,x_j^n)_{n\in\N}
$$ of segments of  trajectories for $\Phi^{\tau\ha C}$ such that 
$
\lim_{n\to\infty}x_k^n= x_k
$ 
for $i\leq k\leq j$.
Then $(x_i,\ldots,x_j)$ is a segment of trajectory for $\Phi^{\tau \ha C}$. 

\item Fix a segment $(\th_i,\ldots,\th_j)$ in $(\T^2)^{j-i+1}$ and assume that there exists  
a sequence 
$$
(\th_i^n,\ldots,\th_j^n)_{n\in\N}
$$ 
of segments of  trajectories for $\Phi^{\tau C}$, with energies 
bounded above, such that 
$\lim_{n\to\infty}\th_k^n\to \th_k$ for $i\leq k\leq j$.  Then $(\th_i,\ldots,\th_j)$ is a segment of 
trajectory for~$\Phi^{\tau C}$.

\end{itemize}
\end{lemma}

\paraga We will focus of periodic trajectories in $\T^2$ and their lifts to~$\R^2$. 
Given a positive integer $q$ and an integer vector $m\in\Z^2$, we  introduce the space
$$
\{\xi\in(\R^2)^\Z\mid \xi(i+q)=\xi(i)+m,\ \forall i\in\Z\}
$$
of all sequences in $\R^2$ whose projections on $\T^2$ are ``$q$--periodic with rotation vector  $w/q$'',
together with the space
$$
\jX_m^q=\big\{(x_0,\ldots,x_q)\in(\R^2)^{q+1}\mid x_q=x_0+m\big\}.
$$
In the following we will identify an element of $\jX_w^q$ with the corresponding
``$q$--periodic'' complete sequence. Given $\tau\in\,]0,\tau_0]$, one easily sees that $(x_0,\ldots,x_q) \in\jX_m^q$ is a 
trajectory of 
$\Phi^{\tau \ha C}$
if and only if it is a critical point of 
the generalized action $\ha S_\tau: \jX_w^q\to\R$ defined by
$$
\ha S_\tau(x_0,\ldots,x_q)=\sum_{k=0}^{q-1}\ha S_\tau(x_k,x_{k+1}).
$$
We also identify a $q$--periodic sequence  on $\T^2$ with its restriction to $\{0,,\ldots,q\}$ and we therefore
introduce the space 
$$
\jE^q=\big\{(\th_0,\ldots,\th_q)\in(\T^2)^{q+1}\mid \th_0=\th_q\big\}.
$$
together with the generalized action $S_\tau:\jE^q\to\R$ defined by
$$
S_\tau(\th_0,\ldots,\th_q)=\sum_{i=0}^{q-1}S_\tau(\th_i,\th_{i+1}).
$$

\paraga  We say that a segment $(x_j,\ldots,x_k)$ in $(\R^2)^{k-j+1}$ is minimizing for $\ha S_\tau$ when its  action 
$\ha S_\tau(x_j,\ldots,x_k)$ is smaller that the action 
of any other segment with the same length and the same extremities, and we say that a sequence  is minimizing 
when each of its  subsegments is minimizing. 
Finally, we say that a sequence $\xi\in\jX ^q_m$ 
is $(q,m)$--minimizing  for $S_\tau$ if
$$
\ha S_\tau(\xi)=\Min_{\xi'\in\jX ^q_m} \ha S_\tau(\xi').
$$
One easily checks that minimizing sequences and  $(q,m)$--minimizing sequences are discrete trajectories of the system 
$\Phi^{\ha \tau C}$.


 \subsection{The hyperbolic fixed point as an Aubry set}
 \setcounter{paraga}{0}
Here we assume that $C$ satifies Conditions $(D)$.
Without loss of generality, one can assume that $U$ reaches its maximum at $0$ and that $U(0)=0$.
We now examine the variational properties of the fixed point $O=(0,0)$ of $X^{\ha C}$
in the discrete framework. In all this section we fix $\tau\in\,]0,\tau_0]$ and get rid of the corresponding index in the notation.
Note first that the functions $\ha S:\R^2\times\R^2\to\R$ and  $S:\T^2\times\T^2\to\R$ clearly satisfy
$\ha S(x,x')\geq0$, $S(\th,\th')\geq0$ and
\beq\label{eq:minim}
 \Min_{(x,x')\in (\R^2)^2}\ha S(x,x')=\ha S(0,0)=0,\qquad \Min_{(\th,\th')\in (\T^2)^2}S(\th,\th')=S(0,0)=0.
\eeq
 We define the (projected) Aubry set $\jA$ as the subset of $\T^2$ formed by the points $\th$ such that
 there exists a sequence $\big(P^i=(\th_0^i,\ldots,\th_{q_i}^i)\big)_{i\in{\N}}$ of 
 $q_i$--periodic sequences, with $q_i\to+\infty$ when $i\to+\infty$, such that
 \begin{equation}\label{eq:aubry}
 \th_0^i=\th\qquad\textrm{and}\qquad \lim_{i\to+\infty}S(P^i)=0.
 \end{equation}
 One easily obtains the following well-known lemma (see \cite{So,Fa09}).
 
 \begin{lemma}\label{lem:Aubry}
 The projected Aubry set $\jA$ reduces to the maximum $\{0\}$ of the potential function $U$.
 \end{lemma}

\begin{proof} The proof is based on the following remark.
 Let $\th\in\T^2\setminus\{0\}$. Then 
 $$
 \Min_{\th'\in\T^2}S(\th,\th')>0.
 $$

To see this, fix $e^*>0=\Max U$. Consider $\th'\in\T^2$ and fix two lifts $x,x'$ of $\th,\th'$. 
Let  $B\subset\R^2$ be a compact ball, centered at $x$, on which 
 ${\rm Max}_B(\ha U)<0$ (which is possible since $x\notin \Z^2$).
 Consider the unique solution $\ga(t)=(x(t),r(t)):[0,\tau]\to T^*\R^2$ of the vector field $X^{\ha C}$ such that $x(0)=x$
 and $x(\tau)=x'$. 
 Let $e$ be the  value of $\ha C$ on  $\ga$. 
 We consider  two cases.
 \vskip1mm
 -- If $e\geq e^*$, since $\pdemi T(r(t))=e-\ha U(x(t))$ then $\ha L\big(x(t),\dot x(t)\big)=e-2\ha U(x(t))\geq e^*$, so
\begin{equation}
 \int_0^\tau \ha L\big(x(t),\dot x(t)\big)\,dt\geq \tau e^*.
\end{equation}
\vskip-1mm
 -- If $e\leq e^*$,
 since $T_\bullet(\dot x(t))$ is  bounded above by $2(e^*-\Min U)$,
 there exists $\tau'>0$ (independent of $e$) such that $x(t)\in B$ for $t\in[0,\tau']$. 
 So in this case
 \begin{equation}
 \int_0^\tau\ha L(x(t),\dot x(t))\,dt\geq \int_0^{\tau'}\ha L(x(t),\dot x(t))\,dt\geq -\tau' {\rm Max}_B(\ha U).
\end{equation}
Since the lifts $x,x'$  are arbitrary, one finally gets
 $$
\Min_{\th'\in\T^2} S(\th,\th')\geq\Min\big[\tau e^*,-\tau' {\rm Max}_B(\ha U)]> 0,
 $$
which proves our remark.

\vskip2mm

Let us now turn to the proof of our lemma. Clearly $0\in\jA$ by (\ref{eq:minim}). Conversely, if $\th\neq0$
and if $P=(\th_0,\ldots,\th_q)$ is a $q$-periodic sequence with $\th_0^i=\th$, then
 $$
 S(P)\geq\Min_{\th'\in\T^2} S(\th,\th')
 $$
 and the limit of a sequence $(P^i)$ of such sequences is positive by the previous remark, which proves that 
 $\th\notin\jA$.
 \end{proof}


\subsection{Proof of Proposition \ref{prop:polyhom}}
Thoughout this section, we identify $H_1(\T^2,\Z)$ with $\Z^2$ and fix $c\in H_1(\T^2,\Z)\setm\{0\}$, so that  $c$ 
is just an integer vector $m\in\Z^2\setm\{0\}$. We assume that $C$ satisfies Conditions $(D)$.


 \subsubsection{Minimizing sequences}

\begin{lemma}\label{lem:posen}
The energy of any $(q,m)$-minimizing sequence for $\ha S$ in $\jX_m^q$ is nonnegative.
\end{lemma}
 
 \begin{proof} Here we will work both with continuous and discrete trajectories.
Let us first prove that a $(q,m)$-minimizing trajectory is fully minimizing, in the sense that it minimizes the
action between any pair of points. This is an easy consequence of the Morse Crossing Lemma 
(\cite{Mat91} Theorem 2 and \cite{So} Lemma 5.31), which
we recall here in a weak form.  There exists $\eps>0$ such that when two (continuous) trajectories 
$\ze_i:[-\eps,\eps]\to\R^2$ of $\ha L$ satisfy $\ze_1(0)=\ze_2(0)$ and $\dot\ze_1(0)\neq\dot\ze_2(0)$, 
then there exist $C^1$ curves $\al_i:[-\eps,\eps]\to\R^2$ with endpoints $\al_1(-\eps)=\ze_1(-\eps)$, $\al_1(\eps)=\ze_2(\eps)$,
$\al_2(-\eps)=\ze_2(-\eps)$ and $\al_2(\eps)=\ze_1(\eps)$ such that 
$$
A(\al_1)+A(\al_2)<A(\ze_1)+A(\ze_2)
$$
(this result still holds true for higher dimensional systems). 
From this, one deduces  that two distinct $(q,m)$-minimizing trajectories do not intersect one another. 
This in turn easily yields the fact that  a $(q,m)$-minimizing trajectory is also a $(nq,m)$-minimizing trajectory, 
for any integer $n\geq 1$ (see \cite{Ba}, Theorem 3.3 for instance). Since a $(nq,m)$-minimizing trajectory
minimizes the action on each subinterval of $[0,nq]$, this easily proves our claim.  

\vskip2mm

We now fix a $(q,m)$-minimizing trajectory and consider the probability measure  $\mu$ evenly distributed on its orbit
 in $T\T^2$. 
By (\cite{Mat91}, Proposition 2), there exists $c\in H^1(\T^2)$ such that $A_c(\mu)=-\al(c)$.  We already noticed 
that the support of $\mu$ is located in $C\inv(\al(c))$, and
clearly $A_c(\mu)\leq0$ since the action $A_c$ vanishes on the zero trajectory. So $\al(c)\geq 0$, which concludes the proof since
the support of $\mu$ coincides with the orbit.
\end{proof}

 We now note that there exists a trivial upper bound, {\em uniform with respect to  $q$},  for the actions 
 of minimizing sequences in $\jX_m^q$. 
  
 \begin{lemma}\label{lem:unifbound}  Let $m\in\Z^2\setm\{0\}$ be fixed. 
 Then, for any $q\geq1$ and any $(q,m)$--minimizing sequence  $\xi\in\jX^q_m$, the action of $\xi$ satisfies
 $$
 \ha S(\xi)\leq \ha S(0,m).
 $$
 \end{lemma}
 
 \begin{proof} Since $\xi$ is minimizing in $\jX^q_m$
$$
 \ha S(\xi)\leq  \ha S(0,\ldots,0,m)=\sum_{k=0}^{q-1}\ha S(0,0)+\ha S(0,m)=\ha S(0,m),
 $$ 
%
 where of course the segments $(0,\ldots,0,m)$ and $(0,\ldots,0,0)$ are in $\R^{q+1}$. 
 \end{proof}
 
One also deduces from the previous lemma the following useful result.

\begin{cor}\label{cor:unifbound1}  Let $m\in\Z^2\setm\{0\}$ be fixed. The energy of a $(q,m)$--minimizing sequence tends to 
$0$ when $q$ tends to $+\infty$. 
\end{cor}
 
\begin{proof} Let $\xi$ be a $(q,m)$--minimizing sequence with energy $e_q$ and let 
$\ze$ be the corresponding continuous trajectory.
The Lagrangian satisfies 
$$
\ha L\big(\ze(t),\dot\ze(t)\big)=e_q-2\ha U(\ze(t))\geq e_q,
$$
so $\ha L\big(\ze(t),\dot\ze(t)\big)\geq e_q\geq 0$ and, as a consequence
$$
q\tau e_q \leq \ha S(\xi)\leq \ha S(0,m),
$$ 
which proves that $\lim_{q\to\infty} e_q=0$.
\end{proof}


 \subsubsection{Homoclinic behavior} 
The next lemma proves the existence of homoclinic trajectories.
 
  \begin{lemma}\label{lem:homorb} 
Fix a sequence $(\xi^q)_{q\in{\N^*}}$ of elements of $\jX_m^q$, such that $\xi^q$ is minimizing for $\ha S$.
Set $\Th^q=\pi\circ\xi^q$. 
Consider a limit point $\Th=(\th_k)_{k\in\Z}$ of the sequence $(\Th^{q})_{q\in{\N^*}}$, relatively to the (compact) product toplogy
of $(\T^2)^\Z$.  Then $\Th$ is a  {\em trajectory} for $\Phi^{\tau \ha C}$ and is homoclinic to $0$, that is:
\begin{equation}\label{eq:hom}
\lim_{k\to -\infty}\th_k=\lim_{k\to+\infty}\th_k=0.
\end{equation}
 \end{lemma} 
 
 \begin{proof} The fact that $\Th$ is a trajectory
 is an immediate consequence of Lemma~\ref{lem:rel} and Corollary \ref{cor:unifbound1}, 
 since each subsegment of $\Th$ is the limit of segments of
 trajectories with energy bounded above.
 
 We will prove  (\ref{eq:hom}) for the $\om$--limit of $\Th$, the $\al$--limit case being similar. More precisely,
we will show that~$0$ is the only  limit point for the sequence $(\th_k)_{k\in\Z}$ when $k\to+\infty$, which proves our claim
by compactness of $\T^2$.
Consider such a limit point $\th$. There exists an increasing sequence $(k_i)_{i\in{\N}}$ such that 
$(\th_{k_i})_{i\in{\N}}$ converges to~$\th$ for $i\to+\infty$, and one can moreover assume that 
 $$
 d(\th_{k_i},\th)\leq\Frac{1}{2^i},\qquad \forall i\in{\N}.
 $$
 Consider the $(k_{(i+1)}-k_i+1)$-periodic sequence
 $
 P^i=(\th,\th_{k_i},\ldots,\th_{k_{(i+1)}},\th).
 $
Then obviously
 $$
 0\leq S(P^i)\leq S(\th_{k_i},\th_{k_i+1},\ldots,\th_{k_{(i+1)}-1},\th_{k_{(i+1)}})+\Frac{{\textstyle c}}{2^i},
 $$
 where $c=2\,\Lip S$. Therefore, given a positive integer $i_1$, adding these inequalities yields
\begin{equation}\label{eq:series}
 0\leq \sum_{i=0}^{i_1}S(P^i)
\leq S(\th_{k_{0}},\ldots,\th_{k_{(i_1+1)}}) + 2c.
\end{equation}
Consider  now a subsequence $(\Th^{q_\ell})_{\ell\in{\N}}$ which converges to $\Th$, and set 
$\Th^{q_\ell}=(\th^{q_\ell}_k)_{k\in\Z}$. For $\ell\geq\ell_0$
large enough, the period $q_\ell$ is larger than $(k_{(i_1+1)}-k_0+2)$ and therefore, by Lemma~\ref{lem:unifbound},
there is $M$ independent of $i_1$ such that
$$
S(\th_{k_{0}}^{q_\ell},\ldots,\th_{k_{(i_1+1)}}^{q_\ell})\leq M,\qquad \ell\geq \ell_0.
$$
Taking the limit when $\ell\to\infty$ shows that
$$
S(\th_{k_{0}},\ldots,\th_{k_{(i_1+1)}}) \leq M
$$
since $S$ is continuous. Therefore by (\ref{eq:series})
the series $\sum S(P^i)$ converge, hence $S(P^i)$ tends to~$0$. This proves
 that $\th$ is in the projected Aubry set, so $\th=0$ by Lemma \ref{lem:Aubry}.
 \end{proof}


 \subsubsection{Shifts and nontrivial limit points} 
Now, nothing prevents the previous limit point $\Th$ to be  the trivial zero sequence. 
To detect nontrivial limit points we will have to consider  new sequences, deduced from  $(\Th^q)_{q\in{\N^*}}$ by translation of the indices in order to ``center the convergence process''. 

\vskip2mm

We begin with two simple results  for which we explicitely take advantage of the simple features of the Hamiltonian flow in the
neighborhood of the fixed point $O$.

\begin{lemma}\label{lem:limpoints} 
Fix a sequence $(\xi^q)_{q\in{\N^*}}$ of elements of $\jX_m^q$, such that $\xi^q$ is minimizing for $\ha S$.
Set $\Th^q=\pi\circ\xi^q$. 
There exists an infinite subsequence $(\Th^q)_{q\in N}$ and a finite set $\{\La_1,\dots,\La_p\}$ of nonzero pairwise distinct 
trajectories $\La_i=(\la^i_k)_{k\in \Z}$ homoclinic to $0$ in $\T^2$, such that for any shift-sequence $\ka=(k_q)_{q\in N}$ and any limit point $\Th=(\th_k)_{k\in \Z}$ of the sequence $(\Th^{q,\ka})_{q\in N}$, there exists $i\in \{1,\dots,p\}$ and $\ell\in \Z$ such that for any $k\in \Z$, $\th_k=\la_{k+\ell}$.
\end{lemma}

\begin{proof} 
We write $\xi^q=(x_k^q)_{k\in\Z}$ and $\Th^q=(\th_k^q)_{k\in\Z}$.
By the previous remark and Corollary~\ref{cor:unifbound1}, there exists $q_0$ such that for $q\geq q_0$
$$
\norm{x_{i+1}^q-x_i^q}< \tfrac{1}{8},\qquad \forall i\in\Z.
$$
We denote by $B_{\T^2}(0,\frac{1}{8})$ the ball of center $0$ and radius~$\frac{1}{8}$ in $\T^2$.
Therefore, for $q\geq  q_0$,  if some points $x_i^q$ and $x_j^q$,  $i<j$, belong to two different lifts of $B_{\T^2}(0,\frac{1}{8})$ in $\R^2$, then there is an index $k$ with $i<k<j$ such that $\pi(x_k^q)\notin B_{\T^2}(0,\frac{1}{8})$. 

\vskip2mm

$\bullet$  \emph{Existence of one nontrivial trajectory homoclinic to $0$}.  
Given $q\geq q_0\geq 3$, by the previous remark,  there exists $k\in\{0,\ldots,q-1\}$ such that $\th^q_k\notin B_{\T^2}(0,\frac{1}{8})$. 

\noindent Let $k_q:=\min\{j\in\{0,\dots,q-1\}\mid\th_j^q\notin B_{\T^2}(0,\frac{1}{8})\}$, so $\ka_1:=(k_q)_{q\geq q_0}$ is a shift-sequence. 
By compactness of $(\T^2)^\Z$ there exists an infinite subset $N_1$ of $\N$ such that the subsequence $(\Th^{\ka_1,q})_{q\in N_1}$ converges. 
Let  $\La_1=(\la_k^1)_{k\in\Z}$ be its limit. By Lemma \ref{lem:homorb}, $\La_1$ is a trajectory homoclinic to $0$.
Now, since $\th^q_{k_q}=\th^{\ka,q}_0$ belongs to the compact set $\T^2\setm B_{\T^2}(0,\frac{1}{8})$, then
$$
\la_0^1=\lim_{\substack{q\to+\infty\\ q\in N_1}}\th^{\ka_1,q}_0\in\T^2\setm B_{\T^2}(0,\tfrac{1}{8}).
$$
So the trajectory $\La_1$ is nontrivial.

\vskip2mm

$\bullet$  \emph{Search for other distinct homoclinic trajectories.} 
For $\de>0$, we denote by $V_\de\subset\T^2$ the $\de$--neighborhood of the set $\{\la_k\mid k\in\Z\}$.  
The following two cases only occur:
\begin{enumerate}
\item  for all $\de>0$ there exists $q_1$ such that for any $q\geq q_1$ the set $\{\th^{\ka,q}_0,\ldots,\th^{\ka,q}_{q-1}\}\subset V_\de$ ,
\item there exists $\de_1>0$ such that the set $\{q\in N_1\mid\{\th^{\ka,q}_0,\ldots,\th^{\ka,q}_{q-1}\}\not\subset V_{\de_1}\}$ is infinite.
\end{enumerate}

In the first case, there is no other homoclinic trajectory, we set $N:=N_1$.

In the second case, set $N_1':=\{q\in N_1\mid\{\th^{\ka_1,q}_0,\ldots,\th^{\ka_1,q}_{q-1}\}\not\subset V_{\de_1}\}$. 
For $q\in N_1'$, let $k_q^2:=\min\{j\in \{0,\ldots,q-1\}\mid \th_j^q\notin V_{\de_1}\}$, so $\ka_2:=(k^2_q)_{q\in N_2}$ is a shift-sequence.
As before, there exists an infinite set $N_2\subset N_1'$ such that the subsequence $(\Th^q)_{q\in N_2}$ converges to a limit $\La_2:=(\la^2_k)_{k\in \Z}$ whose image is not contained in $V_{\de_1}$. 
As before, one cheks that $\La_2$ is a nontrivial trajectory homoclinic to $0$ and by construction, $\La_2\neq \La_1$.
We then examine the same alternative considering neighborhoods of $\{\la_k^2\mid k\in\Z\}$ and the sequence $(\Th^{q,\ka_1+\ka_2})_{q\in N_2}$ and continue the process until the first term of the alternative holds true.

\vskip2mm 

$\bullet$  \emph{The process stops after a finite number of steps.}
For $j\geq 1$, we denote by $N_j$ the index set we get after $j$ steps, by $\ka_j:=(k^j_q)_{q\in N_j}$ the associated shift-sequence and by $\La_j:=(\la^j_k)_{k\in \Z}$ the corresponding trajectory homoclinic to $0$. 
Let $\bar{\ka}_j:= \sum_{i=1}^j \ka_i$, that is, $\bar{\ka}_j=(\bar{k}^j_q)_{q\in N_j}$ with $\bar{k}^j_q:=\sum_{i=1}^jk^i_q$.
In particular, for $1\leq j\leq p$, the following convergence hold true:
$$
\lim_{\substack{q\to+\infty\\ q\in N_j}}\th^{\bar{\ka}_j,q}_0=\la^j_0.
$$
Assume there are $p$ steps with $p\geq 1$. By construction the points $\la_0^j$ with $1\leq j\leq p$ are pairwise distinct and contained in $\T^2\setminus B_{\T^2}(0,\tfrac{1}{8})$, so for $q\in N_p$ large enough, the set $\{\th^q_0,\ldots, \th_{q-1}^q\}$ contains at least $p$ points in the compact set $\T^2\setminus B_{\T^2}(0,\tfrac{1}{8})$.
By the first remark in Lemma~\ref{lem:Aubry}, there exists $a>0$ such that for any $(\th,\th')\in (\T^2\setminus B_{\T^2}(0,\tfrac{1}{8}))\times \T^2$, $S(\th,\th')\geq a$.
Therefore, by Lemma~\ref{lem:unifbound}, for $q\in N_p$ large enough,
$$
pa \leq S(\Th^q) \leq \ha S(0,m)
$$
and so $p$ is bounded above.

\vskip2mm 

$\bullet$  \emph{Conclusion.} Let $p$ be the number of steps. With the previous notation, we set $N=N_p$ and
$$
K:=\bigcup_{j=1}^p\{\la_k^j\,|\, k\in \Z\},
$$
so $K$ is the union  of the images of the trajectories $\La^j$. 
Consider  a shift-sequence $\ka$ and let $\Th=(\th_k)_{k\in \Z}$ be a limit point of $(\Th^{\ka,q})_{q\in N}$.
By construction, the image $\{\th_k\mid k\in\Z\}$ of $\Th$ is contained in the intersection $\cap _{\de>0} V_\de$, where $V_\de$ is the $\de$-neigborhood of $K$, hence $\{\th_k\mid k\in\Z\}\subset K$. 
Moreover, by Lemma~\ref{lem:homorb}, $\Th$ is a nontrivial trajectory homoclinic to $0$. 
This proves that the image of $\Th$ coincides with the image of some $\La_i$, which proves Lemma~\ref{lem:limpoints}, since two trajectories with the same image are deduced from one another by a shift of indices. 
\end{proof}

 \subsubsection{The continuous setting} 
We now prove the geometric convergence to the polyhomoclinic orbits in the sense of Section~\ref{ssec:convergence}.
We keep the notation of the proof of Lemma~\ref{lem:limpoints}. 
 
 \vskip2mm

$\bullet$ Let  $\om^i$  be the continuous  solution associated with the homoclinic trajectory $\La^i\in\jH$, with initial condition
$\om^i(0)=\la^i_0$. The limit polyhomoclinic orbit will be a concatenation of the orbits $\Om_i$ of the solutions $\om_i$,
ordered in a suitable way that we will now make explicit.
The notion of convergence being independent of the choice of sections, we choose as  exit and  entrance sections 
$\Sig^u$ and $\Sig^s$ for each homoclinic orbit suitable lifts to  $T^*\T^2$ of small arcs  of the  boundary circle of the disc $\jO$  
(which can be assumed to be transverse to each homoclinic trajectory).

\vskip2mm

$\bullet$ Let  $\ga^q$  be the continuous solution associated with $\Th^q$, such that $\ga^q(0)=\th^q_0$.
One easily deduces from Lemma \ref{lem:rel} that the sequence of functions 
$\big(\ga^q(t+\ka^1_q)\big)$ pointwise converges  to the function  $\om^1(t)$.

\vskip2mm

$\bullet$ Let $t^1_u<t^1_s$ be the exit and entrance times for $\om^1$, defined by $\om^1(t^1_u)\in \Sig^u$ and 
$\om^1(t^1_s)\in\Sig^s$, and $\om^1(]-\infty,t^u])\in\pi\inv(\jO)$, $\om^1([t^s,+\infty[)\in\pi\inv(\jO)$. 
By transversality and the previous property, there exist sequences $(t^1_u(q))_{q\in N}$ and 
$(t^1_s(q))_{q\in N}$, with limits $t^1_u$ and $t^1_s$ respectively, such that for $q$ large enough
$$
\ga_q(t^1_u(q))\in \Sig^u,\qquad \ga_q(t^1_s(q))\in \Sig^s.
$$
Obviously $[t^1_u(q),t^1_s(q)]\subset [0,q]$ for $q$ large enough. Moreover, $\big(\ga^q(t+\ka^1_q)\big)$ converges
uniformly to $\om^1(t)$ on $[t^1_u-\de,t^1_s+\de]$ for $\de$ small enough.

\vskip2mm

$\bullet$ By definition of $\jO$, there exists a minimal time $t^2_u(q)>t^1_s(q)$ such that $\ga^q(t^2_u(q))\in \Sig^u$. 
Up to extraction of a subsequence one can assume that $\ga^q(t^2_u(q))$ converges, and the limit necessarily belongs
to some homoclinic orbit $\Om^j$, so that there exists $t^j_u$ such that
$$
\lim_{q\to\infty} \ga^q(t^2_u(q))=\om^j(t^j_u).
$$

\vskip2mm

$\bullet$ The previous process may be continued and provides us, for $q$ large enough, with a finite sequence of consecutive intervals
$$
0\leq [t^1_u(q),t^1_s(q)]<[t^2_u(q),t^2_s(q)]<\cdots<[t^\ell_u(q),t^\ell_s(q)]\leq q
$$
such that the intersection of the orbit of $\ga^q$ with the lift of $\jO$ to $T\T^2$ is the union
$$
\bigcup_{1\leq j\leq\ell} \ga^q\big([t^j_u(q),t^j_s(q)]\big),
$$
and such that (up to extraction) each sequence $\big(\ga^q(t^j_u(q))\big)$ and $\big(\ga^q(t^j_s(q))\big)$ is convergent.
Moreover,  for $1\leq j\leq \ell$ the limits 
$$
\lim_{q\to+\infty}\ga^q(t^{j-1}_s(q))
\qquad\textrm{and}\qquad
\lim_{q\to+\infty}\ga^q(t^j_u(q))
$$
belong to the same homoclinic orbit (with the cyclic convention $0=\ell$).
Note that $\ell$ is larger or equal to the number of homoclinic orbits (some of them may be shadowed more than once).


 \subsubsection{The positive energies} 
To conclude the proof of Proposition \ref{prop:polyhom} it only remains to show the existence of a convergent
sequence of minimizing periodic orbits with positive energy to the previous polyhomoclinic orbit.
By Conditions $(D)$, starting with a sequence
$(\Ga_n)$ of periodic orbits with energy $\geq 0$ converging to the polyhomoclinic orbit $\Om$, one can slightly perturb 
each of them to get another sequence $(\Ga^+_n)$ of periodic orbits {\em with positive energy} 
$1/2^n$ close to the initial ones in the $C^1$ topology. This new sequence obviouly converges to the same
polyhomoclinic orbit in the sense of  Definition~\ref{def:conv}, which concludes the proof.


 \subsubsection{The simple homoclinic orbits}
We are now in a position to prove  Lemma~\ref{lem:simppos}. Assume that $C$ satisfies Conditions $(D)$. Following $(D_4)$
consider a homoclinic orbit$\Om$ whose amended action is strictly minimal among the amended actions of the 
homoclinic orbits. Let $A$ be the lower bound of these latter actions.
Let $c\sim m$ be the homology class. Then for $q$ large enough, by properly chosing an initial condition on $\Om$
close enough to $O$, one can produce a $q$--periodic sequence whose rotation vector is $m$ and whose action is
smaller that $A$. As a consequence, a minimizing sequence in $\jX_m^q$ also has an action smaller that $A$. 
By semicontinuity and the previous section, this proves that the associated periodic orbits converge to $\Om$, hence
$\Om$ is positive.


\section{The Hamiltonian Birhoff-Smale theorem}\label{sec:proofBS}

We give here a complete proof of Theorems \ref{thm:hypdyn1} and \ref{thm:hypdyn2}, which follows the lines of 
 \cite{Mar98} with more details. Since in the Hamiltonian setting the (Lagrangian) invariant manifolds cannot be 
transverse in the ambient space,  we have to restrict the system to energy levels and consider Poincar\'e sections for the flow. 
This process is {\em not} general and depends on the spectrum of the equilibrium point as well as on the ``disposition'' of 
the homoclinic orbits. Our result is related to the study of Shilnikov and Turaev \cite{ST89}, but our assumptions are different. We 
moreover need to get a very precise localization of the horseshoes we will construct,  which we were unable to extract from 
\cite{ST89}. To this aim we find very helpful the existence of a proper coordinate system for the fixed point. Another
approach may be found in \cite{KZ}.


\subsection{The setting}
\paraga We consider a $C^\infty$ Hamiltonian system $H$ on $\A^2=T^*\T^2$ with a hyperbolic fixed point~$O$.
We assume that there exists  a neighborhood of $O$ endowed with a $C^\infty$  symplectic coordinate 
system  $(u_1,u_2,s_1,s_2)$, 
with values in some ball $B$ centered at $0$ in $\R^4$, in which $H$ takes the normal form
\begin{equation}\label{eq:normform1}
H(u,s)=\la_1 u_1s_1+\la_2 u_2s_2+R(u_1s_1,u_2s_2),
\end{equation}
with $\la_1>\la_2>0$, $R(0,0)=0$ and $D_{(0,0)}R=0$ (we do not assume any equivariance condition at this point, 
see (\ref{eq:equivariance})).  The products $u_is_i$ are local first integrals of the system and in $B$ the vector field $X^H$ reads
\begin{equation}\label{eq:normform2}
X^H\ \left\vert
\begin{array}{lll}
\dot u_i&=&\la_i(u_1s_1,u_2s_2)\,u_i\\
\dot s_i&=&-\la_i(u_1s_1,u_2s_2)\,s_i\\
\end{array}
\right.
\end{equation}
with
\begin{equation}\label{eq:normform3}
\la_i(u_1s_1,u_2s_2)=\la_i+\partial_{x_i}R(u_1s_1,u_2s_2), \qquad i=1,2.
\end{equation}
We can assume that for $(u,s)\in B$:
\begin{equation}\label{eq:major}
\la_i(u_1s_1,u_2s_2)\geq \ov\la_i>0,\qquad i\in\{1,2\},
\end{equation}
and
\begin{equation}\label{eq:estlamda1}
\ha\La\geq\frac{\la_1(u_1s_1,u_2s_2)}{\la_2(u_1s_1,u_2s_2)}\geq\ov\La>1,
\end{equation}
for some suitable constants $\ov\La$ and $\ha\La$,
which moreover satisfy
\begin{equation}\label{eq:estlamb2}
\ha\La-\ov\La<2(\ov\La-1).
\end{equation}
The local stable and unstable manifolds $W^s_\ell $ and $W^u_\ell$ of $O$ in $B$ are straightened:
$$
W^u_\ell=\{s=0\}, \qquad W^s_\ell=\{u=0\}
$$ 
and the Hamiltonian vector field 
on these  manifolds is purely linear. Some stable and unstable orbits are depicted in Figure 1.
As in Section 2,  we introduce the subsets
$$
W^{ss}_\ell=\{s_2=0\},\  \
W^{sc}_\ell=\{s_1=0\},\ \
W^{uu}_\ell=\{u_2=0\},\ \
W^{uc}_\ell=\{u_1=0\}.
$$
Observe that their germs at $O$
are independent of the normalization.  Indeed, $W_\ell^{ss}$ and $W_\ell^{uu}$ are
the strong (local) stable and unstable manifolds of $O$, while 
$W^{sc}_\ell=\{s_1=0\}$ and $W^{uc}_\ell=\{u_1=0\}$ are the only $C^{\infty}$ invariant
lines  through the fixed points which are contained in in $W_\ell^{ss}$ and $W_\ell^{uu}$
and transverse to $W_\ell^{ss}$ and $W_\ell^{uu}$. Note also that these objects depend continuously
on the Hamiltonian $H$: this comes from the usual (parametrized) Grobman-Hartman theorem.

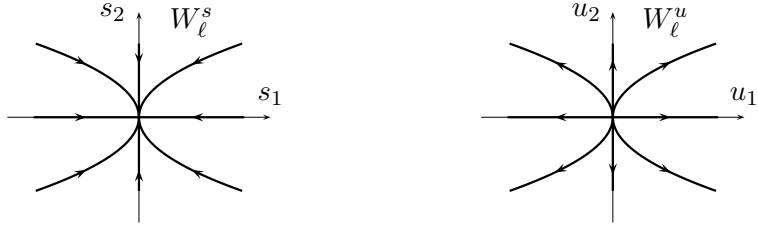
\begin{figure}[h]
\begin{center}
\begin{pspicture}(0cm,4cm)
\psset{xunit=.7cm,yunit=.7cm}
\rput(4.5,3){
\psline [linewidth=0.1mm]{->}(-2.5,0)(2.5,0)
\psline [linewidth=0.1mm]{->}(0,-2)(0,2) 
\parametricplot[linewidth=0.3mm] {0}{1.4}{t  t  mul t} 
\parametricplot[linewidth=0.3mm] {0}{1.4}{t  t  mul t -1 mul} 
\parametricplot[linewidth=0.3mm] {0}{1.4}{t  t  mul -1 mul t -1 mul} 
\parametricplot[linewidth=0.3mm] {0}{1.4}{t  t  mul -1 mul t }
\psline [linewidth=0.3mm](0,-1.4)(0,1.4) 
\psline [linewidth=0.3mm](-2,0)(2,0)
\psline [linewidth=0.3mm]{->}(0,1)(0,1.1)
\psline [linewidth=0.3mm]{->}(0,-1)(0,-1.1)
\psline [linewidth=0.3mm]{->}(1,0)(1.1,0)
\psline [linewidth=0.3mm]{->}(-1,0)(-1.1,0)
\psline [linewidth=0.3mm]{->}(1,1)(1.1,1.05) 
\psline [linewidth=0.3mm]{->}(1,-1)(1.1,-1.05)
\psline [linewidth=0.3mm]{->}(-1,1)(-1.1,1.05)
\psline [linewidth=0.3mm]{->}(-1,-1)(-1.1,-1.05)
\rput(1,1.8){$W^u_\ell$}
\rput(2.5,.4){$u_1$}
\rput(-.5,2){$u_2$}
}
\rput(-4.5,3){
\psline [linewidth=0.1mm]{->}(-2.5,0)(2.5,0)
\psline [linewidth=0.1mm]{->}(0,-2)(0,2) 
\parametricplot[linewidth=0.3mm] {0}{1.4}{t  t  mul t} 
\parametricplot[linewidth=0.3mm] {0}{1.4}{t  t  mul t -1 mul} 
\parametricplot[linewidth=0.3mm] {0}{1.4}{t  t  mul -1 mul t -1 mul} 
\parametricplot[linewidth=0.3mm] {0}{1.4}{t  t  mul -1 mul t }
\psline [linewidth=0.3mm](0,-1.4)(0,1.4) 
\psline [linewidth=0.3mm](-2,0)(2,0)
\psline [linewidth=0.3mm]{<-}(0,1)(0,1.1)
\psline [linewidth=0.3mm]{<-}(0,-1)(0,-1.1)
\psline [linewidth=0.3mm]{<-}(1,0)(1.1,0)
\psline [linewidth=0.3mm]{<-}(-1,0)(-1.1,0)
\psline [linewidth=0.3mm]{<-}(1,1)(1.1,1.05) 
\psline [linewidth=0.3mm]{<-}(1,-1)(1.1,-1.05)
\psline [linewidth=0.3mm]{<-}(-1,1)(-1.1,1.05)
\psline [linewidth=0.3mm]{<-}(-1,-1)(-1.1,-1.05)
\rput(1,1.8){$W^s_\ell$}
\rput(2.5,.4){$s_1$}
\rput(-.5,2){$s_2$}
}
\end{pspicture}
\vskip-10mm
\caption{The flows on $W^s_\ell$ and $W^u_\ell$.}
\end{center}
\end{figure}

\paraga  Given $\eps>0$ small enough and $\sig\in\{-1,+1\}$,  
as in Section 5  we denote by $B(\eps)$ the ball of $\R^4$
centered at $0$ with radius $\eps$ for the Max norm and  we introduce 
the sections
\beq
\Sig_{\sig}^u[\eps]=\{(u,s)\in \ov B(\eps)\mid u_2=\sig\eps\},\qquad
\Sig_{\sig}^s[\eps]=\{(u,s)\in \ov B(\eps)\mid s_2=\sig\eps\},
\eeq
and
\beq
\Sig_{\sig}^u[\eps,e]=\Sig_{\sig}^u[\eps]\cap (C_{\vert B})\inv(e),\qquad
\Sig_{\sig}^s[\eps,e]=\Sig_{\sig}^s[\eps]\cap (C_{\vert B})\inv(e).
\eeq
To define suitable coordinates on these latter sets, we write $x_i=u_is_i$, $i=1,2$ and we fix 
two intervals $X_1=\,]-\ha x_1,\ha x_1[$, $E=\,]-\ha e,\ha e\,[$, together with a neighborhood $X$ of $0$
in $\R^2$ such that the equation 
\beq\label{eq:energie}
H(x_1,x_2)=e, \qquad (x_1,x_2)\in X,\quad e\in E,
\eeq
is equivalent to 
\begin{equation}\label{eq:u2s20}
x_2=\chi(x_1,e), \qquad (x_1,e)\in X_1\times E,
\end{equation}
where $\chi$ is a smooth function  on $X_1\times E$.
Therefore
\begin{equation}\label{eq:deriv}
\chi(x_1,e)=\frac{1}{\la_2}(e-\la_1 x_1)+\ov \chi(x_1,e),\qquad   \ov \chi(0,0)=0,
\quad D_{(0,0)}\ov \chi=0.
\end{equation}
As a consequence, there exists $\ha \eps>0$ such that for $0<\eps<\ha\eps$,  
the equation of $\Sig_{\sig}^s[\eps,e]$ reads
\begin{equation}\label{eq:U2}
u_2=\frac{1}{\sig\eps}\chi(u_1s_1,e):=\phi_{\eps,{\sig}}(u_1s_1,e),\qquad \norm{(u_1,s_1)}_\infty\leq\eps,
\end{equation}
while the equation of  $\Sig_{\sig}^u[\eps,e]$ reads
\begin{equation}\label{eq:S2}
s_2=\phi_{\eps,{\sig}}(u_1s_1,e),\qquad \norm{(u_1,s_1)}_\infty\leq\eps.
\end{equation}
We can now introduce our coordinates. 
Since we implicitely use the conservation of energy through the choice of our sections, we can take 
advantage of only one of the first integrals $u_is_i$ and we will choose the product $u_1s_1$. 

\begin{itemize}

\item On the subset $\Sig_{\sig}^{s*}[\eps,e]=\big\{(u,s)\in \Sig_{\sig}^s[\eps,e]\mid s_1\neq0\big\}$
we define $(x_s,y_s)$ by
\begin{equation}\label{eq:coords0}
x_s = s_1,\qquad y_s = u_1s_1.
\end{equation}
The full set of coordinates of the point $m=(x_s,y_s)\in\Sig_{\sig}^{s*}[\eps,e]$ reads 
\begin{equation}\label{eq:coords}
m=\Big(u_1=\frac{y_s}{x_s},\ u_2=\phi_{\eps,\sig}(y_s,e),\ s_1=x_s,\  s_2=\sig\eps\Big).
\end{equation}

\item Similarly, on $\Sig_{\sig}^{u*}[\eps,e]=\big\{(u,s)\in \Sig_{\sig}^u[\eps,e]\mid u_1\neq0\big\}$, we set
\begin{equation}\label{eq:coordu0}
x_u = u_1,\qquad y_u = u_1s_1,
\end{equation}
so that the coordinates of  $m=(x_s,y_s)\in\Sig_{\sig}^{u*}[\eps,e]$ read
\begin{equation}\label{eq:coordu}
m=\Big(u_1=x_u,\  u_2=\sig\eps,\  s_1=\frac{y_u}{x_u},\ s_2=\phi_{\eps,\sig}(y_u,e)\Big).
\end{equation}
\end{itemize}

Note  that the intersections of the local invariant manifolds $W^u_\ell$ and $W^s_\ell$
with the sections $\Sig^u[\eps,0]$ and $\Sig^s[\eps,0]$ admit the simple equations
$y_u=0$ and $y_s=0$ respectively.

\paraga The following lemma is a simple remark which will enable us to properly localize our construction of the horseshoes.
Note that, due to the form of the flow on $W^s_\ell$, for $\eps>0$ small enough (in particular $\eps<\ha\eps$ defined above)
 $\Sig^s[\eps]$ is an {\em entrance section} for any
orbit $\Ga$ in $W^s_\ell$, in the sense that there exists a first intersection point in $\Ga\cap \Sig^s[\eps]$, according to the
flow induced orientation of $\Ga$ (see the proof below). We call this point the entrance point of $\Ga$ relatively to $\Sig^s[\eps]$. 
We define analogously the exit point of an orbit $\Ga\subset W^u$

\begin{lemma}\label{lem:majcoord}
Let $\Om$ be an orbit in $W^s\setm(W^{ss}_\ell\cup W^{sc}_\ell)$.
Then, if $\eps_a>0$ is small enough,  for $0<\eps\leq\eps_a$ the entrance point  $b$ of $\Om$ 
relatively to $\Sig^s[\eps]$ belongs to a well-defined subset  $\Sig^{s*}_{\sig}[\eps]$. 
Its coordinates (\ref{eq:coords0}) read $b=(\eta,0)$ with
\beq\label{eq:estiment}
\abs{\eta}\leq \pdemi\eps^{\ov \La},\qquad 1\leq i\leq i^*.
\eeq
One has a similar statement when orbits $\Om\subset W^u\setm(W^{us}_\ell\cup W^{uc}_\ell)$,
in this case its exit point $a=(\xi,0)\in \Sig^{u*}_{\sig}[\eps]$ satisfies
\beq\label{eq:estimex}
\abs{\xi}\leq \pdemi\eps^{\ov \La},\qquad 1\leq i\leq i^*.
\eeq
\end{lemma}

\begin{proof} It is of course enough to prove the first claim.  
Since $\Om\subset W^s\setm(W^{ss}_\ell\cup W^{sc}_\ell)$,
$\Om\cap W^s_\ell$ admits an equation of the form
$$
s_1=c\,\abs{s_2}^{\la_1/\la_2},\qquad u=0,
$$
with $c\in\R^*$. One therefore sees that  $\Sig^s[\eps]$ is an entrance section when $\eps$ is small enough.
Therefore the entrance point $b$ is well-defined and belongs to a subset $\Sig^{s*}_{\sig}[\eps]$ since
$c\neq0$. Moreover $\eta=c \eps^{\la_1/\la_2}$, so our claims then easily follow from the condition  
$\ov\La<\la_1/\la_2$. 
\end{proof}

Given a finite set $(\Om_i)_{1\leq i\leq\ell}$ of orbits homoclinic to $O$, which do not intersect the exceptional 
set $W^{ss}_\ell\cup W^{sc}_\ell\cup W^{us}_\ell\cup W^{uc}_\ell$, we say that $\eps_a>0$ is admissible
for $(\Om_i)$ when it satisfies both conditions of Lemma \ref{lem:majcoord}.


\subsection{The Poincar\'e return map}

Throughout this section we fix two compatible polyhomoclinic orbits $\Om^0=(\Om^0_1,\ldots,\Om^0_{\ell^0})$ and 
$\Om^1=(\Om^1_1,\ldots,\Om^1_{\ell^1})$ and we fix an admissible $\eps_a>0$. 
The Poincar\'e map $\Phi$ will be the composition of a flow-induced outer maps $\Phi_{out}$  along the homoclinic orbits
with an inner  map $\Phi_{in}$, for which we will use the normal form~(\ref{eq:normform1}).
Given a positive $\eps<\eps_a$,  we denote by $a_i^\nu$ and $b_i^\nu$ the exit and  entrance point of  $\Om_i^\nu$ 
relatively to the sections  $\Sig^u[\eps]$ and $\Sig^s[\eps]$.  


\subsubsection{The outer map.} 
The outer map $\Phi_{out}$  will be defined  over the union of small  3-dimensional neighborhoods $\cR_i^\nu$ of the 
points  $a^\nu_i$ in $\Sig^u[\eps]$ and will take its values in the union of 3-dimensional neighborhoods of the points 
$b^\nu_i$ in $\Sig^s[\eps]$. 

\paraga  In the $(x_u,y_u)$--coordinates, we set
\begin{equation}\label{eq:coordainu}
a_i^\nu=(\xi_i^\nu,0),
\end{equation}
so that in particular $\xi_i^\nu\neq0$. 
To define the neighborhoods $\cR_i^\nu$, we first need to introduce
the following notation for $2$-dimensional rectangles in $\Sig^u[\eps,e]$.

\begin{itemize}
\item
For $\xi\in\,]-\eps,\eps[\setm\{0\}$ and for $0<\de<\abs{\xi}$ and $\de'>0$, we set
\begin{equation}\label{eq:rectangle0}
R[\xi,\de,\de',e]=\{(x_u,y_u)\in\Sig^u[\eps,e]\mid \abs{x_u-\xi}\leq \de,\ \abs{y_u}\leq \de'\}.
\end{equation}
\end{itemize}
 
The neighborhood $\cR_i^\nu\subset \Sig^u[\eps]$ will be the union of a one-parameter
family of such rectangles:
\begin{equation}\label{eq:gloreg}
\cR_i^\nu=\bigcup_{\abs{e}\leq e_0} R[\xi_i^\nu,\de,\de',e].
\end{equation}
where the parameters $\de,\de'$ will be chosen independently of $e$ (and of $i$ and $\nu)$. The determination
of $e_0,\de,\de'$ will necessitate several steps which will be made explicit in the following.

\paraga Let us introduce the first constraint on $e_0$ and $\de,\de'$. 
By the transversality condition of the invariant manifolds along the homoclinic orbits, a small (one-dimensional) 
segment of $\Sig^u[\eps,0]\cap W^u_\ell=\{y_u=0\}$ around $a_i^\nu$ is sent by $\Phi_{out}$ on a small curve in $\Sig^s[\eps,0]$ 
which is transverse to $\Sig^s[\eps,0]\cap W^s_\ell=\{y_s=0\}$ at $b_i^\nu$.   
Now  the image of $R[\xi_i^\nu,\de,\de',e]$ by $\Phi_{out}$
is the union of the images of the horizontals $\{y_u={\rm cte}\}$ of the rectangle.  We  require the following condition 
(see Figure \ref{fig:rectout}).

\begin{itemize}
\item The energy $e_0>0$ and the constants $\de,\de'$ are small enough so that for $\abs{e}<e_0$, the images of
the horizontals of $R[\xi_i^\nu,\de,\de',e]$ by $\Phi_{out}$ transversely intersect the line
$\{y_s=0\}\subset\Sig^s[\eps,0]$.
\end{itemize}

\begin{figure}[h]
\begin{center}
\begin{pspicture}(0cm,3cm)
\rput(-5.5,1.5){
\psline [linewidth=0.1mm]{->}(-.5,0)(3,0)
\psline [linewidth=0.1mm]{->}(0,-.5)(0,1.5) 
\rput(3,.2){$x_u$}
\rput(.3,1.5){$y_u$}
\rput(2.5,1.3){$\Sig^u[\eps,0]$}
\rput(0,-0.15){
\psline [linewidth=0.3mm](1,0)(2,0)
\psline [linewidth=0.3mm](1,0.3)(2,0.3)
\psline [linewidth=0.3mm](1,0)(1,0.3)
\psline [linewidth=0.3mm](2,0)(2,0.3)
\psline [linewidth=0.1mm](1,0.06)(2,0.06)
\psline [linewidth=0.1mm](1,0.12)(2,0.12)
\psline [linewidth=0.1mm](1,0.18)(2,0.18)
\psline [linewidth=0.1mm](1,0.24)(2,0.24)
}
\pscircle(1.5,0){.07}
\rput(1.5,-.45){$a_i^\nu$}
\psline{->}(4,.5)(7,.5)
\rput(5.5,.8){$\Phi_{out}$}
}
\rput(2.5,1.5){
\psline [linewidth=0.1mm]{->}(-.5,0)(3,0)
\psline [linewidth=0.1mm]{->}(0,-.5)(0,1.5)
\rput(3,.2){$x_s$}
\rput(.3,1.5){$y_s$} 
\rput(2.5,1.3){$\Sig^s[\eps,0]$}
\rput(0.11,-.15){
\psline [linewidth=0.3mm](1,-.5)(1.3,.5)
\psline [linewidth=0.3mm](1.3,-0.2)(1.6,0.8)
\psline [linewidth=0.3mm](1,-.5)(1.3,-.2)
\psline [linewidth=0.3mm](1.3,.5)(1.6,.8)
\psline [linewidth=0.1mm](1.06,-0.44)(1.36,.56)
\psline [linewidth=0.1mm](1.12,-0.38)(1.42,0.62)
\psline [linewidth=0.1mm](1.18,-0.32)(1.48,0.68)
\psline [linewidth=0.1mm](1.24,-0.26)(1.54,0.74)
}
\pscircle(1.42,0){.07}
\rput(1.85,-.3){$b_i^\nu$}
}
\end{pspicture}
\vskip-8mm
\caption{The image of a rectangle in $\Sig^u[\eps,0]$ under the map $\Phi_{out}$.}\label{fig:rectout}
\end{center}
\end{figure}
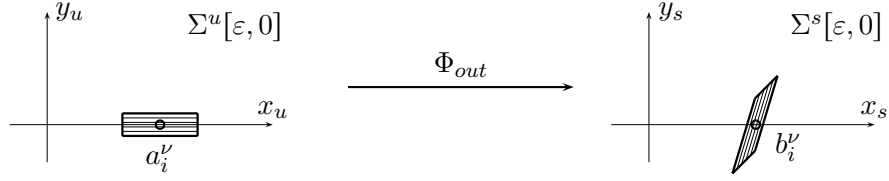
\vspace{-5mm}


\subsubsection{The inner map} \label{ssec:innermap}
The inner map $\Phi_{in}$ sends a point $m\in\Sig^s[\eps]\setm W^s(O)$ on the first intersection point of its orbit
with the section $\Sig^u[\eps]$, provided moreover that the segment of orbit defined by these two points stays inside 
the coordinate ball $B$. 


\begin{lemma} \label{lem:phiin} For $0<\eps<\eps_a$ and $\abs{e}<\ha e$ (see (\ref{eq:energie})),  let us set
\begin{equation}\label{eq:coordphi1}
X_u(x_s,y_s)=\eps^{\La(y_s,e)}\,\frac{y_s}{x_s}\abs{\,\phi_{\eps,\sig}(y_s,e)}^{-\La(y_s,e)},\qquad 
\La(y_s,e)=\frac{\la_1(y_s,\chi(y_s,e))}{\la_2(y_s,\chi(y_s,e))},
\end{equation}
where $\chi$ was introduced in (\ref{eq:u2s20}) and $\phi_{\eps\sig}=\frac{1}{\eps,\sig}\chi$ as in (\ref{eq:U2}).
Then 
\begin{equation}\label{eq:defdom}
\jD[\eps,e]=\{(x_s,y_s)\in \Sig^s[\eps,e]\mid \abs{X_u(x_s,y_s)}<\eps\}.
\end{equation} 
and 
for  $m=(x_s,y_s)\in \jD[\eps,e]$, $\Phi_{in}(m)\in\Sig_{\sig'}^u[\eps,e]$ with 
\begin{equation}\label{eq:coordcomp}
\sig'=\sgn\big(\phi_{\eps,\sig}(y_s,e)\big),\qquad
\Phi_{in}(m)=\Big(X_u(x_s,y_s),Y_u(x_s,y_s)=y_s\Big).
\end{equation}
\end{lemma}

\vskip2mm

\begin{proof} 
Let  $m=(u_1^0,u_2^0,s_1^0,s_2^0)\in B$.  Then if the orbit of $m$ stays inside $B$, its associated solution reads
\beq\label{eq:linflow}
u_i(t)=u^0_i\,e^{\la(u_1^0s_1^0,u_2^0s_2^0)t},\qquad s_i(t)=s^0_i\,e^{-\la(u_1^0s_1^0,u_2^0s_2^0)t}, \qquad i=1,2.
\eeq
Assume moreover that $m\in\Sig^s_\sig[\eps,e]$, with coordinates $(x_s,y_s)$ in this section, so that $y_s=u_1^0s_1^0$.
Assume that $u_2^0=\phi_{\eps,\sig}(y_s,e)$ is nonzero. 
Then the transition time $\tau(m)$ to reach the section $\abs{u_2}=\eps$  is well defined and reads
$$
\tau(m)=\frac{1}{\La_2(y_s,e)}\Ln\frac{\eps}{\abs{\,\phi_{\eps,\sig}(y_s,e)}},
$$
which immediately yields the equality
$$
u_1\big(\tau(m)\big)=X_u(x_s,y_s)
$$
by (\ref{eq:linflow}) and (\ref{eq:coordu0}), provided that the orbit stays inside $\norm{(u,s)}_\infty\leq \eps$.
Since the $u$-coordinates increase while the $s$-coordinates decrease along the orbits contained in this
domain, one easily checks that the inequality $\abs{X_u(x_s,y_s)}\leq \eps$ is a necessary and sufficient 
condition for $m$ being in $\jD[\eps,e]$.  This proves 
(\ref{eq:defdom}) and the expression of $\Phi_{in}$ directly follows from the previous remarks.
Finally, since the sign of the coordinates is preserved by the
flow, one immediately sees that $\sig'=\sgn\big(\phi_{\eps,\sig}(y_s,e)\big)$
\end{proof}

\vskip2mm

\begin{figure}[h]\label{fig:philin1}
\begin{center}
\begin{pspicture}(0cm,3.2cm)
\rput(-6.5,1.5){
\psframe[fillstyle=solid,fillcolor=lightgray](0,-.2)(2,.2)
\psline [linewidth=0.1mm]{->}(-.5,0)(3,0)
\psline [linewidth=0.1mm]{->}(0,-1.5)(0,1.5) 
\psline [linewidth=0.1mm](0,.3)(3,.3)
\psline [linewidth=0.3mm,linestyle=dashed](.8,-1.3)(1,0) 
\psline [linewidth=0.3mm](1,0)(1.2,1.3) 
\pscircle [linewidth=0.2mm,fillcolor=white](1,0){0.05}
\rput(-.3,.3){$\frac{e}{\la_1}$}
\rput(3,-.2){$x_s$}
\rput(.3,1.5){$y_s$}
\rput(2.5,1.3){$\Sig^s_{\sig}[\eps,e]$}
\psline{->}(4,0)(6.5,0)
\rput(5.3,.3){$\Phi_{in}$}
}
\rput(3,1.5){
\psframe[fillstyle=solid,fillcolor=lightgray](-1,-.2)(1,.2)
\psline [linewidth=0.1mm]{->}(-2,0)(4,0)
\psline [linewidth=0.1mm]{->}(0,-1.5)(0,1.5) 
\psline [linewidth=0.1mm](-2,.3)(4,.3)
\parametricplot[linewidth=0.3mm] {0}{.068}{t  t -.2 add div t -.2 add div    t} 
\parametricplot[linewidth=0.3mm,linestyle=dashed] {0}{1.3}{t  -1 mul  t .15 add div t .15 add div   t -1 mul} 
\rput(4,-.2){$x_u$}
\rput(.3,1.5){$y_u$}
\rput(2.5,1.3){$\Sig^u_{\sig}[\eps,e]$}
}
\end{pspicture}
\caption{Case $e>0$.}
\end{center}
\end{figure}

\begin{figure}[h]\label{fig:philin2}
\begin{center}
\begin{pspicture}(0cm,3.2cm)
\rput(-6.5,1.5){
\psframe[fillstyle=solid,fillcolor=lightgray](0,-.2)(2,.2)
\psline [linewidth=0.1mm]{->}(-.5,0)(3,0)
\psline [linewidth=0.1mm]{->}(0,-1.5)(0,1.5) 
\psline [linewidth=0.1mm](0,-.3)(3,-.3)
\psline [linewidth=0.3mm,linestyle=dashed](.8,-1.3)(1,0) 
\psline [linewidth=0.3mm](1,0)(1.2,1.3) 
\pscircle [linewidth=0.2mm,fillcolor=white](1,0){0.05}
\rput(-.3,-.3){$\frac{e}{\la_1}$}
\rput(3,.2){$x_s$}
\rput(.3,1.5){$y_s$}
\rput(2.5,1.3){$\Sig^s_{\sig}[\eps,e]$}
\psline{->}(4,0)(6.5,0)
\rput(5.3,.3){$\Phi_{in}$}
}
\rput(5,1.5){
\psframe[fillstyle=solid,fillcolor=lightgray](-1,-.2)(1,.2)
\psline [linewidth=0.1mm]{->}(-4,0)(2,0)
\psline [linewidth=0.1mm]{->}(0,-1.5)(0,1.5) 
\psline [linewidth=0.1mm](-4,-.3)(2,-.3)
\parametricplot[linewidth=0.3mm,linestyle=dashed] {0}{.068}{t -1 mul  t -.2 add div t -.2 add div    t  -1 mul} 
\parametricplot[linewidth=0.3mm] {0}{1}{t    t .15 add div t .15 add div   t } 
\rput(2,.2){$x_u$}
\rput(.3,1.5){$y_u$}
\rput(-2.5,1.3){$\Sig^u_{-\sig}[\eps,e]$}
}
\end{pspicture}
\caption{Case $e<0$.}
\end{center}
\end{figure}

Figures \ref{fig:philin1} and \ref{fig:philin2} 
depict the  image of a  transverse segment under the map $\Phi_{in}$. We limited ourselves to the part of its image which
will prove useful in the following constructions. The ``useful domain'' in the section $\Sig^u$ is delimited by 
gray rectangles.

\vskip2mm

We will also need the following result for $\Phi_{in}\inv$, whose proof is analogous to that of Lemma~\ref{lem:phiin}.

\begin{lemma} \label{lem:phiininv} For $0<\eps<\eps_a$ and $\abs{e}<\ha e$,  we set
\begin{equation}\label{eq:coordphiinv1}
X_s(x_u,y_u)=\eps^{\La(y_u,e)}\,\frac{y_u}{x_u}\abs{\,\phi_{\eps,\sig}(y_u,e)}^{-\La(y_u,e)},\qquad 
\La(y_u,e)=\frac{\la_1(y_u,\chi(y_u,e))}{\la_2(y_u,\chi(y_u,e))},
\end{equation}
Then the domain of $\Phi_{in}\inv$ reads
\begin{equation}\label{eq:defdominv}
\jD\inv[\eps,e]=\{(x_u,y_u)\in \Sig^s[\eps,e]\mid \abs{X_u(x_u,y_u)}<\eps\}\subset \Sig^u[\eps,e],
\end{equation} 
and 
for  $m=(x_u,y_u)\in \jD\inv[\eps,e]$, $\Phi_{in}\inv(m)\in\Sig_{\sig'}^s[\eps,e]$ with $\sig'=\sgn\big(\phi_{\eps,\sig}(y_u,e)\big)$.
Finally
\begin{equation}\label{eq:coordcompinv}
\Phi_{in}\inv(m)=\Big(X_s(x_u,y_u),Y_s(x_u,y_u)=y_u\Big).
\end{equation}
\end{lemma}

The behavior of $\Phi_{in}\inv$ is therefore immediately deduced from the one of $\Phi_{in}$, by a simple inversion of
the subscripts $s$ and $u$.


\subsubsection{A picture of the Poincar\'e map $\Phi=\Phi_{in}\circ\Phi_{out}$}
Here  we limit ourselves to the case of two simple homoclinic orbits $\Om^0$ and $\Om^1$, with exit points
in the same section $\Sig^u_\sig[\eps,0]$ and  entrance points in the same section $\Sig^s_\sig$, we moreover
assume $e>0$. With the notation of (\ref{eq:gloreg}), we set
$R^\nu=R[\xi^\nu,\de,\de',e]$ for $\nu=0,1$, where as usual $(\xi^\nu,0)$ stands for the coordinates of
the exit point of $\Om^\nu$.
Gathering the previous descriptions, one gets the following picture for the images of rectangles 
$R^\nu$  (we have limited the images $\Phi(R^\nu)$ to their ``useful parts'').

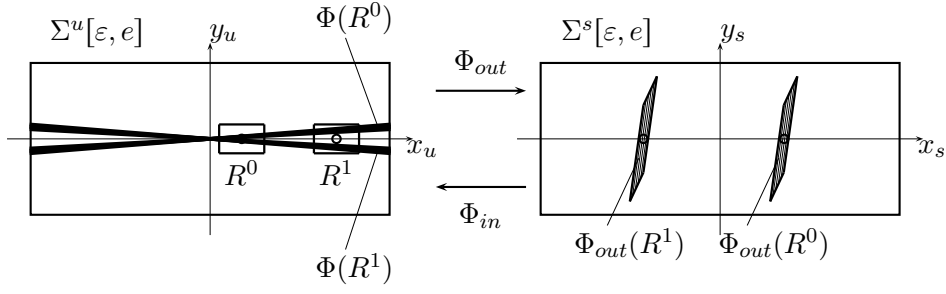
\begin{figure}[h]
\begin{center}
\begin{pspicture}(0cm,3.9cm)
\psset{xunit=.6cm,yunit=1.28cm}
\rput(-5.8,1.5){
\psframe(-4,-.8)(4,.8)
\psline [linewidth=0.1mm]{->}(-4.5,0)(4.5,0)
\psline [linewidth=0.1mm]{->}(0,-1)(0,1) 
\psline [linewidth=0.5mm](0,0)(4,0.1)
\psline [linewidth=0.5mm](0,0)(4,0.12)
\psline [linewidth=0.5mm](0,0)(4,0.13)
\psline [linewidth=0.5mm](0,0)(4,0.15)
\psline [linewidth=0.5mm](0,0)(4,-0.1)
\psline [linewidth=0.5mm](0,0)(4,-0.12)
\psline [linewidth=0.5mm](0,0)(4,-0.13)
\psline [linewidth=0.5mm](0,0)(4,-0.15)
\psline [linewidth=0.5mm](0,0)(-4,0.1)
\psline [linewidth=0.5mm](0,0)(-4,0.12)
\psline [linewidth=0.5mm](0,0)(-4,0.13)
\psline [linewidth=0.5mm](0,0)(-4,0.15)
\psline [linewidth=0.5mm](0,0)(-4,-0.1)
\psline [linewidth=0.5mm](0,0)(-4,-0.12)
\psline [linewidth=0.5mm](0,0)(-4,-0.13)
\psline [linewidth=0.5mm](0,0)(-4,-0.15)
\rput(4.7,-.1){$x_u$}
\rput(.3,1.1){$y_u$}
\rput(-2.5,1.1){$\Sig^u[\eps,e]$}
\rput(-.8,-0.15){
\psline [linewidth=0.3mm](1,0)(2,0)
\psline [linewidth=0.3mm](1,0.3)(2,0.3)
\psline [linewidth=0.3mm](1,0)(1,0.3)
\psline [linewidth=0.3mm](2,0)(2,0.3)
\rput(4,1.4){$\Phi(R^0)$}
\rput(4,-1.2){$\Phi(R^1)$}
\psline[linewidth=0.1mm](3.8,1.2)(4.5,.3)
\psline[linewidth=0.1mm](3.8,-1)(4.5,0)
}
\rput(1.3,-0.15){
\psline [linewidth=0.3mm](1,0)(2,0)
\psline [linewidth=0.3mm](1,0.3)(2,0.3)
\psline [linewidth=0.3mm](1,0)(1,0.3)
\psline [linewidth=0.3mm](2,0)(2,0.3)
}
\pscircle(.7,0){.07}
\pscircle(2.8,0){.07}
\rput(.7,-.35){$R^0$}
\rput(2.8,-.35){$R^1$}
\psline{->}(5,.5)(7,.5)
\psline{<-}(5,-.5)(7,-.5)
\rput(6,.8){$\Phi_{out}$}
\rput(6,-.8){$\Phi_{in}$}
}
\rput(5.5,1.5){
\psframe(-4,-.8)(4,.8)
\psline [linewidth=0.1mm]{->}(-4.5,0)(4.5,0)
\psline [linewidth=0.1mm]{->}(0,-1)(0,1)
\rput(4.7,-.1){$x_s$}
\rput(.3,1.1){$y_s$} 
\rput(-2.5,1.1){$\Sig^s[\eps,e]$}
\rput(0.11,-.15){
\psline [linewidth=0.3mm](1,-.5)(1.3,.5)
\psline [linewidth=0.3mm](1.3,-0.2)(1.6,0.8)
\psline [linewidth=0.3mm](1,-.5)(1.3,-.2)
\psline [linewidth=0.3mm](1.3,.5)(1.6,.8)
\psline [linewidth=0.1mm](1.06,-0.44)(1.36,.56)
\psline [linewidth=0.1mm](1.12,-0.38)(1.42,0.62)
\psline [linewidth=0.1mm](1.18,-0.32)(1.48,0.68)
\psline [linewidth=0.1mm](1.24,-0.26)(1.54,0.74)
}
\rput(-3,-.15){
\psline [linewidth=0.3mm](1,-.5)(1.3,.5)
\psline [linewidth=0.3mm](1.3,-0.2)(1.6,0.8)
\psline [linewidth=0.3mm](1,-.5)(1.3,-.2)
\psline [linewidth=0.3mm](1.3,.5)(1.6,.8)
\psline [linewidth=0.1mm](1.06,-0.44)(1.36,.56)
\psline [linewidth=0.1mm](1.12,-0.38)(1.42,0.62)
\psline [linewidth=0.1mm](1.18,-0.32)(1.48,0.68)
\psline [linewidth=0.1mm](1.24,-0.26)(1.54,0.74)
}
\pscircle(1.42,0){.07}
\pscircle(-1.7,0){.07}
\rput(1.2,-1.1){$\Phi_{out}(R^0)$}
\rput(-2,-1.1){$\Phi_{out}(R^1)$}
\psline [linewidth=0.1mm](.6,-.9)(1.3,-.2)
\psline [linewidth=0.1mm](-2.5,-.9)(-1.8,-.2)
}
\end{pspicture}
\vskip-1mm
\caption{The Poincar\'e return map.}\label{fig:Poincare}
\end{center}
\end{figure}


\subsection{Technical estimates for the inner map} 


\paraga We begin with an easy lemma on the behavior of the function $\phi_{\eps,\sig}$ introduced in (\ref{eq:U2}).

\begin{lemma}\label{lem:estimates1} Fix $\eps\in\, ]0,\ha\eps\,]$ and
$\abs{e}<\ha e$   (see  (\ref{eq:U2}), (\ref{eq:energie})). Set $\phi:=\phi_{\eps,\sig}(.,e):[-\eps^2,\eps^2]\to\R$. Then
$\phi$ is monotone, with $\sgn(\phi')=-\sig$, and vanishes at a single point $y_s^*$ such that
\begin{equation}\label{eq:zeroU}
y_s^*=\frac{e}{\la_1}+O_2(e).
\end{equation}
\end{lemma}

\begin{proof} The function 
$
\phi_{\eps,\sig}(y_s,e)
$
is well defined on $[-\eps^2,\eps^2]\times E$. The following derivative is immediately deduced 
from the implicit expression of $\chi$ in (\ref{eq:u2s20}):
$$
\partial_{y_s}\phi_{\eps,\sig}(y_s,e)=\displaystyle-\frac{\la_1+\partial_1 R(y_s,\chi(y_s,e))}
{\sig\eps\big(\la_2+\partial_2 R(y_s,\chi(y_s,e))\big)}, 
$$
By (\ref{eq:major}) and (\ref{eq:estlamda1}), this shows that  $\phi$ is monotone with $\sgn(\phi')=-\sig$ on its 
domain.  
Finally,  (\ref{eq:zeroU}) is immediate by (\ref{eq:deriv}). 
\end{proof}


\paraga We now restrict ourselves to suitable horizontal strips inside the sections, in order to get asymptotic estimates
on the various quantities involved in the construction of the horseshoes. In the following we write $\phi$ instead of 
 $\phi_{\eps,\sig}$ when the context is clear.

\begin{lemma}\label{lem:estimates} For  $0<\eps<\ha\eps$ fix $\xi(\eps)\in\ ]0,\eps[$. Then  
there exist positive constants $\ka$, $C_\eps, C'_\eps, e_0$, {\em with $\ka<1/{2\la_1}$ independent of $\eps$},
such that for $\abs{e}<e(\eps)$ and for $(x_s,y_s)$ in the domain $\jD[\eps,e]$ such that
$\abs{y_s}\leq \ka \abs{e}$, the function 
$
X_u(x_s,y_s)
$
introduced in (\ref{eq:coordphi1}) satisfies the following estimates:
\vskip-1mm
\begin{equation}\label{eq:inegX}
\abs{X_u(x_s,y_s)}\geq C_\eps\abs{y_s}\abs{e}^{-\ov\La},
\end{equation}
\begin{equation}\label{eq:inegdxX}
C'_\eps\frac{1}{\abs{x_s}^2}\abs{e}^{-\ha \La+1}\geq \abs{\partial_{x_s}X_u(x_s,y_s)}\geq C_\eps\abs{y_s}\abs{e}^{-\ov\La},
\end{equation}
and if $\abs{x_s}\geq \xi(\eps)$:
\begin{equation}\label{eq:inegdyX}
\abs{\partial_{y_s}X_u(x_s,y_s)}\geq C_\eps \abs{e}^{-\ov\La}.
\end{equation}
\end{lemma}

\vskip2mm
\begin{proof}  We can obviously assume that $\eps_0<1$ and $\abs{\phi}<1$ so that (\ref{eq:coordphi1}) yields
$$
\eps^{\ha\La}\,\frac{\abs{y_s}}{\abs{x_s}}\abs{\,\phi(y_s)}^{-\ov \La}\leq \abs{X_u(x_s,y_s)}
\leq \eps^{\ov\La}\,\frac{\abs{y_s}}{\abs{x_s}}\abs{\,\phi(y_s)}^{-\ha \La}.
$$
We assume first that $\ka<1/{2\la_1}$, so with the upper bound (\ref{eq:deriv}) we easily get the
inequalities
\begin{equation}\label{eq:provin}
\frac{1}{4\la_2\eps}\abs{e}\leq \abs{\phi(y_s)}\leq \frac{2}{\la_2\eps} \abs{e}
\end{equation}
for $\abs{e}\leq \ov e_0$, and (\ref{eq:inegX}) follows easily with $e(\eps)=\Min(\ov e_0,\la_2\eps/2)$. The derivatives of $X_u$ read
\begin{equation}
\partial_{x_s}X_u(x_s,y_s)=-\eps^{\La(y_s,e)}\frac{y_s}{x_s^2}\abs{\,\phi(y_s)}^{-\La(y_s,e)},
\end{equation}
\begin{equation}\label{eq:dyX}
\begin{array}{lll}
\partial_{y_s}X_u(x_s,y_s)
&=& \frac{\eps^{\La(y_s,e)}}{x_s}\big[\abs{\,\phi(y_s)}-\La(y_s,e)\,y_s\abs{\phi}'(y_s)\big]\abs{\,\phi(y_s)}^{-(\La(y_s,e)+1)}\\
&+& \partial_{y_s} \La(y_s,e)\big(\Ln\eps-\Ln\abs{\,\phi(y_s)}\big)X_u(x_s,y_s).\\
\end{array}
\end{equation}
The estimates (\ref{eq:inegdxX}) are also immediate from (\ref{eq:provin}). To prove (\ref{eq:inegdyX}) observe first that
by Lemma \ref{lem:estimates1}:
$$
\abs{\phi}'(y_s)\leq \frac{\ha\La}{\eps}.
$$
Therefore, by (\ref{eq:provin}), for $\ka>0$ small enough,  there exists a positive constant $c>0$ such that
\begin{equation}\label{eq:ineg4}
\abs{\abs{\phi(y_s)}-\La(y_s)\,y_s\abs{\phi}'(y_s)}\geq  c\abs{e}.
\end{equation}
The estimate (\ref{eq:inegdyX}) immediately follows from (\ref{eq:dyX}), (\ref{eq:ineg4}) and  (\ref{eq:inegX}).
\end{proof}


\paraga The following lemma will enable us to make precise the localization of our horseshoes. We keep the notation of 
Lemma~\ref{lem:phiin}.

\begin{lemma}\label{lem:confin}  Fix a constant  $\ka_0>0$. Then there exists $\eps_0>0$ and $e_0>0$ such that for 
$0<\eps<\eps_0$ and $\abs{e}\leq e_0$ the subset $\jE[\eps,e]$ of $\Sig^s[\eps,e]$ defined by
$$
\jE[\eps,e]=\Big\{(x_s,y_s)\in \Sig^s[\eps,e]\mid \abs{x_s}\leq \eps^{\ov\La},\  \abs{X_u(x_s,y_s)}\leq\eps^{\ov \La}\Big\}
$$
is contained in the horizontal strip $\abs{y_s}\leq \ka_0 e$. Moreover, given $0<c<1$, the set
$$
\jE[\eps,e,c]=\Big\{(x_s,y_s)\in \Sig^s[\eps,e]\mid c\eps^{\ov\La}<\abs{x_s}\leq \eps^{\ov\La},\  \abs{X_u(x_s,y_s)}\leq\eps^{\ov \La}\Big\}
$$
is bounded above and below by two $\mu(e)$ horizontal curves over $c\eps^{\ov\La}<\abs{x_s}\leq \eps^{\ov\La}$, 
with $\mu(e)\to 0$ when $e\to 0$.
\end{lemma}

\begin{proof} Assume that $\abs{y_s}>\ka_0 \abs{e}$. Then $\abs{\la_1y_s-e}\leq (\la_1+\ka\inv)\abs{y_s}$ and, with the
notation of (\ref{eq:deriv}),
$\abs{\ov\chi(y_s,e)}\leq \abs{y_s}$ for $e$ small enough. As a consequence
$$
\abs{\phi(y_s,e)}\leq \frac{c}{\eps}\abs{y_s}
$$
with $c=\frac{1}{\la_2}(\la_1+\ka\inv)+1$. Hence, if $\abs{x_s}\leq\eps^{\ov\La}$,
$$
\abs{X_u(y_s,e)}\geq \eps^{\ha\La} \frac{\abs{y_s}}{\eps^{\ov\La}}\Big(\frac{c}{\eps}\abs{y_s}\Big)^{-\ov\La}
=c^{-\ov\La}\eps^{\ha\La}\abs{y_s}^{1-\ov\La}
\geq c^{-\ha\La}\eps^{\ha\La+2(1-\ov\La)}
$$
since $\abs{y_s}\leq\eps^2$. Now recall that we have assumed from the beginning that $\ha\La-\ov\La< 2(\ov\La-1)$ 
(see (\ref{eq:estlamb2})). Therefore $\ha\La+2(1-\ov\La)<\ov\La$ and
$$
c^{-\ha\La}\eps^{\ha\La+2(1-\ov\La)}\geq \eps^{\ov\La}
$$
for $\eps$ small enough, which proves our first claim. The second one is an immediate consequence of the Implicit Function
Theorem and estimates (\ref{eq:inegdxX}) and (\ref{eq:inegdyX}).
\end{proof}


We can finally prove Lemma~\ref{lem:poscomp}, which we recall here in an explicit form. 

\begin{lemma}\label{lem:compsign}  Let $\Om=(\Om_1,\ldots,\Om_\ell)$ be a positive
polyhomoclinic orbit.   We fix an admissible $\eps>0$ such that the estimates (\ref{eq:estiment}) and (\ref{eq:estimex})
are satified for the exit points $a_i=(\xi_i,0)$  and $b_i=(\eta_i,0)$  of $\Om_i$
relatively to $\Sig^u[\eps]$ and $\Sig^s[\eps]$. Then, with the usual cyclic order
$$
\sig(b_i)=\sig(a_{i+1}),
$$
so that $\Om$ is compatible.
\end{lemma}

\begin{proof} 
There exists a sequence $(\Ga_n)_{n\geq0}$ of orbits with energies $e_n>0$, which converges to  $\Om$.  
So, for $n$ large enough,  
$\Ga_n\cap\Sig^u=\{a_1^{n},\ldots,a_p^{n}\}$ and $\Ga_n\cap\Sig^s=\{b_1^{n},\ldots,b_p^{n}\}$
with the following cyclic order
$$
a_1^{n}<b_1^{n}<a_2^{n}<b_2^{n}<\cdots<a_p^{n}<b_p^{n},
$$
according to the orientation on $\Ga_n$ induced by the flow, and moreover 
$$
\lim_{n\to\infty} a_i^{n}=a_i,\qquad
\lim_{n\to\infty} b_i^{n}=b_i.
$$
In particular, for $n$ large enough, the signs $\sig(a_i^{n})$  and 
$\sig(b_i^{n})$ are well-defined and equal to $\sig(a_i)$ and $\sig(b_i)$ respectively.  
Let us set $b_i^n=(\eta_i^n,y_i^n)$ in the coordinate system of (the appropriate part of) 
$\Sig^s[\eps,e_n]$.
For $n$ large enough, $b_i^{n}$ belongs to the domain of $\Phi_{in}$,  $a_{i+1}^{n}=\Phi_{in}(b_i^{n})$ 
and their coordinates satisfy
$$
\abs{\eta_i^n}\leq \eps^{\ov\La},\quad  \abs{\xi_{i+1}^n}\leq \eps^{\ov\La}.
$$
By Lemma \ref{lem:confin}, 
these inequalities prove that $\abs{y_i^n}<\ka e_n$ and since $\ka<1/2\la_1$, this shows by (\ref{eq:zeroU}) 
and Lemma \ref{lem:estimates1} that 
$$
\sgn\phi_{\eps,\sig(b_i^n)}(y_i^n,e_n)=\sig(b_i^n).
$$ 
Finally, by (\ref{eq:coordcomp}) and Lemma \ref{lem:phiin}, this proves that 
$$ 
\sig(a_{i+1}^{n})=\sig\big(\Phi_{in}(b_i^{n})\big)=\sig(b_i^n)
$$ 
which finally yields our result by taking the limit $n\to\infty$.
\end{proof}


\subsection{Proof of Theorems \ref{thm:hypdyn1} and \ref{thm:hypdyn2}}
Here we first examine the combinatorics of horseshoes associated with two homoclinic orbits, according
to their exit data, from which the proof of Theorems \ref{thm:hypdyn1} and \ref{thm:hypdyn2} easily follows.


\subsubsection{Construction of the parametrized horseshoes} 
Here we fix  two  homoclinic orbits $\Om_0$ and 
$\Om_1$ and we fix an admissible $\eps_a>0$. 
For $0<\eps<\eps_a$,  $\Sig^u[\eps]$ and $\Sig^s[\eps]$  are exit and entrance sections 
for $\Om_i$, relatively to which the exit and  entrance points are well-defined. 
We denote them by $a_i=(\xi_i,0)$  and $b_i=(\eta_i,0)$  respectively 
(recall that both depend on $\eps$).

In the following we {\em fix} $\eps>0$ small enough so that there exists $e_0>0$
for which Lemma~\ref{lem:estimates} 
and Lemma~\ref{lem:confin} apply to each $e\in\,]-e_0,e_0[$,
with the choice
\beq\label{eq:choicexi}
\xi:=\xi(\eps)=\demi\Min_{i}\abs{\xi_i}.
\eeq
in Lemma~\ref{lem:estimates},  and the constant $\ka_0$ of Lemma~\ref{lem:confin} chosen 
equal to the constant $\ka$ of Lemma~\ref{lem:estimates}.


\paraga {\bf The images of verticals curves by the inner map.} 
We begin with the behavior of the inner map with respect to vertical curves in the entrance section.
We refer to the appendix for the definition of horizontal and vertical curves.

\begin{lemma} \label{lem:intcond}
Fix $\mu_*>0$ and consider a family $(v_e)_{\abs{e}<e_0}$ of $\mu_*$--vertical curves
over some fixed interval $I$ containing $0$ in $\Sig^{s*}_{\sig}[\eps,e]$. 
Then there exists $0<e_1<e_0$ such that for $\abs{e}<e_1$, the image
$
\Phi_{in}\big(v_e\cap\jE[\eps,e]\big)
$
is contained in the section  $\Sig_{\sig'}^u[\eps,e]$, with $\sig'=\sgn(e)\sig$. Moreover,  in this section,
the intersection 
$$
\Phi_{in}\big(v_e\cap\jE[\eps,e]\big)\cap \Big\{(x_u,y_u)\mid\abs{x_u}\leq\eps^{\ov\La}\Big\}
$$
is a $\mu(e)$--horizontal curve over the interval $\abs{x_u}\leq\eps^{\ov\La}$, with $\mu(e)\to0$ when  
$e\to0$.
\end{lemma}

\begin{proof}  As usual we use the same notation for a curve and its underlying function, so that
$$
L_e:=\Phi_{in}\big(\til v_e\cap\jE[\eps,e]\big)=\Big\{\big(X_u\big(v_e(t),t\big)\mid t\in I,\ \big(v_e(t),t\big)\in \jE[\eps,e]\Big\}.
$$
This is a curve contained in  $\Sig_{\sig'}^u[\eps,e]$ by Lemma \ref{lem:confin}. 
Moreover, the slope $s_e(t)$ of $L_e$ at $\ell_e(t)$ satisfies
$$
\abs{s_e(t)}=\frac{1}{\abs{\partial_x X_u(\ze(t),t) \ze_h'(t)+\partial_y X_u(\ze(t),t)}}\leq 
\frac{1}{\abs{\partial_y X_u(\ze(t),t)}-\mu^*\abs{\partial_x X_u(\ze(t),t)}}
$$
and since $\jE[\eps,e]$ is contained in the strip $\abs{y_s}\leq\ka \abs{e}$, 
Lemma \ref{lem:estimates1} proves that $\abs{s_e(t)}$ converges
uniformly to $0$ when $e\to 0$. 
Observe finally that by the second claim of Lemma \ref{lem:confin}, since $v_e$ is a $\mu_*$ vertical curve, then 
for $\abs{e}$ small enough it transversely intersects the horizontal curves of the boundary of $\jE[\eps,e]$
and the set of $t\in I$ such that $\big(v_e(t),t\big)\in \jE[\eps,e]$ is an interval containing $0$. This
proves that $L_e$ is connected and intersects the vertical segments $\abs{x_u}=\eps^{\ov\La}$. 
As a consequence, by the previous estimate of the slope,  $L_e$ is  a $\mu(e)$ horizontal
curve in the rectangle $\{\abs{x_u}\leq\eps^{\ov\La}\}$, with $\mu(e)\to0$ when $e\to 0$. 
\end{proof}


\paraga {\bf Rectangles and intersection conditions.}
We will now apply the previous lemma  to get our horseshoe. To simplify the notation, given $D\subset \Sig^u$, 
we write $\Phi(D)$  for the image by $\Phi$ of the intersection of $D$ with the domain of definition of $\Phi$, 
with a similar convention for all the maps involved in the construction.

\begin{lemma}\label{lem:intcases} Fix 
$0<\de<\abs{\xi}$ (where $\xi$ was defined in (\ref{eq:choicexi})) and fix $0<\de' <\eps^2$. Consider the rectangles  
$$
R_{i}(e)=R[\xi_i,\de,\de',e]\subset \Sig^{u*}_{\sig(a_i)}[\eps,e],\qquad i=0,1.
$$
Then the pair $(R_{0},R_1)$ satisfies the intersection condition (see Definition~\ref{def:horse}) for the Poincar\'e map
$\Phi=\Phi_{in}\circ\Phi_{out}$. More precisely:
\begin{itemize}\label{item:cases}
\item if $\sig(a_{i'})=\sgn(e)\sig(b_{i})$, $\Phi(R_{i}(e))\cap R_{i'}(e)$ is a 
$\mu(e)$ horizontal strip in $R_{i'}(e)$;
\item if $\sig(a_{i'})=-\sgn(e)\sig(b_{i})$, $\Phi(R_{i}(e))\cap R_{i'}(e)=\emptyset$.
\end{itemize}
\end{lemma}

\begin{proof} Observe first that for $\abs{e}$ small enough the image $\Phi_{out}(R_{i}(e))$ is a ``rectangle'' 
contained in $\Sig^s_{\sig(b_i}[\eps,e]$. Moreover, the images of the horizontals of $R_{i}(e)$ are curves
which transversely intersect the axis $\{y_s=0\}$, whose intersection with the domain $\jE[\eps,e]$ are $\mu_*$
vertical curves for a suitable $\mu_*$. Therefore, by Lemma~\ref{lem:phiin}, 
$$
\Phi(R_{i}(e))\subset \Sig^u_{\sig'}[\eps,e], \qquad \sig'=(\sgn(e)\sig(b_i)),
$$
and, by Lemma~\ref{lem:intcond}, $\Phi(R_{i}(e))$
is a $\mu(e)$-horizontal strip in $\{(x_u,y_u)\mid\abs{x_u}\leq\eps^{\ov\La}\}$, defined over  
the whole interval $\abs{x_u}\leq\eps^{\ov\La}$, with $\mu(e)\to0$ when $e\to 0$. 
This proves our claim for $\abs{e}$ small enough.
\end{proof}


\paraga {\bf Sector and hyperbolicity.} 
We will now prove that the sector conditions and hyperbolicity constraints of Definition \ref{def:horse} hold true in the 
rectangles we considered above. 
We begin with a lemma on the derivative of the Poincar\'e map $\Phi$.  
Let us write the matrix of the derivative of $\Phi_{out}$ at the point $a_i$ relative to the 
coordinates $(x_u,y_u)$ and $(x_s,y_s)$  in the form
\begin{equation}
D_{a_i}\Phi_{out}=\left[
\begin{array}{lll}
p_0&q_0\\
r_0&s_0\\
\end{array}
\right]
\end{equation}
with $p_0s_0-q_0r_0\neq0$ and  $r_0\neq0$ 
(since $\Phi_{out}(W^u_\ell(O))$ is 
transverse to $W_\ell^s(O)$).

\begin{lemma}\label{lem:hypcond} For $\de$ and $\de'$ small enough, for $m\in R[\xi,\de,\de';e]$, the map $D_m\Phi$ 
admits two real eigenvalues $\la^-(m),\la^+(m)$ with 
$$
\la^-(m,e)\sim_{e\to 0}r(m)\,\partial_{y_s} X_u(\Phi_{in}(m))
$$
$$
\la^+(m,e)\sim_{e\to 0}\frac{(p(m)r(m)-q(m)s(m))\partial_{x_s} X_u(\Phi_{in}(m))}{r(m)\,\partial_{y_s} X_u(\Phi_{in}(m))},
$$
uniformly with respect to $m$. In particular
$$
\la^-(m,e)\to+\infty\quad\textit{and}\quad \la^+(m,e)\to 0 \quad\textit{when}\quad e\to0,
$$
uniformly with respect to $m$.
The associated eigenlines are spanned by the
vectors
$$
w^-(m,e)=\Big(\frac{1}{r(m)}\big(\la_e^h-s(m)\big), 1\Big),\qquad w^+(m,e)=\Big(\frac{1}{r(m)}\big(\la_e^v-s(m)\big), 1\Big).
$$
The line $\R\, w^-(m,e)$ converges to the line $\{y_u=0\}\subset \Sig^u[\eps,e]$, while the line $\R\, w^+(m,e)$ converges to 
$\Phi_{out}(\{y_s=0\})\subset\Sig^u[\eps,e]$ when $e\to 0$, uniformly with respect to $m$.
\end{lemma}

\begin{proof}
For $m=(x_u,y_u)\in R_i$,
\begin{equation}
D_{m}\Phi_{out}
=\left[
\begin{array}{lll}
p(m)&q(m)\\
r(m)&s(m)\\
\end{array}
\right]=\left[
\begin{array}{lll}
p_0&q_0\\
r_0&s_0\\
\end{array}
\right]
+o(\de,\de').
\end{equation}
We can therefore assume that $ps-qr$ and $r$ are bounded  below by positive constants 
$\De$ and $\rho$ over the rectangle $R_i$.
Now the derivative of the Poincar\'e map $\Phi$ reads
\begin{equation}
D_m\Phi=
\left[
\begin{array}{lll}
p(m)\,\partial_{x_s} X_u+r(m)\,\partial_{y_s} X_u&q(m)\,\partial_{x_s} X_u+s(m)\,\partial_{y_s} X_u\\
r(m)&s(m)\\
\end{array}
\right]
\end{equation}
where the derivatives of $X_u$ are computed at $\Phi_{in}(m)$.
The trace of $D_m\Phi$ satisfies
\begin{equation}
\abs{{{\rm Tr}_e}(m) }= \abs{p(m)\,\partial_{x_s} X_u+r(m)\,\partial_{y_s} X_u+s(m)}\geq C \abs{e}^{-\ov\La}
\end{equation}
for a suitable constant $C>0$, since $r(m)\geq \rho>0$, by Lemma \ref{lem:estimates}.
The determinant of $D_m\Phi$ satisfies
\begin{equation}
\De_e(m)=(pr-qs)\partial_{x_s} X_u=o({{\rm Tr}_e}(m)),
\end{equation}
By standard computation, one immediately gets the estimates
\begin{equation}
\la_e^h(m)\sim_{e\to 0} {{\rm Tr}_e}(m),\qquad \la_e^v(m)\sim_{e\to 0} \frac{\De_e(m)}{{{\rm Tr}_e}(m)},
\end{equation}
which proves our first claim. The second one on the eigenvectors is immediate.
Finally, the convergence of the line $\R\, w^-(m,e)$ to the line $\{y_u=0\}$ is immediate, while the convergence of the 
line $\R\, w^+(m,e)$  to  $\Phi_{out}(\{y_s=0\})$ is proved by a completely analogous reasoning on $\Phi\inv$,
using now Lemma~\ref{lem:phiininv} (the uniformity with respect to $m$ comes from the compactness of the 
domain and range
of the various maps).
\end{proof}

\vskip2mm

Observe that the line $\{y_u=0\}\subset \Sig[\eps,e]$ converges to $W^u_\ell\cap\Sig^u[\eps,0]$,  while the curve
$\Phi_{out}\inv(\{y_s=0\})$ converges to $\Phi_{out}\inv(W^s_\ell\cap\Sig^s[\eps,0])$ (where the sections are endowed 
with the appropriate sign).
Note also that $\Phi_{out}\inv(W^s_\ell\cap\Sig^s[\eps,0])$ is nothing but (some connected component of) 
the intersection
$W^s\cap\Sig^s[\eps,0]$.

\vskip2mm

\begin{lemma}\label{lem:hypcond2}
For $m\in R_i$, the sectors $S^h_m$ and $S^v_m$ in $T_m\Sig^u(e)$ defined by
$$
S^h_z=\{(\xi,\eta)\in\R^2\mid \abs{\eta}\leq \mu_e\abs{\xi}\},\qquad
S^v_z=\{(\xi,\eta)\in\R^2\mid \abs{\xi}\leq \mu_e\abs{\eta}\}
$$
satisfy the stability condition and the dilatation conditions of Definition \ref{def:horse} with 
$$
\mu_e=\frac{1}{2\Max_{m\in R_i}{{\rm Tr}_e}(m)}.
$$
\end{lemma}

\begin{proof} 
This is an immediate consequence of the form of the eigenvectors.
\end{proof}


\subsubsection{Proof of Theorem \ref{thm:hypdyn1}}  
This will be an immediate consequence of Lemmas~\ref{lem:intcases}, \ref{lem:hypcond}, \ref{lem:hypcond2}.
We  fix a compatible polyhomoclinic orbit $\Om=(\Om_1,\ldots,\Om_\ell)$ and keep the previous  assumptions 
and notation for the sections and the entrance and exit points. In particular
$$
\sig(b_i)=\sig(a_{i+1})
$$
for $1\leq i\leq \ell$, with the cyclic order.
By Lemma~\ref{lem:intcases}, the transition matrix of the horseshoe satisfies 
$$
\al(i,i+1)=1\quad\textrm{for}\quad1\leq i\leq {\ell}.
$$
The existence and hyperbolicity of the horseshoe come from Lemma~\ref{lem:hypcond} and 
Lemma~\ref{lem:hypcond2}.
The statement on $m(e)$ is a direct consequence of Lemma~\ref{lem:hypcond} applied to the 
iterate $\Phi^{\ell}$.
Theorem~\ref{thm:hypdyn1} is proved.


\subsubsection{Proof of Theorem \ref{thm:hypdyn2}} 
Now $\Om_0$ and $\Om_1$ are compatible and satisfy the sign condition (\ref{eq:condsign}).
We will work at negative energies, therefore, by Lemma~\ref{lem:intcases} 
$$
\sig(a^0)=\sig(b^0),\qquad \sig(a^1)=\sig(b^1),\qquad \sig(a^0)=-\sig(a^1).
$$ 
One immediately checks that the transition matrix of the horseshoe reads 
$$
A=\left[
\begin{array}{lll}
0&1\\
1&0\\
\end{array}
\right].
$$
The statement on the existence and hyperbolicity of the horseshoe immediately comes from Lemma~\ref{lem:hypcond} and 
Lemma~\ref{lem:hypcond2}, while the statement on the periodic point $m(e)$ is a direct consequence of Lemma~\ref{lem:hypcond} 
applied to $\Phi^2$.
Theorem~\ref{thm:hypdyn2} is proved.

\section{A reminder on horseshoes}\label{sec:horses} \setcounter{paraga}{0}
We will need some additional definitions concerning horseshoes and hyperbolic dynamics. We will  follow
the approach of Moser in \cite{Mos}, which is perfectly adapted to our two--dimensional situation. 
See also \cite{HK} for a more general point of view. 

\paraga Consider a rectangular subset $R$ of $\R^2$ of the form $R=I^h\times I^v$, where $I^h$ and $I^v$ are two compact 
nontrivial intervals of $\R$.  Given $\mu>0$, a {\em $\mu$--horizontal curve} is the graph of a $\mu$--Lipschitzian map $c^h:I^h\to I^v$, while 
a {\em $\mu$--vertical curve} is the graph of a $\mu$--Lipschitzian map $c^v:I^v\to I^h$. A $\mu$--horizontal strip is a subset of $R$
limited by two nonintersecting horizontal curves, that is a set of the form
$$
\{(x,y)\in R\mid c^h(x)\leq y\leq d^h(x)\}
$$
where $c^h$ and $d^h$ are two $\mu$--Lipschitzian maps satisfying $c^h(x)<d^h(x)$ for $x\in I^h$. One defines similarly
the $\mu$--vertical strips.

\begin{figure}[h]
\begin{center}
\begin{pspicture}(0cm,2.5cm)
\psset{xunit=.8cm,yunit=.8cm}
\rput(-2.5,0){
\psframe(0,0)(5,3)
\pscurve(0,1)(1,1.3)(2,.9)(3,1.2)(4,1)(5,1.2)
\pscurve(0,1.3)(1,1.5)(2,1.3)(3,1.4)(4,1.3)(5,1.4)
\pscurve(2,0)(2.2,1)(1.9,2)(2.3,3)
\pscurve(2.3,0)(2.4,1)(2.2,2)(2.5,3)
}
\end{pspicture}
\caption{Horizontal and vertical strips}\label{Fig:strips}
\end{center}
\end{figure}
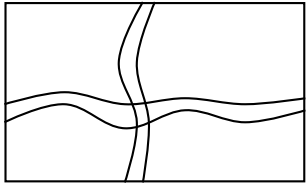

\paraga The definition of a horseshoe we use here is from \cite{Mos}
\begin{Def}\label{def:horse} Consider a finite family of rectangles $R_i=I^h_i\times I^v_i$, $i\in\{0,\ldots,m\}$, in $\R^2$ and let $\Phi$ be a 
$C^1$--diffeomorphism  defined on a neighborhood of $R=R_0\cup \cdots\cup R_m$. We say
that $R$ is a horseshoe for $\Phi$ when there exists $\mu>0$ such that
\begin{enumerate} 
\item for $(i,j)\in\{0,\ldots,m\}^2$, $\Phi(R_i)\cap R_j$ is either $\emptyset$ or a $\mu$--horizontal strip $H_{ij}$ 
and $\Phi\inv(R_j)\cap R_i$ is either $\emptyset$ or a $\mu$--vertical strip $V_{ij}$, so that $\Phi(V_{ij})=H_{ij}$;

\item for each $z\in I=\big(\cup_{ij}H_{ij}\big)\cap \big(\cup_{ij}V_{ij}\big)$, there exists a sector $S^h_z\subset T_z\R^2\sim\R^2$, of the form
$$
S^h_z=\{(\xi,\eta)\in\R^2\mid \abs{\eta}\leq \mu\abs{\xi}\},
$$
which satisfies the stability condition $D_z\Phi(S^h_z)\subset S_{\Phi(z)}$ for all $z\in H$ together with the dilatation condition
$$
\forall (\xi,\eta)\in S^h_z, \quad\textit{setting}\quad D_z\Phi(\xi,\eta)=(\xi',\eta'), \quad\textit{then}\quad \abs{\xi'}\geq \mu\inv\abs{\xi};
$$

\item  for each $z\in I$, there exists a sector $S^v_z\subset T_z\R^2$, of the form
$$
S^v_z=\{(\xi,\eta)\in\R^2\mid \abs{\xi}\leq \mu\abs{\eta}\},
$$
which satisfies the stability condition $D_z\Phi\inv(S^v_z)\subset S^v_{\Phi\inv(z)}$ for all $z\in V$ together with the dilatation condition
$$
\forall (\xi,\eta)\in S^v_z, \quad\textit{setting}\quad D_z\Phi\inv(\xi,\eta)=(\xi',\eta'), \quad\textit{then}\quad \abs{\eta'}\geq \mu\inv\abs{\eta};
$$

\item for all $z\in V$, $\abs{\det D_z\Phi}\leq \pdemi\mu^{-2}$ and for all $z\in H$,
$\abs{\det D_z\Phi\inv}\leq \pdemi\mu^{-2}$. 
\end{enumerate}
\end{Def}

\vskip2mm

Given a horseshoe $R=(R_k)_{1\leq k\leq m}$ for $\Phi$, one defines its {\em transition matrix} as the matrix 
$A=\big(\al(i,j)\big)\in M_m(\{0,1\})$ whose coefficient
$\al(i,j)$ is $0$ when $H_{ij}=\emptyset$ and is $1$ when $H_{ij}\neq\emptyset$.
Given such a transition matrix $A=\big(\al(i,j)\big)$, one defines as usual the $A$--admissible subset $\jS_A$ of 
$\{0,\ldots,m\}^\Z$ by 
$$
(s_k)_{k\in\Z}\in \jS_A\Longleftrightarrow \al(s_k,s_{k+1})=1,\quad \forall k\in\Z.
$$

\paraga We can now set out the main result on horseshoes.

\vskip3mm

\noindent{\bf Theorem \cite{Mos}.}  {\em Let $R=(R_k)_{1\leq k\leq m}$ be a horseshoe for the $C^1$ diffeomorphism $\Phi$,  
with transition matrix $A$.
Equip $\jS_A$ with the induced product topology and let
$$
\jI=\bigcap_{k\in\Z}\,\Phi^{-k} (R)
$$
be the the maximal invariant set for $\Phi$ contained in $R$.
Then there exists a homeomorphism $\jC:\jS_A\to \jI$ such that $\jC\circ \sig=\Phi\circ \jC$, where $\sig$ stands for the right Bernoulli subshift on 
$\jS_A$,
defined by
$$
\sig\big((s_k)\big)_{k\in\Z}=(\bar s_k)_{k\in\Z},\qquad \bar s_k=s_{k+1}.
$$
Morover, the invariant set $\jI$ is hyperbolic in the sense that there exists two continuous line bundles $L^h$ and $L^v$
defined over $\jI$ and invariant under $D\Phi$, such that for all $z\in \jI$:
$$
\norm{D_z\Phi(\ze)}\geq\mu\inv\norm{\ze},\  \forall \ze\in L_z^h,\qquad
\norm{D_z\Phi\inv(\ze)}\geq\mu\inv\norm{\ze}, \ \forall \ze\in L_z^v.
$$
}

Finally, let us notice that one can be more explicit for the coding of points of $\jI$ by sequences induced by~$\jC$. Namely, for $m\in\jI$:
$$
\jC(m)=(s_k)_{k\in\Z}
\Longleftrightarrow
\Phi^k(m)\in R_{s_k},\quad \forall k\in\Z.
$$


 \def\cprime{$'$}


\begin{thebibliography}{DdlLS06b}
\expandafter\ifx\csname fonteauteurs\endcsname\relax
\def\fonteauteurs{\scshape}\fi

\bibitem[AR67]{AR}
Ralph \bgroup\fonteauteurs\bgroup Abraham\egroup\egroup{} et Joel
  \bgroup\fonteauteurs\bgroup Robbin\egroup\egroup{} :
\newblock {\em Transversal mappings and flows}.
\newblock An appendix by Al Kelley. W. A. Benjamin, Inc., New York-Amsterdam,
  1967.

\bibitem[Arn64]{A64}
V.~I. \bgroup\fonteauteurs\bgroup Arnol{\cprime}d\egroup\egroup{} :
\newblock Instability of dynamical systems with many degrees of freedom.
\newblock {\em Dokl. Akad. Nauk SSSR}, 156\string:\penalty500\relax 9--12,
  1964.

\bibitem[Arn94]{A94}
V.~I. \bgroup\fonteauteurs\bgroup Arnol{\cprime}d\egroup\egroup{} :
\newblock Mathematical problems in classical physics.
\newblock \emph{In} {\em Trends and perspectives in applied mathematics},
  volume 100 de {\em Appl. Math. Sci.}, pages 1--20. Springer, New York, 1994.

\bibitem[Art91]{Art}
Michael \bgroup\fonteauteurs\bgroup Artin\egroup\egroup{} :
\newblock {\em Algebra}.
\newblock Prentice Hall, Inc., Englewood Cliffs, NJ, 1991.

\bibitem[Ban88]{Ba}
V.~\bgroup\fonteauteurs\bgroup Bangert\egroup\egroup{} :
\newblock Mather sets for twist maps and geodesics on tori.
\newblock \emph{In} {\em Dynamics reported, {V}ol.\ 1}, volume~1 de {\em Dynam.
  Report. Ser. Dynam. Systems Appl.}, pages 1--56. Wiley, Chichester, 1988.
  
\bibitem{BB03}{BB03} 
M.L. \bgroup\fonteauteurs\bgroup Betotti\egroup\egroup{},
S.  \bgroup\fonteauteurs\bgroup Bolotin\egroup\egroup{} :
\newblock Chaotic trajectories for natural systems on a torus.
\newblock {\em Discrete Contin. Dyn. Syst.},  {9}\string:\penalty500\relax 2003.


\bibitem[BB13]{BB13}
Pierre \bgroup\fonteauteurs\bgroup Berger\egroup\egroup{} et Abed
  \bgroup\fonteauteurs\bgroup Bounemoura\egroup\egroup{} :
\newblock A geometrical proof of the persistence of normally hyperbolic
  submanifolds.
\newblock {\em Dyn. Syst.}, 28(4)\string:\penalty500\relax 567--581, 2013.

\bibitem[Ber00]{Be00}
Patrick \bgroup\fonteauteurs\bgroup Bernard\egroup\egroup{} :
\newblock Homoclinic orbits to invariant sets of quasi-integrable exact maps.
\newblock {\em Ergodic Theory Dynam. Systems}, 20(6)\string:\penalty500\relax
  1583--1601, 2000.

\bibitem[Ber08]{B08}
Patrick \bgroup\fonteauteurs\bgroup Bernard\egroup\egroup{} :
\newblock The dynamics of pseudographs in convex {H}amiltonian systems.
\newblock {\em J. Amer. Math. Soc.}, 21(3)\string:\penalty500\relax 615--669,
  2008.

\bibitem[Ber10a]{Berg10}
Pierre \bgroup\fonteauteurs\bgroup Berger\egroup\egroup{} :
\newblock Persistence of laminations.
\newblock {\em Bull. Braz. Math. Soc. (N.S.)}, 41(2)\string:\penalty500\relax
  259--319, 2010.

\bibitem[Ber10b]{B10}
Patrick \bgroup\fonteauteurs\bgroup Bernard\egroup\egroup{} :
\newblock Large normally hyperbolic cylinders in a priori stable {H}amiltonian
  systems.
\newblock {\em Ann. Henri Poincar\'e}, 11(5)\string:\penalty500\relax 929--942,
  2010.

\bibitem[BK87]{BK}
David \bgroup\fonteauteurs\bgroup Bernstein\egroup\egroup{} et Anatole
  \bgroup\fonteauteurs\bgroup Katok\egroup\egroup{} :
\newblock Birkhoff periodic orbits for small perturbations of completely
  integrable {H}amiltonian systems with convex {H}amiltonians.
\newblock {\em Invent. Math.}, 88(2)\string:\penalty500\relax 225--241, 1987.

\bibitem[BK02]{BK02}
G.~R. \bgroup\fonteauteurs\bgroup Belitskii\egroup\egroup{} et A.~Ya.
  \bgroup\fonteauteurs\bgroup Kopanskii\egroup\egroup{} :
\newblock Equivariant {S}ternberg-{C}hen theorem.
\newblock {\em J. Dynam. Differential Equations},
  14(2)\string:\penalty500\relax 349--367, 2002.

\bibitem[BKZ]{BKZ13}
Patrick \bgroup\fonteauteurs\bgroup Bernard\egroup\egroup{} et Vadim \bgroup\fonteauteurs\bgroup Kaloshin\egroup\egroup{} et Ke \bgroup\fonteauteurs\bgroup Zhang\egroup\egroup{} :
\newblock Arnold diffusion in arbitrary degrees of freedom and crumpled 3-dimensional normally hyperbolic invariant cylinders.
\newblock arXiv :1112.2773

\bibitem[Bol78]{Bo78}
S.~V. \bgroup\fonteauteurs\bgroup Bolotin\egroup\egroup{} :
\newblock Libration motions of natural dynamical systems.
\newblock {\em Vestnik Moskov. Univ. Ser. I Mat. Mekh.},
  (6)\string:\penalty500\relax 72--77, 1978.

\bibitem[Bol83]{Bo83}
S.~V. \bgroup\fonteauteurs\bgroup Bolotin\egroup\egroup{} :
\newblock Existence of homoclinic motions.
\newblock {\em Moscow Univ. Math. Bull.}, (6)\string:\penalty500\relax
  117--123, 1983.

\bibitem[Bou10]{Bou10}
Abed \bgroup\fonteauteurs\bgroup Bounemoura\egroup\egroup{} :
\newblock Nekhoroshev estimates for finitely differentiable quasi-convex
  {H}amiltonians.
\newblock {\em J. Differential Equations}, 249(11)\string:\penalty500\relax
  2905--2920, 2010.

\bibitem[BR02]{BR02}
S.~V. \bgroup\fonteauteurs\bgroup Bolotin\egroup\egroup{} et P.~H.
  \bgroup\fonteauteurs\bgroup Rabinowitz\egroup\egroup{} :
\newblock Some geometrical conditions for the existence of chaotic geodesics on
  a torus.
\newblock {\em Ergodic Theory Dynam. Systems}, 22(5)\string:\penalty500\relax
  1407--1428, 2002.

\bibitem[BT99]{BT99}
S.~\bgroup\fonteauteurs\bgroup Bolotin\egroup\egroup{} et
  D.~\bgroup\fonteauteurs\bgroup Treschev\egroup\egroup{} :
\newblock Unbounded growth of energy in nonautonomous {H}amiltonian systems.
\newblock {\em Nonlinearity}, 12(2)\string:\penalty500\relax 365--388, 1999.

\bibitem[C]{C}
Chong-Qing \bgroup\fonteauteurs\bgroup Cheng\egroup\egroup{} :
\newblock Arnold diffusion in nearly integrable Hamiltonian systems.
\newblock ArXiv : 1207.4016

\bibitem[Car95]{DC95}
M.~J.~Dias \bgroup\fonteauteurs\bgroup Carneiro\egroup\egroup{} :
\newblock On minimizing measures of the action of autonomous {L}agrangians.
\newblock {\em Nonlinearity}, 8(6)\string:\penalty500\relax 1077--1085, 1995.

\bibitem[Cha86]{C85}
Marc \bgroup\fonteauteurs\bgroup Chaperon\egroup\egroup{} :
\newblock G\'eom\'etrie diff\'erentielle et singularit\'es de syst\`emes
  dynamiques.
\newblock {\em Ast\'erisque}, (138-139)\string:\penalty500\relax 440, 1986.

\bibitem[Cha04]{C04}
Marc \bgroup\fonteauteurs\bgroup Chaperon\egroup\egroup{} :
\newblock Stable manifolds and the {P}erron-{I}rwin method.
\newblock {\em Ergodic Theory Dynam. Systems}, 24(5)\string:\penalty500\relax
  1359--1394, 2004.

\bibitem[Cha08]{C08}
Marc \bgroup\fonteauteurs\bgroup Chaperon\egroup\egroup{} :
\newblock The {L}ipschitzian core of some invariant manifold theorems.
\newblock {\em Ergodic Theory Dynam. Systems}, 28(5)\string:\penalty500\relax
  1419--1441, 2008.

\bibitem[CY09]{CY09}
Chong-Qing \bgroup\fonteauteurs\bgroup Cheng\egroup\egroup{} et Jun
  \bgroup\fonteauteurs\bgroup Yan\egroup\egroup{} :
\newblock Arnold diffusion in {H}amiltonian systems: a priori unstable case.
\newblock {\em J. Differential Geom.}, 82(2)\string:\penalty500\relax 229--277,
  2009.

\bibitem[DdlLS00]{DLS00}
Amadeu \bgroup\fonteauteurs\bgroup Delshams\egroup\egroup{}, Rafael de~la
  \bgroup\fonteauteurs\bgroup Llave\egroup\egroup{} et Tere~M.
  \bgroup\fonteauteurs\bgroup Seara\egroup\egroup{} :
\newblock A geometric approach to the existence of orbits with unbounded energy
  in generic periodic perturbations by a potential of generic geodesic flows of
  {${\bf T}^2$}.
\newblock {\em Comm. Math. Phys.}, 209(2)\string:\penalty500\relax 353--392,
  2000.

\bibitem[DdlLS06a]{DLS06a}
Amadeu \bgroup\fonteauteurs\bgroup Delshams\egroup\egroup{}, Rafael de~la
  \bgroup\fonteauteurs\bgroup Llave\egroup\egroup{} et Tere~M.
  \bgroup\fonteauteurs\bgroup Seara\egroup\egroup{} :
\newblock A geometric mechanism for diffusion in {H}amiltonian systems
  overcoming the large gap problem: heuristics and rigorous verification on a
  model.
\newblock {\em Mem. Amer. Math. Soc.}, 179(844)\string:\penalty500\relax
  viii+141, 2006.

\bibitem[DdlLS06b]{DLS06b}
Amadeu \bgroup\fonteauteurs\bgroup Delshams\egroup\egroup{}, Rafael de~la
  \bgroup\fonteauteurs\bgroup Llave\egroup\egroup{} et Tere~M.
  \bgroup\fonteauteurs\bgroup Seara\egroup\egroup{} :
\newblock Orbits of unbounded energy in quasi-periodic perturbations of
  geodesic flows.
\newblock {\em Adv. Math.}, 202(1)\string:\penalty500\relax 64--188, 2006.

\bibitem[Den89]{D89}
Bo~\bgroup\fonteauteurs\bgroup Deng\egroup\egroup{} :
\newblock The \v {S}il\cprime nikov problem, exponential expansion, strong
  {$\lambda$}-lemma, {$C^1$}-linearization, and homoclinic bifurcation.
\newblock {\em J. Differential Equations}, 79(2)\string:\penalty500\relax
  189--231, 1989.

\bibitem[DH09]{DH09}
Amadeu \bgroup\fonteauteurs\bgroup Delshams\egroup\egroup{} et Gemma
  \bgroup\fonteauteurs\bgroup Huguet\egroup\egroup{} :
\newblock Geography of resonances and {A}rnold diffusion in a priori unstable
  {H}amiltonian systems.
\newblock {\em Nonlinearity}, 22(8)\string:\penalty500\relax 1997--2077, 2009.

\bibitem[Fa0]{Fa09}
Albert \bgroup\fonteauteurs\bgroup Fathi\egroup\egroup{} :
\newblock Weak KAM theory.
\newblock {\em Unpublished manuscript.}


\bibitem[FM03]{FM03}
Ernest \bgroup\fonteauteurs\bgroup Fontich\egroup\egroup{} et Pau
  \bgroup\fonteauteurs\bgroup Mart{\'{\i}}n\egroup\egroup{} :
\newblock Hamiltonian systems with orbits covering densely submanifolds of
  small codimension.
\newblock {\em Nonlinear Anal.}, 52(1)\string:\penalty500\relax 315--327, 2003.

\bibitem[GdlL06]{GL06}
Marian \bgroup\fonteauteurs\bgroup Gidea\egroup\egroup{} et Rafael de~la
  \bgroup\fonteauteurs\bgroup Llave\egroup\egroup{} :
\newblock Topological methods in the instability problem of {H}amiltonian
  systems.
\newblock {\em Discrete Contin. Dyn. Syst.}, 14(2)\string:\penalty500\relax
  295--328, 2006.

\bibitem[GLS]{GLS}
Marian \bgroup\fonteauteurs\bgroup Gidea\egroup\egroup{} et Rafael de~la
  \bgroup\fonteauteurs\bgroup Llave\egroup\egroup{}et Tere  \bgroup\fonteauteurs\bgroup Seara\egroup\egroup{} :
  \newblock A General Mechanism of Diffusion in Hamiltonian Systems: Qualitative Results.
  \newblock ArXiv : 1405.0866

\bibitem[GM]{GM}
Marian \bgroup\fonteauteurs\bgroup Gidea\egroup\egroup{} et Jean-Pierre \bgroup\fonteauteurs\bgroup Marco\egroup\egroup{} :
\newblock Diffusion orbits along chains of hyperbolic cylinders.
\newblock ArXiv

\bibitem[Gou07]{G07}
Nikolaz \bgroup\fonteauteurs\bgroup Gourmelon\egroup\egroup{} :
\newblock Adapted metrics for dominated splittings.
\newblock {\em Ergodic Theory Dynam. Systems}, 27(6)\string:\penalty500\relax
  1839--1849, 2007.

\bibitem[GR13]{GR13}
Marian \bgroup\fonteauteurs\bgroup Gidea\egroup\egroup{} et Clark
  \bgroup\fonteauteurs\bgroup Robinson\egroup\egroup{} :
\newblock Diffusion along transition chains of invariant tori and
  {A}ubry-{M}ather sets.
\newblock {\em Ergodic Theory Dynam. Systems}, 33(5)\string:\penalty500\relax
  1401--1449, 2013.

\bibitem[GT]{GT}
Vassily \bgroup\fonteauteurs\bgroup Gelfreich\egroup\egroup{} et Dmitry \bgroup\fonteauteurs\bgroup Turaev\egroup\egroup{} :
\newblock Arnold Diffusion in a priory chaotic Hamiltonian systems.
\newblock arXiv : 1406.2945

\bibitem[Her83]{Herman}
Michael-R. \bgroup\fonteauteurs\bgroup Herman\egroup\egroup{} :
\newblock {\em Sur les courbes invariantes par les diff\'eomorphismes de
  l'anneau. {V}ol. 1}, volume 103 de {\em Ast\'erisque}.
\newblock Soci\'et\'e Math\'ematique de France, Paris, 1983.
\newblock With an appendix by Albert Fathi, With an English summary.

\bibitem[Her86]{Herman2}
Michael-R. \bgroup\fonteauteurs\bgroup Herman\egroup\egroup{} :
\newblock Sur les courbes invariantes par les diff\'eomorphismes de l'anneau.
  {V}ol.\ 2.
\newblock {\em Ast\'erisque}, (144)\string:\penalty500\relax 248, 1986.
\newblock With a correction to: {{\i}t On the curves invariant under
  diffeomorphisms of the annulus, Vol. 1} (French) [Ast{\'e}risque No. 103-104,
  Soc. Math. France, Paris, 1983; MR0728564 (85m:58062)].

\bibitem[HPS77]{HPS}
M.~W. \bgroup\fonteauteurs\bgroup Hirsch\egroup\egroup{}, C.~C.
  \bgroup\fonteauteurs\bgroup Pugh\egroup\egroup{} et
  M.~\bgroup\fonteauteurs\bgroup Shub\egroup\egroup{} :
\newblock {\em Invariant manifolds}.
\newblock Lecture Notes in Mathematics, Vol. 583. Springer-Verlag, Berlin-New
  York, 1977.

\bibitem[KH95]{HK}
Anatole \bgroup\fonteauteurs\bgroup Katok\egroup\egroup{} et Boris
  \bgroup\fonteauteurs\bgroup Hasselblatt\egroup\egroup{} :
\newblock {\em Introduction to the modern theory of dynamical systems},
  volume~54 de {\em Encyclopedia of Mathematics and its Applications}.
\newblock Cambridge University Press, Cambridge, 1995.
\newblock With a supplementary chapter by Katok and Leonardo Mendoza.

\bibitem[KZ]{KZ}
Vadim \bgroup\fonteauteurs\bgroup Kaloshin\egroup\egroup{} et
  Ke~\bgroup\fonteauteurs\bgroup Zhang\egroup\egroup{} :
\newblock A strong form of arnold diffusion for two and a half degrees of
  freedom.

\bibitem[LC86]{LC86}
Patrice \bgroup\fonteauteurs\bgroup Le~Calvez\egroup\egroup{} :
\newblock Existence d'orbites quasi-p\'eriodiques dans les attracteurs de
  {B}irkhoff.
\newblock {\em Comm. Math. Phys.}, 106(3)\string:\penalty500\relax 383--394,
  1986.

\bibitem[LC87]{LC87}
Patrice \bgroup\fonteauteurs\bgroup Le~Calvez\egroup\egroup{} :
\newblock Propri\'et\'es dynamiques des r\'egions d'instabilit\'e.
\newblock {\em Ann. Sci. \'Ecole Norm. Sup. (4)},
  20(3)\string:\penalty500\relax 443--464, 1987.

\bibitem[LC91]{LC}
Patrice \bgroup\fonteauteurs\bgroup Le~Calvez\egroup\egroup{} :
\newblock Propri\'et\'es dynamiques des diff\'eomorphismes de l'anneau et du
  tore.
\newblock {\em Ast\'erisque}, (204)\string:\penalty500\relax 131, 1991.

\bibitem[LC07]{LC07}
Patrice \bgroup\fonteauteurs\bgroup Le~Calvez\egroup\egroup{} :
\newblock Drift orbits for families of twist maps of the annulus.
\newblock {\em Ergodic Theory Dynam. Systems}, 27(3)\string:\penalty500\relax
  869--879, 2007.

\bibitem[LM]{LM}
Laurent \bgroup\fonteauteurs\bgroup Lazzarini\egroup\egroup{} et Jean-Pierre
  \bgroup\fonteauteurs\bgroup Marco\egroup\egroup{} :
\newblock More hyperbolic kam tori in nearly-integrable systems.
\newblock {\em  <hal-01172729>}, \string:\penalty500\relax 2015.

\bibitem[LMS]{LMS}
Pierre \bgroup\fonteauteurs\bgroup Lochak\egroup\egroup{}, 
Jean-Pierre \bgroup\fonteauteurs\bgroup Marco\egroup\egroup{}, 
David \bgroup\fonteauteurs\bgroup Sauzin\egroup\egroup{}:
\newblock Measure and capacity of wandering domains in Gevrey near-integrable exact symplectic systems.
\newblock {\em In preparation}


\bibitem[Mar]{M}
Jean-Pierre \bgroup\fonteauteurs\bgroup Marco\egroup\egroup{} :
\newblock Arnold diffusion for cusp-generic nearly integrable systems on $A^3$.
\newblock {\em Preprint on ArXiv}

\bibitem[Mar98]{Mar98}
Jean-Pierre \bgroup\fonteauteurs\bgroup Marco\egroup\egroup{} :
\newblock Dynamics in the vicinity of double resonances.
\newblock \emph{In} {\em Proceedings of the {IV} {C}atalan {D}ays of {A}pplied
  {M}athematics ({T}arragona, 1998)}, pages 123--137. Univ. Rovira Virgili,
  Tarragona, 1998.

\bibitem[Mar08]{Mar08}
Jean-Pierre \bgroup\fonteauteurs\bgroup Marco\egroup\egroup{} :
\newblock Mod\`eles pour les applications fibr\'ees et les polysyst\`emes.
\newblock {\em C. R. Math. Acad. Sci. Paris}, 346(3-4)\string:\penalty500\relax
  203--208, 2008.

\bibitem[Mat91]{Mat91}
John~N. \bgroup\fonteauteurs\bgroup Mather\egroup\egroup{} :
\newblock Action minimizing invariant measures for positive definite
  {L}agrangian systems.
\newblock {\em Math. Z.}, 207(2)\string:\penalty500\relax 169--207, 1991.

\bibitem[Mat03]{Mat04}
John \bgroup\fonteauteurs\bgroup Mather\egroup\egroup{} :
\newblock Arnol\cprime d diffusion. {I}. {A}nnouncement of results.
\newblock {\em Sovrem. Mat. Fundam. Napravl.}, 2\string:\penalty500\relax
  116--130, 2003.

\bibitem[Mat10]{Mat10}
John~N. \bgroup\fonteauteurs\bgroup Mather\egroup\egroup{} :
\newblock Order structure on action minimizing orbits.
\newblock \emph{In} {\em Symplectic topology and measure preserving dynamical
  systems}, volume 512 de {\em Contemp. Math.}, pages 41--125. Amer. Math.
  Soc., Providence, RI, 2010.

\bibitem[Moe02]{M02}
Richard \bgroup\fonteauteurs\bgroup Moeckel\egroup\egroup{} :
\newblock Generic drift on {C}antor sets of annuli.
\newblock \emph{In} {\em Celestial mechanics ({E}vanston, {IL}, 1999)}, volume
  292 de {\em Contemp. Math.}, pages 163--171. Amer. Math. Soc., Providence,
  RI, 2002.

\bibitem[Mos01]{Mos}
J{\"u}rgen \bgroup\fonteauteurs\bgroup Moser\egroup\egroup{} :
\newblock {\em Stable and random motions in dynamical systems}.
\newblock Princeton Landmarks in Mathematics. Princeton University Press,
  Princeton, NJ, 2001.
\newblock With special emphasis on celestial mechanics, Reprint of the 1973
  original, With a foreword by Philip J. Holmes.

\bibitem[MS]{MS13}
Jean-Pierre \bgroup\fonteauteurs\bgroup Marco\egroup\egroup{} et Lara
  \bgroup\fonteauteurs\bgroup Sabbagh\egroup\egroup{} :
\newblock Examples of nearly integrable systems on $\A^3$ with asymptotically
  dense projected orbits.

\bibitem[Neh79]{N77}
N.~N. \bgroup\fonteauteurs\bgroup Nehoro{\v{s}}ev\egroup\egroup{} :
\newblock An exponential estimate of the time of stability of nearly integrable
  {H}amiltonian systems. {II}.
\newblock {\em Trudy Sem. Petrovsk.}, (5)\string:\penalty500\relax 5--50, 1979.

\bibitem[NP12]{NP12}
Meysam \bgroup\fonteauteurs\bgroup Nassiri\egroup\egroup{} et Enrique~R.
  \bgroup\fonteauteurs\bgroup Pujals\egroup\egroup{} :
\newblock Robust transitivity in {H}amiltonian dynamics.
\newblock {\em Ann. Sci. \'Ec. Norm. Sup\'er. (4)},
  45(2)\string:\penalty500\relax 191--239, 2012.

\bibitem[Oli08]{O08}
Elismar~R. \bgroup\fonteauteurs\bgroup Oliveira\egroup\egroup{} :
\newblock Generic properties of {L}agrangians on surfaces: the {K}upka-{S}male
  theorem.
\newblock {\em Discrete Contin. Dyn. Syst.}, 21(2)\string:\penalty500\relax
  551--569, 2008.

\bibitem[Pal68]{P69}
J.~\bgroup\fonteauteurs\bgroup Palis\egroup\egroup{} :
\newblock On {M}orse-{S}male dynamical systems.
\newblock {\em Topology}, 8\string:\penalty500\relax 385--404, 1968.

\bibitem[Poi87]{P}
H.~\bgroup\fonteauteurs\bgroup Poincar{\'e}\egroup\egroup{} :
\newblock {\em Les m\'ethodes nouvelles de la m\'ecanique c\'eleste. {T}ome
  {III}}.
\newblock Les Grands Classiques Gauthier-Villars. Librairie Scientifique et Technique Albert Blanchard, Paris, 1987.

\bibitem[P{\"o}s93]{Po93}
J{\"u}rgen \bgroup\fonteauteurs\bgroup P{\"o}schel\egroup\egroup{} :
\newblock Nekhoroshev estimates for quasi-convex {H}amiltonian systems.
\newblock {\em Math. Z.}, 213(2)\string:\penalty500\relax 187--216, 1993.

\bibitem[Rud91]{Ru}
Walter \bgroup\fonteauteurs\bgroup Rudin\egroup\egroup{} :
\newblock {\em Functional analysis}.
\newblock International Series in Pure and Applied Mathematics. McGraw-Hill,
  Inc., New York, second \'edition, 1991.

\bibitem[Sab15]{S13}
Lara \bgroup\fonteauteurs\bgroup Sabbagh\egroup\egroup{} :
\newblock An inclination lemma for normally hyperbolic manifolds with an
  application to diffusion.
\newblock {\em Ergodic Theory Dynam. Systems}, 35(7)\string:\penalty500\relax
  2269--2291, 2015.

\bibitem[Sor]{So}
Alfonso \bgroup\fonteauteurs\bgroup Sorrentino\egroup\egroup{}.
\newblock Lecture notes on Mather's theory for Lagrangian systems.
\newblock ArXiv : 1011.0590

\bibitem[Tak83]{Ta83}
Floris \bgroup\fonteauteurs\bgroup Takens\egroup\egroup{} :
\newblock Mechanical and gradient systems; local perturbations and generic
  properties.
\newblock {\em Bol. Soc. Brasil. Mat.}, 14(2)\string:\penalty500\relax
  147--162, 1983.

\bibitem[Tre04]{T04}
D.~\bgroup\fonteauteurs\bgroup Treschev\egroup\egroup{} :
\newblock Evolution of slow variables in a priori unstable {H}amiltonian
  systems.
\newblock {\em Nonlinearity}, 17(5)\string:\penalty500\relax 1803--1841, 2004.

\bibitem[TS89]{ST89}
D.~V. \bgroup\fonteauteurs\bgroup Turaev\egroup\egroup{} et L.~P.
  \bgroup\fonteauteurs\bgroup Shil{\cprime}nikov\egroup\egroup{} :
\newblock Hamiltonian systems with homoclinic saddle curves.
\newblock {\em Dokl. Akad. Nauk SSSR}, 304(4)\string:\penalty500\relax
  811--814, 1989.

\bibitem[Yan]{Y}
Dennis~Guang \bgroup\fonteauteurs\bgroup Yang\egroup\egroup{} :
\newblock An invariant manifold theory for odes and its applications.
\newblock ArXiv : 0909.1103

\bibitem[Zeh76]{Ze76}
Eduard \bgroup\fonteauteurs\bgroup Zehnder\egroup\egroup{} :
\newblock Moser's implicit function theorem in the framework of analytic
  smoothing.
\newblock {\em Math. Ann.}, 219(2)\string:\penalty500\relax 105--121, 1976.

\end{thebibliography}
\end{document}